\documentclass{amsart}[14pt,a4paper]

\usepackage{geometry}
\usepackage{amsfonts}
\usepackage{hyperref}
\usepackage{amsmath}
\usepackage{galois}
\usepackage{bm}
\DeclareMathSizes{10}{10}{7}{5}
\usepackage{mathtools}
\numberwithin{equation}{section}

\usepackage{enumitem}
\usepackage{amssymb}

\numberwithin{equation}{section}

\usepackage{enumitem}
\usepackage{amssymb}
\usepackage{mathrsfs}

\newtheorem{theorem}{Theorem}[section]
\newtheorem{question}{Question}[section]

\newtheorem{lemma}{Lemma}[section]
\newtheorem{corollary}{Corollary}[section]
\usepackage{chngcntr}
\counterwithin{figure}{section}
\begin{document}
	\title{Arithmetic conditions for Julia sets to have positive area}
	\author{Jianyong Qiao and Hongyu Qu\textsuperscript{*}}
	\address{School of Mathematical Sciences, Beijing University of Posts and Telecommunications, Beijing
		100786, P. R. China. \textit{Email:\ qjy@bupt.edu.cn}}
	\address{Key Laboratory of Mathematics and Information Networks (Beijing University of Posts and Telecommunications), Ministry of Education, People's Republic of China}
	\address{School of Mathematical Sciences, Beijing University of Posts and Telecommunications, Beijing
		100786, P. R. China. \textit{Email:} \textit{hongyuqu2022@126.com}}
	\address{Key Laboratory of Mathematics and Information Networks (Beijing University of Posts and Telecommunications), Ministry of Education, People's Republic of China}
	\renewcommand{\thefootnote}{\fnsymbol{footnote}}
	\footnotetext[1]{Corresponding author, Email: hongyuqu2022@126.com}
	\maketitle
	\begin{abstract}
		In this paper, for quadratic polynomials with Cremer points or Siegel disks whose boundaries don't contain critical points,
		we give an arithmetic condition for their Julia sets to have positive area. Our proof relies on the Siegel Return Machinery as well as the pseudo-Siegel disk theory developed by Dudko and Lyubich.
	\end{abstract}
	\tableofcontents
	
	\section{Introduction}
	The quadratic polynomial $$P_{\alpha}(z)=e^{2\pi i\alpha}z+z^2,$$
	where $\alpha$ is an irrational number with the continued fraction expansion 
	\begin{align*}
		\alpha&=[a_1,a_2,a_3,\cdots]\\
		&=\frac{1}{a_1+\frac{1}{a_2+\frac{1}{a_3+\frac{1}{\ddots}}}},
	\end{align*}
	here $\{a_n\}$ is a sequence of positive integers.
	The origin is a fixed point of $P_{\alpha}$ and $P_{\alpha}'(0)=e^{2\pi i\alpha}$. 
	$0$ is called a Siegel point of $P_{\alpha}$ if $P_{\alpha}$ can be holomorphically conjugate
	to a linear map in a neighborhood of $0$.
	Otherwise, 
	$0$ is called a Cremer point of $P_{\alpha}$.
	In \cite{Yoc95}, Yoccoz proved that $0$ is a Siegel point if and only if $\alpha$ is a Brjuno number (refer also to \cite{Br}, \cite{Che19}, \cite{Sie42}). Here
	an irrational number $\alpha$ is called a Brjuno number if the Brjuno series
	\[\sum_{n=1}^{+\infty}\frac{\log{q_{n+1}}}{q_n}<+\infty,\]
	where $q_n$ is the denominator of the $n$th convergent to $\alpha$.
	This result reveals a profound connection between the arithmetic of $\alpha$ and the dynamics of $P_{\alpha}$.
	Over the past few decades, more wonderful connections between the arithmetic of $\alpha$ and the dynamics of $P_{\alpha}$ were founded, such as
	the Marmi-Moussa-Yoccoz conjecture
	(see \cite{BC06}, \cite{Ca03}, \cite{CC15} and \cite{MMY97}),
	the topology of the irrationally indifferent attractor of $P_{\alpha}$ (see \cite{Che22}, \cite{Dou87}, \cite{Her87}, \cite{SY24}) and so on.
	This paper is focus on exploring an arithmetic condition of $\alpha$ for the Julia set of $P_{\alpha}$ to have positive area. There have been some breakthroughs on the research of the connection between the arithmetic of $\alpha$ and the Lebesgue measure of $J(P_{\alpha})$ as follows:
	When $\alpha$ is of bounded type, McMullen proved that
	the Hausdorff dimension of $J(P_{\alpha})$ is less than $2$ and
	hence the Lebesgue measure of $J(P_{\alpha})$ is equal to $0$ (see \cite{McM98}); when $\alpha$ is of PZ-type\footnote[2]{Here PZ-type means $\ln a_n=O(\sqrt{n})$ as $n\to+\infty$. The set of all irrational numbers of PZ-type has full measure in $\mathbb{R}/\mathbb{Z}$.},
	Petersen and Zakeri proved that the Lebesgue measure of $J(P_{\alpha})$ is equal to $0$ (see \cite{PZ04}); 
	on the other hand, 
	Buff and Ch\'eritat proved that
	positive area of 	$J(P_{\alpha})$ can appear at both Brjuno and non-Brjuno cases (see \cite{BC}; see also \cite{QQ24}, \cite{Q24} and \cite{QQZ23} for improving some key techniques used in their proof).
	
	Although some significant progresses have been made, there still are many secrets about the Lebesgue measure of	$J(P_{\alpha})$ worth studying. Here list three open questions about Lebesgue measure of quadratic Julia sets (see \cite{Che09}):
	
	\begin{question}
		\label{Q1}Does there exist a quadratic polynomial $P_{\alpha}$ with
		a Cremer point, but whose Julia set has Lebesgue measure $=0$?\footnote{This question was proposed for rational functions in \cite{Che09}.}
	\end{question}	
	\begin{question}
		\label{Q2}Is it equivalent for a quadratic polynomial $P_{\alpha}$ with a fixed Siegel disk $\Delta_{\alpha}$ centering at $0$, to have a positive measure Julia set and to have $\partial\Delta_{\alpha}$ not containing the critical point?
	\end{question}
	In \cite{PZ04}, Peterson and Zakeri proved that for almost everywhere $\alpha$, ${\rm area}(J(P_{\alpha}))=0$ and the Siegel disk boundary $\partial\Delta_{\alpha}$ contains the unique finite critical point $c_0=-\frac{e^{2\pi i\alpha}}{2}$.
	
	According to \cite{Che22} and \cite{SY24}, if $\alpha$ is of sufficiently high type (there exists a sufficiently large $N>0$ such that for all $n$, $a_n\geq N$), then the Siegel disk boundary contains the critical point $c_0$ if and only if $\alpha$ is a Herman number.
	In \cite{CDY20}, Cheraghi, DeZotti and Yang proved that if $\alpha$ is an irrational number of sufficiently high type and does not belong to Herman numbers, then the Hausdorff dimension of $J(P_{\alpha})$ is $2$.
	
	\begin{question}
		\label{Q3}Is it possible to characterize the set of irrational numbers $\alpha$ such that
		${\rm area}(J(P_\alpha))>0$?
	\end{question}
	The famous work of Buff and Ch\'eritat shows the existence of $\alpha$ such that ${\rm area}(J(P_\alpha))>0$. 
	However, to our knowledge, so far there is no simple arithmetic condition for $\alpha$ to make ${\rm area}(J(P_\alpha))>0$ been provided explicitly.
	
	\subsection{Statement of main results.}
	
	\vspace{0.2cm}
	We denote by $\mathcal{C}$ the set of all non-Brjuno numbers between $0$ and $1$.
	We denote by $\mathcal{E}$ the set of all irrational number $\alpha\in(0,1)$ such that $P_{\alpha}$ has a Siegel disk centering at $0$ whose boundary doesn't contain the critical point $c_0$.
	
	Given positive integers $\mathfrak{M}$, $\mathfrak{l}$ and $\mathfrak{p}$,
	an irrational number $\alpha=[a_1,a_2,\cdots]$ is said to satisfy the arithmetic condition $Q(\mathfrak{M},\mathfrak{l},\mathfrak{p})$ if there exists a positive integer $j$ such that for all $j_1>j_2\geq j$ with $a_{j_1}>\mathfrak{M}$ and $a_{j_2}>\mathfrak{M}$, we have $j_1-j_2\geq\mathfrak{p}$ or $j_1-j_2\leq\mathfrak{l}-1$. See Picture \ref{f20260721a}.
	\begin{figure}
		\centering
		\includegraphics[scale=0.7]{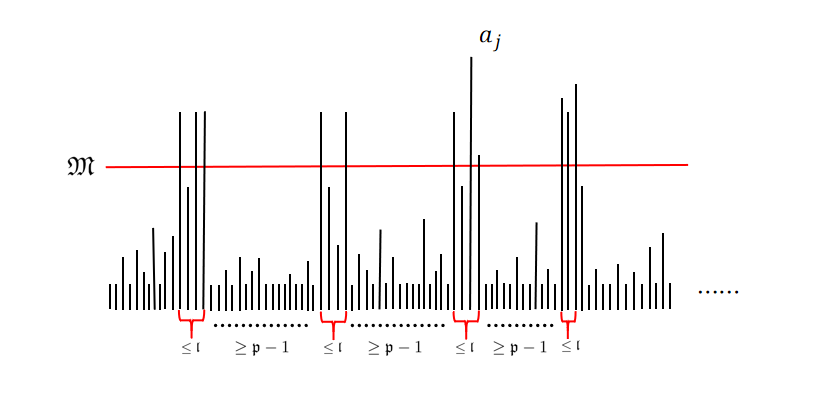}
		\caption{An illustration of the arithmetic condition $Q(\mathfrak{M},\mathfrak{l},\mathfrak{p})$.}
		\label{f20260721a}
	\end{figure}
	We denote by 
	$\mathcal{C}_{\mathfrak{l},\mathfrak{p}}^{(\mathfrak{M})}$ the set consisting of $\alpha=[a_1,a_2,\cdots]\in\mathcal{C}$ satisfying the above arithmetic condition $Q(\mathfrak{M},\mathfrak{l},\mathfrak{p})$; we denote by 
	$\mathcal{E}_{\mathfrak{l},\mathfrak{p}}^{(\mathfrak{M})}$ the set consisting of $\alpha=[a_1,a_2,\cdots]\in\mathcal{E}$ satisfying the above arithmetic condition $Q(\mathfrak{M},\mathfrak{l},\mathfrak{p})$.
	Our main results are the following two theorems:
	
	\begin{theorem}[Cremer case]
		\label{T2}Let $\mathfrak{M}\geq1$ and $\mathfrak{l}\geq1$ be two natural numbers. For every natural number $\mathfrak{p}\geq1$, if $\alpha\in\mathcal{C}_{\mathfrak{l},\mathfrak{p}}^{(\mathfrak{M})}$, then
		$P_{\alpha}$ has a positive area Julia set, with at most finitely many exceptional $\mathfrak{p}$.
	\end{theorem}
	
	\begin{theorem}[Siegel case]
		\label{T3}Let $\mathfrak{M}\geq1$ and $\mathfrak{l}\geq1$ be two natural numbers. For every natural number $\mathfrak{p}\geq1$, if $\alpha\in\mathcal{E}_{\mathfrak{l},\mathfrak{p}}^{(\mathfrak{M})}$, then
		$P_{\alpha}$ has a positive area Julia set, with at most finitely many exceptional $\mathfrak{p}$.
	\end{theorem}
	Theorems \ref{T2} and \ref{T3} give arithmetic conditions for Cremer quadratic Julia sets and Siegel quadratic Julia sets to have positive areas, respectively. 
	
	It is easy to see that $\mathcal{C}_{\mathfrak{l},\mathfrak{p}}^{(\mathfrak{M})}$ in Theorem \ref{T2} is an uncountable set. Moreover, because of the arbitrariness of $j$ in the definition of $Q(\mathfrak{M},\mathfrak{l},\mathfrak{p})$ and the fact that changing finitely many entries in the continued fraction expansion of an irrational number does not affect whether it belongs to $\mathcal{C}$,
	we have that $\mathcal{C}_{\mathfrak{l},\mathfrak{p}}^{(\mathfrak{M})}$ in Theorem \ref{T2} is also dense in $(0,1)$.
	
	Note that $$\cdots\subsetneqq\mathcal{C}_{1,\mathfrak{p}}^{(1)}\subsetneqq\cdots\subsetneqq\mathcal{C}_{1,2}^{(1)}\subsetneqq\mathcal{C}_{1,1}^{(1)}=\mathcal{C}$$
	and $$\cdots\subsetneqq\mathcal{E}_{1,\mathfrak{p}}^{(1)}\subsetneqq\cdots\subsetneqq\mathcal{E}_{1,2}^{(1)}\subsetneqq\mathcal{E}_{1,1}^{(1)}=\mathcal{E}.$$
	Theorems \ref{T2} and \ref{T3} seem to support the following conjecture:
	All quadratic polynomials with a Cremer point or a Siegel disk whose boundary doesn't contain a critical point have a Julia set with positive area.

	\subsection{Strategy of the proof.}
	We shall give an outline of the proof of the main theorem for the special case: the non-Brjuno number $\alpha=[a_1,a_2,\cdots]$ satisfies $a_n=1$ for all $\mathfrak{p}\nmid n$, where $\mathfrak{p}\gg1$. The main idea is also applicable for the other cases with an appropriate revision.
	
	We consider the perturbations of $\alpha$:
	$$\alpha_k=[a_1,a_2,\cdots,a_k,1,1,\cdots],\ k\geq1.$$
	By the upper semi-continuity of areas of filled-in Julia sets (see [Proposition $2$, \cite{BC}]), we have $$\limsup\limits_{k\to+\infty}{\rm area}(K(P_{\alpha_k}))\leq{\rm area}(K(P_{\alpha}))={\rm area}(J(P_{\alpha})).$$
	So in order to prove ${\rm area}(J(P_{\alpha}))>0$, we need only to
	prove
	\begin{equation}
		\label{e20260724a}\limsup\limits_{k\to+\infty}{\rm area}(K(P_{\alpha_k}))>0.
	\end{equation}
	
	By the Douady-Ghys surgery, for all $k\geq1$, the Siegel disk $\Delta_{\alpha_k}$ is a quasidisk and ${\rm area}(\Delta_{\alpha_k})>0$. However, as $k\to\infty$, the shape of $\Delta_{\alpha_k}$ will develop deep fjords towards $0$ on many different renormalization levels, and at last the interior completely disappears. 
	For all integer $n\geq-1$, the pseudo-Siegel disk $\hat{\Delta}_{\alpha_k}^n$ of level $n$ constructed by Dudko and Lyubich, obtained from $\Delta_{\alpha_k}$ by adding all deep fjords of level $\geq n$ by an appropriate way, has controllable dynamics and bounded geometry on scale $\geq n$ (see (P1) and (P2) in Section \ref{s2.3}).
	The remaining proof is divided into two steps.
	
	In the first step, we apply the Siegel Return Machinery to the pseudo-Siegel disk: similar to that in [Avila and Lyubich, \cite{AL22}], at this setting we also construct ``safe trapping disks'' (see $\tilde{\Omega}_{c_0}^n\setminus\Lambda_\alpha^{n+\mathfrak{q}_3}$) at various levels, and then apply the Siegel Return Machinery to ensure so many times (depending on $\mathfrak{p}$) of returns back to these ``safe trapping disks'' that the probability of landing at a deeper level pseudo-Siegel disk is as high as possible (see Lemma \ref{key1}). 
	
	In the second step, based on the result in the first step, we estimate the density of $K(P_{\alpha_k})$ from deep to shallow fjords (see Theorem \ref{key5}), and thus obtain that for sufficiently large $n$, the density of $K(P_{\alpha_k})$ in $\hat{\Delta}_{\alpha_k}^n$ (in fact, we only do it for some quasidisk with a fixed dilatation contained in $\hat{\Delta}_{\alpha_k}^n$ here) has a positive low bound, independent of $k$, which implies (\ref{e20260724a}).

	\subsection{The other case: infinitely renormalizable case}
	For quadratic polynomials, except in the irrational indifferent case, Julia sets can also occur positive area in the infinitely  renormalizable case. 
	In \cite{BC} 
	Buff and Ch\'eritat proved that there exist infinitely renormalizable quadratic polynomials with unbounded satellite combinatorics having Julia sets of positive area. In \cite{AL22}
	Avila and Lyubich proved that there exists a Feigenbaum quadratic polynomial with
	primitive combinatorics whose Julia set has positive area. In \cite{DL23}
	Dudko and Lyubich proved that there exist quadratic polynomials with bounded
	satellite combinatorics whose Julia sets have positive area.
	
	\vspace{0.2cm}
	\noindent{\bf Acknowledgements} The research work was supported by
	the National Natural Science Foundation of China (12301102, 12471084 and 12171264) and the Fundamental Research Funds for the Central Universities.
	
	\section{The conformal geometry and pseudo-Siegel disks\label{S2}}
	\subsection{Notations}
	We list some notations used later:
	\begin{itemize}
		\item For all nonzero $a,b,c$, $a\preceq_cb$ (or $b\succeq_ca$) means that $\frac{b}{a}$ has a positive low bound depending only on $c$; $a\asymp_cb$ means that $a\preceq_cb$ and $a\succeq_cb$.
		\item For any $a,b,c,c',d>0$, $a\ll_cb$ (or $b\gg_ca$) means that $\frac{b}{a}$ has a positive low bound much larger than $1$ depending only on $c$; $b=O_c(a)$ means that $b\preceq_ca$. Furthermore, $\gg_c$ and $b=O_c(a)$ satisfy that if $a\gg_cb$ and $b\succeq_cd$, then
		$a\gg_cd$; if $a\gg_cb$ and $c>c'>0$, then $a\gg_{c'}b$; if $b=O_c(a)$ and $c'>c>0$, then $b=O_{c'}(a)$.
		\item $\preceq$, $\succeq$, $\ll$ and $\gg$ are based on a  universal constant $c$ ($>0$), in particular, independent of the rotation number $\alpha$.
	\end{itemize}
	Let $D$ be a closed quasidisk in $\mathbb{C}$.
	\begin{itemize}
		\item For any $n$ ($\geq3$) different points $a_1,a_2,\cdots,a_n$ on $\partial D$, $a_1<a_2<\cdots<a_n$ means that $a_1,a_2,\cdots,a_n$ are clockwise on $\partial D$. 
		\item For any $n$ ($\geq3$) intervals $I_1,I_2,\cdots,I_n$ with pairwise disjoint interiors on $\partial D$, $I_1<I_2<\cdots<I_n$ means that $I_1,I_2,\cdots,I_n$ are clockwise on $\partial D$.
		\item ${\rm int}(D)$ means the interior of $D$.
		\item For any Jordan arc $I\subseteq\mathbb{C}$, the notation ${\rm int}(I)$, the interior of $I$, means the difference between $I$ and its two endpoints.
	\end{itemize}
	\begin{itemize}
		\item We denote by $c_{0,\alpha}$ the unique finite critical point $-\frac{e^{2\pi i\alpha}}{2}$ of $P_{\alpha}(z)=e^{2\pi i\alpha}z+z^2$, 
		and $c_{0,\alpha}$ is usually simply written as $c_0$ if no confusion.
	\end{itemize}
	\begin{itemize}
		\item We denote by $\rho_{U}(z)$, $z\in U\cap\mathbb{C}$ the metric density of the hyperbolic domain $U$ ($\subseteq\hat{\mathbb{C}}$) under the standard coordinate of $U\cap\mathbb{C}$.
		\item For all hyperbolic Riemann surface $U$, any nonempty subset $E_1\subseteq U$ and any curve $\gamma_1\subseteq U$, we denote by ${\rm diam}_{U}(E_1)$ the hyperbolic diameter of $E_1$ on $U$ and by $l_{U}(\gamma_1)$ the hyperbolic length of $\gamma_1$ on $U$. Moreover, for any subset $E_2\subseteq\mathbb{C}$ and any curve $\gamma_2\subseteq\mathbb{C}$, we denote by ${\rm diam}(E_2)$ the Euclidean diameter of $E_2$ and by $l(\gamma_1)$ the Euclidean length of $\gamma_1$ on $U$.
	\end{itemize}

	\subsection{The extremal width of a family of curves.}
	\quad
	
	\vspace{0.2cm}
	\noindent2.2.1. {\bf A family of curves.}
	Given a family of curves $\mathcal{F}$ in $\hat{\mathbb{C}}$, we denote by $\mathcal{W}(\mathcal{F})$ the extremal width of $\mathcal{F}$, that is the reciprocal of the extremal length of $\mathcal{F}$. The extremal width satisfies the following laws:
	\begin{itemize}
		\item If $\mathcal{F}_1$ and $\mathcal{F}_2$ are two families of curves such that $\mathcal{F}_2\subseteq\mathcal{F}_1$, then $\mathcal{W}(\mathcal{F}_2)\leq\mathcal{W}(\mathcal{F}_1)$. 
		\item If $\mathcal{F}_j$, $1\leq j\leq n$ are $n$ families of curves, then $\mathcal{W}(\bigcup_{j=1}^n\mathcal{F}_j)\leq\sum_{j=1}^n\mathcal{W}(\mathcal{F}_j)$; Moreover, if $\mathcal{F}_j$, $1\leq j\leq n$ are contained
		in $n$ pairwise disjoint measurable sets, respectively, then $\mathcal{W}(\bigcup_{j=1}^n\mathcal{F}_j)=\sum_{j=1}^n\mathcal{W}(\mathcal{F}_j)$.
	\end{itemize}
	A curve is said to overflow $\mathcal{F}$ if this curve contains a subcurve which belongs to $\mathcal{F}$;
	a family of curves $\mathcal{G}$ is said to overflow $\mathcal{F}$ if every curve in $\mathcal{G}$ overflows $\mathcal{F}$.
	\begin{itemize}
		\item If $\mathcal{G}$ overflows $\mathcal{F}$, then the extremal widths have the following relation:
		$$\mathcal{W}(\mathcal{G})\leq\mathcal{W}(\mathcal{F}).$$
	\end{itemize}
	
	\vspace{0.2cm}
	\noindent2.2.2. {\bf A rectangle.}
	By a (topological) rectangle in $\hat{\mathbb{C}}$ we mean a closed Jordan disk $\mathcal{R}$ together with a conformal map $h:\mathcal{R}\to E_x$, where
	$E_x:=[0,x]\times[0,1]\subseteq\mathbb{C}$. In this case, we define
	\begin{itemize}
		\item the base $\partial^{h,0}\mathcal{R}$ is $h^{-1}([0,x]\times\{0\})$;
		\item the roof $\partial^{h,1}\mathcal{R}$ is $h^{-1}([0,x]\times\{1\})$; 
		\item every $h^{-1}(\{t\}\times[0,1])$, $0\leq t\leq x$ is a vertical curve of $\mathcal{R}$, in particular, $h^{-1}(\{\frac{x}{2}\}\times[0,1])$ is called the center arc of $\mathcal{R}$, the left side $\partial^{v,0}\mathcal{R}$ is $h^{-1}(\{0\}\times[0,1])$, the right side $\partial^{v,1}\mathcal{R}$ is $h^{-1}(\{1\}\times[0,1])$;
		\item the vertical family $\mathcal{F}^{v}(\mathcal{R})$ is the family of all vertical curves of $\mathcal{R}$; the full family $\mathcal{F}^{full}(\mathcal{R})$ is the family of curves connecting the base and the roof on $\mathcal{R}$;
		\item the extremal width of $\mathcal{R}$ is $\mathcal{W}(\mathcal{R}):=x$.
		It is well known that $$\mathcal{W}(\mathcal{R})=\mathcal{W}(\mathcal{F}^{full}(\mathcal{R}))=\mathcal{W}(\mathcal{F}^{v}(\mathcal{R})).$$
	\end{itemize}
	A family of curves $\mathcal{F}$ or a rectangle $\mathcal{R}$ is said to 
	\begin{itemize}
		\item overflow a rectangle $\mathcal{R}'$ if every curve in $\mathcal{F}$ or every vertical curve of $\mathcal{R}$ overflows the full family $\mathcal{F}^{full}(\mathcal{R}')$;
		\item cross a rectangle $\mathcal{R}'$ if every curve in $\mathcal{F}$ or every vertical curve of $\mathcal{R}$ overflows the dual rectangle $(\mathcal{R}')^*$ of $\mathcal{R}'$, where the dual rectangle $(\mathcal{R}')^*$ means that $(\mathcal{R}')^*=\mathcal{R}'$ and the base, roof of $(\mathcal{R}')^*$ are exactly the left, right sides of $\mathcal{R}'$.
	\end{itemize}
	If a family of curves $\mathcal{F}$ (resp. a rectangle $\mathcal{R}$) crosses a rectangle $\mathcal{R}'$, then
	$$\mathcal{W}(\mathcal{F})\leq1/\mathcal{W}(\mathcal{R}')\ {\rm(resp.}\ \mathcal{W}(\mathcal{R})\leq1/\mathcal{W}(\mathcal{R}')).$$
	
	\vspace{0.2cm}
	\noindent2.2.3. {\bf An annulus.} Let $\mathcal{A}$ be an annulus on $\hat{\mathbb{C}}$. Then there exists a conformal map
	$$\phi_{\mathcal{A}}:\mathcal{A}\to \mathbb{D}_r:=\{z:r<|z|<1\}$$
	for some $0<r<1$. The conformal module of $\mathcal{A}$ is defined by ${\rm mod}(\mathcal{A}):=\frac{1}{2\pi}\log\frac{1}{r}$. We denote by $\mathcal{F}^v(\mathcal{A})$ the family of curves connecting two components of $\partial\mathcal{A}$ on $\overline{\mathcal{A}}$ and by $\mathcal{F}^h(\mathcal{A})$ the family of Jordan curves not homotopic to a single point in $\mathcal{A}$. The following properties hold:
	\begin{itemize}
	\item ${\rm mod}(\mathcal{A})=\mathcal{W}(\mathcal{F}^h(\mathcal{A}))=1/\mathcal{W}(\mathcal{F}^v(\mathcal{A}))$.
	\end{itemize}
	Let $\mathcal{A}_j$, $1\leq j\leq n$ be $n$ pairwise disjoint annuluses contained in $\mathcal{A}$ such that each $\mathcal{A}_j$ separates two components of $\partial\mathcal{A}$.
	\begin{itemize}
		\item ${\rm mod}(\mathcal{A})\geq\sum_{j=1}^n{\rm mod}(\mathcal{A}_j)$.
	\end{itemize}

	\subsection{The geometry of nests of tilings.} 
	\quad
	
	\vspace{0.2cm}
	\noindent 2.3.1. {\bf A nest of tilings.}
	
	\vspace{0.2cm}
	\noindent
	The notations and definitions used here refer to [Section $11.1$, \cite{DL}]. Consider a closed quasidisk $D\subseteq\mathbb{C}$.
	Let $\mathcal{T}=\{\mathcal{T}_n\}_{n\geq -1}$ be a system of finite partitions of $\partial D$ into finitely many closed intervals such that $\mathcal{T}_{n+1}$ is a refinement of $\mathcal{T}_n$. $\mathcal{T}$ is called a nest of tilings of $D$ if
	\begin{itemize}
		\item the maximal diameter of intervals in $\mathcal{T}_n$ tends to $0$ as $n\to\infty$, and
		\item every interval in $\mathcal{T}_n$ for $n\geq -1$ decomposes into at least two intervals of $\mathcal{T}_{n+2}$.
	\end{itemize}
	Intervals in $\mathcal{T}_n$ are also called intervals of level $n$ of $\mathcal{T}$. 
	For any positive integer $M\geq2$, a nest of tilings $\mathcal{T}$ is said to have $M$-bounded combinatorics if $\mathcal{T}$ satisfies
	\begin{itemize}
		\item[(${\rm a}1$)] $\mathcal{T}_{-1}$ has at most $M$ intervals;
		\item[(${\rm b}1$)] for all $n\geq-1$, every interval in $\mathcal{T}_n$ decomposes into at most $M$ intervals in $\mathcal{T}_{n+1}$.
	\end{itemize}
	If $\mathcal{T}$ is only required to satisfy (b1), then we call that $\mathcal{T}$ has post-$M$-bounded combinatorics.
	
	For any two disjoint intervals $I,J\subseteq\partial D$, we denote by
	\begin{itemize}
		\item $\mathcal{F}(I,J)$ the family of curves in $\hat{\mathbb{C}}$ connecting $I$ and $J$:
		$$\mathcal{F}(I,J):=\{\gamma:[0,1]\to\hat{\mathbb{C}}:\gamma(0)\in I,\ \gamma(1)\in J\};$$
		\item $\mathcal{F}_D^-(I,J)$ the family of curves in $D$ connecting $I$ and $J$:
		$$\mathcal{F}_D^-(I,J):=\{\gamma:[0,1]\to D:\gamma(0)\in I,\ \gamma(1)\in J\};$$
		\item $\mathcal{F}_D^+(I,J)$ the family of curves in $\hat{\mathbb{C}}\setminus{\rm int}(D)$ connecting $I$ and $J$:
		$$\mathcal{F}_D^+(I,J):=\{\gamma:[0,1]\to\hat{\mathbb{C}}\setminus{\rm int}(D):\gamma(0)\in I,\ \gamma(1)\in J\};$$
		\item $\mathcal{W}(I,J)$ the extremal width of $\mathcal{F}(I,J)$;
		\item $\mathcal{W}_D^-(I,J)$ the extremal width of $\mathcal{F}_D^-(I,J)$;
		\item $\mathcal{W}_D^+(I,J)$ the extremal width of $\mathcal{F}_D^+(I,J)$.
	\end{itemize}
	
	For any interval $I\in\mathcal{T}_n$, let $I_l, I_r\in\mathcal{T}_n$ be two its neighboring intervals. We denote by $[3I]^c$ the closure of $\partial D\setminus(I_l\cup I\cup I_r)$, that is
	$$[3I]^c=\overline{\partial D\setminus(I_l\cup I\cup I_r)},$$
	and define
	\begin{itemize}
		\item $\mathcal{F}^-_{3,\mathcal{T}}(I)$ to be the family of curves in $D$ connecting $I$ and $[3I]^c$;
		\item $\mathcal{F}^+_{3,\mathcal{T}}(I)$ to be the family of curves in $\hat{\mathbb{C}}\setminus{\rm int}(D)$ connecting $I$ and $[3I]^c$;
		\item $\mathcal{F}_{3,\mathcal{T}}(I)$ to be the family of curves in $\hat{\mathbb{C}}$ connecting $I$ and $[3I]^c$;
		\item $\mathcal{W}^{\pm}_{3,\mathcal{T}}(I)=\mathcal{W}(\mathcal{F}^{\pm}_{3,\mathcal{T}}(I))$; $\mathcal{W}_{3,\mathcal{T}}(I)=\mathcal{W}(\mathcal{F}_{3,\mathcal{T}}(I))$.
	\end{itemize}
	
	A nest of tilings $\mathcal{T}$ is said to have essentially bounded outer geometry if there exists a constant $C$ such that for every interval $I$ of $\mathcal{T}$ we have $\mathcal{W}^+_{3,\mathcal{T}}(I)\leq C$. In this case, one also says that $\mathcal{T}$ has essentially $C$-bounded outer geometry. If moreover, $\mathcal{T}$ has (post-)$M$-bounded combinatorics, then one says that $\mathcal{T}$ has (post-) bounded outer geometry or (post-) $(C,M)$-bounded outer geometry. Similarly, essentially bounded inner geometry, (post-) bounded inner geometry, essentially bounded geometry, (post-) bounded geometry are defined.
	
	\vspace{0.2cm}
	\noindent 2.3.2. {\bf The length associated to a nest of tilings.}
	
	\vspace{0.2cm}
	\noindent
	Let $\mathcal{T}$ be a nest of tilings of the closed quasidisk $D$. Assume that $\mathcal{T}$ has $M$-bounded combinatorics. For any interval $I$ on $\partial D$, we denote by
	$|I|_{\mathcal{T}}$ the supremum of $\sum\limits_{j=1}^k\frac{1}{M^{n_j+1}}$, where $n_1,n_2,\cdots,n_k$ are levels of any finitely many pairwise interior-disjoint intervals $I_1, I_2, \cdots, I_k$ of $\mathcal{T}$ such that $I_1<I_2<\cdots<I_k$ and $\bigcup\limits_{j=1}^kI_j\subseteq I$. It is easy to check the following properties:
	\begin{itemize}
		\item For any two intervals $I,J\subseteq\partial D$, if $I\subseteq J$, then $|I|_{\mathcal{T}}\leq|J|_{\mathcal{T}}$, and moreover, the equality holds if and only if $I=J$.
		\item For any interval $I$ of $\mathcal{T}$ with level $n\geq-1$, $|I|_{\mathcal{T}}=1$ if $n=-1$; $|I|_{\mathcal{T}}=\frac{1}{M^{n+1}}\ {\rm or}\ \frac{1}{M^{n}}$ if $n\geq0$, for $I$ may be also an interval with level $n-1$.
		\item For any interval $I\subseteq\partial D$ and nonnegative integer $m\geq-1$, if $|I|_{\mathcal{T}}\geq\frac{1}{M^{m}}$, then $I$ contains at least one interval of level $m$.
	\end{itemize}
	
	For any intervals $I_1, I_2, \cdots, I_k$ on $\partial D$ with pairwise disjoint interiors,  
	we write $I_1\#I_2\# \cdots\#I_k$
	if $I_1< I_2<\cdots<I_k$ and for all $j\in\{1,2,\cdots,k-1\}$, $I_j$ and $I_{j+1}$ are adjacent.
	
	\begin{lemma}
		\label{l811}Assume that a nest of tilings $\mathcal{T}$ of the closed quasidisk $D$ has post-$(C,M)$-bounded inner and outer geometries.
		Let $I_1$, $I_2$ and $I_3$ be three intervals on $\partial D$ such that
		$$I_1\#I_2\#I_3,\ \min\{|I_1|_{\mathcal{T}},|I_3|_{\mathcal{T}}\}\succeq|I_2|_{\mathcal{T}}\ {\rm and}\ |I_2|_{\mathcal{T}}<1.$$
		Then $$\mathcal{W}_D^+(I_2,\overline{\partial D\setminus(I_1\cup I_2\cup I_3)})\preceq_{C,M}1\ {\rm and}\ \mathcal{W}_D^-(I_2,\overline{\partial D\setminus(I_1\cup I_2\cup I_3)})\preceq_{C,M}1.$$
	\end{lemma}
	\begin{proof}
		Since $|I_2|_{\mathcal{T}}<1$, $\min\{|I_1|_{\mathcal{T}},|I_3|_{\mathcal{T}}\}\succeq|I_2|_{\mathcal{T}}$ and $\mathcal{T}$ has post-$M$-bounded combinatorics, we can choose $t$ ($=O_M(1)$) pairwise interior-disjoint intervals
		$I_2^{(1)}, I_2^{(2)},\cdots, I_2^{(t)}$ of level $m$ ($\geq2$)
		with
		$$I_2^{(1)}\#I_2^{(2)}\#\cdots\#I_2^{(t)}$$ such that
		\begin{equation}
			\label{e20260720a}I_2\subseteq I_2^{(1)}\cup I_2^{(2)}\cup\cdots\cup I_2^{(t)},
		\end{equation}
		\begin{equation}
			\label{e20260720b}I_2\not\subseteq I_2^{(1)}\cup I_2^{(2)}\cup\cdots\cup I_2^{(t-1)},
		\end{equation}
		\begin{equation}
			\label{e20260720c}I_2\not\subseteq I_2^{(2)}\cup I_2^{(3)}\cup\cdots\cup I_2^{(t)},
		\end{equation}
		$$\frac{1}{M^{m}}\leq\min\{|I_1|_{\mathcal{T}},|I_2|_{\mathcal{T}},|I_3|_{\mathcal{T}}\}$$
		and $$\frac{1}{M^{m}}\asymp_M|I_2|_{\mathcal{T}}.$$ 
		Since $\frac{1}{M^{m}}\leq\min\{|I_1|_{\mathcal{T}},|I_2|_{\mathcal{T}},|I_3|_{\mathcal{T}}\}$, we have that $I_j$ contains at least one interval of level $m$ for $j=1,2,3$.
		By (\ref{e20260720a}), (\ref{e20260720b}) and (\ref{e20260720c}),
		we have that both $I_2\cap I_2^{(1)}$ and $I_2\cap I_2^{(t)}$ are non-degenerate intervals. It follows that both $I_1$ and $I_3$ contain an interval of level $m$ different from $I_2^{(s)}$, $1\leq s\leq t$.
		Since $I_1\#I_2\#I_3$,
		we have that $\mathcal{F}_D^+(I_2,\overline{\partial D\setminus(I_1\cup I_2\cup I_3)})\subseteq\bigcup\limits_{s=1}^t\mathcal{F}_{3,\mathcal{T}}^+(I_2^{(s)})$ and $\mathcal{F}_D^-(I_2,\overline{\partial D\setminus(I_1\cup I_2\cup I_3)})\subseteq\bigcup\limits_{s=1}^t\mathcal{F}_{3,\mathcal{T}}^-(I_2^{(s)})$.
		Observe also that $\mathcal{T}$ has essentially $C$-bounded inner and outer geometries. Thus $$\mathcal{W}(\mathcal{F}_D^+(I_2,\overline{\partial D\setminus(I_1\cup I_2\cup I_3)}))\leq\sum_{s=1}^t\mathcal{W}_{3,\mathcal{T}}^+(I_2^{(s)})\preceq_{C,M}1$$
		and
		$$\mathcal{W}(\mathcal{F}_D^-(I_2,\overline{\partial D\setminus(I_1\cup I_2\cup I_3)}))\leq\sum_{s=1}^t\mathcal{W}_{3,\mathcal{T}}^-(I_2^{(s)})\preceq_{C,M}1.$$
	\end{proof}
	The following lemma is stated for a nest of tilings with post-bounded inner and outer geometries, see [Lemma $4$, \cite{Q24}]. One can see the original corresponding result for pseudo-Siegel disks in \cite{DL}, see [Lemma\ 11.4\ (II),\ \cite{DL}] for details.
	\begin{lemma}
		\label{l12}Assume that a nest of tilings $\mathcal{T}$ of the closed quasidisk $D$ has post-bounded inner and outer geometries. Then $\mathcal{T}$ has essentially bounded geometry.
	\end{lemma}
	
	\begin{corollary}
		\label{c8201}Assume that a nest of tilings $\mathcal{T}$ of the closed quasidisk $D$ has post-{\rm(}$C,M${\rm)}-bounded inner and outer geometries.
		Let $I_1$, $I_2$ and $I_3$ be three intervals on $D$ such that $I_1\#I_2\#I_3$, $|I_2|_{\mathcal{T}}\preceq\min\{|I_1|_{\mathcal{T}},|I_3|_{\mathcal{T}}\}$ and $|I_2|_{\mathcal{T}}<1$. Then $\mathcal{W}(\mathcal{F}(I_2,\overline{\partial D\setminus(I_1\cup I_2\cup I_3)}))\preceq_{C,M}1$.
	\end{corollary}
	
	\begin{proof}
		By Lemma \ref{l12}, $\mathcal{T}$ has essentially bounded geometry. Since $\mathcal{T}$ has post-$M$-bounded combinatorics, 
		$\mathcal{T}$ has post-bounded geometry.
		Replacing post-bounded inner and outer geometries by post-bounded geometry, in the same way as proving Lemma \ref{l811}, we can obtain the corollary.
		
	\end{proof}


	\subsection{Pseudo-Siegel disks.\label{s2.3}}
	Assume that $\alpha=[a_1,a_2,\cdots]$ is an eventually golden-mean rotation number, that is, there exists a positive integer $j_0$ such that
	$a_j=1$ for all $j\geq j_0$ (Note that by the notation $\alpha$ we denote an eventually golden-mean rotation number in the whole Sections \ref{S2}-\ref{S4}). We consider
	the sequence of the best approximations of $\alpha$:
	$$\frac{p_n}{q_n}:=\left\{ \begin{matrix}
		[a_1, a_2,a_3,...,a_n]& {\rm if}\ a_1 >1\\
		[1, a_2,a_3,...,a_{n+1}]& {\rm if}\ a_1=1\end{matrix}\right.,\ n\geq1,$$
	where $p_n,q_n$ are two coprime positive integers and set $q_0:= 1$.
	
	Due to the Douady-Ghys surgery, the quadratic polynomial $P_{\alpha}(z)=e^{2\pi i\alpha}z+z^2$ has a Siegel disk $\Delta_{\alpha}$ centering at $0$ whose boundary is a quasicircle containing the critical point $c_{0}=-\frac{e^{2\pi i\alpha}}{2}$.
	
	Let $\Phi$ be a conformal map fixing $0$ from $\mathbb{D}$ to $\Delta_{\alpha}$ with the homeomorphism extension $\tilde{\Phi}$ on $\overline{\mathbb{D}}$.
	For all interval $I$ on $\partial\Delta_{\alpha}$, the combinatorial length $|I|_{\partial\Delta_\alpha}$ of $I$ is defined by the Euclidean length of $\tilde{\Phi}^{-1}(I)$ divided by $2\pi$. For any two points $a,b$ on $\partial\Delta_{\alpha}$, the combinatorial distance ${\rm dist}_{\partial\Delta_{\alpha}}(a,b)$ between $a$ and $b$ is defined by the minimum of lengths of two intervals on $\partial\Delta_{\alpha}$ with endpoints $a$ and $b$.
	
	For $n\geq-1$, we denote by ${\rm CP}_n$ the set of critical points of $P_{\alpha}^{\comp q_{n+1}}$.
	The diffeo-tiling $\mathfrak{D}_n(\alpha)$ of level $n$ is the partition of $\partial\Delta_{\alpha}$ induced by ${\rm CP}_n$: every interval in $\mathfrak{D}_n(\alpha)$, also called an interval of level $n$, is the closure of a component of $\partial\Delta_{\alpha}\setminus{\rm CP}_n$. It is easy to check that
	\begin{itemize}
		\item For $n=-1$, $\mathfrak{D}_n(\alpha)$ consists of a single ``interval'' with a single endpoint $c_0$;
		\item For all $n\geq-1$, every interval in $\mathfrak{D}_n(\alpha)$ is divided into at least two intervals in $\mathfrak{D}_{n+2}(\alpha)$.
	\end{itemize}
	We denote by $x_{n}$ the unique point on $\partial\Delta_{\alpha}$ such that $P_{\alpha}^{\comp n}(x_{n})=c_{0}$ for all $n\geq0$. For all $I\in\mathfrak{D}_n(\alpha)$, we also write $I=I_{[x_{m},x_{m'}]}$, where $x_m$ and $x_{m'}$ are the two endpoints of $I$. We write $\iota_n:={\rm dist}_{\partial\Delta_\alpha}(c_{0},x_{q_{n}})$ for all $n\geq0$.
	
	Let $D\subseteq\mathbb{C}$ be a closed quasidisk and $I$ be an interval on $\partial D$. A curve $\gamma$ or rectangle $\mathcal{R}$ is said to be based on $I$ with respect to $D$ if
	\begin{itemize}
		\item except two endpoints, $\gamma$ is contained in $\hat{\mathbb{C}}\setminus D$ or $\mathcal{R}\setminus(\partial^{h,0}\mathcal{R}\cup\partial^{h,1}\mathcal{R})\subseteq\hat{\mathbb{C}}\setminus D$,
		\item two endpoints of $\gamma$ are contained in $I$ or $\partial^{h,0}\mathcal{R}\cup\partial^{h,1}\mathcal{R}\subseteq I$.
	\end{itemize}
	In this case, we denote by $O_I(\gamma)$ (resp. $O_I(\mathcal{R})$)
	the difference between  the closure of the union of all components of $\hat{\mathbb{C}}\setminus (I\cup\gamma)$ (resp. $\hat{\mathbb{C}}\setminus (I\cup\partial\mathcal{R})$) not containing ${\rm int}(D)$ and $D$. Moreover, $\gamma$ (resp. $\mathcal{R}$), based on $I$, is said to be non-winding with respect to $D$ if $\infty\not\in O_I(\gamma)$ (resp. $\infty\not\in O_I(\mathcal{R})$).
	
	We assign each $I_n\in\mathfrak{D}_n(\alpha)$ a non-winding geodesic $\tau_n(I_n)$ based on $I_n$ with respect to $\overline{\Delta_{\alpha}}$ on the hyperbolic space $\hat{\mathbb{C}}\setminus\overline{\Delta_{\alpha}}$ such that two endpoints of $\tau_n(I_n)$ are in $I_n\cap{\rm CP}_{n+1}\setminus{\rm CP}_n$. Let $F_n$ be the filling-in of $$\partial\Delta_{\alpha}\cup\left(\bigcup_{I_n\in\mathfrak{D}_n(\alpha)}\tau_n(I_n)\right),$$
	that is the union of $\partial\Delta_{\alpha}\cup\left(\bigcup_{I_n\in\mathfrak{D}_n(\alpha)}\tau_n(I_n)\right)$
	and all bounded components of $$\mathbb{C}\setminus\left(\partial\Delta_{\alpha}\cup\left(\bigcup_{I_n\in\mathfrak{D}_n(\alpha)}\tau_n(I_n)\right)\right).$$
	The set $F_n$ is called a geodesic pre-filling-in of $\Delta_{\alpha}$ of level $n$ associated with $\tau_n$.
	
	A nest $\{\frak{L}_n\}_{n\geq-1}$ is called a geodesic filling-in of $\Delta_{\alpha}$ if
	\begin{itemize}
		\item $\frak{L}_n=\overline{\Delta_{\alpha}}$ for all sufficiently large $n$,
		\item for all $n\geq0$,
		$$\frak{L}_{n-1}=\frak{L}_n\ {\rm or}\ \frak{L}_{n-1}=\frak{L}_n\cup F_{n-1},$$
		where $F_{n-1}$ is a geodesic pre-filling-in of $\Delta_{\alpha}$ of level $n-1$. (see Figure \ref{f20260720a})
	\end{itemize}
	\begin{figure}
		\centering
		\includegraphics[scale=0.5]{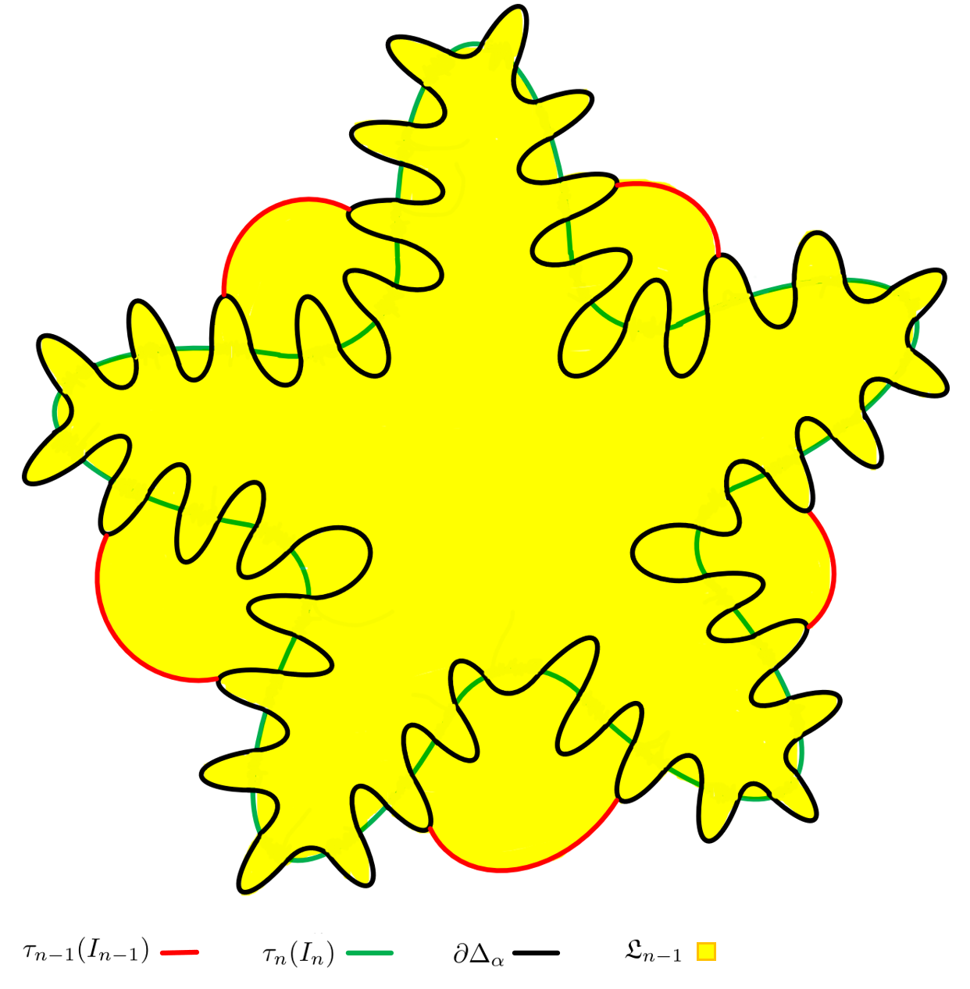}
		\caption{An illustration of the case: $\frak{L}_{n-1}=\frak{L}_n\cup F_{n-1}$, $\frak{L}_{n}=\frak{L}_{n+1}\cup F_{n}$, $\frak{L}_{n+1}=\overline{\Delta_{\alpha}}$.}
		\label{f20260720a}
	\end{figure}
	Each $\frak{L}_n$ is called a geodesic filling-in of $\Delta_{\alpha}$ of level $n$ and each $\tau_n(I_n)$, if exists, is called a dam of level $n$ based on $I_n$ with respect to $\{\frak{L}_n\}_{n\geq-1}$.
	It is easy to see that $\frak{L}_n$ is a closed quasidisk. 
	For any interval $I$ on $\partial\Delta_{\alpha}$, if two endpoints of $I$ are contained in $\partial\frak{L}_n$, then $I$ is called regular rel $\partial\frak{L}_n$. For all regular interval $I$ rel $\partial\frak{L}_n$,
	the interval $\hat{I}$ on $\partial\frak{L}_n$, that has the same endpoints as $I$ and is homotopic to $I$ on $\mathbb{C}^*$, is called the projection of $I$ onto $\partial\frak{L}_n$.
	
	In \cite{DL}, Dudko and Lyubich proved that there exists a geodesic filling-in of $\Delta_{\alpha}$: $\{\frak{L}_n\}_{n\geq-1}$ such that the following properties hold:\footnote{Dudko and Lyubich obtained more nice properties about their geodesic filling-ins, such as ``quasi-invariance'' up to $q_{n+1}$, and the part listed here will be used in our argument later in this paper.}
	\begin{itemize}
		\item[(P1)] {\bf Uniform a priori bounds:} there exist universal constants $C>0, M\geq2$ (not depending on $\alpha$) such that for all $n\geq-1$, there exists a nest of tilings $\mathcal{T}=\{\mathcal{T}_k\}_{k\geq-1}$ of $\frak{L}_n$ with post-$(C,M)$-bounded inner and outer geometries. Moreover,
		for all $m\geq n\geq-1$ and $I_m\in\mathfrak{D}_m$, if $I_m$ is regular rel $\partial\frak{L}_n$, then the projection $\hat{I}_m$ of $I_m$ onto $\partial\frak{L}_n$ belongs to $\mathcal{T}_{m-n-1}$.
		(refer to Section $11.4$ of \cite{DL} for details)
		\item[(P2)] {\bf Protection of dams:}
		for each dam $\beta_{I_n}$ of level $n\geq-1$ based on $I_{n}\in\mathfrak{D}_n$, 
		there exists a non-winding (with respect to $\overline{\Delta_{\alpha}}$) geodesic $\hat{\beta}_{I_{n}}$ (with respect to the hyperbolic metric on $\hat{\mathbb{C}}\setminus\overline{\Delta_{\alpha}}$) based on $I_n$ with endpoints in $I_{n}\cap{\rm CP}_{n+1}$ and ${\rm int}(\beta_{I_n})\subseteq O_{I_n}(\hat{\beta}_{I_{n}})$ such that
		\begin{itemize}
			\item the interval between two endpoints of $\beta_{I_{n}}$ on $I_n$ has a combinatorial distance at least $\frac{4}{5}|I_n|_{\partial\Delta_\alpha}$;
			\item the rectangle $\mathcal{R}_{I_n}$ between $\beta_{I_n}$ and $\hat{\beta}_{I_{n}}$ on $\hat{\mathbb{C}}\setminus\Delta_\alpha$, based on $I_n':=I_{n}\cap P_{\alpha}^{\comp q_{n+1}}(I_n)$ (with respect to $\overline{\Delta_{\alpha}}$) with $\partial^{v,0}\mathcal{R}_{I_n}=\beta_{I_n}$ and $\partial^{v,1}\mathcal{R}_{I_n}=\hat{\beta}_{I_{n}}$, satisfies that the width $\mathcal{W}(\mathcal{R}_{I_n})$ is greater than a universal $\Delta+4\gg1$;
			\item the combinatorial distance between two endpoints of $\hat{\beta}_{I_{n}}$ and two endpoints of $I_n$ is at least $11\iota_{n+1}$. (refer to Assumptions $4$ and $6$, Section $5.1.9$ and Lemma $5.6$.)
		\end{itemize}
	\end{itemize}
	Each $\frak{L}_n$, satisfying (P1) and (P2), is called a pseudo-Siegel disk of level $n$, written specially as $\hat{\Delta}_{\alpha}^n$. 
	
	If $\beta_{n}$ is a dam based on $I_{n}\in\mathfrak{D}_n$ with respect to $\{\hat{\Delta}_{\alpha}^n\}_{n\geq-1}$, then both $\beta_{n}$ and $\hat{\beta}_{n}$, based on $I_{n}$, are non-winding with respect to $\overline{\Delta_{\alpha}}$.
	The set $O_{I_n}(\beta_n)$ is
	called a fjord of level $n$ of $\Delta_{\alpha}$, written specially as $\mathcal{J}(\beta_n)$, and
	$O_{I_n}(\hat{\beta}_n)$ is called a big fjord of level $n$ of $\Delta_{\alpha}$, written specially as $\mathcal{J}(\hat{\beta}_n)$. We view $\overline{\mathcal{J}(\hat{\beta}_n)\setminus\mathcal{J}(\beta_n)}$ as a rectangle with $\partial^{v,1}\overline{\mathcal{J}(\hat{\beta}_n)\setminus\mathcal{J}(\beta_n)}=\hat{\beta}_n$ and $\partial^{v,0}\overline{\mathcal{J}(\hat{\beta}_n)\setminus\mathcal{J}(\beta_n)}=\beta_n$; that is, $\overline{\mathcal{J}(\hat{\beta}_n)\setminus\mathcal{J}(\beta_n)}$ is exactly $\mathcal{R}_{I_n}$ appearing in (P2). Let $\varphi_{I_n}$ be a conformal map from $\overline{\mathcal{J}(\hat{\beta}_n)\setminus\mathcal{J}(\beta_n)}$ to $E_x$ for some positive real number $x$ ($\geq\Delta+4$) such that $$\varphi_{I_n}^{-1}(\{(x,z):0\leq z\leq 1\})=\partial^{v,1}\overline{\mathcal{J}(\hat{\beta}_n)\setminus\mathcal{J}(\beta_n)}$$
	and $$\varphi_{I_n}^{-1}(\{(0,z):0\leq z\leq 1\})=\partial^{v,0}\overline{\mathcal{J}(\hat{\beta}_n)\setminus\mathcal{J}(\beta_n)}.$$
	We set $$\tilde{\beta}_n:=\varphi_{I_n}^{-1}(\{(y,z):0\leq z\leq 1, y=x-4\}).$$
	See Figure \ref{f20260720b}.
	\begin{figure}
		\centering
		\includegraphics[scale=0.5]{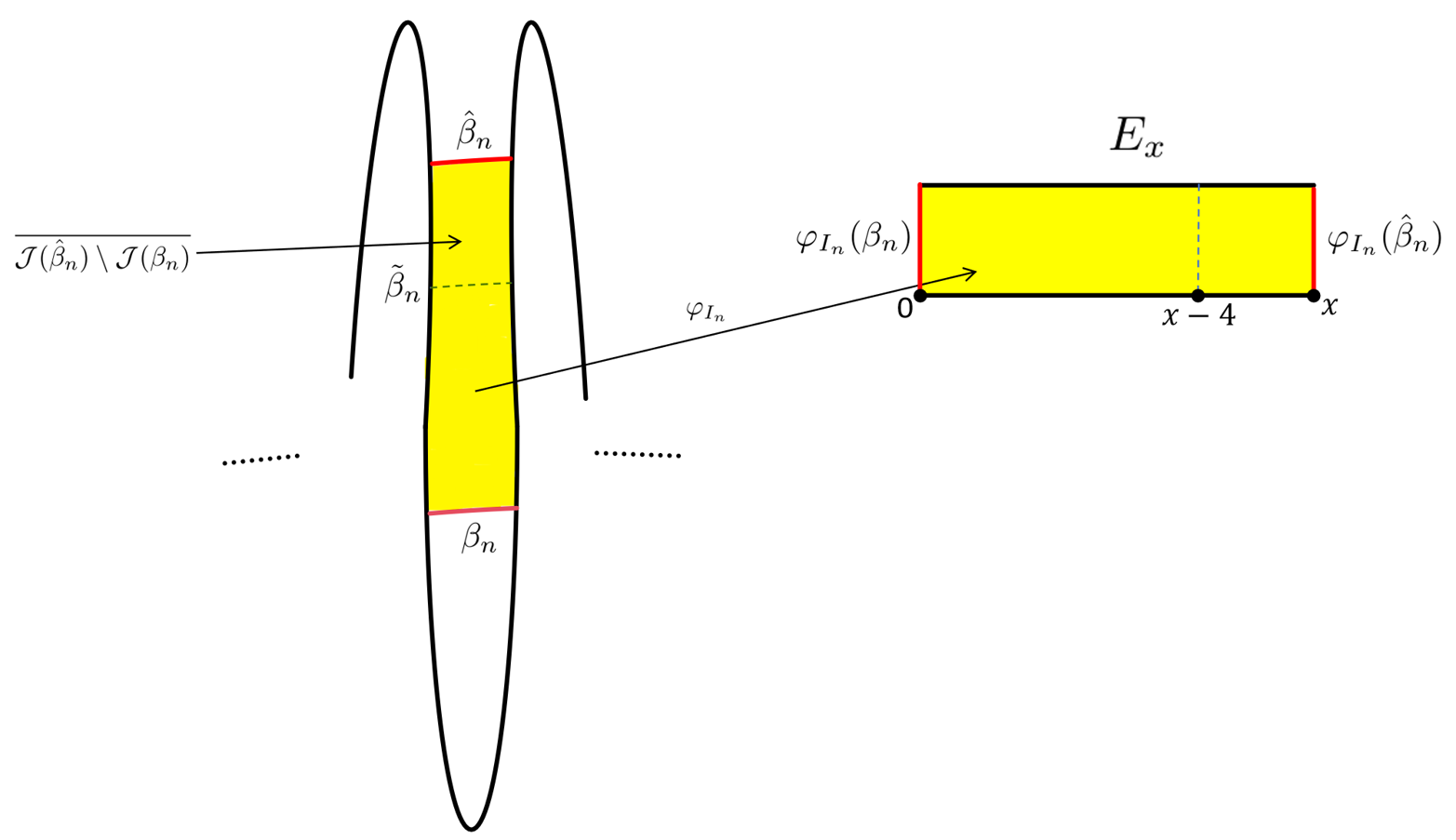}
		\caption{}
		\label{f20260720b}
	\end{figure}
	We set $\mathcal{J}(\tilde{\beta}_n):=O_{I_n}(\tilde{\beta}_{n})$.
	We view $\overline{\mathcal{J}(\tilde{\beta}_n)\setminus\mathcal{J}(\beta_n)}$ as a rectangle with $$\partial^{v,1}\overline{\mathcal{J}(\tilde{\beta}_n)\setminus\mathcal{J}(\beta_n)}=\tilde{\beta}_n$$
	and $$\partial^{v,0}\overline{\mathcal{J}(\tilde{\beta}_n)\setminus\mathcal{J}(\beta_n)}=\beta_n.$$ Then 
	\begin{equation}
		\label{e20251214a1}\mathcal{W}\left(\overline{\mathcal{J}(\tilde{\beta}_n)\setminus\mathcal{J}(\beta_n)}\right)\geq\Delta\gg1.
	\end{equation}
	
	Since for all $m\geq-1$, $x_{q_{m+1}-1}\in{\rm CP}_{m}$ and ${\rm dist}_{\partial\Delta_{\alpha}}(x_{q_{m+1}-1},P_{\alpha}(c_{0}))=\iota_{m+1}$, by (P2) we have that $P_{\alpha}(c_{0})$ is not contained in the closure of any big fjord of level $m$ and hence $P_{\alpha}(c_{0})$ is contained in $\partial\hat{\Delta}_{\alpha}^n$ ($n\geq-1$). 
	Then we have
	\begin{lemma}
		\label{l26821}
		For all $n\geq-1$, $\partial\hat{\Delta}_{\alpha}^n$ is a quasicircle containing $P_{\alpha}(c_{0})$ in $\mathbb{C}$.
	\end{lemma}
	For all $n\geq-1$, we denote by $\hat{U}_{\alpha}^{n}$ the closure of the component of $P_{\alpha}^{-1}(\hat{\Delta}_{\alpha}^{n})\setminus\{c_0\}$ not containing $0$ and denote by $\hat{\Delta}_{\alpha}^{n,-1}$ the closure of the other component of $P_{\alpha}^{-1}(\hat{\Delta}_{\alpha}^{n})\setminus\{c_0\}$.
	Every interval on $\partial\hat{\Delta}_{\alpha}^n$, whose endpoints are not contained in the closure of any big fjord of level $\geq n$, is called a
	well-grounded interval with respect to $\hat{\Delta}_{\alpha}^n$. 
	\begin{lemma}
		\label{l9301}For all $n\geq-1$, all $x_{q_m-1},$ $x_{q_{m+1}-q_m-1},$ $ x_{2q_m-1},$  $x_{q_{m+1}+q_{m+10}-1}$,  $x_{q_m+q_{m+11}-q_{m+10}-1}$ {\rm(}$m\geq0${\rm)} are contained in $\partial\hat{\Delta}_{\alpha}^n$, but not contained in the closure of any big fjord of level $n$. As a consequence, for all $n\geq-1$, any interval whose endpoints are contained in
		$$\{x: x=x_{q_m-1}\ {\rm or}\ x_{2q_m-1}\ {\rm or}\ x_{q_{m+1}+q_{m+10}-1}\ {\rm or}\ x_{q_{m+1}-q_m-1}\ {\rm or}\ x_{q_m+q_{m+11}-q_{m+10}-1}, m\geq0\}$$
		is well-grounded with respect to 
		$\hat{\Delta}_{\alpha}^n$.
	\end{lemma}
	\begin{proof}
		For all $m\geq0$, $x_{q_m-1}\in{\rm CP}_{m-1}$ (resp. $x_{q_{m+1}-q_m-1}\in{\rm CP}_{m}$). Then for all $n\geq m-1$ (resp. $n\geq m$), $x_{q_m-1}$ (resp. $x_{q_{m+1}-q_m-1}$) is contained in $\partial\hat{\Delta}_{\alpha}^n$, but not contained in the closure of any big fjord of level $n$.
		Next, we consider the case $n\leq m-2$ (resp. $n\leq m-1$). In this case, $x_{q_{n+1}-1}\in{\rm CP}_{n}$ and
		the combinatorial distance between $x_{q_m-1}$ (resp. $x_{q_{m+1}-q_m-1}$) and $x_{q_{n+1}-1}$
		is less than or equal to
		$3\iota_{n+1}$. Then by (P2) $x_{q_m-1}$ (resp. $x_{q_{m+1}-q_m-1}$) is not contained in the closure of any big fjord of level $n$.
		Thus for all $n\geq-1$, all $x_{q_m-1}$ (resp. $x_{q_{m+1}-q_m-1}$) {\rm(}$m\geq0${\rm)} are contained in $\partial\hat{\Delta}_{\alpha}^n$, but not contained in the closure of any big fjord of level $n$.
		
		For all $m\geq0$,
		if $q_{m+1}>q_{m}+q_{m-1}$, then 
		$x_{2q_m-1}\in{\rm CP}_{m}$. Then for all $n\geq m$, $x_{2q_m-1}$ is contained in $\partial\hat{\Delta}_{\alpha}^n$, but not contained in the closure of any big fjord of level $n$.
		Next, we consider the case $n\leq m-1$. In this case, $x_{q_{n+1}-1}\in{\rm CP}_{n}$ and
		the combinatorial distance between $x_{2q_m-1}$ and $x_{q_{n+1}-1}$
		is less than or equal to
		$3\iota_{n+1}$. Then by (P2) $x_{2q_m-1}$ is not contained in the closure of any big fjord of level $n$.
		Thus for all $n\geq-1$,  $x_{2q_m-1}$ is contained in $\partial\hat{\Delta}_{\alpha}^n$, but not contained in the closure of any big fjord of level $n$.
		If $q_{m+1}=q_{m}+q_{m-1}$, then 
		$x_{2q_m-1}=x_{q_m-q_{m-1}+q_{m+1}-1}\in{\rm CP}_{m+1}$. Then for all $n\geq m+1$, $x_{2q_m-1}$ is contained in $\partial\hat{\Delta}_{\alpha}^n$, but not contained in the closure of any big fjord of level $n$.
		Next, we consider the case $n\leq m-1$. In this case, $x_{q_{n+1}-1}\in{\rm CP}_{n}$ and
		the combinatorial distance between $x_{2q_m-1}$ and $x_{q_{n+1}-1}$
		is less than or equal to
		$3\iota_{n+1}$. Then by (P2) $x_{2q_m-1}$ is not contained in the closure of any big fjord of level $n$.
		At last, we consider the case $n=m$. In this case, $x_{q_n-q_{n-1}-1}\in{\rm CP}_{n}$ and
		the combinatorial distance between $x_{2q_m-1}=x_{q_n-q_{n-1}+q_{n+1}-1}$ and $x_{q_n-q_{n-1}-1}$
		is equal to
		$\iota_{n+1}$. Then by (P2) $x_{2q_m-1}$ is not contained in the closure of any big fjord of level $n$.
		Thus for all $n\geq-1$, all $x_{2q_m-1}$ {\rm(}$m\geq0${\rm)} are contained in $\partial\hat{\Delta}_{\alpha}^n$, but not contained in the closure of any big fjord of level $n$.
		
		For all $m\geq0$, $x_{q_{m+1}+q_{m+10}-1}\in{\rm CP}_{m+10}$ (resp. $x_{q_m+q_{m+11}-q_{m+10}-1}\in{\rm CP}_{m+10}$). Then for all $n\geq m+10$, $x_{q_{m+1}+q_{m+10}-1}$ (resp. $x_{q_m+q_{m+11}-q_{m+10}-1}$) is contained in $\partial\hat{\Delta}_{\alpha}^n$, but not contained in the closure of any big fjord of level $n$.
		Next, we consider the case $n\leq m$ (resp. $n\leq m-1$). In this case, $x_{q_{n+1}-1}\in{\rm CP}_{n}$ and
		the combinatorial distance between $x_{q_{m+1}+q_{m+10}-1}$ (resp. $x_{q_m+q_{m+11}-q_{m+10}-1}$) and $x_{q_{n+1}-1}$
		is less than or equal to
		$2\iota_{n+1}$. Then by (P2) $x_{q_{m+1}+q_{m+10}-1}$ (resp. $x_{q_m+q_{m+11}-q_{m+10}-1}$) is not contained in the closure of any big fjord of level $n$.
		At last, we consider the case $m+1\leq n\leq m+9$ (resp. $m\leq n\leq m+9$). In this case, $x_{q_{m+1}-1}\in{\rm CP}_{n}$ (resp. $x_{q_{m}-1}\in{\rm CP}_{n}$) and
		the combinatorial distance between $x_{q_{m+1}+q_{m+10}-1}$ (resp. $x_{q_m+q_{m+11}-q_{m+10}-1}$) and $x_{q_{m+1}-1}$ (resp. $x_{q_{m}-1}$)
		is less than or equal to
		$2\iota_{n+1}$. Then by (P2) $x_{q_{m+1}+q_{m+10}-1}$ (resp. $x_{q_m+q_{m+11}-q_{m+10}-1}$) is not contained in the closure of any big fjord of level $n$.
		Thus for all $n\geq-1$, all $x_{q_{m+1}+q_{m+10}-1}$ (resp. $x_{q_m+q_{m+11}-q_{m+10}-1}$) {\rm(}$m\geq0${\rm)} are contained in $\partial\hat{\Delta}_{\alpha}^n$, but not contained in the closure of any big fjord of level $n$.
	\end{proof}
	
	The following five lemmas \ref{l8271a}—\ref{l825a} follow from \cite{DL}, see [Lemma $4.6$, \cite{DL}], [Lemma $5.6$, \cite{DL}], [Lemma $5.8$, \cite{DL}], [Uniform Bounds Theorem $1.1$, \cite{DL}] and [Lemma A.5, \cite{DL}].
	
	\begin{lemma}[Dudko and Lyubich]
		\label{l8271a}For all dam $\beta_{n}$ of level $n$ based on $I_{n}\in\mathfrak{D}_n$ satisfying {\rm(P}$2${\rm)}, $P_{\alpha}^{\comp q_{n+1}}$ is injective on $\mathcal{J}(\hat{\beta}_n)\setminus\varphi_{I_n}^{-1}\left(\{(y,z):0\leq z\leq 1,\ x-3< y\leq x\}\right)$. As a consequence, $P_{\alpha}^{\comp q_{n+1}}|_{\mathcal{J}(\tilde{\beta}_n)}$ is injective on
		$$\varphi_{I_n}^{-1}\left(\{(y,z):0\leq z\leq 1,\ x-4\leq y\leq x-3\}\right).$$
	\end{lemma}
	
	Let $\mathcal{R}$ and $\mathcal{R}_1$ be two rectangles such that $\mathcal{R}_1\subseteq\mathcal{R}$, $\partial^{v,0}\mathcal{R}_1=\partial^{v,0}\mathcal{R}$ and $\partial^{v,1}\mathcal{R}_1=\partial^{v,1}\mathcal{R}$. If for every component $F$ of $\partial^{h,s}\mathcal{R}_1\setminus\partial^{h,s}\mathcal{R}$ ($s\in\{0,1\}$),
	there exists a rectangle $\mathcal{S}\subseteq\mathcal{R}$ with $\mathcal{W}(\mathcal{S})\gg1$ and $\partial^{h,0}\mathcal{S}\cup\partial^{h,1}\mathcal{S}\subseteq\partial^{h,s}\mathcal{R}$ such that $F$ and $\partial^{h,1-s}\mathcal{R}$ are contained in different components of $\mathcal{R}\setminus\mathcal{S}$, then we say that $\mathcal{R}_1$ is protected in $\mathcal{R}$. In this case, each such $\mathcal{S}$ is called a protection of $\mathcal{R}_1$ in $\mathcal{R}$.
	\begin{lemma}[Dudko and Lyubich]
		\label{l861}Let $\mathcal{R}$ and $\mathcal{R}_1$ be two rectangles such that $\mathcal{R}_1\subseteq\mathcal{R}$, $\partial^{v,0}\mathcal{R}_1=\partial^{v,0}\mathcal{R}$ and $\partial^{v,1}\mathcal{R}_1=\partial^{v,1}\mathcal{R}$. If $\mathcal{R}_1$ is protected in $\mathcal{R}$, then $\mathcal{W}(\mathcal{R}_1)\asymp\mathcal{W}(\mathcal{R})$.
	\end{lemma}
	For any two disjoint intervals $I$ and $J$ on $\partial\Delta_\alpha$ with their projections $\hat{I},\hat{J}$ onto $\partial\hat{\Delta}_{\alpha}^n$ respectively, if $\hat{I},\hat{J}$ are well-grounded with respect to the pseudo-Siegel disk $\hat{\Delta}_{\alpha}^n$, then
	by (P2), $\mathcal{F}_{\hat{\Delta}_{\alpha}^n}^+(\hat{I},\hat{J})$ is protected in $\mathcal{F}_{\overline{\Delta_{\alpha}}}^+(I,J)$. Thus it follows from Lemma \ref{l861} that the following lemma \ref{l7291} holds:
	\begin{lemma}[Dudko and Lyubich]
		\label{l7291}
		$\mathcal{W}(\mathcal{F}_{\hat{\Delta}_{\alpha}^n}^+(\hat{I},\hat{J}))\asymp\mathcal{W}(\mathcal{F}_{\overline{\Delta_{\alpha}}}^+(I,J)).$
	\end{lemma}
	
	\begin{lemma}[Dudko and Lyubich]
		\label{l8271}For all $n\geq-1$ and all three intervals $I_1$, $I_2$ and $I_3$ of level $n$ with $I_1\#I_2\#I_3$, there exists an absolute constant $C>0$ such that $\mathcal{W}(\mathcal{F}(I_2,[3I_2]^c))\leq C$,
		where $[3I_2]^c=\overline{\partial\Delta_{\alpha}\setminus(I_1\cup I_2\cup I_3)}$. As a consequence, if $\mathfrak{D}_n$ contains at least four intervals, then $\mathcal{W}(\mathcal{F}(I_2,[3I_2]^c))\asymp 1$.
	\end{lemma}
	
	\begin{lemma}[Dudko and Lyubich]
		\label{l825a}
		Let $D$ be a closed Jordan disk on $\mathbb{C}$ and $\mathcal{R}$ be a rectangle contained in $D$ with $\partial^{h,0}\mathcal{R}\cup\partial^{h,1}\mathcal{R}\subseteq\partial D$ and $\mathcal{W}(\mathcal{R})>\frac{1}{2}$. Let $T$ be a component of $\overline{\partial D\setminus(\partial^{h,0}\mathcal{R}\cup\partial^{h,1}\mathcal{R})}$ and $a,b$ be two endpoints of $T$.
		Let $\gamma$ be the geodesic connecting $a$ and $b$ with respect to the hyperbolic metric on ${\rm int}(D)$.
		Then we have that $\gamma$ is contained in the filling-in of $T\cup\mathcal{R}$. 
	\end{lemma}
	Here we give a new proof of Lemma \ref{l825a} as follows:
	\begin{proof}[The proof of Lemma \ref{l825a}]
		We prove the lemma by contradiction. We suppose that $\gamma$ is not contained in the filling-in of $T\cup\mathcal{R}$. Then there exists a point $d$ on $\gamma$ not contained in the filling-in of $T\cup\mathcal{R}$.
		Let $\phi$ be a conformal map from ${\rm int}(D)$ to $\mathbb{D}$ with the homeomorphism extension $\tilde{\phi}$ on $\overline{D}$ such that  $\phi(d)=0$ and
		$\tilde{\phi}(T)$ is the lower half unit circle $\{e^{2\pi i\theta}: \frac{1}{2}\leq\theta\leq1\}$. Then $\tilde{\phi}(\gamma)=[-1,1]$ and
		$\tilde{\phi}(\partial^{h,0}\mathcal{R})\cup\tilde{\phi}(\partial^{h,1}\mathcal{R})\subseteq\{e^{2\pi i\theta}: 0\leq\theta\leq\frac{1}{2}\}$ (Without loss of generality, we assume $\tilde{\phi}(\partial^{h,0}\mathcal{R})<\tilde{\phi}(\partial^{h,1}\mathcal{R})<\{e^{2\pi i\theta}: \frac{1}{2}\leq\theta\leq1\}$).
		We take a point $c=e^{2\pi i\theta_0}$ ($\theta_0\in(0,\frac{1}{2})$) separating $\tilde{\phi}(\partial^{h,0}\mathcal{R})$ and $\tilde{\phi}(\partial^{h,1}\mathcal{R})$ on $\{e^{2\pi i\theta}: 0\leq\theta\leq\frac{1}{2}\}$ and a Jordan arc $\gamma_1$ connecting $0$ and $c$
		on $\overline{\mathbb{D}}$ not intersecting $\tilde{\phi}(R)$. 
		Observe that $\sqrt{z}$ has two single-value continuous branches on $\overline{\mathbb{D}}\setminus\gamma_1$, written as $\psi_1$ (with $\psi_1(-1)=i$) and $\psi_2$ (with $\psi_1(-1)=-i$).
		We view $\psi_1\comp\tilde{\phi}(\mathcal{R})$ (resp. $\psi_2\comp\tilde{\phi}(\mathcal{R})$)
		as a rectangle with $\partial^{h,0}\psi_1\comp\tilde{\phi}(\mathcal{R})=\psi_1\comp\tilde{\phi}(\partial^{h,0}\mathcal{R})$ (resp. $\partial^{h,1}\psi_2\comp\tilde{\phi}(\mathcal{R})=\psi_2\comp\tilde{\phi}(\partial^{h,0}\mathcal{R})$) and $\partial^{h,1}\psi_1\comp\tilde{\phi}(\mathcal{R})=\psi_1\comp\tilde{\phi}(\partial^{h,1}\mathcal{R})$ (resp. $\partial^{h,0}\psi_2\comp\tilde{\phi}(\mathcal{R})=\psi_2\comp\tilde{\phi}(\partial^{h,1}\mathcal{R})$). Then $\mathcal{W}(\psi_1\comp\tilde{\phi}(\mathcal{R}))=\mathcal{W}(\mathcal{R})$
		and $\mathcal{W}(\psi_2\comp\tilde{\phi}(\mathcal{R}))=\mathcal{W}(\mathcal{R})$.
		We set $\mathcal{C}_-=\{e^{2\pi i\theta}:0\leq\theta\leq\frac{1}{4}\}$ and $\mathcal{C}_+=\{e^{2\pi i\theta}:\frac{1}{2}\leq\theta\leq\frac{3}{4}\}$. 
		It is easy to see $\mathcal{F}(\psi_1\comp\tilde{\phi}(\mathcal{R}))\cup\mathcal{F}(\psi_2\comp\tilde{\phi}(\mathcal{R}))\subseteq\mathcal{F}_{\overline{\mathbb{D}}}^-(\mathcal{C}_-,\mathcal{C}_+)$.
		Thus
		\begin{align*}
			1<2\mathcal{W}(\mathcal{R})&=\mathcal{W}(\psi_1\comp\tilde{\phi}(\mathcal{R}))+\mathcal{W}(\psi_2\comp\tilde{\phi}(\mathcal{R}))\\
			&=\mathcal{W}(\mathcal{F}(\psi_1\comp\tilde{\phi}(\mathcal{R}))\cup\mathcal{F}(\psi_2\comp\tilde{\phi}(\mathcal{R})))\\
			&\leq\mathcal{W}(\mathcal{F}_{\overline{\mathbb{D}}}^-(\mathcal{C}_-,\mathcal{C}_+))\\
			&=1.
		\end{align*}
		This is impossible. 
		
	\end{proof}
	
	For all $n\geq-1$ and any interval $I$ on $\partial\hat{\Delta}_{\alpha}^n$, we define
	$$|I|_{\hat{\Delta}_{\alpha}^n}:=|I|_{\mathcal{T}},$$
	where $\mathcal{T}$ is the nest of tilings as described in (P1). 
	Moreover, if
	$J$ is an interval on $\partial\hat{\Delta}_{\alpha}^n$ having a disjoint interior with $I$, then we define
	$${\rm dist}_{\hat{\Delta}_{\alpha}^n}(I,J):=\min\left\{|T_1|_{\mathcal{T}},|T_2|_{\mathcal{T}}\right\},$$
	where $T_1$ and $T_2$ are the two components of $\overline{\partial\hat{\Delta}_{\alpha}^n\setminus(I\cup J)}$.
	
	\begin{lemma}
		\label{l831}Assume that $I$ and $J$ are two disjoint intervals on $\partial\Delta_{\alpha}$ and are well-grounded with respect to the pseudo-Siegel disk $\hat{\Delta}_{\alpha}^n$ of level $n$.
		Let $\hat{I}$ and $\hat{J}$ be projections of $I$ and $J$ onto $\partial\hat{\Delta}_{\alpha}^n$, respectively.
		For all $c>0$, if ${\rm dist}_{\hat{\Delta}_{\alpha}^n}(\hat{I},\hat{J})\asymp_c\min\{|\hat{I}|_{\hat{\Delta}_{\alpha}^n},|\hat{J}|_{\hat{\Delta}_{\alpha}^n}\}\preceq_c1$, then $\mathcal{W}_{\over{\Delta}_{\alpha}}^+(I,J)\asymp_c1$.
	\end{lemma}
	\begin{proof}
		Let $\hat{T}$ be the closure of the component of $\partial\hat{\Delta}_{\alpha}^n\setminus(\hat{I}\cup\hat{J})$ such that $|\hat{T}|_{\hat{\Delta}_{\alpha}^n}={\rm dist}_{\hat{\Delta}_{\alpha}^n}(\hat{I},\hat{J})$.
		We split $\hat{T}$ into $I_1, I_2, \cdots, I_k$ such that $k\asymp_c1$, $I_1\#I_2\# \cdots\#I_k$ and for all $1\leq j\leq k$, $|I_j|_{\hat{\Delta}_{\alpha}^n}\leq\min\{|\hat{I}|_{\hat{\Delta}_{\alpha}^n},|\hat{J}|_{\hat{\Delta}_{\alpha}^n}\}$ and $|I_j|_{\hat{\Delta}_{\alpha}^n}<1$.
		Let $\hat{T}_1$ be the closure of the other component of $\partial\hat{\Delta}_{\alpha}^n\setminus(\hat{I}\cup\hat{J})$.
		By Lemma \ref{l811} we have 
		$\mathcal{W}(\mathcal{F}_{\hat{\Delta}_{\alpha}^n}^+(I_j,\hat{T}_1))\preceq1$ for all $1\leq j\leq k$.
		Observe that $\mathcal{F}_{\hat{\Delta}_{\alpha}^n}^+(\hat{T},\hat{T}_1)\subseteq\bigcup\limits_{j=1}^k\mathcal{F}_{\hat{\Delta}_{\alpha}^n}^+(I_j,\hat{T}_1)$. Then 
		$$\mathcal{W}(\mathcal{F}_{\hat{\Delta}_{\alpha}^n}^+(\hat{T},\hat{T}_1))\leq\sum\limits_{j=1}^k\mathcal{W}(\mathcal{F}_{\hat{\Delta}_{\alpha}^n}^+(I_j,\hat{T}_1))\preceq_c1.$$
		Thus
		$$\mathcal{W}(\mathcal{F}_{\hat{\Delta}_{\alpha}^n}^+(\hat{I},\hat{J}))\succeq_c1.$$
		Observe that  
		${\rm dist}_{\hat{\Delta}_{\alpha}^n}(\hat{T},\hat{T}_1)\asymp_c\min\{|\hat{T}|_{\hat{\Delta}_{\alpha}^n},|\hat{T}_1|_{\hat{\Delta}_{\alpha}^n}\}\preceq_c1$. By the above similar argument, we can obtain that
		$$\mathcal{W}(\mathcal{F}_{\hat{\Delta}_{\alpha}^n}^+(\hat{T},\hat{T}_1))\succeq_c1$$
		and hence
		$$\mathcal{W}(\mathcal{F}_{\hat{\Delta}_{\alpha}^n}^+(\hat{I},\hat{J}))\preceq_c1.$$
		This implies 
		$$\mathcal{W}(\mathcal{F}_{\hat{\Delta}_{\alpha}^n}^+(\hat{I},\hat{J}))\asymp_c1.$$
		Since $I$ and $J$ are well-grounded with respect to $\hat{\Delta}_{\alpha}^n$, by Lemma \ref{l7291} 
		$$\mathcal{W}(\mathcal{F}_{\over{\Delta}_{\alpha}}^+(I,J))\asymp\mathcal{W}(\mathcal{F}_{\hat{\Delta}_{\alpha}^n}^+(\hat{I},\hat{J}))\asymp_c1$$
		and hence
		$$\mathcal{W}(\mathcal{F}_{\over{\Delta}_{\alpha}}^+(I,J))\asymp_c1.$$
		
	\end{proof}

	For any $I_n(\alpha)\in\mathfrak{D}_n(\alpha)$ ($n\geq1$),
	we let $\hat{I}_n(\alpha)$ be the projection of $I_n(\alpha)$ onto $\partial\hat{\Delta}_{\alpha}^n$ and
	$\gamma_n(\alpha)$ be the geodesic connecting two endpoints of $\hat{I}_n(\alpha)$
	with respect to the hyperbolic metric on ${\rm int}(\hat{\Delta}_{\alpha}^n)$. We denote by $\hat{I}_n(\alpha)+\gamma_n(\alpha)$ the closed curve obtained by connecting $\hat{I}_n(\alpha)$ and $\gamma_n(\alpha)$ end to end.
	Then we have the following lemma.
	\begin{lemma}[Hongyu Qu, \cite{Q24}]
		\label{l84a}There is an absolute constant $\mathfrak{K}$ $>1$ such that
		for any $I_n(\alpha)\in\mathfrak{D}_n(\alpha)$, the closed curve $\hat{I}_n(\alpha)+\gamma_n(\alpha)$ is a $\mathfrak{K}$-quasicircle.
	\end{lemma}
	For all closed quasidisk $D$, we let $I$ and $J$ be two  intervals on $\partial D$ with disjoint interiors. We denote by $\mathcal{G}_D^+(I,J)$ the rectangle on $\hat{\mathbb{C}}\setminus\overline{D}$ with $\partial^{h,0}\mathcal{G}_D^+(I,J)=I$, $\partial^{h,1}\mathcal{G}_D^+(I,J)=J$ and $\partial^{v,0}\mathcal{G}_D^+(I,J)$ (resp. $\partial^{v,1}\mathcal{G}_D^+(I,J)$) is a geodesic with respect to the hyperbolic metric on $\hat{\mathbb{C}}\setminus\overline{D}$. Similarly, we denote by $\mathcal{G}_D^-(I,J)$ the rectangle on  $D$ with $\partial^{h,0}\mathcal{G}_D^-(I,J)=I$, $\partial^{h,1}\mathcal{G}_D^-(I,J)=J$ and $\partial^{v,0}\mathcal{G}_D^-(I,J)$ (resp. $\partial^{v,1}\mathcal{G}_D^-(I,J)$) is a geodesic with respect to the hyperbolic metric on ${\rm int}(D)$.
	By [Log-Rule, \cite{DL}], the following two results hold:
	\begin{equation}
		\label{e8221}\mathcal{W}_D^+(I,J)\preceq ({\rm resp.} \succeq)1\Rightarrow\mathcal{W}(\mathcal{G}_D^+(I,J))\preceq ({\rm resp.} \succeq)1,
	\end{equation}
	and
	\begin{equation}
		\label{e8222}
		\mathcal{W}_D^-(I,J)\preceq ({\rm resp.} \succeq)1\Rightarrow\mathcal{W}(\mathcal{G}_D^-(I,J))\preceq ({\rm resp.} \succeq)1.
	\end{equation}
	Indeed, we can choose an appropriate conformal map $\phi$ from ${\rm int}(D)$ to $\mathbb{D}$ with the homeomorphism extension  $\tilde{\phi}$ on $\overline{D}$ such that 
	$\tilde{\phi}(I)$ and $\tilde{\phi}(J)$ have the same length, and at the same time the two components of $\partial\mathbb{D}\setminus(\tilde{\phi}(I)\cup\tilde{\phi}(J))$ also have the same length. If $\mathcal{W}_D^-(I,J)\preceq ({\rm resp.} \succeq)1$, then $\mathcal{W}_{\overline{\mathbb{D}}}^-(\tilde{\phi}(I),\tilde{\phi}(J))\preceq ({\rm resp.} \succeq)1$. By [Log-Rule, \cite{DL}] the length of $\tilde{\phi}(I)$ and  $\tilde{\phi}(J)$ is at most (resp. at least) comparable with the length of two components of $\partial\mathbb{D}\setminus(\tilde{\phi}(I)\cup\tilde{\phi}(J))$. This implies $\mathcal{W}(\mathcal{G}_{\overline{\mathbb{D}}}^-(\tilde{\phi}(I),\tilde{\phi}(J)))\preceq ({\rm resp.} \succeq)1$. Pulling back to $D$ by $\tilde{\phi}$, we obtain (\ref{e8222}).
	Similarly, we can also obtain (\ref{e8221}).
	Moreover,
	by Lemma \ref{l825a}, we have
	\begin{equation}
		\label{e8223} \mathcal{W}_D^+(I,J)\gg1\Rightarrow\mathcal{W}(\mathcal{G}_D^+(I,J))\gg1,
	\end{equation}
	and
	\begin{equation}
		\label{e8224}
		\mathcal{W}_D^-(I,J)\gg1\Rightarrow\mathcal{W}(\mathcal{G}_D^-(I,J))\gg1.
	\end{equation}
	
	By (\ref{e8221}) and Lemma \ref{l8271}, the following lemma holds:
	
	
	\begin{lemma}
		\label{l7301}For all $n\geq4$ and any positive integer $k\gg1$, $I_n\in\mathfrak{D}_n(\alpha)$ and $I_{n+k}\in\mathfrak{D}_{n+k}(\alpha)$ with $I_{n+k}\subseteq I_n$, we have $$\mathcal{W}(\mathcal{F}_{\over{\Delta}_{\alpha}}^+(I_{n+k},[3I_n]^c))\preceq\frac{1}{k}.$$
		As a consequence, if $I$ is an interval on $\partial\Delta_{\alpha}$, $I_{n+k}\subseteq I$ and both two components of $\overline{I\setminus I_{n+k}}$ contain at least one interval of $\mathfrak{D}_n$, then 
		$$\mathcal{W}(\mathcal{F}_{\over{\Delta}_{\alpha}}^+(I_{n+k},\overline{\partial\Delta_{\alpha}\setminus I}))\preceq\frac{1}{k}.$$
	\end{lemma}
	\begin{proof}
		Let $k\gg1$.
		Let $I_{n+k-4}^{(1)}$, $I_{n+k-4}^{(2)}$, $I_{n+k-4}^{(3)}$ and $I_{n+k-4}^{(4)}$ be four intervals of level $n+k-4$ such that
		\begin{itemize}
			\item $I_{n+k-4}^{(1)}\#I_{n+k-4}^{(2)}\#I_{n+k-4}^{(3)}\#I_{n+k-4}^{(4)}$,
			\item $I_{n+k}\subseteq I_{n+k-4}^{(2)}\cup I_{n+k-4}^{(3)}$.
		\end{itemize}
		For all $s<\frac{k}{4}$, we let 
		$I_{n+k-4s}^{(1)}$, $I_{n+k-4s}^{(2)}$, $I_{n+k-4s}^{(3)}$ and $I_{n+k-4s}^{(4)}$ be four intervals of level $n+k-4s$ such that
		\begin{itemize}
			\item $I_{n+k-4s}^{(1)}\#I_{n+k-4s}^{(2)}\#I_{n+k-4s}^{(3)}\#I_{n+k-4s}^{(4)}$,
			\item $I_{n+k-(s-1)4}^{(1)}\cup I_{n+k-(s-1)4}^{(2)}\cup I_{n+k-(s-1)4}^{(3)}\cup I_{n+k-(s-1)4}^{(4)}\subseteq I_{n+k-4s}^{(2)}\cup I_{n+k-4s}^{(3)}$.
		\end{itemize}
		We take $s= \lfloor\frac{k}{8}\rfloor$.
		Then
		$\mathcal{G}_{\overline{\Delta_\alpha}}^+(I_{n+k-4t}^{(1)},I_{n+k-4t}^{(4)})$, $1\leq t\leq s$ are pairwise disjoint, and separate $I_{n+k}$ and $[3I_n]^c$ on $\hat{\mathbb{C}}\setminus\Delta_\alpha$.
		Thus
		\begin{equation}
			\label{e20260607a}\mathcal{W}(\mathcal{F}_{\overline{\Delta_\alpha}}^+(I_{n+k},[3I_n]^c))\leq\frac{1}{\sum_{t=1}^s\mathcal{W}(\mathcal{G}_{\overline{\Delta_\alpha}}^+(I_{n+k-4t}^{(1)},I_{n+k-4t}^{(4)}))}.
		\end{equation}
		For all $t$ with $1\leq t\leq s$,
		Lemma \ref{l8271} gives
		$\mathcal{W}_{\overline{\Delta_\alpha}}^+(I_{n+k-4t}^{(j)},[3I_{n+k-4t}^{(j)}]^c)\preceq1$ for $j=1,2,3,4$.
		Then 
		\begin{align*}
			\mathcal{W}_{\overline{\Delta_\alpha}}^+(I_{n+k-4t}^{(1)},I_{n+k-4t}^{(4)})&=\frac{1}{\mathcal{W}_{\overline{\Delta_\alpha}}^+\left(I_{n+k-4t}^{(2)}\cup I_{n+k-4t}^{(3)},\overline{\partial\Delta_{\alpha}\setminus\bigcup_{j=1}^4I_{n+k-4t}^{(j)}}\right)}\\
			&\geq\frac{1}{\mathcal{W}_{\overline{\Delta_\alpha}}^+(I_{n+k-4t}^{(2)},[3I_{n+k-4t}^{(2)}]^c)+\mathcal{W}_{\overline{\Delta_\alpha}}^+(I_{n+k-4t}^{(3)},[3I_{n+k-4t}^{(3)}]^c)}\\
			&\succeq1.
		\end{align*}
		By (\ref{e8221}) we have
		$$\mathcal{W}(\mathcal{G}_{\overline{\Delta_\alpha}}^+(I_{n+k-4t}^{(1)},I_{n+k-4t}^{(4)}))\succeq1,\ 1\leq t\leq s.$$
		By combining (\ref{e20260607a}), it follows that $$\mathcal{W}(\mathcal{F}_{\overline{\Delta_\alpha}}^+(I_{n+k},[3I_n]^c))\preceq\frac{1}{s}\asymp\frac{1}{k}.$$
	\end{proof}

	\subsection{Welding of inner and outer geometries.\label{s2.4}}
	\begin{lemma}
		\label{l7311}Let  $\mathcal{T}=\{\mathcal{T}_n\}_{n\geq-1}$ be a nest of tilings of a closed quasidisk $D$ with post-{\rm(}C,M{\rm)}-bounded inner and outer geometries. Let $I$ and $J$ be two disjoint intervals on $\partial D$ such that
		$|I|_{\mathcal{T}}\asymp|J|_{\mathcal{T}}$, $|I|_{\mathcal{T}}<1$ and $|J|_{\mathcal{T}}<1$.
		Let $\mathcal{R}_1$ be a rectangle on $\overline{\hat{\mathbb{C}}\setminus D}$ such that $\partial^{h,0}\mathcal{R}_1=I$, $\partial^{h,1}\mathcal{R}_1=J$ and $\mathcal{W}(\mathcal{R}_1)\succeq1$. Let $\mathcal{R}_2$ be a rectangle on $D$ such that $\partial^{h,0}\mathcal{R}_2=I$, $\partial^{h,1}\mathcal{R}_2=J$ and $\mathcal{W}(\mathcal{R}_2)\succeq1$. We spit $I$ into $I_1\# I_2\# I_3$ such that $|I_1|_{\mathcal{T}}\asymp|I_3|_{\mathcal{T}}\preceq|I_2|_{\mathcal{T}}$ and spit $J$ into $J_1\# J_2\# J_3$ such that $|J_1|_{\mathcal{T}}\asymp|J_3|_{\mathcal{T}}\preceq|J_2|_{\mathcal{T}}$. If there exist rectangles $\mathcal{R}_1^{(1)}$, $\mathcal{R}_1^{(2)}$, $\mathcal{R}_2^{(1)}$ and $\mathcal{R}_2^{(2)}$
		such that
		\begin{itemize}
			\item $\mathcal{R}_1^{(1)}\cup\mathcal{R}_1^{(2)}\subseteq\mathcal{R}_1$, $\mathcal{R}_1^{(1)}\cap\mathcal{R}_1^{(2)}=\emptyset$, $\partial^{h,0}\mathcal{R}_1^{(1)}=I_1$, $\partial^{h,1}\mathcal{R}_1^{(1)}=J_3$, $\partial^{h,0}\mathcal{R}_1^{(2)}=I_3$, $\partial^{h,1}\mathcal{R}_1^{(2)}=J_1$, $\mathcal{W}(\mathcal{R}_1^{(1)})\succeq1$ and $\mathcal{W}(\mathcal{R}_1^{(2)})\succeq1$,
			\item $\mathcal{R}_2^{(1)}\cup\mathcal{R}_2^{(2)}\subseteq\mathcal{R}_2$, $\mathcal{R}_2^{(1)}\cap\mathcal{R}_2^{(2)}=\emptyset$, $\partial^{h,0}\mathcal{R}_2^{(1)}=I_1$, $\partial^{h,1}\mathcal{R}_2^{(1)}=J_3$, $\partial^{h,0}\mathcal{R}_2^{(2)}=I_3$ and $\partial^{h,1}\mathcal{R}_2^{(2)}=J_1$, $\mathcal{W}(\mathcal{R}_2^{(1)})\succeq1$ and $\mathcal{W}(\mathcal{R}_2^{(2)})\succeq1$,
		\end{itemize}
		then the annulus $\mathcal{R}=(\mathcal{R}_1\cup\mathcal{R}_2)^{\comp}$, the interior of the union $\mathcal{R}_1\cup\mathcal{R}_2$, satisfies ${\rm mod}(\mathcal{R})\succeq_{C,M}1$.
	\end{lemma}
	\begin{proof}
		Since $I_1<I_2<I_3<J$, $|I_3|_{\mathcal{T}}\preceq|I_2|_{\mathcal{T}}<|I|_{\mathcal{T}}\asymp|J|_{\mathcal{T}}$ and $|I_3|_{\mathcal{T}}\leq|I|_{\mathcal{T}}<1$, by Corollary \ref{c8201} we have 
		$\mathcal{W}(I_3,I_1)\preceq_{C,M}1$;
		since 
		$I_2<I_3<J_2<J_3$, $|I_3|_{\mathcal{T}}\preceq|I_2|_{\mathcal{T}}$, $|I_3|_{\mathcal{T}}\leq|I|_{\mathcal{T}}\asymp|J|_{\mathcal{T}}\preceq|J_2|_{\mathcal{T}}$ and $|I_3|_{\mathcal{T}}\leq|I|_{\mathcal{T}}<1$, by Corollary \ref{c8201} we have 
		$\mathcal{W}(I_3,J_3)\preceq_{C,M}1$. Then $\mathcal{W}(I_3,I_1\cup J_3)\preceq_{C,M}1$. Similarly, we can obtain $\mathcal{W}(J_1,I_1\cup J_3)\preceq_{C,M}1$.
		Thus 
		\begin{equation}
			\label{e20260221d}\mathcal{W}(I_3\cup J_1,I_1\cup J_3)\preceq_{C,M}1.
		\end{equation}
		
		Consider the family of curves in $\mathcal{R}$ connecting two components of $\mathbb{C}\setminus\mathcal{R}$, written as $\mathcal{F}^v(\mathcal{R})$. Since $\mathcal{W}(\mathcal{R}_1^{(1)})\succeq1$, $\mathcal{W}(\mathcal{R}_1^{(2)})\succeq1$, $\mathcal{W}(\mathcal{R}_2^{(1)})\succeq1$ and $\mathcal{W}(\mathcal{R}_2^{(2)})\succeq1$,
		we have that the width of curves in $\mathcal{F}^v(\mathcal{R})$ not intersecting $I_3\cup J_1$ or $I_1\cup J_3$ is $\preceq1$. 
		By combining (\ref{e20260221d}), it follows that $\mathcal{W}(\mathcal{F}^v(\mathcal{R}))\preceq_{C,M}1$. Thus 
		${\rm mod}(\mathcal{R})\succeq_{C,M}1$.
		
	\end{proof}
	
	For all $n\geq2$, we let $T_1$ be the interval of $\mathfrak{D}_{n}(\alpha)$ containing $P_{\alpha}(c_{0})$, that is $T_1=I_{[x_{q_{n+1}-1},x_{q_n-1}]}$, and let $T_2$ be the interval of $\mathfrak{D}_{n+10}(\alpha)$ containing $P_{\alpha}(c_{0})$, that is $T_2=I_{[x_{q_{n+11}-1},x_{q_{n+10}-1}]}$.
	Let $I_{2,l}$ and $I_{2,r}$ be the two components of $\overline{T_1\setminus T_2}$ such that $I_{2,l}<T_2<I_{2,r}$. Let $a$ and $b$ be two endpoints of $I_{2,l}$, where $a$ is also an endpoint of $T_1$; let $c$ and $d$ be two endpoints of $I_{2,r}$, where $d$ is also an endpoint of $T_1$. Let $\gamma_{a,d}$ be the geodesic connecting $a$ and $d$ with respect to the hyperbolic metric on $\hat{\mathbb{C}}\setminus\overline{\Delta_{\alpha}}$; let $\gamma_{b,c}$ be the geodesic connecting $b$ and $c$ with respect to the hyperbolic metric on $\hat{\mathbb{C}}\setminus\overline{\Delta_{\alpha}}$.
	By Lemma \ref{l9301} we have that $I_{2,l}$ and $I_{2,r}$ are regular intervals rel $\partial\hat{\Delta}_{\alpha}^{m}$ ($m\leq n$). Let $\hat{I}_{2,l}$ and $\hat{I}_{2,r}$ be projections of $I_{2,l}$ and $I_{2,r}$ onto $\partial\hat{\Delta}_{\alpha}^{m}$, respectively.
	We let $\mathcal{G}$ be the interior of the union $\mathcal{G}^-_{\hat{\Delta}_{\alpha}^m}(\hat{I}_{2,l},\hat{I}_{2,r})\cup\hat{\mathcal{G}}^+_{\overline{\Delta_{\alpha}}}(\hat{I}_{2,l},\hat{I}_{2,r})$, where $\hat{\mathcal{G}}^+_{\overline{\Delta_{\alpha}}}(\hat{I}_{2,l},\hat{I}_{2,r})$ is the rectangle bounded by $\hat{I}_{2,l}$, $\hat{I}_{2,r}$, $\gamma_{a,d}$ and $\gamma_{b,c}$ with $\partial^{h,0}\hat{\mathcal{G}}^+_{\overline{\Delta_{\alpha}}}(\hat{I}_{2,l},\hat{I}_{2,r})=\hat{I}_{2,l}$ and $\partial^{h,1}\hat{\mathcal{G}}^+_{\overline{\Delta_{\alpha}}}(\hat{I}_{2,l},\hat{I}_{2,r})=\hat{I}_{2,r}$. See Figure \ref{f20260720c}.
	\begin{figure}
		\centering
		\includegraphics[scale=0.5]{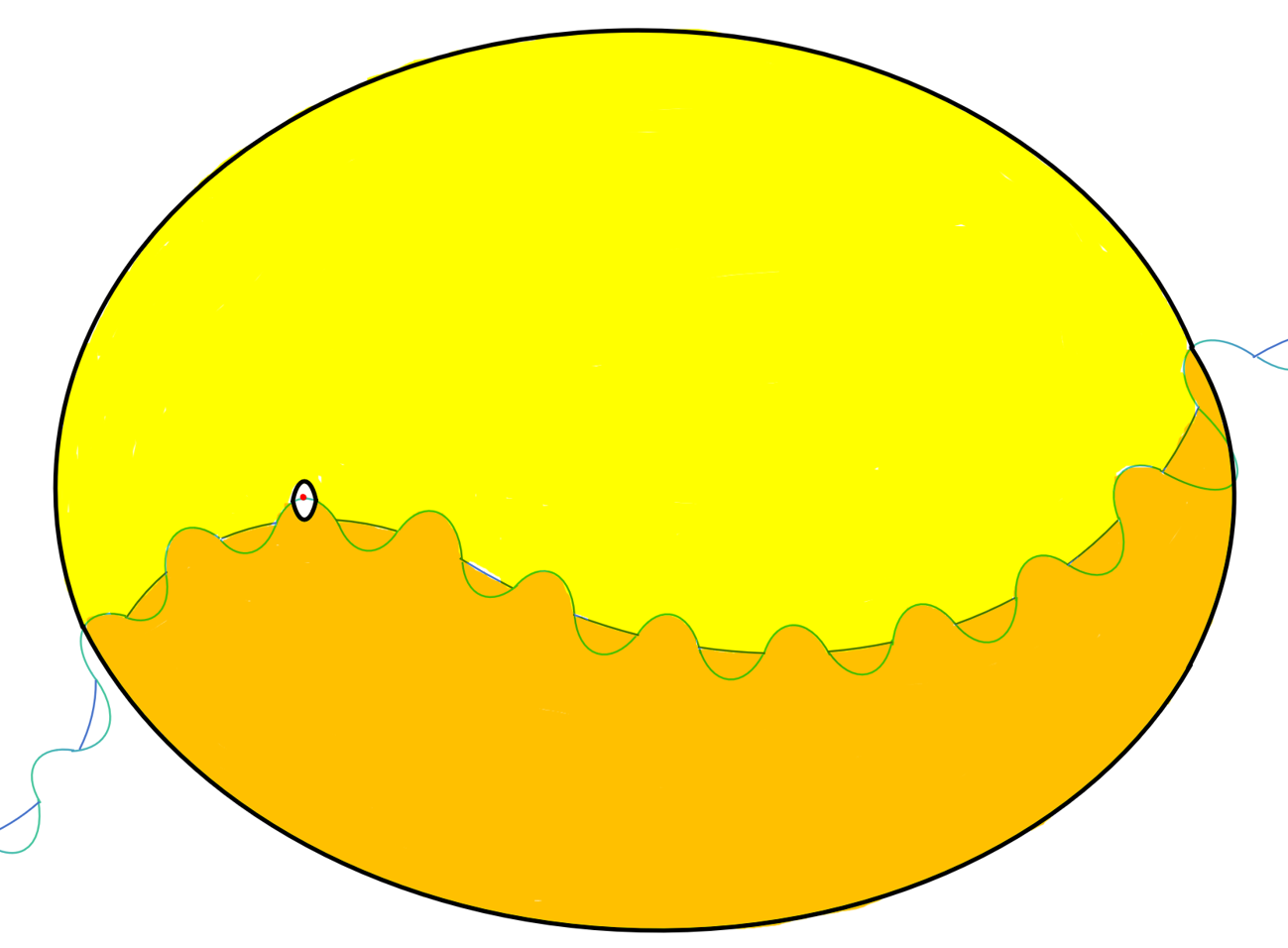}
		\caption{The orange area represents $\mathcal{G}^-_{\hat{\Delta}_{\alpha}^m}(\hat{I}_{2,l},\hat{I}_{2,r})$; the yellow area represents $\hat{\mathcal{G}}^+_{\overline{\Delta_{\alpha}}}(\hat{I}_{2,l},\hat{I}_{2,r})$; the red point represents $P_{\alpha}(c_0)$.}
		\label{f20260720c}
	\end{figure}
	We have the following lemma.
	
	\begin{lemma}
		\label{l832}${\rm mod}(\mathcal{G})\succeq1$.
	\end{lemma}
	\begin{proof}
		We let $T_3$ and $T_4$ be two intervals of $\mathfrak{D}_{n+10}(\alpha)$ such that $T_3\cup T_4\subseteq T_1$ and $T_3\#T_2\#T_4$;
		we let $T_5$ and $T_6$ be two intervals of $\mathfrak{D}_{n+10}(\alpha)$ such that $T_5\cup T_6\subseteq T_1$, $T_5<T_6<\partial\Delta_{\alpha}\setminus T_1$ and both $T_5$ and $T_6$ have a common endpoint with $T_1$.
		Observe that all endpoints of $I_{2,l}, I_{2,r}, T_1, T_2, T_3, T_4, T_5, T_6$ are among 
		$x_{q_{n+1}-1}, x_{q_{n+11}-1},\ x_{q_{n+10}-1}, x_{q_{n}-1},\ x_{q_{n+11}-q_{n+10}-1},\ x_{2q_{n+10}-1},\ x_{q_{n+10}-q_{n+9}-1},\ x_{q_{n+1}+q_{n+10}-1},\\ x_{q_{n}+q_{n+11}-q_{n+10}-1}$.
		By Lemma \ref{l9301} we have that $I_{2,l}, I_{2,r}, T_1, T_2, T_3, T_4, T_5, T_6$ are well-grounded with respect to $\hat{\Delta}_{\alpha}^m$. We consider projections of $T_1$, $T_2$, $T_3$, $T_4$, $T_5$, $T_6$, $T_7:=\overline{I_{2,l}\setminus(T_5\cup T_3)}$ and $T_8:=\overline{I_{2,r}\setminus(T_6\cup T_4)}$ onto $\partial\hat{\Delta}_{\alpha}^{m}$, written as $\hat{T}_1$, $\hat{T}_2$, $\hat{T}_3$, $\hat{T}_4$, $\hat{T}_5$, $\hat{T}_6$,
		$\hat{T}_7$ and $\hat{T}_8$,
		respectively.
		Then by (P1) in Section \ref{s2.3} we have
		$$|\hat{T}_1|_{\hat{\Delta}_{\alpha}^m}=\frac{1}{M^{n-m}}$$
		and
		$$\frac{1}{M^{n+10-m}}\leq|\hat{T}_j|_{\hat{\Delta}_{\alpha}^m}\leq\frac{1}{M^{n+9-m}},\ 2\leq j\leq6.$$
		Observe that $\hat{I}_{2,r}\subsetneqq\hat{T}_1$ and
		$\hat{I}_{2,r}$ contains at least three intervals 
		of level $n+10$. Then
		$\frac{1}{M^{n+10-m}}\leq|\hat{I}_{2,r}|_{\hat{\Delta}_{\alpha}^m}<\frac{1}{M^{n-m}}$ and $\frac{1}{M^{n+10-m}}\leq|T_8|_{\hat{\Delta}_{\alpha}^m}<\frac{1}{M^{n-m}}$. Similarly, $\frac{1}{M^{n+10-m}}\leq|\hat{I}_{2,l}|_{\hat{\Delta}_{\alpha}^m}<\frac{1}{M^{n-m}}$ and $\frac{1}{M^{n+10-m}}\leq|T_7|_{\hat{\Delta}_{\alpha}^m}<\frac{1}{M^{n-m}}$. 
		Thus
		\begin{equation}
			\label{e831}|\hat{I}_{2,l}|_{\hat{\Delta}_{\alpha}^m}\asymp|\hat{I}_{2,r}|_{\hat{\Delta}_{\alpha}^m}\asymp|\hat{T}_2|_{\hat{\Delta}_{\alpha}^m}<1,
		\end{equation}
		\begin{equation}
			\label{e833}
			|\hat{T}_3|_{\hat{\Delta}_{\alpha}^m}\asymp|\hat{T}_4|_{\hat{\Delta}_{\alpha}^m}\asymp|\hat{T}_2|_{\hat{\Delta}_{\alpha}^m}<1,
		\end{equation}
		\begin{equation}
			\label{e20260222a}
			|\hat{T}_3|_{\hat{\Delta}_{\alpha}^m}\asymp|\hat{T}_5|_{\hat{\Delta}_{\alpha}^m}\preceq|\hat{T}_7|_{\hat{\Delta}_{\alpha}^m}<|\hat{I}_{2,l}|_{\hat{\Delta}_{\alpha}^m}<1
		\end{equation}
		and
		\begin{equation}
			\label{e20260222b}
			|\hat{T}_4|_{\hat{\Delta}_{\alpha}^m}\asymp|\hat{T}_6|_{\hat{\Delta}_{\alpha}^m}\preceq|\hat{T}_8|_{\hat{\Delta}_{\alpha}^m}<|\hat{I}_{2,r}|_{\hat{\Delta}_{\alpha}^m}<1.
		\end{equation}
		Since $\hat{T}_2\subseteq\overline{\hat{T}_1\setminus(\hat{T}_5\cup\hat{T}_6)}\subseteq\hat{T}_1$, we have
		$$\frac{1}{M^{n+10-m}}\leq\left|\overline{\hat{T}_1\setminus(\hat{T}_5\cup\hat{T}_6)}\right|_{\hat{\Delta}_{\alpha}^m}\leq\frac{1}{M^{n-m}}$$
		and hence
		\begin{equation}
			\label{e832}
			|\hat{T}_5|_{\hat{\Delta}_{\alpha}^m}\asymp|\hat{T}_6|_{\hat{\Delta}_{\alpha}^m}\asymp\left|\overline{\hat{T}_1\setminus(\hat{T}_5\cup\hat{T}_6)}\right|_{\hat{\Delta}_{\alpha}^m}<1.
		\end{equation}
		Together with (\ref{e831}), (\ref{e833}) and (\ref{e832}), Lemma \ref{l811} gives that 
		$$\mathcal{W}(\mathcal{F}_{\hat{\Delta}_{\alpha}^m}^-(\hat{I}_{2,l},\hat{I}_{2,r}))\succeq1,\ \mathcal{W}(\mathcal{F}_{\hat{\Delta}_{\alpha}^m}^-(\hat{T}_3,\hat{T}_4))\succeq1\ {\rm and}\ \mathcal{W}(\mathcal{F}_{\hat{\Delta}_{\alpha}^m}^-(\hat{T}_5,\hat{T}_6))\succeq1.$$
		By (\ref{e8222}) we have that
		\begin{equation}
			\label{e20260221b}\mathcal{W}(\mathcal{G}_{\hat{\Delta}_{\alpha}^m}^-(\hat{I}_{2,l},\hat{I}_{2,r}))\succeq1,\ \mathcal{W}(\mathcal{G}_{\hat{\Delta}_{\alpha}^m}^-(\hat{T}_3,\hat{T}_4))\succeq1\ {\rm and}\ \mathcal{W}(\mathcal{G}_{\hat{\Delta}_{\alpha}^m}^-(\hat{T}_5,\hat{T}_6))\succeq1.
		\end{equation}
		Together with (\ref{e831}), (\ref{e833}) and (\ref{e832}), Lemma \ref{l831} gives that
		$$\mathcal{W}(\mathcal{F}_{\overline{\Delta_{\alpha}}}^+(I_{2,l},I_{2,r}))\succeq1,\ \mathcal{W}(\mathcal{F}_{\overline{\Delta_{\alpha}}}^+(T_5,T_6))\succeq1\ {\rm and}\ \mathcal{W}(\mathcal{F}_{\overline{\Delta_{\alpha}}}^+(T_3,T_4))\succeq1.$$
		By (\ref{e8221}) we have that
		$$\mathcal{W}(\mathcal{G}_{\overline{\Delta_{\alpha}}}^+(I_{2,l},I_{2,r}))\succeq1,\ \mathcal{W}(\mathcal{G}_{\overline{\Delta_{\alpha}}}^+(T_5,T_6))\succeq1\ {\rm and}\ \mathcal{W}(\mathcal{G}_{\overline{\Delta_{\alpha}}}^+(T_3,T_4))\succeq1.$$
		Since $I_{2,l}, I_{2,r}, T_3, T_4, T_5, T_6$ are well-grounded with respect to $\hat{\Delta}_{\alpha}^m$,
		by (P2) in Section \ref{s2.3}, we have that $\hat{\mathcal{G}}_{\overline{\Delta_{\alpha}}}^+(\hat{I}_{2,l},\hat{I}_{2,r})$, $\hat{\mathcal{G}}_{\overline{\Delta_{\alpha}}}^+(\hat{T}_5,\hat{T}_6)$ and $\hat{\mathcal{G}}_{\overline{\Delta_{\alpha}}}^+(\hat{T}_3,\hat{T}_4)$ are protected in $\mathcal{G}_{\overline{\Delta_{\alpha}}}^+(I_{2,l},I_{2,r})$, $\mathcal{G}_{\overline{\Delta_{\alpha}}}^+(T_5,T_6)$ and $\mathcal{G}_{\overline{\Delta_{\alpha}}}^+(T_3,T_4)$, respectively.
		It follows from Lemma \ref{l861} that
		\begin{equation}
			\label{e20260221c}\mathcal{W}(\hat{\mathcal{G}}_{\overline{\Delta_{\alpha}}}^+(\hat{I}_{2,l},\hat{I}_{2,r}))\succeq1,\ \mathcal{W}(\hat{\mathcal{G}}_{\overline{\Delta_{\alpha}}}^+(\hat{T}_5,\hat{T}_6))\succeq1\ {\rm and}\ \mathcal{W}(\hat{\mathcal{G}}_{\overline{\Delta_{\alpha}}}^+(\hat{T}_3,\hat{T}_4))\succeq1.
		\end{equation}
		At last, together with (\ref{e831}), (\ref{e20260222a}), (\ref{e20260222b}), (\ref{e20260221b}) and (\ref{e20260221c}), Lemma \ref{l7311} gives
		$${\rm mod}(\mathcal{G})\succeq1.$$
		
	\end{proof}

	\begin{corollary}
		\label{c831}For all $n\geq2$ and $k\gg1$, we let $T_1$ be the interval of $\mathfrak{D}_{n}(\alpha)$ containing $P_{\alpha}(c_0)$, that is $T_1=I_{[x_{q_{n+1}-1},x_{q_n-1}]}$, and let $T_2$ be the interval of $\mathfrak{D}_{n+k}(\alpha)$ containing $P_{\alpha}(c_0)$, that is $T_2=I_{[x_{q_{n+k+1}-1},x_{q_{n+k}-1}]}$.
		Let $T_{2,l}$ and $T_{2,r}$ be the two components of $\overline{T_1\setminus T_2}$ such that $T_{2,l}<T_2<T_{2,r}$, and let
		$\hat{T}_{2,l}$ and $\hat{T}_{2,r}$ be projections of $T_{2,l}$ and $T_{2,r}$ onto $\partial\hat{\Delta}_{\alpha}^{m}$ {\rm(}$m\leq n${\rm)}, respectively.
		We let $\mathcal{G}$ be the interior of the union $\mathcal{G}^-_{\hat{\Delta}_{\alpha}^m}(\hat{T}_{2,l},\hat{T}_{2,r})\cup\hat{\mathcal{G}}^+_{\overline{\Delta_{\alpha}}}(\hat{T}_{2,l},\hat{T}_{2,r})$.
		Then
		${\rm mod}(\mathcal{G})\succeq k\gg1$.
	\end{corollary}
	\begin{proof}
		For all $1\leq s\leq[\frac{k}{10}]$, we let $I_{s,l}$ and  $I_{s,r}$ be the two components of $$\overline{I_{[x_{q_{n+1+10(s-1)}-1},x_{q_{n+10(s-1)}-1}]}\setminus  I_{[x_{q_{n+1+10s}-1},x_{q_{n+10s}-1}]}}.$$
		For all $1\leq s\leq[\frac{k}{10}]$,
		we define $\mathcal{G}_s$ as the interior of the union $\mathcal{G}^-_{\hat{\Delta}_{\alpha}^{m}}(\hat{I}_{s,l},\hat{I}_{s,r})\cup\hat{\mathcal{G}}^+_{\overline{\Delta_{\alpha}}}(\hat{I}_{s,l},\hat{I}_{s,r})$, where $\hat{I}_{s,l}$ and $\hat{I}_{s,r}$ be projections of $I_{s,l}$ and $I_{s,r}$ onto $\partial\hat{\Delta}_{\alpha}^{m}$, respectively.
		By Lemma \ref{l832}, we have ${\rm mod}(\mathcal{G}_s)\succeq1$. Thus by Gr\"otzsch inequality we have
		$${\rm mod}(\mathcal{G})\geq\sum_{s=1}^{[\frac{k}{10}]}{\rm mod}(\mathcal{G}_s)\succeq k\gg1.$$

	\end{proof}
	
	For all $n\geq2$ and $k\geq10$, we let $T_1$ be the interval of $\mathfrak{D}_{n}(\alpha)$ containing $P_{\alpha}(c_0)$ and let $T_2$ be the interval of $\mathfrak{D}_{n+k}(\alpha)$ containing $P_{\alpha}(c_0)$.
	Let $I_{2,l}$ and $I_{2,r}$ be the two components of $\overline{T_1\setminus T_2}$ such that $I_{2,l}<T_2<I_{2,r}$, and let
	$\hat{I}_{2,l}$ and $\hat{I}_{2,r}$ be projections of $I_{2,l}$ and $I_{2,r}$ onto $\partial\hat{\Delta}_{\alpha}^{m}$ ($m\leq n$), respectively.
	We let $\mathcal{G}$ be the interior of the union $\mathcal{G}^-_{\hat{\Delta}_{\alpha}^m}(\hat{I}_{2,l},\hat{I}_{2,r})\cup\hat{\mathcal{G}}^+_{\overline{\Delta_{\alpha}}}(\hat{I}_{2,l},\hat{I}_{2,r})$.
	Then by Lemma \ref{l832} we have
	${\rm mod}(\mathcal{G})\succeq1$. Moreover, if $k\gg1$, then by Corollary \ref{c831} ${\rm mod}(\mathcal{G})\gg1$.
	
	We assume $\mathcal{G}\subseteq\mathbb{C}$ and $P_{\alpha}(c_0)$ is contained in the bounded component of $\mathbb{C}\setminus\mathcal{G}$. Let $\mathcal{G}^{-1}$ be the preimage of $\mathcal{G}$ under $P_{\alpha}$. Then $P_{\alpha}|_{\mathcal{G}^{-1}}$ is a covering of degree $2$ from $\mathcal{G}^{-1}$ to $\mathcal{G}$. Thus
	\begin{equation}
		\label{e20260221a}{\rm mod}(\mathcal{G}^{-1})=\frac{{\rm mod}(\mathcal{G})}{2}. 
	\end{equation}
	Let
	$\hat{I}_{2,l}^{-1}\subseteq\partial\hat{\Delta}_{\alpha}^{m,-1}$
	be the interval such that $P_{\alpha}(\hat{I}_{2,l}^{-1})=\hat{I}_{2,l}$; let $\hat{I}_{2,r}^{-1}\subseteq\partial\hat{\Delta}_{\alpha}^{m,-1}$
	be the interval such that $P_{\alpha}(\hat{I}_{2,r}^{-1})=\hat{I}_{2,r}$.
	Then $\overline{\mathcal{G}^{-1}\cap\partial\hat{\Delta}_{\alpha}^{m,-1}}=\hat{I}_{2,l}^{-1}\cup\hat{I}_{2,r}^{-1}$.
	Thus $\overline{\mathcal{G}^{-1}\setminus\hat{\Delta}_{\alpha}^{m,-1}}$ can be viewed as 
	a rectangle with $\partial^{h,0}\overline{\mathcal{G}^{-1}\setminus\hat{\Delta}_{\alpha}^{m,-1}}=\hat{I}_{2,l}^{-1}$ and
	$\partial^{h,1}\overline{\mathcal{G}^{-1}\setminus\hat{\Delta}_{\alpha}^{m,-1}}=\hat{I}_{2,r}^{-1}$.
	Since every Jordan curve, not homotopic to a point in $\mathcal{G}^{-1}$, overflows some element of $\mathcal{F}^{full}(\overline{\mathcal{G}^{-1}\setminus\hat{\Delta}_{\alpha}^{m,-1}})$, we have $\mathcal{W}(\overline{\mathcal{G}^{-1}\setminus\hat{\Delta}_{\alpha}^{m,-1}})\geq{\rm mod}(\mathcal{G}^{-1})$. Let $I_{2,l}^{-1}$ and $I_{2,r}^{-1}$ be pre-images of $I_{2,l}$ and $I_{2,r}$ contained in $\partial\Delta_{\alpha}$ under $P_{\alpha}$ respectively (Note that $\hat{I}_{2,l}^{-1}$ and $\hat{I}_{2,r}^{-1}$ are pre-images of $\hat{I}_{2,l}$ and $\hat{I}_{2,r}$ contained in $\partial\hat{\Delta}_{\alpha}^{m,-1}$ under $P_{\alpha}$ respectively).
	We denote by $\mathcal{R}_{+}^{(-1)}$ the rectangle bounded by $I_{2,l}^{-1}\cup I_{2,r}^{-1}\cup\partial^{v,0}\overline{\mathcal{G}^{-1}\setminus\hat{\Delta}_{\alpha}^{m,-1}}\cup\partial^{v,1}\overline{\mathcal{G}^{-1}\setminus\hat{\Delta}_{\alpha}^{m,-1}}$ with
	$\partial^{h,0}\mathcal{R}_{+}^{(-1)}=I_{2,l}^{-1}$ and $\partial^{h,1}\mathcal{R}_{+}^{(-1)}=I_{2,r}^{-1}$. See Figure \ref{f20260720d}.
	\begin{figure}
		\centering
		\includegraphics[scale=0.6]{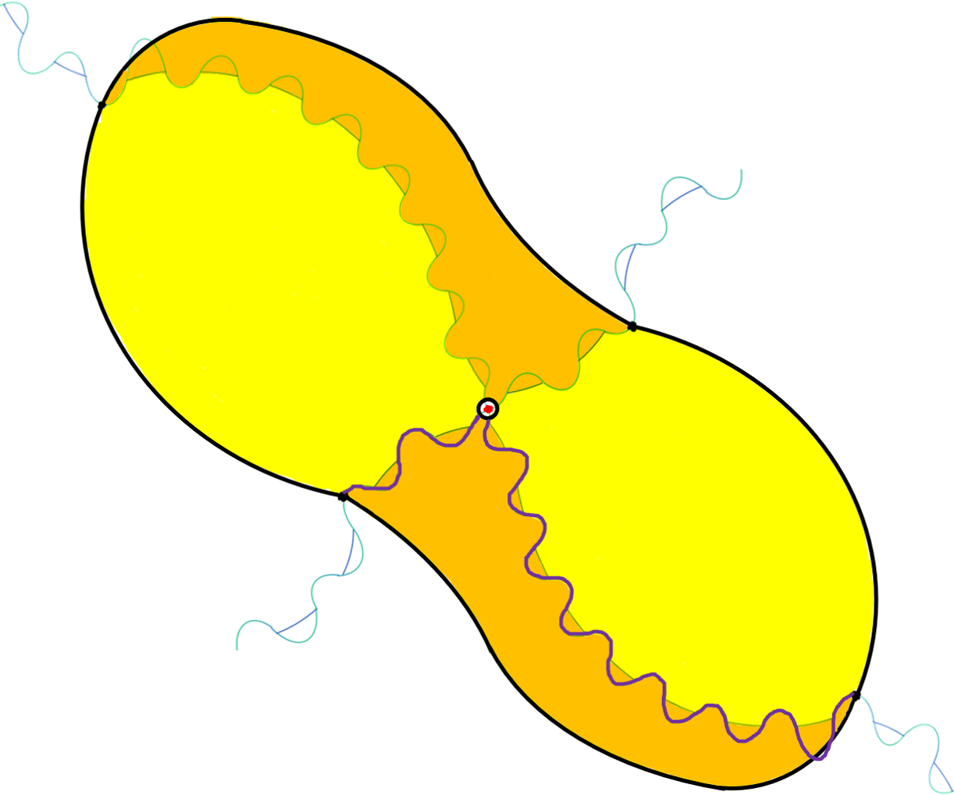}
		\caption{The union of orange areas and yellow areas represents $\mathcal{G}^{-1}$; the yellow (resp. orange) area represents pre-image of $\hat{\mathcal{G}}^+_{\overline{\Delta_{\alpha}}}(\hat{I}_{2,l},\hat{I}_{2,r})$ (resp. $\mathcal{G}^-_{\hat{\Delta}_{\alpha}^m}(\hat{I}_{2,l},\hat{I}_{2,r})$) under $P_{\alpha}$; the red point represents $c_0$; the short purple arc represents $I_{2,l}^{-1}$; the long purple arc represents $I_{2,r}^{-1}$.}
		\label{f20260720d}
	\end{figure}
	By Lemma \ref{l9301} $I_{2,l}$ and $I_{2,r}$ are well-grounded intervals with respect to $\partial\Delta_{\alpha}^m$. Then 
	by (P2) in Section \ref{s2.3}, $\hat{\mathcal{G}}^+_{\overline{\Delta_{\alpha}}}(\hat{I}_{2,l},\hat{I}_{2,r})$ is protected in $\mathcal{G}^+_{\overline{\Delta_{\alpha}}}(I_{2,l},I_{2,r})$.
	Pulling back these protections by $P_{\alpha}$, we obtain that 
	$\overline{\mathcal{G}^{-1}\setminus\hat{\Delta}_{\alpha}^{m,-1}}$ is protected in
	$\mathcal{R}_{+}^{(-1)}$.
	By Lemma \ref{l861}
	$$\mathcal{W}(\mathcal{R}_{+}^{(-1)})\asymp\mathcal{W}(\overline{\mathcal{G}^{-1}\setminus\hat{\Delta}_{\alpha}^{m,-1}})\geq{\rm mod}(\mathcal{G}^{-1}).$$
	By combining (\ref{e20260221a}), it follows that
	\begin{equation}
		\label{e85a}\mathcal{W}(\mathcal{R}_{+}^{(-1)})\succeq{\rm mod}(\mathcal{G}).
	\end{equation}

	\section{Dynamics near the Siegel disk $\Delta_{\alpha}$\label{S3}}
	
	\subsection{Geodesic mountain ranges of Siegel disk $\Delta_{\alpha}$}
	For all $n\geq4$, we denote by ${\rm EP}(\mathfrak{D}_n(\alpha))$ the set of endpoints of all intervals in $\mathfrak{D}_n(\alpha)$. For any $x\in{\rm EP}(\mathfrak{D}_n(\alpha))$, we denote by $I_{l,x}^n,I_{r,x}^n\in\mathfrak{D}_n(\alpha)$ be a pair of adjacent intervals with the common endpoint $x$ such that $y<x<z$, where $y$ and $z$ are the two other endpoints of $I_{l,x}^n$ and $I_{r,x}^n$, respectively. We denote by $\gamma_x^n$ the geodesic connecting $y$ and $z$ with respect to the hyperbolic metric on $\hat{\mathbb{C}}\setminus\overline{\Delta_{\alpha}}$. 
	Recall that the notation $\overline{O_{I_{l,x}^n\cup I_{r,x}^n}(\gamma_x^n)}$ means the closed region bounded by $\gamma_x^n\cup I_{l,x}^n\cup I_{r,x}^n$ not intersecting $\Delta_{\alpha}$.
	We say that $\Omega_x^n:=\overline{O_{I_{l,x}^n\cup I_{r,x}^n}(\gamma_x^n)}\setminus(I_{l,x}^n\cup I_{r,x}^n)$ is a (geodesic) mountain of level $n$ centering at $x$.
	For all $n\geq4$, we set
	$$\Lambda_{\alpha}^n:=\bigcup_{x\in{\rm EP}(\mathfrak{D}_{n}(\alpha))}\Omega_x^{n},$$
	which is called the (geodesic) mountain range of level $n$ of the Siegel disk $\Delta_{\alpha}$.
	The difference $$\tilde{\Omega}_{c_0}^n:=\Omega_{c_0}^n\setminus\hat{\Delta}_{\alpha}^n$$ is called the main (geodesic) mountain of level $n$ of $\Delta_{\alpha}$.

	\subsection{Geodesics of deep levels are non-winding}
	
	\begin{lemma}
		\label{l04231}There exists a positive integer $\mathfrak{q}_1\geq6$, independent of $\alpha$, such that for all $n\geq\mathfrak{q}_1$ and all adjacent intervals $I_3,J_3\in\mathfrak{D}_{n}(\alpha)$ with a common endpoint $c_0$, $\gamma_{c_0}^{n}$ is non-winding based on $I_3\cup J_3$ with respect to the closed quasidisk $\overline{\Delta_{\alpha}}$.
	\end{lemma}
	\begin{proof}
		We consider $\mathfrak{D}_4(\alpha)$. Observe that $\mathfrak{D}_4(\alpha)$ has at least four intervals. Let $T$ be the interval of $\mathfrak{D}_4(\alpha)$ containing the critical value $P_{\alpha}(c_0)$, that is $T=[x_{q_{5}-1},x_{q_{4}-1}]$. Let $T_1$ and $T_2$ be the two intervals of $\mathfrak{D}_4(\alpha)$ having a common endpoint $c_0$.
		Then $T$, $T_1$ and $T_2$ are pairwise different.
		We firstly consider the case: the harmonic measure of $T$ in $(\hat{\mathbb{C}}\setminus\overline{\Delta_{\alpha}},\infty)$ is less than $\frac{1}{2}$. We let $T_4$ be the interval of $\mathfrak{D}_{14}(\alpha)$ containing $P_{\alpha}(c_0)$.
		Let $I_{4,l}$ and $I_{4,r}$ be the two components of $\overline{T\setminus T_4}$ such that $I_{4,l}<T_4<I_{4,r}$. 
		It follows from Lemma \ref{l9301} that $T$, $T_4$, $I_{4,l}$ and $I_{4,r}$ are regular intervals of $\partial\hat{\Delta}_{\alpha}^{4}$. We consider projections of $T$, $T_4$, $I_{4,l}$ and $I_{4,r}$ onto $\partial\hat{\Delta}_{\alpha}^{4}$, written as $\hat{T}$, $\hat{T}_4$, $\hat{I}_{4,l}$ and $\hat{I}_{4,r}$, respectively. 
		We define $\mathcal{G}$ as the interior of the union $\mathcal{G}^-_{\hat{\Delta}_{\alpha}^4}(\hat{I}_{4,l},\hat{I}_{4,r})\cup\hat{\mathcal{G}}^+_{\overline{\Delta_{\alpha}}}(\hat{I}_{4,l},\hat{I}_{4,r})$. 
		By Lemma \ref{l832} ${\rm mod}(\mathcal{G})\succeq1$. Since the harmonic measure of $T$ in $(\hat{\mathbb{C}}\setminus\overline{\Delta_{\alpha}},\infty)$ is less than $\frac{1}{2}$, both $\partial^{v,0}\hat{\mathcal{G}}^+_{\overline{\Delta_{\alpha}}}(\hat{I}_{4,l},\hat{I}_{4,r})$ and $\partial^{v,1}\hat{\mathcal{G}}^+_{\overline{\Delta_{\alpha}}}(\hat{I}_{4,l},\hat{I}_{4,r})$ are non-winding based on $T$ with respect to $\overline{\Delta_{\alpha}}$. Observe also that $P_{\alpha}(c_0)\in\hat{T}_4$ (Refer to Lemma \ref{l26821}). Thus $\mathcal{G}\subseteq\mathbb{C}$ and $P_{\alpha}(c_0)$ is contained in the bounded component of $\mathbb{C}\setminus\mathcal{G}$.
		We define $\mathcal{G}^{-1}:=P_{\alpha}^{-1}(\mathcal{G})$. Then $P_{\alpha}|_{\mathcal{G}^{-1}}$ is a covering of degree $2$ from $\mathcal{G}^{-1}$ to $\mathcal{G}$. 
		Let $I_{4,l}^{-1}$ and $I_{4,r}^{-1}$ be pre-images of $I_{4,l}$ and $I_{4,r}$ contained in $\partial\Delta_{\alpha}$ under $P_{\alpha}$, respectively.
		As the last paragraph in Section \ref{s2.4},
		we denote by $\mathcal{R}_{+}^{(-1)}$ the bounded rectangle bounded by $I_{4,l}^{-1}\cup I_{4,r}^{-1}\cup\partial^{v,0}\overline{\mathcal{G}^{-1}\setminus\hat{\Delta}_{\alpha}^{4,-1}}\cup\partial^{v,1}\overline{\mathcal{G}^{-1}\setminus\hat{\Delta}_{\alpha}^{4,-1}}$  with
		$\partial^{h,0}\mathcal{R}_{+}^{(-1)}=I_{4,l}^{-1}$ and $\partial^{h,1}\mathcal{R}_{+}^{(-1)}=I_{4,r}^{-1}$. 
		By (\ref{e85a}) we have
		$$\mathcal{W}(\mathcal{R}_{+}^{(-1)})\succeq{\rm mod}(\mathcal{G})\succeq1.$$
		
		Let $T_4^{-1}$
		be the component of $P_{\alpha}^{-1}(T_4)$ contained in $\partial\Delta_{\alpha}$. Then $I_{4,l}^{-1}\#T_4^{-1}\#I_{4,r}^{-1}$ and $T_4^{-1}=I_{[c_0,x_{q_{15}}]}\cup I_{[c_0,x_{q_{14}}]}$.
		Let $I_3,J_3\in\mathfrak{D}_{n}(\alpha)$ ($n\geq17$) be the two adjacent intervals having a common endpoint $c_0$.
		Let $I_l$ and $I_r$ be the two components of $\overline{(I_{[c_0,x_{q_{15}}]}\cup I_{[c_0,x_{q_{14}}]})\setminus (I_3\cup J_3)}$. It follows from Lemma \ref{l7301} that we can choose a sufficiently large positive integer $\mathfrak{q}_1\geq17$, independent of $\alpha$, so that for all $n\geq\mathfrak{q}_1$, $\mathcal{W}_{\overline{\Delta_{\alpha}}}^+(I_l,I_r)\gg1$. We view the closed region $D_1$ on $\hat{\mathbb{C}}$ bounded by $\overline{\partial\Delta_{\alpha}\setminus(I_3\cup J_3)}\cup\gamma_{c_0}^{n}$ not containing $\Delta_{\alpha}$ as a rectangle with $\partial^{h,0}D_1=I_l$ and $\partial^{h,1}D_1=I_r$. By Lemma \ref{l825a} $\mathcal{W}(D_1)\geq\mathcal{W}_{\overline{\Delta_{\alpha}}}^+(I_l,I_r)-\frac{1}{2}\gg1$. This implies that $\gamma_{c_0}^{n}$ doesn't cross $\mathcal{R}_{+}^{(-1)}$, for if not, then $D_1$ crosses $\mathcal{R}_{+}^{(-1)}$, which contradicts $\mathcal{W}(\mathcal{R}_{+}^{(-1)})\succeq1$. 
		
		Let $T^{-1}$ (resp. $\hat{T}^{-1}$) be the pre-image of $T$ (resp. $\hat{T}$) contained in $\partial\Delta_{\alpha}$ (resp. $\partial\hat{\Delta}_{\alpha}^{4,-1}$) under $P_{\alpha}$.  
		Since $T$ and $\hat{T}$ are homotopic on $\mathbb{C}^*$, we have that $T^{-1}$ and $\hat{T}^{-1}$ are homotopic on $\mathbb{C}^*$.
		It is also easy to see that $\hat{T}^{-1}$ and the outer vertical boundary $\partial^{v,0}\overline{\mathcal{G}^{-1}\setminus\hat{\Delta}_{\alpha}^{4,-1}}$ have the same endpoints, and they are homotopic on $\mathbb{C}^*$, for $\hat{T}$ is contained in the filling-in of $\mathbb{C}\setminus\mathcal{G}$.
		Thus $T^{-1}$ and $\partial^{v,0}\overline{\mathcal{G}^{-1}\setminus\hat{\Delta}_{\alpha}^{4,-1}}$ are homotopic on $\mathbb{C}^*$. This implies that
		$\mathcal{R}_{+}^{(-1)}$
		is non-winding based on $T^{-1}$ with respect to $\overline{\Delta_{\alpha}}$.
		Thus $\gamma_{c_0}^{n}$ is also non-winding based on $T^{-1}$ (also $I_3\cup J_3$) with respect to $\overline{\Delta_{\alpha}}$.
		
		Next, we consider the other case: the harmonic measure of $T$
		in $(\hat{\mathbb{C}}\setminus\overline{\Delta_{\alpha}},\infty)$ is greater than or equal to $\frac{1}{2}$. In this case, 
		the harmonic measure of $T_1\cup T_2$ is less than or equal to $\frac{1}{2}$. For any positive integer $n\geq6$, we take a pair of adjacent intervals $I_3, J_3\in\mathfrak{D}_n$ such that $I_3$ and $J_3$ have a common endpoint $c_0$. Then $I_3\cup J_3\subsetneqq T_1\cup T_2$ and hence the harmonic measure of $I_3\cup J_3$ is less than $\frac{1}{2}$. It follows that $\gamma_{c_0}^{n}$ is non-winding based on $I_3\cup J_3$ with respect to $\overline{\Delta_{\alpha}}$.
	\end{proof}
	
	We denote by $U_0$ the component of $P_{\alpha}^{-1}(\Delta_{\alpha})$ not containing $0$. 
	In the proof of Lemma \ref{l04231}, since $$P_{\alpha}(c_0)\not\in\partial^{v,1}\mathcal{G}^-_{\hat{\Delta}_{\alpha}^4}(\hat{I}_{4,l},\hat{I}_{4,r})\ {\rm and}\ 
	\partial^{v,1}\mathcal{G}^-_{\hat{\Delta}_{\alpha}^4}(\hat{I}_{4,l},\hat{I}_{4,r})\subseteq\hat{\Delta}_{\alpha}^4,$$
	we have that
	$P_{\alpha}^{-1}(\partial^{v,1}\mathcal{G}^-_{\hat{\Delta}_{\alpha}^4}(\hat{I}_{4,l},\hat{I}_{4,r}))$ has two connected components: the one contained in $\hat{\Delta}_{\alpha}^{4,-1}$ and the other one not intersecting $\hat{\Delta}_{\alpha}^{4,-1}$, written as $\hat{\gamma}_0$. Since two endpoints of
	$\partial^{v,1}\mathcal{G}^-_{\hat{\Delta}_{\alpha}^4}(\hat{I}_{4,l},\hat{I}_{4,r})$, say $x_{q_4-1}$ and $x_{q_5-1}$, aren't contained in the closure of any big fjord of level $n\geq4$, we have
	$\partial^{v,1}\mathcal{G}^-_{\hat{\Delta}_{\alpha}^4}(\hat{I}_{4,l},\hat{I}_{4,r})\cap\Delta_{\alpha}\not=\emptyset$.
	By combining the fact that $\hat{\gamma}_0$ doesn't intersect $\Delta_{\alpha}$, it follows that $\hat{\gamma}_0\cap U_0\not=\emptyset$. 
	Observe also that $\hat{\gamma}_0\subseteq\partial^{v,0}\mathcal{R}_{+}^{(-1)}$. Then
	$\partial^{v,0}\mathcal{R}_{+}^{(-1)}\cap U_0\not=\emptyset$ and hence $U_0\cup\{c_0\}$ connects $c_0$ and $\partial^{v,0}\mathcal{R}_{+}^{(-1)}$.
	This implies that for all $n\geq\mathfrak{q}_1$,
	$\gamma_{c_0}^{n}\cap U_0\not=\emptyset$, for $\gamma_{c_0}^{n}$ doesn't cross $\mathcal{R}_{+}^{(-1)}$.
	
	In the same way 
	we can obtain that if there exists a level $m$ interval $T$ containing $P_{\alpha}(c_0)$ such that the harmonic measure of $T$ in $(\hat{\mathbb{C}}\setminus\overline{\Delta_{\alpha}},\infty)$ is less than $\frac{1}{2}$, then for all $n\geq m+\mathfrak{q}_1$,
	$\gamma_{c_0}^{n}\cap U_0\not=\emptyset$. Together with the following lemma, it gives that for all $n\geq\mathfrak{q}_2+\mathfrak{q}_1$,
	\begin{equation}
		\label{e20251126a}\gamma_{c_0}^{n}\cap U_0\not=\emptyset.
	\end{equation}
	
	\begin{lemma}
		\label{l841}There exists a positive integer $\mathfrak{q}_2$ $(\gg\mathfrak{q}_1)$, independent of $\alpha$, such that for all adjacent intervals $I_6$ and $J_6$ of level $n$ {\rm(}$\geq\mathfrak{q}_2${\rm)} having a common endpoint $a$, $\gamma_{a}^{n}$ is non-winding based on $I_6\cup J_6$ with respect to the closed quasidisk $\overline{\Delta_{\alpha}}$.
	\end{lemma}
	\begin{proof}
		For all $n\geq8+\mathfrak{q}_1$, we let $I_6$ and $J_6$ be the two adjacent intervals of $\mathfrak{D}_n(\alpha)$ having a common endpoint $a$. Since
		every interval of $\mathfrak{D}_{2+\mathfrak{q}_1}(\alpha)$ can be divided into at least $5$ intervals of $\mathfrak{D}_{8+\mathfrak{q}_1}(\alpha)$,
		we can choose a pair of adjacent intervals $I$ and $J$ with a common endpoint $b$ in $\mathfrak{D}_{2+\mathfrak{q}_1}(\alpha)$ such that $I_6\cup J_6\in\overline{(I\cup J)\setminus(I_5\cup J_5)}$, where $I_5$ ($\subseteq I$) and $J_5$ ($\subseteq J$) are the two intervals of $\mathfrak{D}_{8+\mathfrak{q}_1}(\alpha)$ such that $I_5$ (resp. $J_5$) has a common endpoint as $I$ (resp. $J$), but not $b$. Let $I_{l}$ and $I_{r}$ be the two components of
		$\overline{(I\cup J)\setminus(I_6\cup J_6)}$.
		We let $k$ be the smallest nonnegative integer such that $P_{\alpha}^{\comp k}(I)$ or $P_{\alpha}^{\comp k}(J)$ have $c_0$ as an endpoint.
		It follows from the smallestness that $P_{\alpha}^{\comp k}(I)$ (resp. $P_{\alpha}^{\comp k}(J)$) is an interval of $\mathfrak{D}_{2+\mathfrak{q}_1}(\alpha)$, $P_{\alpha}^{\comp k}(I_5)$ (resp. $P_{\alpha}^{\comp k}(J_5)$) is an  interval of $\mathfrak{D}_{8+\mathfrak{q}_1}(\alpha)$ and $P_{\alpha}^{\comp k}(I_6)$ (resp. $P_{\alpha}^{\comp k}(J_6)$) is an  interval of $\mathfrak{D}_{n}(\alpha)$. 
		Since $P_{\alpha}^{\comp k}(I_l)$ (resp. $P_{\alpha}^{\comp k}(I_r)$) contains $P_{\alpha}^{\comp k}(I_5)$ or $P_{\alpha}^{\comp k}(J_5)$, by Lemma \ref{l7301} there exists $\mathfrak{q}_2\gg10+\mathfrak{q}_1$ such that for all $n\geq\mathfrak{q}_2$, we have $\mathcal{W}_{\overline{\Delta_\alpha}}^+(P_{\alpha}^{\comp k}(I_l),P_{\alpha}^{\comp k}(I_r))\gg1$. By (\ref{e8223}) $\mathcal{W}(\mathcal{G}_{\overline{\Delta_{\alpha}}}^+(P_{\alpha}^{\comp k}(I_l),P_{\alpha}^{\comp k}(I_r))\gg1$.
		By Lemma \ref{l04231} we have $\infty\not\in\overline{\Omega_{c_0}^{\mathfrak{q}_1}}$. By combining 
		$\mathcal{G}_{\overline{\Delta_{\alpha}}}^+(P_{\alpha}^{\comp k}(I_l),P_{\alpha}^{\comp k}(I_r))\subseteq\overline{\Omega_{c_0}^{\mathfrak{q}_1}}$, it follows that $\mathcal{G}_{\overline{\Delta_{\alpha}}}^+(P_{\alpha}^{\comp k}(I_l),P_{\alpha}^{\comp k}(I_r))$ is non-winding based on $P_{\alpha}^{\comp k}(I_l)\cup P_{\alpha}^{\comp k}(I_r)$ with respect to $\overline{\Delta_{\alpha}}$.
		Pulling back $\mathcal{G}_{\Delta_\alpha}^+(P_{\alpha}^{\comp k}(I_l),P_{\alpha}^{\comp k}(I_r))$ by $P_{\alpha}^{\comp k}$ along $P_{\alpha}^{\comp k}(I\cup J),P_{\alpha}^{\comp(k-1)}(I\cup J),\cdots, I\cup J$,
		we obtain a non-winding rectangle $\mathcal{R}$ with $\partial^{h,0}\mathcal{R}=I_l$, $\partial^{h,1}\mathcal{R}=I_r$ and $\mathcal{W}(\mathcal{R})\gg1$.
		By Lemma \ref{l825a}, we can obtain that $\gamma_a^n$ doesn't cross $\mathcal{R}$ and hence $\gamma_{a}^{n}$ is non-winding based on $I\cup J$ (also $I_6\cup J_6$) with respect to $\overline{\Delta_\alpha}$.
	\end{proof}
	
	In the proof of Lemma \ref{l04231}, we proved that if the harmonic measure of the interval $T$ of level $4$ containing $P_{\alpha}(c_0)$ in $(\hat{\mathbb{C}}\setminus\overline{\Delta_{\alpha}},\infty)$ is less than $\frac{1}{2}$, then
	for all $n\geq\mathfrak{q}_1$, $\gamma_{c_0}^{n}$ doesn't cross $\mathcal{R}_{+}^{(-1)}$, and thus
	$$\Omega_{c_0}^{n}\subseteq O_{T^{-1}}(\mathcal{R}_{+}^{(-1)})\subseteq\hat{\Delta}_{\alpha}^{4,-1}\cup \hat{\mathcal{G}^{-1}},$$
	where $\hat{\mathcal{G}^{-1}}$ means the filling-in of $\mathcal{G}^{-1}$.
	This implies 
	$$P_{\alpha}(\Omega_{c_0}^{n})\subseteq O_T(\mathcal{G}^+_{\overline{\Delta_{\alpha}}}(I_{4,l},I_{4,r}))\cup\hat{\Delta}_{\alpha}^{4}\subseteq\Lambda_\alpha^4\cup\overline{\Delta_{\alpha}}.$$
	It follows from Lemma \ref{l841} that all intervals of level $n\geq\mathfrak{q}_2$ have harmonic measure $<\frac{1}{2}$. Then for all $n\geq\mathfrak{q}_2$, the interval of level $n$ containing $P_{\alpha}(c_0)$ has harmonic measure $<\frac{1}{2}$ in $(\hat{\mathbb{C}}\setminus\overline{\Delta_{\alpha}},\infty)$. Thus in the same way we can prove 
	
	\begin{corollary}
		\label{l8191}For all $n\geq\mathfrak{q}_2$, $P_{\alpha}(\Omega_{c_0}^{n+\mathfrak{ q}_1})\subseteq\Lambda_\alpha^{n}\cup\overline{\Delta_{\alpha}}$.
	\end{corollary}

	\subsection{Essentially visiting main mountains\label{s3.4}}
	
	\begin{lemma}
		\label{l8251}For all  $n\geq\mathfrak{q}_3:=3\mathfrak{q}_2$ and $a\in{\rm EP}(\mathfrak{D}_n(\alpha))$, there exists a nonnegative integer $k$ such that  $P_{\alpha}^{\comp k}(\Omega_{a}^n)\subseteq\Omega_{c_0}^{n-\mathfrak{q}_2}$ and $P_{\alpha}^{\comp j}(\Omega_{a}^n)\subseteq\Lambda_{\alpha}^{n-\mathfrak{q}_2}$ for all $0\leq j\leq k$.
	\end{lemma}
	\begin{proof}
		We denote by $I_6$ and $J_6$ the two adjacent intervals of $\mathfrak{D}_n(\alpha)$ having the common endpoint $a$. Since
		every interval of $\mathfrak{D}_{n-\mathfrak{q}_2+2}(\alpha)$ can be divided into at least $5$ intervals of $\mathfrak{D}_{n-\mathfrak{q}_2+8}(\alpha)$,
		we can choose a pair of adjacent intervals $I$ and $J$ with a common endpoint $b$ in $\mathfrak{D}_{n-\mathfrak{q}_2+2}(\alpha)$ such that $I_6\cup J_6\in\overline{(I\cup J)\setminus(I_5\cup J_5)}$, where $I_5$ ($\subseteq I$) and $J_5$ ($\subseteq J$) are the two intervals of $\mathfrak{D}_{n-\mathfrak{q}_2+8}(\alpha)$ such that $I_5$ (resp. $J_5$) has a common endpoint as $I$ (resp. $J$), but not $b$. Let $I_{l}$ and $I_{r}$ be the two components of
		$\overline{(I\cup J)\setminus(I_6\cup J_6)}$ with $I_{l}<I_6\cup J_6<I_{r}$.
		We let $k$ be the smallest nonnegative integer such that $P_{\alpha}^{\comp k}(I)$ or $P_{\alpha}^{\comp k}(J)$ have $c_0$ as an endpoint.
		It follows from the smallestness that $P_{\alpha}^{\comp k}(I)$ (resp. $P_{\alpha}^{\comp k}(J)$) is an interval of $\mathfrak{D}_{n-\mathfrak{q}_2+2}(\alpha)$, $P_{\alpha}^{\comp k}(I_5)$ (resp. $P_{\alpha}^{\comp k}(J_5)$) is an  interval of $\mathfrak{D}_{n-\mathfrak{q}_2+8}(\alpha)$ and $P_{\alpha}^{\comp k}(I_6)$ (resp. $P_{\alpha}^{\comp k}(J_6)$) is an  interval of $\mathfrak{D}_{n}(\alpha)$. 
		Since $P_{\alpha}^{\comp k}(I_l)$ (resp. $P_{\alpha}^{\comp k}(I_r)$) contains $P_{\alpha}^{\comp k}(I_5)$ or $P_{\alpha}^{\comp k}(J_5)$, by Lemma \ref{l7301} we have $\mathcal{W}_{\overline{\Delta_\alpha}}^+(P_{\alpha}^{\comp k}(I_l),P_{\alpha}^{\comp k}(I_r))\gg1$. By (\ref{e8223}) $\mathcal{W}(\mathcal{G}_{\overline{\Delta_{\alpha}}}^+(P_{\alpha}^{\comp k}(I_l),P_{\alpha}^{\comp k}(I_r))\gg1$.
		Since $n-\mathfrak{q}_2\geq\mathfrak{q}_1$, by Lemma \ref{l04231} we have $\infty\not\in\overline{\Omega_{c_0}^{n-\mathfrak{q}_2}}$. By combining 
		$\mathcal{G}_{\overline{\Delta_{\alpha}}}^+(P_{\alpha}^{\comp k}(I_l),P_{\alpha}^{\comp k}(I_r))\subseteq\overline{\Omega_{c_0}^{n-\mathfrak{q}_2}}$, it follows that $\mathcal{G}_{\overline{\Delta_{\alpha}}}^+(P_{\alpha}^{\comp k}(I_l),P_{\alpha}^{\comp k}(I_r))$ is non-winding based on $P_{\alpha}^{\comp k}(I)\cup P_{\alpha}^{\comp k}(J)$ with respect to $\overline{\Delta_{\alpha}}$.
		Pulling back $\mathcal{G}_{\overline{\Delta_{\alpha}}}^+(P_{\alpha}^{\comp k}(I_l),P_{\alpha}^{\comp k}(I_r))$ by $P_{\alpha}^{\comp k}$ along $P_{\alpha}^{\comp k}(I\cup J),P_{\alpha}^{\comp(k-1)}(I\cup J),\cdots, I\cup J$,
		we obtain a non-winding rectangle $\mathcal{R}$ with $\partial^{h,0}\mathcal{R}=I_l$ and $\partial^{h,1}\mathcal{R}=I_r$ with $\mathcal{W}(\mathcal{R})\gg1$.
		Let $\gamma_{\frac{1}{2}}$ be the center arc of $\mathcal{R}$.
		By Lemma \ref{l825a}, we can obtain that $\gamma_a^n$ doesn't cross $\gamma_{\frac{1}{2}}$.
		Then $\Omega_{a}^n\subseteq O_{I\cup J}(\gamma_{\frac{1}{2}})$. Applying $P_{\alpha}^{\comp k}$ to both sides of the above set relation, we have
		$$P_{\alpha}^{\comp k}(\Omega_{a}^n)\subseteq P_{\alpha}^{\comp k}(O_{I\cup J}(\gamma_{\frac{1}{2}}))=O_{P_{\alpha}^{\comp k}(I)\cup P_{\alpha}^{\comp k}(J)}(P_{\alpha}^{\comp k}(\gamma_{\frac{1}{2}}))\subseteq \Omega_{c_0}^{n-\mathfrak{q}_2}.$$
		
		For all $0\leq j\leq k$, we view $P_{\alpha}^{\comp j}(\mathcal{R})$ as a rectangle with $\partial^{h,0}P_{\alpha}^{\comp j}(\mathcal{R})=P_{\alpha}^{\comp j}(I_l)$ and $\partial^{h,1}P_{\alpha}^{\comp j}(\mathcal{R})=P_{\alpha}^{\comp j}(I_r)$.
		Then $\mathcal{W}(P_{\alpha}^{\comp j}(\mathcal{R}))=\mathcal{W}(\mathcal{R})\gg1$ and $P_{\alpha}^{\comp j}(\gamma_{\frac{1}{2}})$ is the center arc of $P_{\alpha}^{\comp j}(\mathcal{R})$. Observe that $\gamma_{P_{\alpha}^{\comp j}(b)}^{n-\mathfrak{q}_2+2}$ has the same endpoints as $\partial^{v,0}P_{\alpha}^{\comp j}(\mathcal{R})$. Then
		it follows from Lemma \ref{l825a} that $\gamma_{P_{\alpha}^{\comp j}(b)}^{n-\mathfrak{q}_2+2}$ doesn't intersect $P_{\alpha}^{\comp j}(\gamma_{\frac{1}{2}})$, and hence $O_{P_{\alpha}^{\comp j}(I)\cup P_{\alpha}^{\comp j}(J)}(P_{\alpha}^{\comp j}(\gamma_{\frac{1}{2}}))\subseteq\Omega_{P_{\alpha}^{\comp j}(b)}^{n-\mathfrak{q}_2+2}\subseteq\Lambda_{\alpha}^{n-\mathfrak{q}_2}$.
		Applying $P_{\alpha}^{\comp j}$ to both sides of $\Omega_{a}^n\subseteq O_{I\cup J}(\gamma_{\frac{1}{2}})$,
		we obtain $$P_{\alpha}^{\comp j}(\Omega_{a}^n)\subseteq P_{\alpha}^{\comp j}(O_{I\cup J}(\gamma_{\frac{1}{2}}))=O_{P_{\alpha}^{\comp j}(I)\cup P_{\alpha}^{\comp j}(J)}(P_{\alpha}^{\comp j}(\gamma_{\frac{1}{2}}))\subseteq\Lambda_{\alpha}^{n-\mathfrak{q}_2}.$$
	\end{proof}

	\begin{lemma}
		\label{l20257191}For any dam $\beta_{I_n}$ of level $n\geq-1$ based on $I_{n}\in\mathfrak{D}_n$ with respect to $\{\hat{\Delta}_{\alpha}^n\}_{n\geq-1}$ and any $z\in\mathcal{J}(\tilde{\beta}_{I_n})$, we have that $P_{\alpha}^{\comp j}(z)\in\Lambda_\alpha^n$ for all $j\in\{0,1, \cdots,q_{n+1}\}$ and $P_{\alpha}^{\comp j_0}(z)\in\Omega_{c_0}^n$ for some $j_0\in\{1, \cdots,q_{n+1}\}$.
	\end{lemma}
	
	\begin{proof}
		Recall that $\overline{\mathcal{J}(\hat{\beta}_{I_n})\setminus\mathcal{J}(\beta_{I_n})}$ is a non-winding rectangle based on $I_n$ with respect to $\overline{\Delta_{\alpha}}$ together with the conformal map $$\varphi_{I_n}:\overline{\mathcal{J}(\hat{\beta}_n)\setminus\mathcal{J}(\beta_n)}\to E_x$$
		such that $$\varphi_{I_n}^{-1}(\{(x,z):0\leq z\leq 1\})=\hat{\beta}_n$$
		and $$\varphi_{I_n}^{-1}(\{(0,z):0\leq z\leq 1\})=\beta_n.$$
		By Lemma \ref{l8271a}, for all $j\in\{0,1,\cdots, q_{n+1}\}$, $P_{\alpha}^{\comp j}$ is a univalent map on the interior of $$\mathcal{R}=\varphi_{I_n}^{-1}(\{(y,z):0\leq z\leq 1, x-4\leq y\leq x-3\})$$
		and homeomorphism on $\mathcal{R}$.
		Since the combinatorial distance between $\mathcal{R}\cap\partial\Delta_{\alpha}$ ($\subseteq I_n$) and two endpoints of $I_n$ is at least $11\iota_{n+1}$ (Due to {\rm(P}$2${\rm)} in Section \ref{s2.3}), we have that $P_{\alpha}^{\comp j}(\mathcal{R})$ is based on some interval $I_{n,j}$ of level $n$, where $P_{\alpha}^{\comp j}(\mathcal{R})$ is viewed as a rectangle with
		$$\partial^{h,0}P_{\alpha}^{\comp j}(\mathcal{R})=P_{\alpha}^{\comp j}(\varphi_{I_n}^{-1}(\{(y,0):x-4\leq y\leq x-3\}))$$
		and
		$$\partial^{h,1}P_{\alpha}^{\comp j}(\mathcal{R})=P_{\alpha}^{\comp j}(\varphi_{I_n}^{-1}(\{(y,1):x-4\leq y\leq x-3\})).$$
		Then the width $\mathcal{W}(P_{\alpha}^{\comp j}(\mathcal{R}))=1$. By Lemma \ref{l825a} there exists a geodesic $\gamma_j$ of $\hat{\mathbb{C}}\setminus\overline{\Delta_{\alpha}}$ connecting $\partial^{h,0}P_{\alpha}^{\comp j}(\mathcal{R})$ and $\partial^{h,1}P_{\alpha}^{\comp j}(\mathcal{R})$ contained in $P_{\alpha}^{\comp j}(\mathcal{R})$.
		Observe also that $\mathcal{R}$ is non-winding based on $I_n$ with respect to $\overline{\Delta_{\alpha}}$ and $\tilde{\beta}_{I_n}=\varphi_{I_n}^{-1}(\{(y,z):y=x-4\ {\rm and}\ 0\leq z\leq1\})$.
		Thus $P_{\alpha}^{\comp j}(\mathcal{J}(\tilde{\beta}_{I_n}))\subseteq O_{I_{n,j}}(\gamma_j)\subseteq\Lambda_\alpha^n$. At last, we need only to choose an appropriate $j_0\in\{1,2,\cdots, q_{n+1}\}$ so that $I_{n,j_0}$ has an endpoint $c_0$. Then $P_{\alpha}^{\comp j_0}(\mathcal{J}(\tilde{\beta}_{I_n}))\subseteq O_{I_{n,j_0}}(\gamma_{j_0})\subseteq\Omega_{c_0}^n$.
		
	\end{proof}
	
	For all $n\geq-1$, we denote by $\mathfrak{D}_{\geq n}$ the set consisting of all intervals $I$ of level $\geq n$ such that there exists a dam of the same level based on $I$ with respect to $\{\hat{\Delta}_{\alpha}^n\}_{n\geq-1}$. Let $\tilde{\Delta}_{\alpha}^{n}$ be the filling-in of
	$\overline{\Delta_\alpha}\cup\bigcup\limits_{I\in\mathfrak{D}_{\geq n}}\tilde{\beta}_I$, that is $\overline{\Delta_\alpha}\cup\bigcup\limits_{I\in\mathfrak{D}_{\geq n}}\mathcal{J}(\tilde{\beta}_I)$, and let 
	$\tilde{\Delta}_{\alpha}$ be the filling-in of $P_{\alpha}(\tilde{\Delta}_{\alpha}^{-1})\cup\tilde{\Delta}_{\alpha}^{-1}$, that is $\overline{\Delta_\alpha}\cup\bigcup\limits_{I\in\mathfrak{D}_{\geq n}}\mathcal{J}(\tilde{\beta}_I)\cup\bigcup\limits_{I\in\mathfrak{D}_{\geq n}}P_{\alpha}(\mathcal{J}(\tilde{\beta}_I))$.
	For all $n\geq\mathfrak{q}_2$, we set
	$$\Lambda_{\alpha}^{*(n+2)}:=\bigcup_{x\in{\rm EP}(\mathfrak{D}_{n+2}(\alpha))\setminus\{x_{q_{n+2}-1},x_{q_{n+3}-1}\}}\Omega_x^{n+2}.$$
	It is clear that
	\begin{itemize}
		\item $\hat{\Delta}_{\alpha}^{-1}\subseteq\tilde{\Delta}_{\alpha}^{-1}$ (Due to $\mathcal{J}(\beta_I)\subseteq\mathcal{J}(\tilde{\beta}_I)\subseteq\mathbb{C}$ for all $I\in\mathfrak{D}_{\geq -1}$),
		\item $\tilde{\Delta}_{\alpha}\cup\Lambda_{\alpha}^{*(n+2)}$ is a bounded closed set (Due to Lemma \ref{l841}).
	\end{itemize}
	Furthermore, we have the following result:
	\begin{corollary}
		\label{c20251027}For all $n\geq\mathfrak{q}_2$,
		$P_{\alpha}(c_0)\in\partial(\tilde{\Delta}_{\alpha}\cup\Lambda_{\alpha}^{*(n+2)})$ and there exists a Jordan arc $\Gamma_{\alpha}$ outside $(\tilde{\Delta}_{\alpha}\cup\Lambda_{\alpha}^{*(n+2)})\setminus\{P_{\alpha}(c_0)\}$ that connects $P_{\alpha}(c_0)$ and $\infty$.
	\end{corollary}
	\begin{proof}
		For every interval $I\in\mathfrak{D}_{\geq-1}$ with the level $m$, as that in the proof of Lemma \ref{l20257191}, we 
		can obtain that there exists a non-winding geodesic $\gamma_I$ of $\hat{\mathbb{C}}\setminus\overline{\Delta_{\alpha}}$ (with respect to $\overline{\Delta_{\alpha}}$) based on some interval $I_{m}$ of level $m$ such that 
		\begin{itemize}
			\item[(${\rm a}2$)] $P_{\alpha}(\mathcal{J}(\tilde{\beta}_I))\subseteq O_{I_m}(\gamma_I)$,
			\item[(${\rm b}2$)] the combinatorial distance between two endpoints of $\gamma_I$ and two endpoints of $I_m$ is at least $10\iota_{m+1}$.
		\end{itemize}
		We denote by $F(\Delta_{\alpha})$ the filling-in of $\Delta_{\alpha}\cup\bigcup\limits_{I\in\mathfrak{D}_{\geq-1}}(\gamma_I\cup\tilde{\beta}_I)$. It is clear that $\tilde{\Delta}_{\alpha}^{-1}\subseteq F(\Delta_{\alpha})$, and by (a2) we have $P_{\alpha}(\tilde{\Delta}_{\alpha}^{-1})\subseteq F(\Delta_{\alpha})$.
		Then
		\begin{equation}
			\label{e251027a}\tilde{\Delta}_{\alpha}\subseteq F(\Delta_{\alpha}).
		\end{equation}
		Observe that for any $I\in\mathfrak{D}_{\geq-1}$ with the level $m$, if $P_{\alpha}(c_0)\in I$, then the combinatorial distance between two endpoints of $I$ and $P_{\alpha}(c_0)$ is  $\iota_{m+1}$. Then by (b2) and (P2) in Section \ref{s2.3} we have $P_{\alpha}(c_0)\not\in\overline{O_I(\tilde{\beta}_I)\cup O_{I}(\gamma_{I'})}$, where $I'\in\mathfrak{D}_m$  and $\gamma_{I'}$ is based on $I$. Thus
		\begin{equation}
			\label{e20260302a}P_{\alpha}(c_0)\in\partial F(\Delta_{\alpha}).
		\end{equation}
		For all $x\in{\rm EP}(\mathfrak{D}_{n+2}(\alpha))\setminus\{x_{q_{n+2}-1},x_{q_{n+3}-1}\}$, by $n\geq\mathfrak{q}_2$ and Lemma \ref{l841}, $\gamma_x^{n+2}$ is non-winding based on $\overline{\partial\Delta_{\alpha}\setminus I_{[x_{q_{n+2}-1},x_{q_{n+3}-1}]}}$ with respect to $\overline{\Delta_{\alpha}}$.
		This implies $\overline{\Lambda_{\alpha}^{*(n+2)}}\cap {\rm int}(I_{[x_{q_{n+2}-1},x_{q_{n+3}-1}]})=\emptyset$.
		Since $P_{\alpha}(c_0)\in{\rm int}(I_{[x_{q_{n+2}-1},x_{q_{n+3}-1}]})$, we have $P_{\alpha}(c_0)\not\in\overline{\Lambda_{\alpha}^{*(n+2)}}$. By combining (\ref{e20260302a}), it follows that $P_{\alpha}(c_0)\in\partial (F(\Delta_{\alpha})\cup\Lambda_{\alpha}^{*(n+2)})$.
		Since the above each $\gamma_x^{n+2}$ is a non-winding geodesic of $\hat{\mathbb{C}}\setminus\overline{\Delta_{\alpha}}$ based on $\overline{\partial\Delta_{\alpha}\setminus I_{[x_{q_{n+2}-1},x_{q_{n+3}-1}]}}$ with respect to $\overline{\Delta_{\alpha}}$ and each $\gamma_I$ (resp. $\tilde{\beta}_I$), $I\in\mathfrak{D}_{\geq-1}$ is a non-winding geodesic of $\hat{\mathbb{C}}\setminus\overline{\Delta_{\alpha}}$ based on $I$ with respect to $\overline{\Delta_{\alpha}}$,
		we have that
		$\partial(F(\Delta_{\alpha})\cup\Lambda_{\alpha}^{*(n+2)})$ is a Jordan curve on $\mathbb{C}$. Thus there exists a Jordan arc $\Gamma_{\alpha}$ outside $(F(\Delta_{\alpha})\cup\Lambda_{\alpha}^{*(n+2)})\setminus\{P_{\alpha}(c_0)\}$ that connects $P_{\alpha}(c_0)$ and $\infty$.
		By (\ref{e251027a}), we have $$(\tilde{\Delta}_{\alpha}\cup\Lambda_{\alpha}^{*(n+2)})\setminus\{P_{\alpha}(c_0)\}\subseteq(F(\Delta_{\alpha})\cup\Lambda_{\alpha}^{*(n+2)})\setminus\{P_{\alpha}(c_0)\}.$$
		Thus $\Gamma_{\alpha}$ is outside $(\tilde{\Delta}_{\alpha}\cup\Lambda_{\alpha}^{*(n+2)})\setminus\{P_{\alpha}(c_0)\}$ and connects $P_{\alpha}(c_0)$ and $\infty$.
		Together with $P_{\alpha}(c_0)\in\tilde{\Delta}_{\alpha}^{-1}\subseteq\tilde{\Delta}_{\alpha}$, it gives 
		$P_{\alpha}(c_0)\in\partial(\tilde{\Delta}_{\alpha}\cup\Lambda_{\alpha}^{*(n+2)})$.
		
	\end{proof}
	
	Let $\Gamma_{\alpha}$ be the same as that in Corollary \ref{c20251027}. Since $\hat{\mathbb{C}}\setminus\Gamma_{\alpha}$ is a simple connected region not containing critical values of $P_{\alpha}$, $P_{\alpha}^{-1}$ has two single-value analytic branch: $h_{\Gamma_{\alpha}}^{(1)}$ and $h_{\Gamma_{\alpha}}^{(2)}$, where $h_{\Gamma_{\alpha}}^{(1)}(0)=0$ and $h_{\Gamma_{\alpha}}^{(2)}(0)\not=0$.
	For all $n\geq-1$, we
	define $$\tilde{U}_0^{n}:=\overline{h_{\Gamma_{\alpha}}^{(2)}(\tilde{\Delta}_{\alpha}^{n}\setminus\{P_{\alpha}(c_0)\})}\ \left(=h_{\Gamma_{\alpha}}^{(2)}(\tilde{\Delta}_{\alpha}^{n}\setminus\{P_{\alpha}(c_0)\})\cup\{c_0\}\right),$$
	that is the closure of the component of $P_{\alpha}^{-1}(\tilde{\Delta}_{\alpha}^{n})\setminus\{c_0\}$ not containing $0$.
	\begin{lemma}
		\label{l20257193}For all $n\geq\mathfrak{q}_2$, $\tilde{U}_0^{-1}\cap\left(\tilde{\Delta}_{\alpha}^{-1}\cup(\Lambda_{\alpha}^{n+\mathfrak{q}_2}\setminus\Omega_{c_0}^n)\right)=\{c_0\}$.
	\end{lemma}
	\begin{proof}
		By Corollary \ref{c20251027},  $P_{\alpha}^{-1}\left((\tilde{\Delta}_{\alpha}\cup\Lambda_{\alpha}^{*(n+2)})\setminus\{P_{\alpha}(c_0)\}\right)$ has two components: $W_1^{n+2}$ (containing $0$) and $W_2^{n+2}$ (not containing $0$), and $\overline{W_1^{n+2}}\cap \overline{W_2^{n+2}}=\{c_0\}$.
		Since $\tilde{\Delta}_{\alpha}^{-1}\subseteq\tilde{\Delta}_{\alpha}\cup\Lambda_{\alpha}^{*(n+2)}$, we have 
		\begin{equation}
			\label{e251030a}\tilde{U}_0^{-1}\subseteq\overline{W_2^{n+2}}.
		\end{equation}
		Since $\tilde{\Delta}_{\alpha}^{-1}$ is contained in the closure of the component of $P_{\alpha}^{-1}\left(\tilde{\Delta}_{\alpha}\setminus\{P_{\alpha}(c_0)\}\right)$ containing $0$, we have 
		\begin{equation}
			\label{e251030b}\tilde{\Delta}_{\alpha}^{-1}\subseteq\overline{W_1^{n+2}}.
		\end{equation}
		
		Next, we prove that for all $n\geq\mathfrak{q}_2$,
		\begin{equation}
			\label{e251030c}\Lambda_{\alpha}^{n+\mathfrak{q}_2}\setminus\Omega_{c_0}^n\subseteq\overline{W_1^{n+2}}.
		\end{equation}
		Let $I$ and $J$ be two adjacent intervals of $\mathfrak{D}_{n}$ having the common endpoint $c_0$.
		For all $b\in{\rm EP}(\mathfrak{D}_{n+\mathfrak{q}_2}(\alpha))$ with $b\not\in{\rm int}(I\cup J)$,
		there exist two adjacent intervals $I_1$ and $J_1$ of $\mathfrak{D}_{n+2}$ having a common endpoint $e$ such that $b\in (I_1\cup J_1)\setminus(I_2\cup I_2'\cup J_2\cup J_2')$, where $I_2\subseteq I_1$ (resp. $ J_2\subseteq J_1$) is an interval of $\mathfrak{D}_{n+8}$ having a common endpoint as $I_1$ (resp. $J_1$) which is not $e$; $I_2'\subseteq I_1$ (resp. $J_2'\subseteq J_1$) is an interval of $\mathfrak{D}_{n+8}$ adjacent to $I_2$ (resp. $J_2$). 
		It is easy to see
		that $c_0\not\in I_1\cup J_1$. Then
		$e\in{\rm EP}(\mathfrak{D}_{n+2}(\alpha))\setminus\{c_0,x_{q_{n+2}}\}$ and hence
		\begin{equation} 
			\label{e20260303b}P_{\alpha}(e)\in{\rm EP}(\mathfrak{D}_{n+2}(\alpha))\setminus\{x_{q_{n+2}-1},x_{q_{n+3}-1}\}.
		\end{equation} 
		Let $I_3$ and $J_3$
		be two intervals of $\mathfrak{D}_{n+\mathfrak{q}_2}$ having the common endpoint $b$. Let $I_4$
		and $J_4$ be two components of $\overline{P_{\alpha}(I_1\cup J_1)\setminus P_{\alpha}(I_3\cup J_3)}$ with $I_4<P_{\alpha}(I_3\cup J_3)<J_4$. Observe that $c_0\not\in I_1\cup J_1$ and $I_2\cup I_3\cup J_3\cup J_2\subseteq I_1\cup J_1$. Then
		$P_{\alpha}(I_1),P_{\alpha}(J_1)\in\mathfrak{D}_{n+2}$; $P_{\alpha}(I_2),P_{\alpha}(J_2)\in\mathfrak{D}_{n+8}$; $P_{\alpha}(I_3),P_{\alpha}(J_3)\in\mathfrak{D}_{n+\mathfrak{q}_2}$. 
		Since each component of $\overline{I_1\cup J_1\setminus(I_3\cup J_3)}$ contains $I_2$ or $J_2$, we have that each component of $\overline{P_{\alpha}(I_1\cup J_1)\setminus P_{\alpha}(I_3\cup J_3)}$ contains $P_{\alpha}(I_2)$ or $P_{\alpha}(J_2)$.
		By Lemma \ref{l7301} $\mathcal{W}_{\overline{\Delta_{\alpha}}}^+(I_4,J_4)\gg1$. 
		By (\ref{e8223}) $\mathcal{W}(\mathcal{G}_{\overline{\Delta_{\alpha}}}^+(I_4,J_4))\gg1$.
		Lemma \ref{l841} gives that $\mathcal{G}_{\overline{\Delta_{\alpha}}}^+(I_4,J_4)$ is non-winding based on $P_{\alpha}(I_1\cup J_1)$ with respect to $\overline{\Delta_{\alpha}}$. Since $P_{\alpha}(c_0)\not\in P_{\alpha}(I_1\cup J_1)$, the rectangle
		$\mathcal{G}_{\overline{\Delta_{\alpha}}}^+(I_4,J_4)$ can be conformally pulled back under $P_{\alpha}$ along
		$P_{\alpha}(I_1\cup J_1)$, $I_1\cup J_1$. We write  $\mathcal{G}_{\overline{\Delta_{\alpha}}}^+(I_4,J_4)^{-1}$ as the pullback and view $\mathcal{G}_{\overline{\Delta_{\alpha}}}^+(I_4,J_4)^{-1}$ as a rectangle with $P_{\alpha}(\partial^{h,0}\mathcal{G}_{\overline{\Delta_{\alpha}}}^+(I_4,J_4)^{-1})=\partial^{h,0}\mathcal{G}_{\overline{\Delta_{\alpha}}}^+(I_4,J_4)$ $(=I_4)$ and $P_{\alpha}(\partial^{h,1}\mathcal{G}_{\overline{\Delta_{\alpha}}}^+(I_4,J_4)^{-1})=\partial^{h,1}\mathcal{G}_{\overline{\Delta_{\alpha}}}^+(I_4,J_4)$ $(=J_4)$. Then $\mathcal{W}(\mathcal{G}_{\overline{\Delta_{\alpha}}}^+(I_4,J_4)^{-1})=\mathcal{W}(\mathcal{G}_{\overline{\Delta_{\alpha}}}^+(I_4,J_4))\gg1$. 
		Observe also that $$\partial^{h,0}\mathcal{G}_{\overline{\Delta_{\alpha}}}^+(I_4,J_4)^{-1}\#(I_3\cup J_3)\#\partial^{h,1}\mathcal{G}_{\overline{\Delta_{\alpha}}}^+(I_4,J_4)^{-1}.$$
		Then Lemma \ref{l825a} gives 
		\begin{equation}
			\label{e20260303c}\Omega_b^{n+\mathfrak{q}_2}\subseteq O_{I_1\cup J_1}(\mathcal{G}_{\overline{\Delta_{\alpha}}}^+(I_4,J_4)^{-1}).
		\end{equation}
		By
		(\ref{e20260303b}) and the definition of $\Lambda_{\alpha}^{*(n+2)}$, we have $\Omega_{P_{\alpha}(e)}^{n+2}\subseteq\Lambda_{\alpha}^{*(n+2)}$ and hence
		$\mathcal{G}_{\overline{\Delta_{\alpha}}}^+(I_4,J_4)\subseteq\Lambda_{\alpha}^{*(n+2)}\cup\overline{\Delta_{\alpha}}$. Then $O_{I_1\cup J_1}(\mathcal{G}_{\overline{\Delta_{\alpha}}}^+(I_4,J_4)^{-1})\subseteq\overline{W_1^{n+2}}$. By combining (\ref{e20260303c}),
		it follows that
		$\Omega_b^{n+\mathfrak{q}_2}\subseteq\overline{W_1^{n+2}}$. 
		Then the arbitrariness of $b$ gives that for all $n\geq\mathfrak{q}_2$,
		$\Lambda_{\alpha}^{n+\mathfrak{q}_2}\setminus\Omega_{c_0}^n\subseteq\overline{W_1^{n+2}}$. 
		
		At last, by (\ref{e251030a}), (\ref{e251030b}) and (\ref{e251030c}), we have that for all $n\geq\mathfrak{q}_2$, $$\tilde{U}_0^{-1}\cap\left(\tilde{\Delta}_{\alpha}^{-1}\cup(\Lambda_{\alpha}^{n+\mathfrak{q}_2}\setminus\Omega_{c_0}^n)\right)\subseteq\overline{W_1^{n+2}}\cap\overline{W_2^{n+2}}=\{c_0\}.$$
		Observe also that $c_0\in\tilde{U}_0^{-1}$ and $c_0\in\tilde{\Delta}_{\alpha}^{-1}$. Thus 
		$$\tilde{U}_0^{-1}\cap\left(\tilde{\Delta}_{\alpha}^{-1}\cup(\Lambda_{\alpha}^{n+\mathfrak{q}_2}\setminus\Omega_{c_0}^n)\right)=\{c_0\}.$$
	\end{proof}

	\begin{lemma}
		\label{l819}
		Assuming $n\geq\mathfrak{q}_3$, for all
		$z\in\Lambda^{n+t\mathfrak{q}_3}$ with
		$P_{\alpha}^{\comp m}(z)\not\in\Lambda^{n}\cup\overline{\Delta_{\alpha}}$ for some two positive integer $m$ and $t$, we have that there exist positive integers $m_1<m_2<\cdots<m_t$ such that for all $1\leq j\leq t$,
		$$P_{\alpha}^{\comp m_j}(z)\in\tilde{\Omega}_{c_0}^{n+t\mathfrak{q}_3-j\mathfrak{q}_3}\setminus\Lambda_\alpha^{n+t\mathfrak{q}_3-(j-1)\mathfrak{q}_3}$$
		and
		$$\left\{P_{\alpha}^{\comp s}(z):\ 0\leq s\leq m_j\right\}\subseteq\Lambda_\alpha^{n+t\mathfrak{q}_3-j\mathfrak{q}_3}.$$
	\end{lemma}
	\begin{proof}
		If $z\in\Omega_{c_0}^{n+t\mathfrak{q}_3}$, by Corollary \ref{l8191} and $P_{\alpha}^{\comp m}(z)\not\in\Lambda^{n}\cup\overline{\Delta_{\alpha}}$, we have $P_{\alpha}(z)\in\Lambda_\alpha^{n+t\mathfrak{q}_3-\mathfrak{q}_1}$. Then by Lemma \ref{l8251} 
		there exists a nonnegative integer $k_1$ such that  $$P_{\alpha}^{\comp k_1}(P_{\alpha}(z))\in\Omega_{c_0}^{n+t\mathfrak{q}_3-\mathfrak{q}_1-\mathfrak{q}_2}\subseteq\Omega_{c_0}^{n+(t-1)\mathfrak{q}_3}$$
		and
		$$P_{\alpha}^{\comp j}(P_{\alpha}(z))\in\Lambda_{\alpha}^{n+t\mathfrak{q}_3-\mathfrak{q}_1-\mathfrak{q}_2}\subseteq\Lambda_{\alpha}^{n+(t-1)\mathfrak{q}_3}$$
		for all $0\leq j\leq k_1$.
		
		If $P_{\alpha}^{\comp( k_1+1)}(z)\not\in\tilde{\Omega}_{c_0}^{n+(t-1)\mathfrak{q}_3}$, then there is a dam $\beta_{I_{l_1}}$ of level $l_1\geq n+(t-1)\mathfrak{q}_3$ based on $I_{l_1}\in\mathfrak{D}_{l_1}$ such that $P_{\alpha}^{\comp( k_1+1)}(z)\in\mathcal{J}(\beta_{I_{l_1}})$.
		Then by Lemma \ref{l20257191}
		$$ \{P_{\alpha}^{\comp j}(P_{\alpha}^{\comp( k_1+1)}(z)): j=0,1,\cdots,q_{l_1+1}\}\subseteq\Lambda_\alpha^{l_1}\subseteq\Lambda_\alpha^{n+(t-1)\mathfrak{q}_3}$$
		and
		$$P_{\alpha}^{\comp b_1}(P_{\alpha}^{\comp( k_1+1)}(z))\subseteq\Omega_{c_0}^{l_1}\subseteq\Omega_{c_0}^{n+(t-1)\mathfrak{q}_3},$$
		where $b_1\in\{1,2,\cdots,q_{l_1+1}\}$.
		If $P_{\alpha}^{\comp( k_1+b_1+1)}(z)\not\in\tilde{\Omega}_{c_0}^{n+(t-1)\mathfrak{q}_3}$, then similarly, there exists a positive integer $l_2\geq n+(t-1)\mathfrak{q}_3$ such that 
		$$ \{P_{\alpha}^{\comp j}(P_{\alpha}^{\comp( k_1+b_1+1)}(z)): j=0,1,\cdots,q_{l_2+1}\}\subseteq\Lambda_\alpha^{l_2}\subseteq\Lambda_\alpha^{n+(t-1)\mathfrak{q}_3}$$
		and
		$$P_{\alpha}^{\comp b_2}(P_{\alpha}^{\comp( k_1+b_1+1)}(z))\subseteq\Omega_{c_0}^{l_2}\subseteq\Omega_{c_0}^{n+(t-1)\mathfrak{q}_3},$$
		where $b_2\in\{1,2,\cdots,q_{l_2+1}\}$.
		In general, if $P_{\alpha}^{\comp( k_1+\sum_{j=1}^kb_j+1)}(z)\not\in\tilde{\Omega}_{c_0}^{n+(t-1)\mathfrak{q}_3}$, then there exists a positive integer $l_{k+1}\geq n+(t-1)\mathfrak{q}_3$ such that 
		$$ \{P_{\alpha}^{\comp j}(P_{\alpha}^{\comp( k_1+\sum_{j=1}^kb_j+1)}(z)): j=0,1,\cdots,q_{l_{k+1}+1}\}\subseteq\Lambda_\alpha^{l_{k+1}}\subseteq\Lambda_\alpha^{n+(t-1)\mathfrak{q}_3}$$
		and
		$$P_{\alpha}^{\comp b_{k+1}}(P_{\alpha}^{\comp( k_1+\sum_{j=1}^kb_j+1)}(z))\subseteq\Omega_{c_0}^{l_{k+1}}\subseteq\Omega_{c_0}^{n+(t-1)\mathfrak{q}_3},$$
		where $b_{k+1}\in\{1,2,\cdots,q_{l_{k+1}+1}\}$.
		Since $P_{\alpha}^{\comp m}(z)\not\in\Lambda_{\alpha}^{n}\cup\overline{\Delta_{\alpha}}$, we have $k_1+\sum_{j=1}^kb_j+1<m$, which implies that $k$ is finite.
		Thus there is a positive integer $t_1<m$
		such that 
		$$ \{P_{\alpha}^{\comp j}(z): j=0,1,\cdots,t_1\}\subseteq\Lambda_\alpha^{n+(t-1)\mathfrak{q}_3}$$
		and
		$$P_{\alpha}^{\comp t_1}(z)\subseteq\tilde{\Omega}_{c_0}^{n+(t-1)\mathfrak{q}_3}.$$
		If $P_{\alpha}^{\comp t_1}(z)\in\Lambda_\alpha^{n+t\mathfrak{q}_3}$, then
		we can apply the above progress again until
		there exists a positive integer $m_1<m$ such that 
		$$ \{P_{\alpha}^{\comp j}(z): j=0,1,\cdots,m_1\}\subseteq\Lambda_\alpha^{n+(t-1)\mathfrak{q}_3}$$
		and
		$$P_{\alpha}^{\comp m_1}(z)\subseteq\tilde{\Omega}_{c_0}^{n+(t-1)\mathfrak{q}_3}\setminus\Lambda_\alpha^{n+t\mathfrak{q}_3}.$$
		
		If $t=1$, then the lemma holds. If $t\geq2$, then, applying the same progress to
		$P_{\alpha}^{\comp m_1}(z)$, we have that
		there exists a positive integer $m_2-m_1$ ($m_2<m$) such that
		$$ \{P_{\alpha}^{\comp(m_1+j)}(z)=P_{\alpha}^{\comp j}(P_{\alpha}^{\comp m_1}(z)): j=0,1,\cdots,m_2-m_1\}\subseteq\Lambda_\alpha^{n+(t-2)\mathfrak{q}_3}$$
		and
		$$P_{\alpha}^{\comp m_2}(z)=P_{\alpha}^{\comp (m_2-m_1)}(P_{\alpha}^{\comp m_1}(z))\subseteq\tilde{\Omega}_{c_0}^{n+(t-2)\mathfrak{q}_3}\setminus\Lambda_\alpha^{n+(t-1)\mathfrak{q}_3}.$$
		Repeating the above progress, we can obtain the lemma.
		
	\end{proof}
	
	We say that the orbit of $z$ essentially visits $\tilde{\Omega}_{c_0}^n$
	$k$ times if there exist $k$ distinct nonnegative integers $m_1,m_2,\cdots,m_k$ such that for all $j\in\{m_1,m_2,\cdots,m_k\}$, $P_{\alpha}^{\comp j}(z)\in\tilde{\Omega}_{c_0}^n\setminus\Lambda_\alpha^{n+\mathfrak{q}_3}$.

	\subsection{Pulling back pseudo-Siegel disks\label{s3.5}}
	For all $n\geq\mathfrak{q}_3$, we call $\tilde{\Omega}_{c_0}^n\setminus\Lambda_\alpha^{n+\mathfrak{q}_3}$ the hovered main (geodesic) mountain of level $n$.
	
	\vspace{0.2cm}
	\noindent 3.4.1. {\bf Geometry of hovered main mountains.}
	
	\noindent
	\begin{lemma}
		\label{l25070716}There exists a positive integer $\mathfrak{q}_4$ {\rm(}$\gg\mathfrak{q}_3$ and independent of $\alpha${\rm)} such that for all $n\geq\mathfrak{q}_4$ and all $z\in\Lambda_{\alpha}^n$, we have $$\rho_{\mathbb{C}\setminus\overline{\Delta_{\alpha}}}(z)\asymp\rho_{\hat{\mathbb{C}}\setminus\overline{\Delta_{\alpha}}}(z)\ {\rm and}\ \rho_{\mathbb{C}\setminus\overline{\Delta_{\alpha}}}(z)\asymp\frac{1}{{\rm dist}(z,\Delta_\alpha)}.$$
	\end{lemma}
	\begin{proof}
		For all $n\geq\mathfrak{q}_3+10$ and $z\in\Lambda_{\alpha}^n$, there exists $\Omega_{a}^n$ such that $z\in\Omega_{a}^n\subseteq\mathbb{C}$.
		Let $I_6$ and $J_6$ be the adjacent intervals of $\mathfrak{D}_{n}$ with the common endpoint $a$.
		Since
		every interval of $\mathfrak{D}_{\mathfrak{q}_3}(\alpha)$ can be divided into at least $5$ intervals of $\mathfrak{D}_{\mathfrak{q}_3+6}(\alpha)$,
		we can choose a pair of adjacent intervals $I$ and $J$ in $\mathfrak{D}_{\mathfrak{q}_3}(\alpha)$ with a common endpoint $b$ such that $I_6\cup J_6\in\overline{(I\cup J)\setminus(I_5\cup J_5)}$,
		where
		$I_5$ ($\subseteq I$) and $J_5$ ($\subseteq J$) be the two intervals of $\mathfrak{D}_{\mathfrak{q}_3+6}(\alpha)$ such that $I_5$ (resp. $J_5$) has a common endpoint as $I$ (resp. $J$), but not $b$.
		Let $I_{l}$ and $J_{r}$ be the two components of
		$\overline{(I\cup J)\setminus(I_6\cup J_6)}$.
		Since $I_{l}$ (resp. $J_{r}$) contains $I_5$ or $J_5$, by Lemma \ref{l7301} there exists a positive integer $\mathfrak{q}_4$ {\rm(}$\gg\mathfrak{q}_3${\rm)} ( independent of $z$ and $\alpha$) such that for all $n\geq\mathfrak{q}_4$, we have $\mathcal{W}_{\overline{\Delta_{\alpha}}}^+(I_{l},J_{r})\gg1$. By (\ref{e8223}) for all $n\geq\mathfrak{q}_4$ and $z\in\Lambda_{\alpha}^n$, $\mathcal{W}(\mathcal{G}_{\overline{\Delta_{\alpha}}}^+(I_{l},J_{r}))\gg1$. By $n\geq\mathfrak{q}_4>\mathfrak{q}_2$, $\mathcal{G}_{\overline{\Delta_{\alpha}}}^+(I_{l},J_{r})$ is non-winding based on $I\cup J$ with respect to $\overline{\Delta_{\alpha}}$. Then $\mathcal{G}_{\overline{\Delta_{\alpha}}}^+(I_{l},J_{r})$ separates $z$ and $\infty$ in $\hat{\mathbb{C}}\setminus\overline{\Delta_{\alpha}}$. Let $B_{\hat{\mathbb{C}}\setminus\overline{\Delta_{\alpha}}}$ be the closed unit hyperbolic ball centering at $\infty$ with radius $1$ with respect to the hyperbolic metric on  $\hat{\mathbb{C}}\setminus\overline{\Delta_{\alpha}}$.
		We claim that $\Omega_a^n\cap B_{\hat{\mathbb{C}}\setminus\overline{\Delta_{\alpha}}}=\emptyset$. In fact, if not, then every vertical curve of $\mathcal{G}_{\overline{\Delta_{\alpha}}}^+(I_{l},J_{r})$ passes through $B_{\hat{\mathbb{C}}\setminus\overline{\Delta_{\alpha}}}$, and hence ${\rm mod}(\hat{\mathbb{C}}\setminus(\overline{\Delta_{\alpha}}\cup B_{\hat{\mathbb{C}}\setminus\overline{\Delta_{\alpha}}}))^{-1}\geq\mathcal{W}(\mathcal{G}_{\overline{\Delta_{\alpha}}}^+(I_{l},J_{r}))\gg1$. This is impossible. By $z\in\Omega_a^n$ and the above claim, we have $z\in (\hat{\mathbb{C}}\setminus\overline{\Delta_{\alpha}})\setminus B_{\hat{\mathbb{C}}\setminus\overline{\Delta_{\alpha}}}$.
		Then by applying a conformal map $\phi$, mapping
		$\infty$ to $0$, to transfer $\hat{\mathbb{C}}\setminus\overline{\Delta_{\alpha}}$ to $\mathbb{D}$ and an immediate computation on the ratio of  hyperbolic metrics on $\mathbb{D}$ and $\mathbb{D}^*$, we can get that
		for all $n\geq\mathfrak{q}_4$ and $z\in\Lambda_{\alpha}^n$,
		$$\rho_{\mathbb{C}\setminus\overline{\Delta_{\alpha}}}(z)\asymp\rho_{\hat{\mathbb{C}}\setminus\overline{\Delta_{\alpha}}}(z).$$
		Moreover, it follows from ${\rm mod}(\hat{\mathbb{C}}\setminus(\overline{\Delta_{\alpha}}\cup B_{\hat{\mathbb{C}}\setminus\overline{\Delta_{\alpha}}}))\asymp1$ and [Theorem $2.3$, \cite{McM94}] that  $$\rho_{\mathbb{C}\setminus\overline{\Delta_{\alpha}}}(z)\asymp\frac{1}{{\rm dist}(z,\Delta_\alpha)}.$$
	\end{proof}

	\begin{lemma}
		\label{l8221}Let $a_1,a_2,a_3,a_4,a_5,a_6\in\partial\Delta_{\alpha}$ with $a_1<a_2<a_3\leq a_4<a_5<a_6<a_1$. 
		Let $\gamma_1$ be the geodesic with respect to $\hat{\mathbb{C}}\setminus\overline{\Delta_{\alpha}}$ connecting $a_1$ and $a_3$; let $\gamma_2$ be the geodesic with respect to $\hat{\mathbb{C}}\setminus\overline{\Delta_{\alpha}}$ connecting $a_4$ and $a_6$; let $\gamma_3$ be the geodesic with respect to $\hat{\mathbb{C}}\setminus\overline{\Delta_{\alpha}}$ connecting $a_2$ and $a_5$. Let $\gamma_3'$ be the subarc of $\gamma_3$ bounded by $\gamma_1$ and $\gamma_2$.
		Let $I_1$ be the interval on $\partial\Delta_{\alpha}$ bounded by $a_1$ and $a_2$ not containing $a_3$; let $I_2$ be the interval on $\partial\Delta_{\alpha}$ bounded by $a_2$ and $a_3$ not containing $a_1$; let $I_3$ be the interval on $\partial\Delta_{\alpha}$ bounded by $a_4$ and $a_5$ not containing $a_6$; let $I_4$ be the interval on $\partial\Delta_{\alpha}$ bounded by $a_5$ and $a_6$ not containing $a_4$. 
		For all $c>0$, if $\mathcal{W}_{\overline{\Delta_{\alpha}}}^+(I_s,I_t)\succeq_c1$ for all $s\in\{1,2\}$ and all $t\in\{3,4\}$,
		then $l_{\hat{\mathbb{C}}\setminus\overline{\Delta_{\alpha}}}(\gamma_3')\preceq_c1$.
	\end{lemma}
	
	\begin{proof}
		Let $\phi$ be the conformal map from $\Delta_{\alpha}$ to $\mathbb{H}_-:=\{z=x+yi\in\mathbb{C}:\ x\in\mathbb{R}, y<0\}$ such that the boundary extension, also written as $\phi$, maps $a_2, a_3, a_5$ to $0, 1, \infty$. 
		Then $\phi(a_6)<\phi(a_1)<\phi(a_2)=0<\phi(a_3)=1\leq\phi(a_4)$ and $\mathcal{W}_{\mathbb{H}}^-(\phi(I_s),\phi(I_t))\succeq_c1$ for all $s\in\{1,2\}$ and all $t\in\{3,4\}$.
		Observe that $\phi(I_1)=[\phi(a_1),\phi(a_2)]$,  $\phi(I_2)=[\phi(a_2),\phi(a_3)]=[0,1]$, $\phi(I_3)=[\phi(a_4),+\infty]$ and $\phi(I_4)=[-\infty,\phi(a_6)]$. Thus
		\begin{equation}
			\label{e20260305a}1=\phi(a_3)\leq\phi(a_4)\asymp_c1,\ \phi(a_1)\asymp_c-1\ {\rm and}\ -1\asymp_c\phi(a_6)<\phi(a_1).
		\end{equation}
		Let $b_1$ (resp. $b_2$) be the intersection point between $\phi(\gamma_1)$ (resp. $\phi(\gamma_2)$) and the imaginary axis. Observe that $\phi(\gamma_3)$ is the positive imaginary axis. Then $\phi(\gamma_3')$ is the segment with endpoints $b_1$ and $b_2$ in the positive imaginary axis. Since $\phi(\gamma_1)$ (resp. $\phi(\gamma_2)$) is the geodesic with endpoints $\phi(a_1)$ and $\phi(a_3)$ (resp. $\phi(a_4)$ and $\phi(a_6)$) with respect to the hyperbolic metric on $\mathbb{H}$, (\ref{e20260305a}) gives $1\asymp_cb_1<b_2\asymp_c1$. Thus $l_{\hat{\mathbb{C}}\setminus\overline{\Delta_{\alpha}}}(\gamma_3')=l_{\mathbb{H}}(\phi(\gamma_3'))=\ln\frac{b_2}{b_1}\preceq_c1$.
		
	\end{proof}
	
	\begin{lemma}
		\label{l20257193a}
		For all $n\geq\mathfrak{q}_4$, ${\rm diam}_{\mathbb{C}\setminus\overline{\Delta_{\alpha}}}\tilde{\Omega}_{c_0}^n\setminus\Lambda_\alpha^{n+\mathfrak{q}_3}\preceq_{\mathfrak{q}_3}1$.
	\end{lemma}
	\begin{proof}
		For all $n\leq m\leq n+\mathfrak{q}_3$, we let $I$ be an interval in $\mathfrak{D}_{m}$ containing at least two intervals in $\mathfrak{D}_{n+\mathfrak{q}_3}$ or be the union of two adjacent intervals in $\mathfrak{D}_{m}$. Let $a$ and $b$ be two endpoints of $I$. Let $I_1$ and $I_2$ be two adjacent intervals in $\mathfrak{D}_{n+\mathfrak{q}_3}$  having the common endpoint $a$; let $J_1$ and $J_2$ be two adjacent intervals in $\mathfrak{D}_{n+\mathfrak{q}_3}$ having the common endpoint $b$. Let $\gamma_I$ be a geodesic connecting two endpoints of $a$ and $b$ with respect to the hyperbolic metric on $\hat{\mathbb{C}}\setminus\overline{\Delta_{\alpha}}$.
		We denote by $\gamma_{I}^{n+\mathfrak{q}_3}$ the subarc of $\gamma_{I}$ bounded by $\gamma_b^{n+\mathfrak{q}_3}$ and $\gamma_a^{n+\mathfrak{q}_3}$.
		Observe that each of endpoints of $I_1$, $I_2$, $J_1$ and $J_2$ has a combinatorial distance at most $2\iota_{m+1}$ with ${\rm EP}(\mathfrak{D}_{m})$. Then $I_1$, $I_2$, $J_1$ and $J_2$ are well-grounded intervals on $\partial\hat{\Delta}_{\alpha}^m$. We let $\hat{I}_1$, $\hat{I}_2$, $\hat{J}_1$ and $\hat{J}_2$ be projections of $I_1$, $I_2$, $J_1$ and $J_2$ onto $\partial\hat{\Delta}_{\alpha}^m$, respectively.
		Then $$\frac{1}{M^{n+\mathfrak{q}_3+1}}\leq|\hat{I}_1|_{\hat{\Delta}_{\alpha}^m},|\hat{I}_2|_{\hat{\Delta}_{\alpha}^m},|\hat{J}_1|_{\hat{\Delta}_{\alpha}^m},|\hat{J}_2|_{\hat{\Delta}_{\alpha}^m}\leq\frac{1}{M^{n+\mathfrak{q}_3}}$$ and for all $s,t\in\{1,2\}$, 
		$${\rm dist}_{\hat{\Delta}_{\alpha}^m}(\hat{I}_s,\hat{J}_t)=0\ {\rm or}\ \frac{1}{M^{n+\mathfrak{q}_3+1}}\leq{\rm dist}_{\hat{\Delta}_{\alpha}^m}(\hat{I}_s,\hat{J}_t)\leq\frac{1}{M^{n-1}}.$$
		By Lemma \ref{l831} $\mathcal{W}_{\overline{\Delta_{\alpha}}}^+(I_s,J_t)\succeq_{\mathfrak{q}_3}1$ for $s=1,2$ and $t=1,2$. Thus
		by Lemma \ref{l8221}, 
		\begin{equation}
			\label{e05042}l_{\hat{\mathbb{C}}\setminus\overline{\Delta_{\alpha}}}(\gamma_{I}^{n+\mathfrak{q}_3})\preceq_{\mathfrak{q}_3}1.
		\end{equation}
		
		We assume that $\beta$ is a dam of an interval $T\in\mathfrak{D}_m$ ($n\leq m\leq n+\mathfrak{q}_3-1$). 
		Let $a(\beta)$ $(\in{\rm CP}_{m+1}\setminus{\rm CP}_m)$ and $b(\beta)$ $(\in{\rm CP}_{m+1}\setminus{\rm CP}_m)$ be two endpoints of $\beta$ and let $a(T)$ and $b(T)$ be two endpoints of $T$ such that $a(T)<a(\beta)<b(\beta)<b(T)$. Let $T_l$ be the subinterval of $T$ such that $a(\beta)$ and $a(T)$
		are two endpoints of $T_l$; let $T_r$ be the subinterval of $T$ such that $b(\beta)$ and $b(T)$
		are two endpoints of $T_r$. For all $x\in\{a(T),b(T),a(\beta),b(\beta)\}$ and $j\in\{m+1,n+\mathfrak{q}_3\}$, we denote by $I_{x}^{j}$ the union of two adjacent intervals of level $j$ having a common endpoint $x$.
		We denote by $E$ the set consisting of the union of two adjacent intervals of level $m+1$ contained in $T_l\cup T_r$.
		By (P1) in Section \ref{s2.3} we have that $\#E$, the number of elements of $E$, satisfies
		\begin{equation}
			\label{e20260306a}\#E\preceq1.
		\end{equation}
		If $a(\beta)$ and $b(\beta)$ are not two endpoints of some common interval of level $n+\mathfrak{q}_3$, then there exists the geodesic arc
		$\beta^{n+\mathfrak{q}_3}$ that is the subarc of $\beta$ bounded by $\gamma_{a(\beta)}^{n+\mathfrak{q}_3}$ and $\gamma_{b(\beta)}^{n+\mathfrak{q}_3}$. 
		It is easy to see that
		\begin{itemize}
			\item[(a3)] two endpoints of $\beta^{n+\mathfrak{q}_3}$ are contained in $\gamma_{I_{a(\beta)}^{n+\mathfrak{q}_3}}^{n+\mathfrak{q}_3}\cup\gamma_{I_{b(\beta)}^{n+\mathfrak{q}_3}}^{n+\mathfrak{q}_3}$.
			\item[(b3)] $\gamma_T^{n+\mathfrak{q}_3}\cup\bigcup\limits_{e\in E}\gamma_e^{n+\mathfrak{q}_3}\cup\gamma_{I_{a(T)}^{m+1}}^{n+\mathfrak{q}_3}\cup\gamma_{I_{b(T)}^{m+1}}^{n+\mathfrak{q}_3}\cup\gamma_{I_{a(\beta)}^{n+\mathfrak{q}_3}}^{n+\mathfrak{q}_3}\cup\gamma_{I_{b(\beta)}^{n+\mathfrak{q}_3}}^{n+\mathfrak{q}_3}$ is a connected set.
		\end{itemize} 
		It follows from (a3), (b3) and the fact that $\beta^{n+\mathfrak{q}_3}$ is a geodesic that
		\begin{align*}
			l_{\hat{\mathbb{C}}\setminus\overline{\Delta_{\alpha}}}(\beta^{n+\mathfrak{q}_3})\leq & l_{\hat{\mathbb{C}}\setminus\overline{\Delta_{\alpha}}}(\gamma_T^{n+\mathfrak{q}_3})+\sum\limits_{e\in E}l_{\hat{\mathbb{C}}\setminus\overline{\Delta_{\alpha}}}(\gamma_e^{n+\mathfrak{q}_3})\\
			&+l_{\hat{\mathbb{C}}\setminus\overline{\Delta_{\alpha}}}(\gamma_{I_{a(T)}^{m+1}}^{n+\mathfrak{q}_3})+l_{\hat{\mathbb{C}}\setminus\overline{\Delta_{\alpha}}}(\gamma_{I_{b(T)}^{m+1}}^{n+\mathfrak{q}_3})+l_{\hat{\mathbb{C}}\setminus\overline{\Delta_{\alpha}}}(\gamma_{I_{a(\beta)}^{n+\mathfrak{q}_3}}^{n+\mathfrak{q}_3})+l_{\hat{\mathbb{C}}\setminus\overline{\Delta_{\alpha}}}(\gamma_{I_{b(\beta)}^{n+\mathfrak{q}_3}}^{n+\mathfrak{q}_3}).
		\end{align*}
		By (\ref{e05042}) and (\ref{e20260306a}) we have
		\begin{align*}
			1\succeq_{\mathfrak{q}_3} & l_{\hat{\mathbb{C}}\setminus\overline{\Delta_{\alpha}}}(\gamma_T^{n+\mathfrak{q}_3})+\sum\limits_{e\in E}l_{\hat{\mathbb{C}}\setminus\overline{\Delta_{\alpha}}}(\gamma_e^{n+\mathfrak{q}_3})\\
			&+l_{\hat{\mathbb{C}}\setminus\overline{\Delta_{\alpha}}}(\gamma_{I_{a(T)}^{m+1}}^{n+\mathfrak{q}_3})+l_{\hat{\mathbb{C}}\setminus\overline{\Delta_{\alpha}}}(\gamma_{I_{b(T)}^{m+1}}^{n+\mathfrak{q}_3})+l_{\hat{\mathbb{C}}\setminus\overline{\Delta_{\alpha}}}(\gamma_{I_{a(\beta)}^{n+\mathfrak{q}_3}}^{n+\mathfrak{q}_3})+l_{\hat{\mathbb{C}}\setminus\overline{\Delta_{\alpha}}}(\gamma_{I_{b(\beta)}^{n+\mathfrak{q}_3}}^{n+\mathfrak{q}_3}).
		\end{align*}
		Thus
		\begin{equation}
			\label{e251019a}l_{\hat{\mathbb{C}}\setminus\overline{\Delta_{\alpha}}}(\beta^{n+\mathfrak{q}_3})\preceq_{\mathfrak{q}_3}1.
		\end{equation}
		
		By Lemma \ref{l25070716}, to prove that for all $n\geq\mathfrak{q}_4$,
		$${\rm diam}_{\mathbb{C}\setminus\overline{\Delta_{\alpha}}}\tilde{\Omega}_{c_0}^n\setminus\Lambda_\alpha^{n+\mathfrak{q}_3}\preceq_{\mathfrak{q}_3}1,$$
		we need only to prove $${\rm diam}_{\hat{\mathbb{C}}\setminus\overline{\Delta_{\alpha}}}\tilde{\Omega}_{c_0}^n\setminus\Lambda_\alpha^{n+\mathfrak{q}_3}\preceq_{\mathfrak{q}_3}1.$$
		The boundary $\partial(\tilde{\Omega}_{c_0}^n\setminus\Lambda_\alpha^{n+\mathfrak{q}_3})$ consists of finitely many ($\asymp\mathfrak{q}_3$) end-to-end geodesics on $\hat{\mathbb{C}}\setminus\overline{\Delta_{\alpha}}$, and each of these geodesics must be one of the following form: (see Figure \ref{f20260720e})
		\begin{itemize}
			\item $\gamma_{e}^{n+\mathfrak{q}_3}$ or its subarcs, where $e$ is the union of two adjacent intervals in $\mathfrak{D}_{m}$
			($m=n\ {\rm or}\ n+\mathfrak{q}_3$), 
			\item $\beta^{n+\mathfrak{q}_3}$,
			where $\beta$ is a dam of an interval of level $m$ with $n\leq m\leq n+\mathfrak{q}_3-1$.
		\end{itemize}
		\begin{figure}
			\centering
			\includegraphics[scale=0.5]{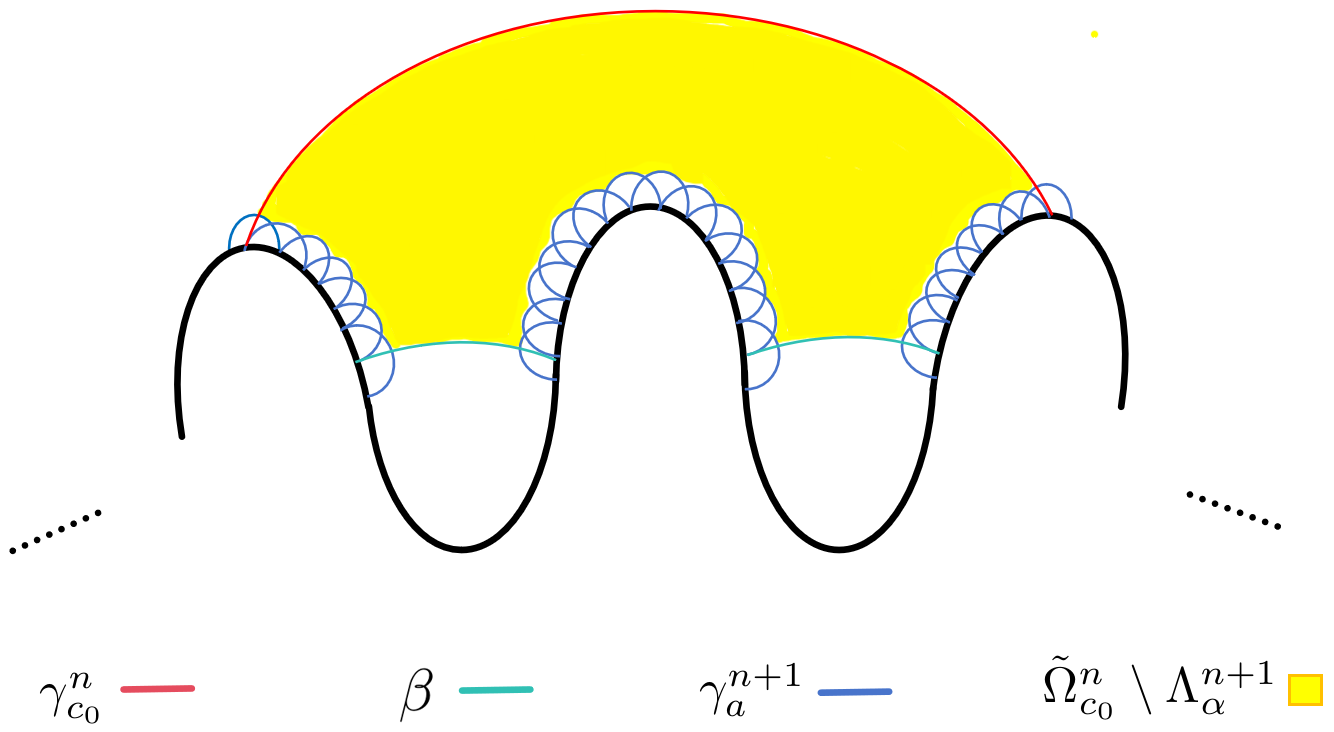}
			\caption{For simplicity, we present an illustration of $\tilde{\Omega}_{c_0}^n\setminus\Lambda_\alpha^{n+\mathfrak{q}_3}$ for only $\mathfrak{q}_3=1$.}
			\label{f20260720e}
		\end{figure}
		By (\ref{e05042}) and (\ref{e251019a}), we have
		$${\rm diam}_{\hat{\mathbb{C}}\setminus\overline{\Delta_{\alpha}}}\tilde{\Omega}_{c_0}^n\setminus\Lambda_\alpha^{n+\mathfrak{q}_3}\preceq_{\mathfrak{q}_3}1.$$
	\end{proof}
	
	\noindent 3.4.2. {\bf Preimages of pseudo-Siegel disks in hovered main mountains.}
	
	\begin{lemma}
		\label{l727a1}Let $D$ $(\subseteq\mathbb{C})$ be a $K$-quasidisk with $P_\alpha(c_0)\in\partial D$. Let $D^{-1}$ be a component of $P_{\alpha}^{-1}(D)$. Then $D^{-1}$ is a $K'$-quasidisk, where $K'$ is determined by $K$ and independent of $\alpha$.
	\end{lemma}
	\begin{proof}
		Let $\varphi$ be a $K$-quasiconformal map from $\mathbb{C}$ to $\mathbb{C}$ such that $\varphi(\partial D)=\mathbb{S}^1$ and $\varphi(P_\alpha(c_0))=1$.
		An immediate computation gives 
		$$\mu_{\varphi\comp P_\alpha}(z)=\mu_{\varphi}(P_\alpha(z))\frac{\overline{P'_{\alpha}(z)}}{P'_{\alpha}(z)},$$
		where $\mu_{\varphi\comp P_\alpha}$ and $\mu_{\varphi}$ are Beltrami coefficients of $\varphi\comp P_\alpha$ and $\varphi$, respectively.
		Then
		$||\mu_{\varphi\comp P_\alpha}||_{\infty}=||\mu_{\varphi}||_{\infty}<1$ and hence Ahlfors-Bers theorem gives that there exists a $K$-quasiconformal map $\phi$ from $\mathbb{C}$ to $\mathbb{C}$ such that $\mu_{\phi}(z)=\frac{\phi_{\overline{z}}(z)}{\phi_z(z)}=\mu_{\varphi\comp P_\alpha}(z)$, $\phi(c_0)=1$ and $\phi$ maps some zero of $\varphi\comp P_\alpha-2$ to $0$.
		Then $\varphi\comp P_\alpha\comp\phi^{-1}$ is an analytic branched covering of degree $2$ on $\mathbb{C}$ having the critical point $1$, fixes $1$ and maps $0$ to $2$.
		Thus $\varphi\comp P_\alpha\comp\phi^{-1}(z)$ is the polynomial $(z-1)^2+1$.
		It follows that
		$(\phi(z)-1)^2+1=\varphi\comp P_\alpha(z)$ and hence
		$$(\phi(D^{-1})-1)^2=\varphi\comp P_\alpha(D^{-1})-1=\mathbb{D}-1.$$
		This implies that
		$\phi(D^{-1})$ is a component of $\sqrt{\mathbb{D}-1}+1$ and hence
		$\phi(D^{-1})$ is a  $K_1$-quasidisk, where $K_1>1$ is an absolute constant. We write $K':=KK_1$.
		Thus $D^{-1}$ is a $K'$-quasidisk.
		
	\end{proof}

	\begin{lemma}
		\label{l822}For $n\geq\mathfrak{q}_4$ and $1\leq\mathfrak{l}\ll\mathfrak{q}_2$, we have that for all $z\in\tilde{\Omega}_{c_0}^n\setminus\Lambda_\alpha^{n+\mathfrak{q}_3}$, there exists an Euclidean ball $B$ on $\mathbb{C}\setminus\overline{\Delta_{\alpha}}$ such that
		$$B\subseteq\hat{U}_0^{n+\mathfrak{l}}\cap\tilde{\Omega}_{c_0}^n\setminus\Lambda_\alpha^{n+\mathfrak{q}_3},\ {\rm diam}_{\mathbb{C}\setminus\overline{\Delta_{\alpha}}}{B}\asymp1\ {\rm and}\ {\rm dist}_{\mathbb{C}\setminus\overline{\Delta_{\alpha}}}(B,z)=O_{\mathfrak{q}_3}(1).$$
	\end{lemma}
	
	\begin{proof}
		Let $\gamma_{n+\mathfrak{l}}$ be the geodesic connecting two endpoints of $\hat{I}_{[x_{q_{n+\mathfrak{l}+1}-1},x_{q_{n+\mathfrak{l}}-1}]}$
		with respect to the hyperbolic metric on ${\rm int}\hat{\Delta}_{\alpha}^{n+\mathfrak{l}}$ (here $\hat{I}_{[x_{q_{n+\mathfrak{l}+1}-1},x_{q_{n+\mathfrak{l}}-1}]}$ is the projection of $I_{[x_{q_{n+\mathfrak{l}+1}-1},x_{q_{n+\mathfrak{l}}-1}]}$ to $\partial\hat{\Delta}_{\alpha}^{n+\mathfrak{l}}$).
		We denote by $D_{c_0}^{n+\mathfrak{l}}$ the bounded region surrounded by
		$\gamma_{n+\mathfrak{l}}$ and $\hat{I}_{[x_{q_{n+\mathfrak{l}+1}-1},x_{q_{n+\mathfrak{l}}-1}]}$.
		Let $\hat{U}_{0,c_0}^{n+\mathfrak{l}}$ be the closure of the component of $P_{\alpha}^{-1}(D_{c_0}^{n+\mathfrak{l}})$ not intersecting $\overline{\Delta_{\alpha}}$.
		By Lemmas \ref{l84a} and \ref{l727a1}, we have that $\hat{U}_{0,c_0}^{n+\mathfrak{l}}$ is a closed ${\bf K'}$-quasidisk,
		where ${\bf K'}$ is determined by ${\bf K}$ and independent of $\alpha$.
		Moreover, it follows from Lemma \ref{l20257193} that
		\begin{equation}
			\label{e84a1}\hat{U}_{0,c_0}^{n+\mathfrak{l}}\subseteq\hat{U}_0^{n+\mathfrak{l}}\subseteq(\hat{U}_0^{n+\mathfrak{l}}\setminus\Lambda_\alpha^{n+2\mathfrak{q}_2})\cup\Omega_{c_0}^{n+\mathfrak{q}_2}\
			{\rm and}\ \hat{U}_{0,c_0}^{n+\mathfrak{l}}\cap\hat{\Delta}_{\alpha}^{-1}=\{c_0\}.
		\end{equation}
		
		For all $1\leq s\leq3$, we let $I_{s,l}$ and  $I_{s,r}$ be the two components of $$\overline{I_{[x_{q_{n+\mathfrak{l}+(s-1)\lfloor\frac{\mathfrak{q}_2}{4}\rfloor}-1},x_{q_{n+\mathfrak{l}+1+(s-1)\lfloor\frac{\mathfrak{q}_2}{4}\rfloor}-1}]}\setminus  I_{[x_{q_{n+\mathfrak{l}+s\lfloor\frac{\mathfrak{q}_2}{4}\rfloor}-1},x_{q_{n+\mathfrak{l}+1+s\lfloor\frac{\mathfrak{q}_2}{4}\rfloor}-1}]}}$$
		with
		$I_{s,l}<I_{[x_{q_{n+\mathfrak{l}+s\lfloor\frac{\mathfrak{q}_2}{4}\rfloor}-1},x_{q_{n+\mathfrak{l}+1+s\lfloor\frac{\mathfrak{q}_2}{4}\rfloor}-1}]}<I_{s,r}$.
		For all $1\leq s\leq3$,
		we denote by 
		$\hat{I}_{s,l}$ and  $\hat{I}_{s,r}$ projections of
		$I_{s,l}$ and  $I_{s,r}$ onto $\partial\hat{\Delta}_{\alpha}^{n+\mathfrak{l}}$, respectively.
		We define $\mathcal{G}_s$ as the interior of the union $\mathcal{G}^-_{\hat{\Delta}_{\alpha}^{n+\mathfrak{l}}}(\hat{I}_{s,l},\hat{I}_{s,r})\cup\hat{\mathcal{G}}^+_{\overline{\Delta_{\alpha}}}(\hat{I}_{s,l},\hat{I}_{s,r})$.
		By Corollary \ref{c831}, we have that for all $1\leq s\leq3$, ${\rm mod}(\mathcal{G}_s)\gg1$.
		By Lemma \ref{l841}, 
		for all $1\leq s\leq3$, $\mathcal{G}_s\subseteq\mathbb{C}$ and $P_{\alpha}(c_0)$ is contained in the bounded component of $\mathbb{C}\setminus\mathcal{G}_s$. 
		For all $1\leq s\leq3$, we define $\mathcal{G}_s^{-1}:=P_{\alpha}^{-1}(\mathcal{G}_s)$.
		Since $P_{\alpha}(c_0)$ is contained in the bounded component of $\mathbb{C}\setminus\mathcal{G}_s$, we have that $P_{\alpha}|_{\mathcal{G}_s^{-1}}$ is a covering of degree $2$ from $\mathcal{G}_s^{-1}$ to $\mathcal{G}_s$. Then ${\rm mod}(\mathcal{G}_s^{-1})=\frac{{\rm mod}(\mathcal{G}_s)}{2}\gg1$.
		
		For all $1\leq s\leq3$, we let $I_{s,l}^{-1}$ and $I_{s,r}^{-1}$ be pre-images of
		$I_{s,l}$ and $I_{s,r}$
		contained in $\partial\Delta_{\alpha}$ under $P_{\alpha}$, respectively.
		We denote by $\mathcal{R}_{s,+}^{(-1)}$ the rectangle bounded by $I_{s,l}^{-1}\cup I_{s,r}^{-1}\cup\partial^{v,0}\overline{\mathcal{G}_s^{-1}\setminus\hat{\Delta}_{\alpha}^{n+\mathfrak{l},-1}}\cup\partial^{v,1}\overline{\mathcal{G}_s^{-1}\setminus\hat{\Delta}_{\alpha}^{n+\mathfrak{l},-1}}$ with
		$\partial^{h,0}\mathcal{R}_{s,+}^{(-1)}=I_{s,l}^{-1}$ and $\partial^{h,1}\mathcal{R}_{s,+}^{(-1)}=I_{s,r}^{-1}$. 
		By (\ref{e85a})
		$$\mathcal{W}(\mathcal{R}_{s,+}^{(-1)})\succeq{\rm mod}(\mathcal{G}_s)\gg1.$$
		Since $\partial\Delta_\alpha\setminus\overline{\Omega_{c_0}^{n}}<\partial^{h,0}\mathcal{R}_{1,+}^{(-1)}<\partial^{h,0}\mathcal{R}_{2,+}^{(-1)}<\partial^{h,0}\mathcal{R}_{3,+}^{(-1)}<\partial\Delta_\alpha\cap\overline{\Omega_{c_0}^{n+\mathfrak{q}_2}}<\partial^{h,1}\mathcal{R}_{3,+}^{(-1)}<\partial^{h,1}\mathcal{R}_{2,+}^{(-1)}<\partial^{h,1}\mathcal{R}_{1,+}^{(-1)}$ (Due to $1\leq\mathfrak{l}\ll\mathfrak{q}_2$),
		by Lemma \ref{l825a} we have that $\gamma_{c_0}^n\cap\mathcal{R}_{2,+}^{(-1)}=\emptyset$ and $\gamma_{c_0}^{n+\mathfrak{q}_2}\cap\mathcal{R}_{2,+}^{(-1)}=\emptyset$.
		It follows that
		\begin{equation}
			\label{e20260307b}\Omega_{c_0}^{n+\mathfrak{q}_2}\cap\mathcal{R}_{2,+}^{(-1)}=\emptyset\ {\rm and}\ 
			\mathcal{R}_{2,+}^{(-1)}\setminus(\partial^{h,0}\mathcal{R}_{2,+}^{(-1)}\cup\partial^{h,1}\mathcal{R}_{2,+}^{(-1)})\subseteq\Omega_{c_0}^{n}.
		\end{equation}
		By (\ref{e84a1}) and (\ref{e20260307b}), we have
		\begin{equation}
			\label{e251102a}\mathcal{R}_{2,+}^{(-1)}\cap\hat{U}_{0,c_0}^{n+\mathfrak{l}}\subseteq\hat{U}_0^{n+\mathfrak{l}}\cap\tilde{\Omega}_{c_0}^n\setminus\Lambda_\alpha^{n+2\mathfrak{q}_2}\subseteq\hat{U}_0^{n+\mathfrak{l}}\cap\tilde{\Omega}_{c_0}^n\setminus\Lambda_\alpha^{n+\mathfrak{q}_3}.
		\end{equation}
		If
		$${\rm diam}(\partial^{v,0}\mathcal{R}_{2,+}^{(-1)})\leq{\rm diam}(\partial^{v,1}\mathcal{R}_{2,+}^{(-1)}),$$
		then we set $\gamma_1=\partial^{v,0}\mathcal{R}_{2,+}^{(-1)}$ and $\gamma_2=\partial^{v,1}\mathcal{R}_{2,+}^{(-1)}$; if $${\rm diam}(\partial^{v,1}\mathcal{R}_{2,+}^{(-1)})\leq{\rm diam}(\partial^{v,0}\mathcal{R}_{2,+}^{(-1)}),$$
		then we set $\gamma_1=\partial^{v,1}\mathcal{R}_{2,+}^{(-1)}$ and $\gamma_2=\partial^{v,0}\mathcal{R}_{2,+}^{(-1)}$.
		Observe that $\mathcal{G}_2^{-1}$ separates $\gamma_1$ and $\gamma_2$. Then it follows from ${\rm mod}(\mathcal{G}_2^{-1})\gg1$ that ${\rm dist}(\gamma_1,\gamma_2)\succeq{\rm diam}(\gamma_1)$.
		
		We set
		$$F_1:=\left\{z\in\mathbb{C}:\frac{{\rm dist}(\gamma_1,\gamma_2)}{3}\leq{\rm dist}(\gamma_1,z)\leq\frac{2{\rm dist}(\gamma_1,\gamma_2)}{3}\right\}.$$
		Since $\mathcal{G}^-_{\hat{\Delta}_{\alpha}^{n+\mathfrak{l}}}(\hat{I}_{2,l},\hat{I}_{2,r})\subseteq  D_{c_0}^{n+\mathfrak{l}}$, we have $$\partial^{v,0}\mathcal{R}_{2,+}^{(-1)}\cap\hat{U}_{0,c_0}^{n+\mathfrak{l}}\not=\emptyset\ {\rm and}\ \partial^{v,1}\mathcal{R}_{2,+}^{(-1)}\cap\hat{U}_{0,c_0}^{n+\mathfrak{l}}\not=\emptyset,$$
		that is $$\hat{U}_{0,c_0}^{n+\mathfrak{l}}\cap\gamma_1\not=\emptyset\ {\rm and}\ \hat{U}_{0,c_0}^{n+\mathfrak{l}}\cap \gamma_2\not=\emptyset.$$
		Together with $\hat{U}_{0,c_0}^{n+\mathfrak{l}}\cap\overline{\Delta_\alpha}=\{c_0\}$ (Due to (\ref{e84a1})), it gives
		$$\partial\hat{U}_{0,c_0}^{n+\mathfrak{l}}\cap\gamma_1\not=\emptyset\ {\rm and}\ \partial\hat{U}_{0,c_0}^{n+\mathfrak{l}}\cap \gamma_2\not=\emptyset.$$
		We take $a\in\partial\hat{U}_{0,c_0}^{n+\mathfrak{l}}\cap\gamma_1$ and $b\in\partial\hat{U}_{0,c_0}^{n+\mathfrak{l}}\cap\gamma_2$.
		Let $\gamma_3$ and $\gamma_4$ be the two arcs of $\partial\hat{U}_{0,c_0}^{n+\mathfrak{l}}$ connecting $a$ and $b$. 
		By the definition of $F_1$, we have that ${\rm dist}(\gamma_s\cap F_1,\gamma_t)\geq\frac{{\rm dist}(\gamma_1,\gamma_2)}{3}$ for $s\in\{3,4\}$ and $t\in\{1,2\}$.
		Since $\partial\hat{U}_{0,c_0}^{n+\mathfrak{l}}$ is a ${\bf K'}$-quasicircle,
		we have
		\begin{equation}
			\label{e84a2}{\rm dist}(\gamma_3\cap F_1,\gamma_4\cap F_1)\succeq{\rm dist}(\gamma_1,\gamma_2).
		\end{equation}
		
		We claim that there exists $e\in F_1\cap{\rm int}(\hat{U}_{0,c_0}^{n+1})\cap\mathcal{R}_{2,+}^{(-1)}$ such that 
		${\rm dist}(\gamma_3\cap F_1,e)\succeq{\rm dist}(\gamma_1,\gamma_2)$,
		${\rm dist}(e,\gamma_4\cap F_1)\succeq{\rm dist}(\gamma_1,\gamma_2)$ and ${\rm dist}(e,\gamma_1)=\frac{{\rm dist}(\gamma_1,\gamma_2)}{2}$. In fact, by (\ref{e84a2}), we could consider $l$-neighborhoods $F_2$ and $F_3$ of $\gamma_3\cap F_1$ and $\gamma_4\cap F_1$, where $l\asymp{\rm dist}(\gamma_1,\gamma_2)$ and $l<\frac{{\rm dist}(\gamma_1,\gamma_2)}{6}$, such that $\overline{F_2}\cap\overline{F_3}=\emptyset$. We take a path $\gamma\subseteq\hat{U}_{0,c_0}^{n+\mathfrak{l}}$ connecting $a$ and $b$ such that $\gamma([0,1])\cap(F_2\cup F_3)=\emptyset$ and $\gamma([0,1])\cap\partial\hat{U}_{0,c_0}^{n+\mathfrak{l}}=\{a,b\}$. Since $\hat{U}_{0,c_0}^{n+\mathfrak{l}}\cap\overline{\Delta_\alpha}=\{c_0\}$ (Due to (\ref{e84a1})), $\gamma$ passes through $\mathcal{R}_{2,+}^{(-1)}$. Then there exists a point $e\in\gamma\cap\mathcal{R}_{2,+}^{(-1)}$ such that ${\rm dist}(e,\gamma_1)=\frac{{\rm dist}(\gamma_1,\gamma_2)}{2}$. Thus $e\in F_1\cap{\rm int}(\hat{U}_{0,c_0}^{n+\mathfrak{l}})\cap\mathcal{R}_{2,+}^{(-1)}$ with 
		${\rm dist}(\gamma_3\cap F_1,e)\succeq{\rm dist}(\gamma_1,\gamma_2)$,
		${\rm dist}(e,\gamma_4\cap F_1)\succeq{\rm dist}(\gamma_1,\gamma_2)$ and ${\rm dist}(e,\gamma_1)=\frac{{\rm dist}(\gamma_1,\gamma_2)}{2}$.
		
		By the above claim, there exists an Euclidean ball $B\subseteq F_1\cap{\rm int}(\hat{U}_{0,c_0}^{n+\mathfrak{l}})\cap\mathcal{R}_{2,+}^{(-1)}$ centering at $e$ such that ${\rm diam}(B)\asymp{\rm dist}(\gamma_1,\gamma_2)$, ${\rm diam}(B)\leq\frac{{\rm dist}(\gamma_1,\gamma_2)}{12}$,
		${\rm dist}(\gamma_3\cap F_1,B)\succeq{\rm dist}(\gamma_1,\gamma_2)$ and
		${\rm dist}(B,\gamma_4\cap F_1)\succeq{\rm dist}(\gamma_1,\gamma_2)$.
		Thus $B\subseteq
		\mathcal{R}_{2,+}^{(-1)}\cap\hat{U}_{0,c_0}^{n+\mathfrak{l}}$ with ${\rm diam}(B)\asymp{\rm dist}(\gamma_1,\gamma_2)$,
		${\rm diam}(B)\asymp{\rm dist}(B,\gamma_1)$ and ${\rm dist}(B,\partial\hat{U}_{0,c_0}^{n+\mathfrak{l}})\succeq{\rm dist}(\gamma_1,\gamma_2)$.
		Since ${\rm diam}(B)\asymp{\rm dist}(\gamma_1,\gamma_2)$,
		${\rm diam}(B)\asymp{\rm dist}(B,\gamma_1)$,  ${\rm dist}(\gamma_1,\gamma_2)\succeq{\rm diam}(\gamma_1)$ and $\gamma_1\cap\overline{\Delta_{\alpha}}\not=\emptyset$, we have
		${\rm diam}(B)\succeq{\rm dist}(B,\Delta_{\alpha})$.
		Since $B\subseteq
		\hat{U}_{0,c_0}^{n+\mathfrak{l}}$, ${\rm dist}(B,\partial\hat{U}_{0,c_0}^{n+\mathfrak{l}})\succeq{\rm dist}(\gamma_1,\gamma_2)$, ${\rm diam}(B)\asymp{\rm dist}(\gamma_1,\gamma_2)$ and $\hat{U}_{0,c_0}^{n+\mathfrak{l}}\cap\Delta_\alpha=\emptyset$, we have
		${\rm diam}(B)\preceq{\rm dist}(B,\Delta_{\alpha})$.
		Thus ${\rm diam}(B)\asymp{\rm dist}(B,\Delta_{\alpha})$.
		By (\ref{e251102a})
		$$B\subseteq
		\mathcal{R}_{2,+}^{(-1)}\cap\hat{U}_{0,c_0}^{n+\mathfrak{l}}\subseteq\hat{U}_0^{n+\mathfrak{l}}\cap\tilde{\Omega}_{c_0}^n\setminus\Lambda_\alpha^{n+\mathfrak{q}_3}\subseteq\Lambda_{\alpha}^n$$ and hence by Lemma \ref{l25070716} and $n\geq\mathfrak{q}_4$ we have ${\rm diam}_{\mathbb{C}\setminus\overline{\Delta_{\alpha}}}(B)\asymp1$.
		Since $z\in\tilde{\Omega}_{c_0}^n\setminus\Lambda_\alpha^{n+\mathfrak{q}_3}$ and $B\subseteq\tilde{\Omega}_{c_0}^n\setminus\Lambda_\alpha^{n+\mathfrak{q}_3}$,
		by Lemma \ref{l20257193a} ${\rm dist}_{\mathbb{C}\setminus\overline{\Delta_{\alpha}}}(B,z)=O_{\mathfrak{q}_3}(1)$.
		
	\end{proof}
	
	\noindent 3.4.3. {\bf Pulling back pseudo-Siegel disks through hovered main mountains.}
	Given any $w\in\Lambda_{\alpha}^{n+1}$ with $n\geq2\mathfrak{q}_4$, we assume that there exist
	positive integers $m_1<m_2<\cdots<m_k$ ($k\gg1$) such that for all $1\leq j\leq k$,
	$$P_{\alpha}^{\comp m_j}(w)\in\tilde{\Omega}_{c_0}^{n_j}\setminus\Lambda^{n_j+\mathfrak{q}_3},\ n_j\geq n+1$$ and
	$$\left\{P_{\alpha}^{\comp t}(w):\ 0\leq t\leq m_k\right\}\subseteq\Lambda_{\alpha}^{n+1}.$$
	By Lemma \ref{l822}, for all $1\leq j\leq k$ and $1\leq\mathfrak{l}\ll\mathfrak{q}_2$,
	there exists an Euclidean ball $B_j$ $\subseteq\mathbb{C}\setminus\overline{\Delta_{\alpha}}$ such that
	$$B_j\subseteq\hat{U}_0^{n_j+\mathfrak{l}}\cap\tilde{\Omega}_{c_0}^{n_j}\setminus\Lambda^{n_j+\mathfrak{q}_3},\ {\rm diam}_{\mathbb{C}\setminus\overline{\Delta_{\alpha}}}{B_j}\asymp1\ {\rm and}\ {\rm dist}_{\mathbb{C}\setminus\overline{\Delta_{\alpha}}}(B_j,P_{\alpha}^{\comp m_j}(w))=O_{\mathfrak{q}_3}(1).$$
	Let $\tilde{B}_j$ be the Euclidean ball with the same center as $B_j$ and a half radius of $B_j$. Then combined with Lemma \ref{l25070716}, we have
	\begin{equation}
		\label{e20260327d1}\tilde{B}_j\subseteq\hat{U}_0^{n_j+\mathfrak{l}}\cap\tilde{\Omega}_{c_0}^{n_j}\setminus\Lambda^{n_j+\mathfrak{q}_3},\ {\rm diam}_{\mathbb{C}\setminus\overline{\Delta_{\alpha}}}{\tilde{B}_j}\asymp1\ {\rm and}\ {\rm dist}_{\mathbb{C}\setminus\overline{\Delta_{\alpha}}}(\tilde{B}_j,P_{\alpha}^{\comp m_j}(w))=O_{\mathfrak{q}_3}(1).
	\end{equation}
	Let $\gamma_j\subseteq \mathbb{C}\setminus\overline{\Delta_{\alpha}}$ be the geodesic connecting $\tilde{B}_j$ and $P_{\alpha}^{\comp m_j}(w)$ satisfying $$l_{\mathbb{C}\setminus\overline{\Delta_{\alpha}}}(\gamma_j)={\rm dist}_{\mathbb{C}\setminus\overline{\Delta_{\alpha}}}(\tilde{B}_j,P_{\alpha}^{\comp m_j}(w)).$$
	Then $l_{\mathbb{C}\setminus\overline{\Delta_{\alpha}}}(\gamma_j)=O_{\mathfrak{q}_3}(1)$. For all $0\leq s\leq m_j$, we
	let $\gamma_j^{(s)}$ be the component of $P_{\alpha}^{-s}(\gamma_j)$ containing $P_{\alpha}^{\comp(m_j-s)}(w)$ and $B_j^{(s)}$ be the component of $P_{\alpha}^{-s}(\tilde{B}_j)$ whose closure intersects $\gamma_j^{(s)}$. Then by the Schwarz lemma we have
	\begin{equation}
		\label{e20260326a}l_{\mathbb{C}\setminus\overline{\Delta_{\alpha}}}(\gamma_j^{(s)})\leq l_{\mathbb{C}\setminus\overline{\Delta_{\alpha}}}(\gamma_j)=O_{\mathfrak{q}_3}(1)\ {\rm and}\ {\rm diam}_{\mathbb{C}\setminus\overline{\Delta_{\alpha}}}{B_j^{(s)}}\leq{\rm diam}_{\mathbb{C}\setminus\overline{\Delta_{\alpha}}}{\tilde{B}_j}\asymp1.
	\end{equation}
	
	\vspace{0.2cm}
	\noindent{\bf Claim 1:} There exists a positive integer $\chi$ $(=\mathfrak{q}_4)$, independent of $n$ ($\geq2\mathfrak{q}_4$) and $\alpha$, such that ${\rm dist}_{\mathbb{C}\setminus\overline{\Delta_{\alpha}}}(\partial\Lambda_{\alpha}^{n-\chi},\partial\Lambda_{\alpha}^{n})\gg_{\mathfrak{q}_3}1$. As a consequence, for all $1\leq j\leq k$ and $0\leq s\leq m_j$,
	$\gamma_j^{(s)}\cup B_j^{(s)}\subseteq\Lambda_{\alpha}^{n+1-\chi}$ (Due to $P_{\alpha}^{\comp(m_j-s)}(w)\in\Lambda_{\alpha}^{n+1}$ and (\ref{e20260326a})).
	\begin{proof}
		For all $n\geq2\mathfrak{q}_4$ and $b\in{\rm EP}(\mathfrak{D}_{n})$, there exist two adjacent intervals $I, J\in\mathfrak{D}_{n-\mathfrak{q}_4}$ having a common endpoint $a$ such that
		$b\in(I\cup J)\setminus(I_1\cup J_1\cup I_2\cup J_2)$, where $I_1\subseteq I$ (resp. $ J_1\subseteq J$) is an interval of $\mathfrak{D}_{n-\mathfrak{q}_4+6}$ having a common endpoint as $I$ (resp. $J$) which is not $a$; $I_2\subseteq I$ (resp. $J_2\subseteq J$) is an interval of $\mathfrak{D}_{n-\mathfrak{q}_4+6}$ adjacent to $I_1$ (resp. $J_1$). Let $I_3$ and $J_3$ be two adjacent intervals of $\mathfrak{D}_{n}$ having the common endpoint $b$. Let $I_{3,l}$ and $J_{3,r}$ be the two components of $\overline{(I\cup J)\setminus(I_3\cup J_3)}$ such that $I_{3,l}<\overline{(I\cup J)\setminus(I_3\cup J_3)}<J_{3,r}$. Then $I_{3,l}$ (resp. $J_{3,r}$) contains $I_1$ or $J_1$. By Lemma \ref{l7301} $\mathcal{W}_{\overline{\Delta_{\alpha}}}^+(I_{3,l},J_{3,r})\succeq{\mathfrak{q}_3}$ (Due to
		$\mathfrak{q}_4>\mathfrak{q}_3$). 
		Then by Lemma \ref{l825a} $\mathcal{W}:=\mathcal{W}(\mathcal{G}_{\overline{\Delta_{\alpha}}}^+(I_{3,l},J_{3,r}))\succeq{\mathfrak{q}_3}$.
		
		For all $z\in{\rm int}(\partial^{v,0}\mathcal{G}_{\overline{\Delta_{\alpha}}}^+(I_{3,l},J_{3,r}))$ and a hyperbolic closed ball $B_z$ centering at $z$ with radius $r=\log(\frac{e^{4\pi/\mathcal{W}}+1}{e^{4\pi/\mathcal{W}}-1})$ on $\hat{\mathbb{C}}\setminus\overline{\Delta_{\alpha}}$, we have
		\begin{equation}
			\label{e20260530a}{\rm mod}(\hat{\mathbb{C}}\setminus(\overline{\Delta_{\alpha}}\cup B_z))=\frac{1}{2\pi}\log\frac{e^r+1}{e^r-1}=\frac{2}{\mathcal{W}}.
		\end{equation}
		We claim that $B_z\cap\partial^{v,1}\mathcal{G}_{\overline{\Delta_{\alpha}}}^+(I_{3,l},J_{3,r})=\emptyset$. In fact, if not, then every vertical curve of $\mathcal{G}_{\overline{\Delta_{\alpha}}}^+(I_{3,l},J_{3,r})$ passes through $B_z$ and hence ${\rm mod}(\hat{\mathbb{C}}\setminus(\overline{\Delta_{\alpha}}\cup B_z))^{-1}\geq\mathcal{W}$. This contradicts (\ref{e20260530a}).
		
		Observe that $\partial^{v,0}\mathcal{G}_{\overline{\Delta_{\alpha}}}^+(I_{3,l},J_{3,r})=\gamma_a^{n-\mathfrak{q}_4}$ and $\partial^{v,1}\mathcal{G}_{\overline{\Delta_{\alpha}}}^+(I_{3,l},J_{3,r})=\gamma_b^{n}$. 
		Then by the above claim we have
		${\rm dist}_{\hat{\mathbb{C}}\setminus\overline{\Delta_{\alpha}}}(\gamma_b^{n},\gamma_a^{n-\mathfrak{q}_4})\geq r\gg_{\mathfrak{q}_3}1$ and hence
		${\rm dist}_{\hat{\mathbb{C}}\setminus\overline{\Delta_{\alpha}}}(\gamma_b^{n},\partial\Lambda_{\alpha}^{n-\mathfrak{q}_4})\gg_{\mathfrak{q}_3}1$. This implies ${\rm dist}_{\hat{\mathbb{C}}\setminus\overline{\Delta_{\alpha}}}(\partial\Lambda_{\alpha}^{n-\mathfrak{q}_4},\partial\Lambda_{\alpha}^{n})\gg_{\mathfrak{q}_3}1$.
		By Lemma \ref{l25070716} and $n\geq2\mathfrak{q}_4$, ${\rm dist}_{\mathbb{C}\setminus\overline{\Delta_{\alpha}}}(\partial\Lambda_{\alpha}^{n-\mathfrak{q}_4},\partial\Lambda_{\alpha}^{n})\gg_{\mathfrak{q}_3}1$.
		At last, by taking $\chi=\mathfrak{q}_4$, we complete the proof.
	\end{proof}
	\begin{corollary}
		\label{c20260331a}Let $D$ be a closed quasidisk on $\mathbb{C}$ and $\mathcal{R}$ be a rectangle such that $\partial^{h,0}\mathcal{R}\cup\partial^{h,1}\mathcal{R}\subseteq\partial D$, $\mathcal{R}\setminus(\partial^{h,0}\mathcal{R}\cup\partial^{h,1}\mathcal{R})\subseteq{\rm int}(D)$. Then ${\rm dist}_{{\rm int}(D)}(\partial^{v,0}\mathcal{R},\partial^{v,1}\mathcal{R})\geq\log(\frac{e^{4\pi/\mathcal{W}(\mathcal{R})}+1}{e^{4\pi/\mathcal{W}(\mathcal{R})}-1})$.
	\end{corollary}
	\begin{proof}
		As the claim in the proof of Claim $1$, for all $z\in{\rm int}(\partial^{v,0}\mathcal{R})$ and a hyperbolic closed ball $B_z$ centering at $z$ with radius $\log(\frac{e^{4\pi/\mathcal{W}(\mathcal{R})}+1}{e^{4\pi/\mathcal{W}(\mathcal{R})}-1})$ on ${\rm int}(D)$, we can obtain that $B_z\cap\partial^{v,1}\mathcal{R}=\emptyset$. Thus ${\rm dist}_{{\rm int}(D)}(\partial^{v,0}\mathcal{R},\partial^{v,1}\mathcal{R})\geq\log(\frac{e^{4\pi/\mathcal{W}(\mathcal{R})}+1}{e^{4\pi/\mathcal{W}(\mathcal{R})}-1})$.
	\end{proof}
	
	In the proof of Claim $1$ we have proved that for all $n\geq2\mathfrak{q}_4$,
	$${\rm dist}_{\hat{\mathbb{C}}\setminus\overline{\Delta_{\alpha}}}(\partial\Lambda_{\alpha}^{n-\mathfrak{q}_4},\partial\Lambda_{\alpha}^{n})\gg_{\mathfrak{q}_3}1.$$
	Then for any $w\in\Lambda_{\alpha}^{2\mathfrak{q}_4}$, there exists a hyperbolic ball contained in $\Lambda_{\alpha}^{\mathfrak{q}_4}$ centering at $w$ with radius $\gg_{\mathfrak{q}_3}1$ with respect to the hyperbolic metric on $\hat{\mathbb{C}}\setminus\overline{\Delta_\alpha}$. Thus by Lemma \ref{l25070716} the hyperbolic ball $\mathbb{B}(w)$ centering at $w$ with the radius $\asymp_{\mathfrak{q}_3}1$ with respect to the hyperbolic ball on $\mathbb{C}\setminus\overline{\Delta_\alpha}$ is well defined.
	
	\begin{lemma}
		\label{l8241}For all nonnegative integer $s$ and $w\in\mathbb{C}\setminus\overline{\Delta_{\alpha}}$ with $P_{\alpha}^{\comp s}(w)\in\Lambda_{\alpha}^{2\mathfrak{q}_4}$, we let $\mathbb{B}(P_{\alpha}^{\comp s}(w))$ be an open hyperbolic ball centering at $P_{\alpha}^{\comp s}(w)$ with the hyperbolic radius $r\preceq_{\mathfrak{q}_3}1$ on $\mathbb{C}\setminus\overline{\Delta_{\alpha}}$, and let $\mathbb{B}^{-1}(P_{\alpha}^{\comp s}(w))$ be the component of the preimage of $\mathbb{B}(P_{\alpha}^{\comp s}(w))$ containing $w$ under $P_{\alpha}^{\comp s}$. Then for all $w'\in\mathbb{B}^{-1}(P_{\alpha}^{\comp s}(w))$, we have that 
		\begin{equation}
			\label{e20251210a1}|(P_{\alpha}^{\comp s})'(w)|\cdot\rho_{\mathbb{C}\setminus\overline{\Delta_{\alpha}}}(P_{\alpha}^{\comp s}(w))\asymp_{\mathfrak{q}_3}|(P_{\alpha}^{\comp s})'(w')|\cdot\rho_{\mathbb{C}\setminus\overline{\Delta_{\alpha}}}(P_{\alpha}^{\comp s}(w')).
		\end{equation}
		Moreover, if $E$ {\rm(}resp. $F${\rm)} is a subset of $\mathbb{B}(P_{\alpha}^{\comp s}(w))$ with at least two points and $E^{-1}$ {\rm(}resp. $F^{-1}${\rm)} is the component of the preimage of $E$ {\rm(}resp. $F${\rm)} under $P_{\alpha}^{\comp s}$ contained in $\mathbb{B}^{-1}(P_{\alpha}^{\comp s}(w))$, then 
		\begin{equation}
			\label{e20251210a2}
			\frac{{\rm diam}_{\mathbb{C}\setminus\overline{\Delta_{\alpha}}}(E^{-1})}{{\rm diam}_{\mathbb{C}\setminus\overline{\Delta_{\alpha}}}(F^{-1})}\asymp_{\mathfrak{q}_3}\frac{{\rm diam}_{\mathbb{C}\setminus\overline{\Delta_{\alpha}}}(E)}{{\rm diam}_{\mathbb{C}\setminus\overline{\Delta_{\alpha}}}(F)}.
		\end{equation}
		As a consequence,
		for any rectifiable curve $\gamma$ in $\mathbb{B}(P_{\alpha}^{\comp s}(w))$ with $l_{\mathbb{C}\setminus\overline{\Delta_{\alpha}}}(\gamma)>0$, we let $\gamma^{-1}$ be the component of the preimage of $\gamma$ under $P_{\alpha}^{\comp s}$ contained in $\mathbb{B}^{-1}(P_{\alpha}^{\comp s}(w))$. Then we have
		\begin{equation}
			\label{e20260328a}
			\frac{l_{\mathbb{C}\setminus\overline{\Delta_{\alpha}}}(\gamma^{-1})}{{\rm diam}_{\mathbb{C}\setminus\overline{\Delta_{\alpha}}}(E^{-1})}\asymp_{\mathfrak{q}_3}\frac{l_{\mathbb{C}\setminus\overline{\Delta_{\alpha}}}(\gamma)}{{\rm diam}_{\mathbb{C}\setminus\overline{\Delta_{\alpha}}}(E)}.
		\end{equation}
	\end{lemma}
	\begin{proof}
		Let $\tilde{\mathbb{B}}(P_{\alpha}^{\comp s}(w))$ be the hyperbolic ball  centering at $P_{\alpha}^{\comp s}(w)$ with the radius $2r$ on $\mathbb{C}\setminus\overline{\Delta_{\alpha}}$.
		By pulling $\tilde{\mathbb{B}}(P_{\alpha}^{\comp s}(w))$ back to the universal covering space $\mathbb{D}$ under a holomorphic covering map $p_1:\mathbb{D}\to\mathbb{C}\setminus\overline{\Delta_{\alpha}}$ such that $0$ is a preimage of $P_{\alpha}^{\comp s}(w)$, it is clear that 
		${\rm mod}(\tilde{\mathbb{B}}(P_{\alpha}^{\comp s}(w))\setminus\overline{\mathbb{B}(P_{\alpha}^{\comp s}(w))})\asymp_{\mathfrak{q}_3}1$.
		Let $\tilde{\phi}$ be the branch of $P_{\alpha}^{-s}$ mapping $P_{\alpha}^{\comp s}(w)$ to $w$ on $\tilde{\mathbb{B}}(P_{\alpha}^{\comp s}(w))$. The distortion theorem gives that for all $z,z'\in\mathbb{B}(P_{\alpha}^{\comp s}(w))$, $|\tilde{\phi}'(z)|\asymp_{\mathfrak{q}_3}|\tilde{\phi}'(z')|$,
		and hence 
		\begin{equation}
			\label{e8251}|(P_{\alpha}^{\comp s})'(\tilde{\phi}(z))|\asymp_{\mathfrak{q}_3}|(P_{\alpha}^{\comp s})'(\tilde{\phi}(z'))|.
		\end{equation}
		Let $\psi$ be a branch of $p_1^{-1}$ on $\tilde{\mathbb{B}}(P_{\alpha}^{\comp s}(w))$ such that $\psi(P_{\alpha}^{\comp s}(w))=0$.
		The distortion theorem gives that for all $z,z'\in\mathbb{B}(P_{\alpha}^{\comp s}(w))$, $|\psi'(z)|\asymp_{\mathfrak{q}_3}|\psi'(z')|$. Observe that $\psi(\mathbb{B}(P_{\alpha}^{\comp s}(w)))$ is a hyperbolic ball centering at $0$ with the radius $r\preceq_{\mathfrak{q}_3}1$ on $\mathbb{D}$. Then a simple computation gives that for all $z,z'\in\mathbb{B}(P_{\alpha}^{\comp s}(w))$, 
		$\rho_{\mathbb{D}}(\psi(z))\asymp_{\mathfrak{q}_3}\rho_{\mathbb{D}}(\psi(z'))$.
		Thus for all $z,z'\in\mathbb{B}(P_{\alpha}^{\comp s}(w))$,
		\begin{equation}
			\label{e20251208b}\rho_{\mathbb{C}\setminus\overline{\Delta_{\alpha}}}(z)=|\psi'(z)|\cdot\rho_{\mathbb{D}}(\psi(z))\asymp_{\mathfrak{q}_3}|\psi'(z')|\cdot\rho_{\mathbb{D}}(\psi(z'))=\rho_{\mathbb{C}\setminus\overline{\Delta_{\alpha}}}(z').
		\end{equation}
		Combining (\ref{e8251}) and (\ref{e20251208b}), we have that for all $w'\in\mathbb{B}^{-1}(P_{\alpha}^{\comp s}(w))$,
		$$|(P_{\alpha}^{\comp s})'(w)|\cdot\rho_{\mathbb{C}\setminus\overline{\Delta_{\alpha}}}(P_{\alpha}^{\comp s}(w))\asymp_{\mathfrak{q}_3}|(P_{\alpha}^{\comp s})'(w')|\cdot\rho_{\mathbb{C}\setminus\overline{\Delta_{\alpha}}}(P_{\alpha}^{\comp s}(w')).$$
		
		Since $P_{\alpha}^{\comp s}$ is a holomorphic covering map from $\mathbb{C}\setminus\overline{P_{\alpha}^{-s}(\Delta_{\alpha}})$ to $\mathbb{C}\setminus\overline{
			\Delta_{\alpha}}$, we have that
		\begin{equation}
			\label{e20251211a1}{\rm diam}_{\mathbb{C}\setminus\overline{\Delta_{\alpha}}}(E)={\rm diam}_{\mathbb{C}\setminus\overline{P_{\alpha}^{-s}(\Delta_{\alpha}})}(E^{-1})\ {\rm and}\
			{\rm diam}_{\mathbb{C}\setminus\overline{\Delta_{\alpha}}}(F)={\rm diam}_{\mathbb{C}\setminus\overline{P_{\alpha}^{-s}(\Delta_{\alpha}})}(F^{-1}).
		\end{equation}
		Let $\tilde{\mathbb{B}}^{-1}(P_{\alpha}^{\comp s}(w))$ be the component of the preimage of $\tilde{\mathbb{B}}(P_{\alpha}^{\comp s}(w))$ containing $w$ under $P_{\alpha}^{\comp s}$.
		Then it follows from ${\rm mod}(\tilde{\mathbb{B}}(P_{\alpha}^{\comp s}(w))\setminus\overline{\mathbb{B}(P_{\alpha}^{\comp s}(w))})\asymp_{\mathfrak{q}_3}1$ that 
		$${\rm mod}(\tilde{\mathbb{B}}^{-1}(P_{\alpha}^{\comp s}(w))\setminus\overline{\mathbb{B}^{-1}(P_{\alpha}^{\comp s}(w))})\asymp_{\mathfrak{q}_3}1.$$
		
		For any hyperbolic region $U\subseteq\mathbb{C}$ with $\tilde{\mathbb{B}}^{-1}(P_{\alpha}^{\comp s}(w))\subseteq U$, we let  $p_U:\mathbb{D}\to U$ be a holomorphic covering map such that $0$ is a preimage of $w$ and let $\psi_U$ be the branch of $p_U^{-1}$ on $\tilde{\mathbb{B}}^{-1}(P_{\alpha}^{\comp s}(w))$ such that $\psi_U(w)=0$.
		Since
		$\psi_U(\tilde{\mathbb{B}}^{-1}(P_{\alpha}^{\comp s}(w))\setminus\overline{\mathbb{B}^{-1}(P_{\alpha}^{\comp s}(w))}))\subseteq\mathbb{D}\setminus\psi_U(\overline{\mathbb{B}^{-1}(P_{\alpha}^{\comp s}(w))})$, we have
		\begin{equation}
			\label{e20260316c}{\rm mod}(\mathbb{D}\setminus\psi_U(\overline{\mathbb{B}^{-1}(P_{\alpha}^{\comp s}(w))}))\geq{\rm mod}(\tilde{\mathbb{B}}^{-1}(P_{\alpha}^{\comp s}(w))\setminus\overline{\mathbb{B}^{-1}(P_{\alpha}^{\comp s}(w))})\asymp_{\mathfrak{q}_3}1.
		\end{equation}
		Then for any two subsets $X,Y\subseteq\overline{\mathbb{B}^{-1}(P_{\alpha}^{\comp s}(w))}$ both of whose diameters are $>0$, 
		we have that
		\begin{equation}
			\label{e20260316d}
			\frac{{\rm diam}(\psi_U(X))}{{\rm diam}(\psi_U(Y))}\asymp_{\mathfrak{q}_3}\frac{{\rm diam}_{\mathbb{D}}(\psi_U(X))}{{\rm diam}_{\mathbb{D}}(\psi_U(Y))}.
		\end{equation}
		We write $W:=\tilde{\mathbb{B}}^{-1}(P_{\alpha}^{\comp s}(w))$. Observe that $p_{W}$ is a conformal map from $\mathbb{D}$ to $W$.
		Then we consider the composition $\psi_U\comp p_{W}: \mathbb{D}\to\mathbb{D}$, that maps $\psi_W(X)$ and $\psi_W(Y)$ to $\psi_U(X)$ and $\psi_U(Y)$, respectively.
		Observe that $\psi_W(X)\cup\psi_W(Y)\subseteq\psi_W(\overline{\mathbb{B}^{-1}(P_{\alpha}^{\comp s}(w))})$
		and ${\rm mod}(\mathbb{D}\setminus\psi_W(\overline{\mathbb{B}^{-1}(P_{\alpha}^{\comp s}(w))}))\succeq_{\mathfrak{q}_3}1$ (Due to (\ref{e20260316c})).
		Then the distortion theorem gives
		\begin{equation}
			\label{e20260316e}
			\frac{{\rm diam}(\psi_W(X))}{{\rm diam}(\psi_W(Y))}\asymp_{\mathfrak{q}_3}\frac{{\rm diam}(\psi_U(X))}{{\rm diam}(\psi_U(Y))}.
		\end{equation}
		By (\ref{e20260316d}) and (\ref{e20260316e}) we have
		$$\frac{{\rm diam}_{\mathbb{D}}(\psi_U(X))}{{\rm diam}_{\mathbb{D}}(\psi_U(Y))}\asymp_{\mathfrak{q}_3}\frac{{\rm diam}_{\mathbb{D}}(\psi_W(X))}{{\rm diam}_{\mathbb{D}}(\psi_W(Y))}$$
		and hence
		$$\frac{{\rm diam}_{U}(X)}{{\rm diam}_{U}(Y)}\asymp_{\mathfrak{q}_3}\frac{{\rm diam}_{W}(X)}{{\rm diam}_{W}(Y)}.$$
		Then by taking $X=E^{-1}$, $Y=F^{-1}$ and $U=\mathbb{C}\setminus\overline{P_{\alpha}^{-s}(\Delta_{\alpha}})$ or $\mathbb{C}\setminus\overline{\Delta_{\alpha}}$,
		we have that
		\begin{equation}
			\label{e20251211a2}\frac{{\rm diam}_{\mathbb{C}\setminus\overline{P_{\alpha}^{-s}(\Delta_{\alpha}})}(E^{-1})}{{\rm diam}_{\mathbb{C}\setminus\overline{P_{\alpha}^{-s}(\Delta_{\alpha}})}(F^{-1})}\asymp_{\mathfrak{q}_3}\frac{{\rm diam}_{\mathbb{C}\setminus\overline{\Delta_{\alpha}}}(E^{-1})}{{\rm diam}_{\mathbb{C}\setminus\overline{\Delta_{\alpha}}}(F^{-1})}.
		\end{equation}
		By (\ref{e20251211a1}) and (\ref{e20251211a2}),
		we have
		$$\frac{{\rm diam}_{\mathbb{C}\setminus\overline{\Delta_{\alpha}}}(E^{-1})}{{\rm diam}_{\mathbb{C}\setminus\overline{\Delta_{\alpha}}}(F^{-1})}\asymp_{\mathfrak{q}_3}\frac{{\rm diam}_{\mathbb{C}\setminus\overline{\Delta_{\alpha}}}(E)}{{\rm diam}_{\mathbb{C}\setminus\overline{\Delta_{\alpha}}}(F)}.$$
		
		We divide $\gamma$ into some sufficiently short arcs $\gamma_j$, $j=1,2,\cdots,n$ such that 
		\begin{itemize}
			\item $l_{\mathbb{C}\setminus\overline{\Delta_{\alpha}}}(\gamma)\asymp\sum_{j=1}^n{\rm diam}_{\mathbb{C}\setminus\overline{\Delta_{\alpha}}}(\gamma_j)$,
			\item $l_{\mathbb{C}\setminus\overline{\Delta_{\alpha}}}(\gamma^{-1})\asymp\sum_{j=1}^n{\rm diam}_{\mathbb{C}\setminus\overline{\Delta_{\alpha}}}(\gamma_j^{-1})$, where each $\gamma_j^{-1}$ is the component of the preimage of $\gamma_j$ under $P_{\alpha}^{\comp s}$ contained in $\mathbb{B}^{-1}(P_{\alpha}^{\comp s}(w))$.
		\end{itemize}
		Then (\ref{e20251210a2}) gives that for all $1\leq j\leq n$, 
		$$
		\frac{{\rm diam}_{\mathbb{C}\setminus\overline{\Delta_{\alpha}}}(\gamma_j^{-1})}{{\rm diam}_{\mathbb{C}\setminus\overline{\Delta_{\alpha}}}(E^{-1})}\asymp_{\mathfrak{q}_3}\frac{{\rm diam}_{\mathbb{C}\setminus\overline{\Delta_{\alpha}}}(\gamma_j)}{{\rm diam}_{\mathbb{C}\setminus\overline{\Delta_{\alpha}}}(E)}.
		$$
		Thus
		$$\frac{l_{\mathbb{C}\setminus\overline{\Delta_{\alpha}}}(\gamma^{-1})}{{\rm diam}_{\mathbb{C}\setminus\overline{\Delta_{\alpha}}}(E^{-1})}
		\asymp\frac{\sum_{j=1}^n{\rm diam}_{\mathbb{C}\setminus\overline{\Delta_{\alpha}}}(\gamma_j^{-1})}{{\rm diam}_{\mathbb{C}\setminus\overline{\Delta_{\alpha}}}(E^{-1})}\asymp_{\mathfrak{q}_3}\frac{\sum_{j=1}^n{\rm diam}_{\mathbb{C}\setminus\overline{\Delta_{\alpha}}}(\gamma_j)}{{\rm diam}_{\mathbb{C}\setminus\overline{\Delta_{\alpha}}}(E)}\asymp\frac{l_{\mathbb{C}\setminus\overline{\Delta_{\alpha}}}(\gamma)}{{\rm diam}_{\mathbb{C}\setminus\overline{\Delta_{\alpha}}}(E)}.
		$$
	\end{proof}
	For all $1\leq j\leq k$, by (\ref{e20260327d1}) and (\ref{e20260326a}), ${\rm diam}_{\mathbb{C}\setminus\overline{\Delta_{\alpha}}}{\tilde{B}_{j}}\asymp1\ {\rm and}\ {\rm dist}_{\mathbb{C}\setminus\overline{\Delta_{\alpha}}}(\tilde{B}_{j},P_{\alpha}^{\comp m_{j}}(w))=O_{\mathfrak{q}_3}(1)$, 
	$l_{\mathbb{C}\setminus\overline{\Delta_{\alpha}}}(\gamma_j)=O_{\mathfrak{q}_3}(1)$.
	Then there exists a hyperbolic ball $\hat{B}_j$ on $\mathbb{C}\setminus\overline{\Delta_{\alpha}}$ centering at $P_{\alpha}^{\comp m_{j}}(w)$ with radius $\preceq_{\mathfrak{q}_3}1$ containing $\tilde{B}_{j}\cup\gamma_j$. For all $0\leq s\leq m_j$, we let $\hat{B}_j^{(s)}$ be the component of $\hat{B}_{j}$ containing $P_{\alpha}^{\comp(m_{j}-s)}(w)$ under $P_{\alpha}^{\comp s}$.
	Then $B_{j}^{(s)}\cup\gamma_j^{(s)}\subseteq\hat{B}_j^{(s)}$. Observe $P_{\alpha}^{\comp m_j}(w)\in\Lambda_{\alpha}^{2\mathfrak{q}_4+1}$. Thus by (\ref{e20251210a1}) of Lemma \ref{l8241} we have that for all $w'\in\hat{B}_j^{(s)}$, 
	\begin{equation}
		\label{e20260528a}|(P_{\alpha}^{\comp s})'(w')|\cdot\rho_{\mathbb{C}\setminus\overline{\Delta_{\alpha}}}(P_{\alpha}^{\comp(m_j-s)}(w'))\asymp_{\mathfrak{q}_3}|(P_{\alpha}^{\comp s})'(P_{\alpha}^{\comp(m_j-s)}(w))|\cdot\rho_{\mathbb{C}\setminus\overline{\Delta_{\alpha}}}(P_{\alpha}^{\comp m_{j}}(w)).
	\end{equation}
	Together with ${\rm diam}_{\mathbb{C}\setminus\overline{\Delta_{\alpha}}}(\hat{B}_{j}^{(s)})\preceq_{\mathfrak{q}_3}1$ (Due to the Schwarz lemma), $P_{\alpha}^{\comp(m_{j}-s)}(w)\in\hat{B}_{j}^{(s)}$ and $P_{\alpha}^{\comp(m_j-s)}(w)\in\Lambda_{\alpha}^{2\mathfrak{q}_4+1}$, (\ref{e20251210a1}) of Lemma \ref{l8241} gives that
	for all $w'\in\hat{B}_{j}^{(s)}$, 
	\begin{equation}
		\label{e20260528d}\rho_{\mathbb{C}\setminus\overline{\Delta_{\alpha}}}(w')\asymp_{\mathfrak{q}_3}\rho_{\mathbb{C}\setminus\overline{\Delta_{\alpha}}}(P_{\alpha}^{\comp(m_j-s)}(w)).
	\end{equation}
	
	\vspace{0.2cm}
	\noindent{\bf Claim 2:}
	There exists a constant $0<\xi<1$ determined by $\mathfrak{q}_3$, independent of $\alpha$, such that
	\begin{itemize}
		\item ${\rm diam}_{\mathbb{C}\setminus\overline{\Delta_{\alpha}}}(B_{j+t}^{(m_{j+t})})=O_{\mathfrak{q}_3}(\xi^t)\cdot{\rm diam}_{\mathbb{C}\setminus\overline{\Delta_{\alpha}}}(B_j^{(m_j)})$
		for all $1\leq j<j+t\leq k$,
		\item ${\rm diam}_{\mathbb{C}\setminus\overline{\Delta_{\alpha}}}(B_j^{(m_j)})=O_{\mathfrak{q}_3}(\xi^j)$ for all $1\leq j\leq k$,
		\item $l_{\mathbb{C}\setminus\overline{\Delta_{\alpha}}}(\gamma_{j}^{(m_j)})=O_{\mathfrak{q}_3}({\rm diam}_{\mathbb{C}\setminus\overline{\Delta_{\alpha}}}(B_j^{(m_j)}))$ for all $1\leq j\leq k$,
		\item for all $1\leq j\leq k$, $1\leq t\leq j$ and $m_{t-1}< s\leq m_t$ (Set $m_0=0$),
		$${\rm diam}_{\mathbb{C}\setminus\overline{\Delta_{\alpha}}}(B_j^{(m_j-s)})=O_{\mathfrak{q}_3}(\xi^{j-t})\
		{\rm and}\
		l_{\mathbb{C}\setminus\overline{\Delta_{\alpha}}}(\gamma_j^{(m_j-s)})=O_{\mathfrak{q}_3}(\xi^{j-t}).$$
	\end{itemize}
	
	\begin{proof}
		By Lemma
		\ref{l20257193}
		we have that for all $n\geq\mathfrak{q}_4$, $\tilde{U}_0^{-1}\cap\left(\tilde{\Delta}_{\alpha}^{-1}\cup(\Lambda_{\alpha}^{n+\mathfrak{q}_2}\setminus\Omega_{c_0}^n)\right)=\{c_0\}$, and hence
		for all $1\leq j\leq k$, 
		\begin{equation*}
			U_0\cap\left(\Lambda_{\alpha}^{n_j+\mathfrak{q}_3}\setminus\Omega_{c_0}^{n_j+\mathfrak{q}_3-\mathfrak{q}_2}\right)=\emptyset. 
		\end{equation*}
		Since
		$\gamma_{c_0}^{n_j}\setminus\Lambda_{\alpha}^{n_j+\mathfrak{q}_3}=\gamma_{c_0}^{n_j}\setminus\left(\Lambda_{\alpha}^{n_j+\mathfrak{q}_3}\setminus\Omega_{c_0}^{n_j+\mathfrak{q}_3-\mathfrak{q}_2}\right)$, we have
		that for all $1\leq j\leq k$, 
		\begin{align*}
			U_0\cap\gamma_{c_0}^{n_j}
			&=U_0\setminus\left(\Lambda_{\alpha}^{n_j+\mathfrak{q}_3}\setminus\Omega_{c_0}^{n_j+\mathfrak{q}_3-\mathfrak{q}_2}\right)\cap\gamma_{c_0}^{n_j}\\
			&=U_0\cap\left(\gamma_{c_0}^{n_j}\setminus\left(\Lambda_{\alpha}^{n_j+\mathfrak{q}_3}\setminus\Omega_{c_0}^{n_j+\mathfrak{q}_3-\mathfrak{q}_2}\right)\right)\\
			&=U_0\cap\left(\gamma_{c_0}^{n_j}\setminus\Lambda_{\alpha}^{n_j+\mathfrak{q}_3}\right)\\
			&\subseteq U_0\cap\left(\tilde{\Omega}_{c_0}^{n_j}\setminus\Lambda_\alpha^{n_j+\mathfrak{q}_3}\right).
		\end{align*}
		By (\ref{e20251126a})
		for all $1\leq j\leq k$, $U_0\cap\gamma_{c_0}^{n_j}\not=\emptyset$.
		Thus for all $1\leq j\leq k$, $U_0\cap\left(\tilde{\Omega}_{c_0}^{n_j}\setminus\Lambda_\alpha^{n_j+\mathfrak{q}_3}\right)\not=\emptyset$. Observe that for all $1\leq j\leq k$,
		$P_{\alpha}^{\comp m_j}(w)\in\tilde{\Omega}_{c_0}^{n_j}\setminus\Lambda^{n_j+\mathfrak{q}_3}$. Then Lemma \ref{l20257193a} gives that for all $1\leq j\leq k$,
		$${\rm dist}_{\mathbb{C}\setminus\overline{\Delta_{\alpha}}}(U_0,P_{\alpha}^{\comp m_j}(w))=O_{\mathfrak{q}_3}(1).$$
		By [Proposition $4.4.2$, \cite{McM96}], for all $1\leq j\leq k-1$,
		\begin{equation}
			\label{e250707}\rho_{\mathbb{C}\setminus\overline{\Delta_{\alpha}}}(P_{\alpha}^{\comp m_j}(w))\leq\xi\cdot|P_{\alpha}'(P_{\alpha}^{\comp m_j}(w))|\cdot\rho_{\mathbb{C}\setminus\overline{\Delta_{\alpha}}}(P_{\alpha}^{\comp(m_j+1)}(w)),
		\end{equation}
		where $0<\xi<1$ is a constant determined by $\mathfrak{q}_3$ and independent of $\alpha$.
		By applying the above inequality and the Schwarz lemma repeatedly, we have that for all $1\leq j<j+t\leq k$, 
		\begin{equation}
			\label{e20251127a}\rho_{\mathbb{C}\setminus\overline{\Delta_{\alpha}}}(P_{\alpha}^{\comp m_j}(w))\leq\xi^{t}\cdot|(P_{\alpha}^{\comp(m_{j+t}-m_{j})})'(P_{\alpha}^{\comp m_j}(w))|\cdot\rho_{\mathbb{C}\setminus\overline{\Delta_{\alpha}}}(P_{\alpha}^{\comp m_{j+t}}(w)).
		\end{equation}
		By (\ref{e20260528a}), (\ref{e20260528d}) and (\ref{e20251127a}), we have that for all $w'\in\hat{B}_{j+t}^{(m_{j+t}-m_j)}\supseteq B_{j+t}^{(m_{j+t}-m_j)}$,
		$$\rho_{\mathbb{C}\setminus\overline{\Delta_{\alpha}}}(w')\leq O_{\mathfrak{q}_3}(\xi^{t})\cdot|(P_{\alpha}^{\comp(m_{j+t}-m_{j})})'(w')|\cdot\rho_{\mathbb{C}\setminus\overline{\Delta_{\alpha}}}(P_{\alpha}^{\comp(m_{j+t}-m_j)}(w'))$$
		and hence
		$${\rm diam}_{\mathbb{C}\setminus\overline{\Delta_{\alpha}}}(B_{j+t}^{(m_{j+t}-m_j)})\leq O_{\mathfrak{q}_3}(\xi^t)\cdot{\rm diam}_{\mathbb{C}\setminus\overline{\Delta_{\alpha}}}(\tilde{B}_{j+t})=O_{\mathfrak{q}_3}(\xi^t).$$
		By combining ${\rm diam}_{\mathbb{C}\setminus\overline{\Delta_{\alpha}}}(\tilde{B}_{j})\asymp1$, it follows that
		\begin{equation}
			\label{e20251210e1}{\rm diam}_{\mathbb{C}\setminus\overline{\Delta_{\alpha}}}(B_{j+t}^{(m_{j+t}-m_j)})=O_{\mathfrak{q}_3}(\xi^t)\cdot {\rm diam}_{\mathbb{C}\setminus\overline{\Delta_{\alpha}}}(\tilde{B}_{j}).
		\end{equation}
		By ${\rm diam}_{\mathbb{C}\setminus\overline{\Delta_{\alpha}}}(B_{j+t}^{(m_{j+t}-m_j)})=O_{\mathfrak{q}_3}(\xi^t)$, ${\rm diam}_{\mathbb{C}\setminus\overline{\Delta_{\alpha}}}(\tilde{B}_{j})\asymp1$, $l_{\mathbb{C}\setminus\overline{\Delta_{\alpha}}}(\gamma_{j})=O_{\mathfrak{q}_3}(1)$ and 
		$l_{\mathbb{C}\setminus\overline{\Delta_{\alpha}}}(\gamma_{j+t}^{(m_{j+t}-m_j)})\preceq_{\mathfrak{q}_3}1$ (Due to the Schwarz lemma), we have that for all $w'\in B_{j+t}^{(m_{j+t}-m_j)}$ and $w''\in \tilde{B}_{j}$, 
		$${\rm dist}_{\mathbb{C}\setminus\overline{\Delta_{\alpha}}}(w',w'')\preceq_{\mathfrak{q}_3}1.$$
		We choose a hyperbolic ball centering at $P_{\alpha}^{\comp m_j}(w)$ with radius $\preceq_{\mathfrak{q}_3}1$ containing $B_{j+t}^{(m_{j+t}-m_j)}\cup\tilde{B}_{j}$.
		Pulling back the ball by $P_{\alpha}^{\comp m_j}$ and then applying (\ref{e20251210a2}) of Lemma \ref{l8241} to $B_{j+t}^{(m_{j+t}-m_j)}$ and $\tilde{B}_{j}$, by (\ref{e20251210e1}) we have 
		\begin{equation*}
			\label{e2507071}{\rm diam}_{\mathbb{C}\setminus\overline{\Delta_{\alpha}}}(B_{j+t}^{(m_{j+t})})= O_{\mathfrak{q}_3}(\xi^t)\cdot{\rm diam}_{\mathbb{C}\setminus\overline{\Delta_{\alpha}}}(B_j^{(m_j)}).
		\end{equation*}
		By the Schwarz lemma ${\rm diam}_{\mathbb{C}\setminus\overline{\Delta_{\alpha}}}(B_1^{(m_1)})\preceq1$.
		Thus
		\begin{equation*}
			\label{e2507072}{\rm diam}_{\mathbb{C}\setminus\overline{\Delta_{\alpha}}}(B_{j}^{(m_{j})})\leq O_{\mathfrak{q}_3}(\xi^{j-1})\cdot{\rm diam}_{\mathbb{C}\setminus\overline{\Delta_{\alpha}}}(B_1^{(m_1)})= O_{\mathfrak{q}_3}(\xi^{j}).
		\end{equation*}
		Observe that $l_{\mathbb{C}\setminus\overline{\Delta_{\alpha}}}(\gamma_{j})=O_{\mathfrak{q}_3}(1)=O_{\mathfrak{q}_3}({\rm diam}_{\mathbb{C}\setminus\overline{\Delta_{\alpha}}}(\tilde{B}_j))$ and ${\rm dist}_{\mathbb{C}\setminus\overline{\Delta_{\alpha}}}(w',w'')\preceq_{\mathfrak{q}_3}1$ for all $w'\in\gamma_{j}$ and $w''\in\tilde{B}_j$. Then pulling back a hyperbolic ball with radius $\preceq_{\mathfrak{q}_3}1$ containing $\gamma_{j}\cup\tilde{B}_j$ by $P_{\alpha}^{\comp m_j}$,
		by (\ref{e20260328a}) of Lemma \ref{l8241} we have
		\begin{equation*}
			\label{e2507073}
			l_{\mathbb{C}\setminus\overline{\Delta_{\alpha}}}(\gamma_{j}^{(m_j)})=O_{\mathfrak{q}_3}({\rm diam}_{\mathbb{C}\setminus\overline{\Delta_{\alpha}}}(B_j^{(m_j)}))
		\end{equation*}
		for all $1\leq j\leq k$. By (\ref{e250707}) and the Schwarz lemma, we have that
		for all $1\leq j\leq k$, $1\leq t\leq j$ and $m_{t-1}< s\leq m_t$,  
		$$\rho_{\mathbb{C}\setminus\overline{\Delta_{\alpha}}}(P_{\alpha}^{\comp s}(w))\leq\xi^{j-t}\cdot|(P_{\alpha}^{\comp(m_{j}-s)})'(P_{\alpha}^{\comp s}(w))|\cdot\rho_{\mathbb{C}\setminus\overline{\Delta_{\alpha}}}(P_{\alpha}^{\comp m_{j}}(w)).$$
		By (\ref{e20260528a}), (\ref{e20260528d}) and the above inequality, we have that by (\ref{e20251210a1}) of Lemma \ref{l8241} we have that for all $w''\in \hat{B}_j^{(m_j-s)}\supseteq B_j^{(m_j-s)}\cup\gamma_j^{(m_j-s)}$,
		\begin{equation}
			\label{e20260525b}\rho_{\mathbb{C}\setminus\overline{\Delta_{\alpha}}}(w'')\leq O_{\mathfrak{q}_3}(\xi^{j-t})\cdot|(P_{\alpha}^{\comp(m_{j}-s)})'(w'')|\cdot\rho_{\mathbb{C}\setminus\overline{\Delta_{\alpha}}}(P_{\alpha}^{\comp(m_{j}-s)}(w'')).
		\end{equation}
		Then
		$${\rm diam}_{\mathbb{C}\setminus\overline{\Delta_{\alpha}}}(B_j^{(m_j-s)})\leq O_{\mathfrak{q}_3}(\xi^{j-t})\cdot{\rm diam}_{\mathbb{C}\setminus\overline{\Delta_{\alpha}}}(\tilde{B}_j)=O_{\mathfrak{q}_3}(\xi^{j-t})$$
		and
		$$l_{\mathbb{C}\setminus\overline{\Delta_{\alpha}}}(\gamma_j^{(m_j-s)})\leq O_{\mathfrak{q}_3}(\xi^{j-t})\cdot l_{\mathbb{C}\setminus\overline{\Delta_{\alpha}}}(\gamma_j)=O_{\mathfrak{q}_3}(\xi^{j-t}).$$
	\end{proof}
	
	Recall that $B_j$ ($\subseteq\mathbb{C}\setminus\overline{\Delta_{\alpha}}$) is an Euclidean ball with the same center and twice the radius as $\tilde{B}_j$, and $P_{\alpha}^{\comp m_j}|_{B_j^{(m_j)}}$ is a conformal map from $B_j^{(m_j)}$ to $\tilde{B}_j$.
	Applying the distortion theorem to the branch of $P_{\alpha}^{-m_j}$ on $B_j$ mapping $\tilde{B}_j$ to $B_j^{(m_j)}$, we obtain  
	that there exists an Euclidean ball ${\bf B}_j$ such that
	\begin{equation}
		\label{e20260326b}{\bf B}_j\subseteq B_j^{(m_j)}\ {\rm and}\ {\rm diam}({\bf B}_j)\asymp {\rm diam}(B_j^{(m_j)})\asymp l(\partial B_j^{(m_j)}).
	\end{equation}
	Further, we have the following lemma:
	\begin{lemma}
		\label{l20260330a}	The above Euclidean balls ${\rm\bf B}_j$, $1\leq j\leq k$ satisfy
		\begin{itemize}
			\item[{\rm(a4)}] ${\rm diam}({\bf B}_{j+t})\leq C\xi^t\cdot{\rm diam}({\bf B}_j)$ for all $1\leq j<j+t\leq k$, where $C$ is a constant depending on $\mathfrak{q}_3$ and independent of $\alpha$,
			\item[{\rm(b4)}] ${\rm diam}({\bf B}_j)={\rm dist}(w,\partial\Delta_{\alpha})\cdot O_{\mathfrak{q}_3}(\xi^j)$ for all $1\leq j\leq k$,
			\item[{\rm(c4)}] ${\rm dist}({\bf B}_j,w)=O_{\mathfrak{q}_3}({\rm diam}({\bf B}_j))$ for all $1\leq j\leq k$,
			\item[{\rm(d4)}] for all $1\leq j\leq k$ and all $w'\in{\bf B}_j$, $$\left\{P_{\alpha}^{\comp t}(w'):\ 0\leq t\leq m_j\right\}\subseteq\Lambda_{\alpha}^{n+1-\chi}\ {\rm and}\ P_{\alpha}^{\comp m_j}(w')\in\hat{U}_0^{n_j+\mathfrak{l}}\cap\tilde{\Omega}_{c_0}^{n_j}\setminus\Lambda^{n_j+\mathfrak{q}_3}.$$
			Moreover, for all $1\leq t\leq j$ and $m_{t-1}\leq s\leq m_t$, there exists a curve $\tilde{\gamma}_{j}^{(s)}$ connecting $P_{\alpha}^{\comp s}(w')$ and $P_{\alpha}^{\comp s}(w)$ such that
			$l_{\mathbb{C}\setminus\overline{\Delta_{\alpha}}}(\tilde{\gamma}_{j}^{(s)})=O_{\mathfrak{q}_3}(\xi^{j-t})$
			and 
			$P_{\alpha}^{\comp (m_j-s)}\comp\tilde{\gamma}_{j}^{(s)}=\tilde{\gamma}_{j}^{(m_j)}$.
		\end{itemize}
	\end{lemma}
	\begin{proof}
		For all $1\leq j\leq k$, by Claim $1$ $\gamma_j^{(m_j)}\cup B_j^{(m_j)}\subseteq\Lambda_{\alpha}^{n+1-\chi}\subseteq\Lambda_{\alpha}^{\mathfrak{q}_4+1}$.
		Then together with $l_{\mathbb{C}\setminus\overline{\Delta_{\alpha}}}(\gamma_j^{(m_j)})=O_{\mathfrak{q}_3}(1)$ and $l_{\mathbb{C}\setminus\overline{\Delta_{\alpha}}}(\partial B_j^{(m_j)})=O_{\mathfrak{q}_3}(1)$ (Due to the Schwarz lemma and (\ref{e20260528d}) with $s=0$),
		Lemma \ref{l25070716} gives 
		$${\rm diam}(\gamma_j^{(m_j)}\cup B_j^{(m_j)})\preceq_{\mathfrak{q}_3}{\rm dist}(\gamma_j^{(m_j)}\cup B_j^{(m_j)},\Delta_{\alpha})$$
		and hence for all $w'\in\overline{\gamma_j^{(m_j)}\cup B_j^{(m_j)}}$,
		\begin{equation}
			\label{e20260525a}\rho_{\mathbb{C}\setminus\overline{\Delta_{\alpha}}}(w')\asymp_{\mathfrak{q}_3}\rho_{\mathbb{C}\setminus\overline{\Delta_{\alpha}}}(w).
		\end{equation}
		Then it follows from (\ref{e20260326b}) and (\ref{e20260525a}) that for all $1\leq j\leq k$,  
		\begin{align*}
			\rho_{\mathbb{C}\setminus\overline{\Delta_{\alpha}}}(w)\cdot{\rm diam}({\bf B}_j)&\preceq_{\mathfrak{q}_3}{\rm diam}_{\mathbb{C}\setminus\overline{\Delta_{\alpha}}}({\bf B}_j)\\
			&\leq{\rm diam}_{\mathbb{C}\setminus\overline{\Delta_{\alpha}}}(B_j^{(m_j)})\\
			&\preceq_{\mathfrak{q}_3}\rho_{\mathbb{C}\setminus\overline{\Delta_{\alpha}}}(w)\cdot l(\partial B_j^{(m_j)})\\
			&\asymp\rho_{\mathbb{C}\setminus\overline{\Delta_{\alpha}}}(w)\cdot{\rm diam}({\bf B}_j),
		\end{align*}
		that is
		\begin{equation}
			\label{e20260327a1}{\rm diam}_{\mathbb{C}\setminus\overline{\Delta_{\alpha}}}(B_j^{(m_j)})\asymp_{\mathfrak{q}_3}\rho_{\mathbb{C}\setminus\overline{\Delta_{\alpha}}}(w)\cdot{\rm diam}({\bf B}_j).
		\end{equation}
		By Claim $2$ we have that for all $1\leq j<j+t\leq k$,
		\begin{equation}
			\label{e20260327b1}{\rm diam}_{\mathbb{C}\setminus\overline{\Delta_{\alpha}}}(B_{j+t}^{(m_{j+t})})=O_{\mathfrak{q}_3}(\xi^t)\cdot{\rm diam}_{\mathbb{C}\setminus\overline{\Delta_{\alpha}}}(B_j^{(m_j)})
		\end{equation}
		and for all $1\leq j\leq k$,
		\begin{equation}
			\label{e20260327b2}{\rm diam}_{\mathbb{C}\setminus\overline{\Delta_{\alpha}}}(B_j^{(m_j)})=O_{\mathfrak{q}_3}(\xi^j).
		\end{equation}
		By substituting (\ref{e20260327a1}) into the equality (\ref{e20260327b1}), we get (a4).
		By substituting (\ref{e20260327a1}) into the equality (\ref{e20260327b2}), we get
		$$\rho_{\mathbb{C}\setminus\overline{\Delta_{\alpha}}}(w)\cdot{\rm diam}({\bf B}_j)=O_{\mathfrak{q}_3}(\xi^j).$$
		By combining $\rho_{\mathbb{C}\setminus\overline{\Delta_{\alpha}}}(w)\asymp\frac{1}{{\rm dist}(w,\Delta_{\alpha})}$ (Due to Lemma \ref{l25070716}),
		it follows that (b4) holds.
		For all $1\leq j\leq k$, we have
		\begin{align*}
			{\rm dist}(w,{\bf B}_j)&\leq{\rm dist}(w,B_j^{(m_j)})+{\rm diam}(B_j^{(m_j)})\\
			&\leq l(\gamma_{j}^{(m_j)})+O({\rm diam}({\bf B}_j))\ {\rm(Due\ to\ (\ref{e20260326b}))}\\
			&\leq\frac{O_{\mathfrak{q}_3}(l_{\mathbb{C}\setminus\overline{\Delta_{\alpha}}}(\gamma_{j}^{(m_j)}))}{\rho_{\mathbb{C}\setminus\overline{\Delta_{\alpha}}}(w)}+O({\rm diam}({\bf B}_j))\ {\rm(Due\ to\ (\ref{e20260525a}))}\\
			&\leq\frac{O_{\mathfrak{q}_3}({\rm diam}_{\mathbb{C}\setminus\overline{\Delta_{\alpha}}}(B_j^{(m_j)}))}{\rho_{\mathbb{C}\setminus\overline{\Delta_{\alpha}}}(w)}+O({\rm diam}({\bf B}_j))\ {\rm(Due\ to\ Claim\ 2)}\\
			&\leq O_{\mathfrak{q}_3}({\rm diam}({\bf B}_j))\ {\rm(Due\ to\ (\ref{e20260327a1}))}.
		\end{align*}
		This implies (c4).
		At last, we will prove (d4). For all $1\leq j\leq k$, $0\leq t\leq m_j$ and all $w'\in{\bf B}_j$, we have
		$$P_{\alpha}^{\comp t}(w')\in P_{\alpha}^{\comp t}({\bf B}_j)\subseteq P_{\alpha}^{\comp t}(B_j^{(m_j)})=B_j^{(m_j-t)},$$
		and in particular,
		$$P_{\alpha}^{\comp m_j}(w')\in P_{\alpha}^{\comp m_j}({\bf B}_j)\subseteq\tilde{B}_j.$$
		Then by $B_j^{(m_j-t)}\subseteq\Lambda_{\alpha}^{n+1-\chi}$ (Due to Claim $1$) and $\tilde{B}_j\subseteq\hat{U}_0^{n_j+\mathfrak{l}}\cap\tilde{\Omega}_{c_0}^{n_j}\setminus\Lambda^{n_j+\mathfrak{q}_3}$ (Due to (\ref{e20260327d1})), we have
		$$\left\{P_{\alpha}^{\comp t}(w'):\ 0\leq t\leq m_j\right\}\subseteq\Lambda_{\alpha}^{n+1-\chi}\ {\rm and}\ P_{\alpha}^{\comp m_j}(w')\in\hat{U}_0^{n_j+\mathfrak{l}}\cap\tilde{\Omega}_{c_0}^{n_j}\setminus\Lambda^{n_j+\mathfrak{q}_3}.$$
		
		Recall that for all $1\leq j\leq k$, $\gamma_j^{(0)}$ connects $P_{\alpha}^{\comp m_j}(w)$ and $B_j^{(0)}=\tilde{B}_j$. We denote by $b_{m_j}$ the endpoint of $\gamma_j^{(0)}$ on $\partial B_j^{(0)}$.
		We let $\beta_{m_j}$ be a curve connecting $b_{m_j}$ and $P_{\alpha}^{\comp m_j}(w')$ on $B_j^{(0)}$ such that $l_{\mathbb{C}\setminus\overline{\Delta_{\alpha}}}(\beta_{m_j})=O_{\mathfrak{q}_3}(1)$   (Due to (\ref{e20260327d1}) and (\ref{e20260528d})), and for all $0\leq s\leq m_j$, we let $\beta_{m_j-s}$ be the preimage of $\beta_{m_j}$ under $P_{\alpha}^{\comp s}$ intersecting $P_{\alpha}^{\comp (m_j-s)}(w')$.
		Now we define $\tilde{\gamma}_{j}^{(m_j-s)}$ as the union curve $\gamma_j^{(s)}+\beta_{m_j-s}$.
		It is clear that 
		for all $1\leq t\leq j$ and $m_{t-1}\leq s\leq m_t$, the curve $\tilde{\gamma}_{j}^{(s)}$ connects $P_{\alpha}^{\comp s}(w')$ and $P_{\alpha}^{\comp s}(w)$, and 
		$P_{\alpha}^{\comp (m_j-s)}\comp\tilde{\gamma}_{j}^{(s)}=\tilde{\gamma}_{j}^{(m_j)}$.
		It follows from Claim $2$ that
		\begin{equation}
			\label{e20260525c}l_{\mathbb{C}\setminus\overline{\Delta_{\alpha}}}(\gamma_j^{(m_j-s)})=O_{\mathfrak{q}_3}(\xi^{j-t}).
		\end{equation}
		Observe that $\beta_{s}\subseteq B_j^{(m_j-s)}$ and $P_{\alpha}^{\comp (m_j-s)}\comp\beta_{s}=\beta_{m_j}$. Integrating both sides of (\ref{e20260525b}) on $\beta_{s}$ respectively, we have
		\begin{equation}
			\label{e20260525d}l_{\mathbb{C}\setminus\overline{\Delta_{\alpha}}}(\beta_{s})=O_{\mathfrak{q}_3}(\xi^{j-t}) l_{\mathbb{C}\setminus\overline{\Delta_{\alpha}}}(\beta_{m_j})=O_{\mathfrak{q}_3}(\xi^{j-t}).
		\end{equation}
		Then by (\ref{e20260525c}) and (\ref{e20260525d}), we have $$l_{\mathbb{C}\setminus\overline{\Delta_{\alpha}}}(\tilde{\gamma}_{j}^{(s)})\leq l_{\mathbb{C}\setminus\overline{\Delta_{\alpha}}}(\gamma_j^{(m_j-s)})+l_{\mathbb{C}\setminus\overline{\Delta_{\alpha}}}(\beta_s)=O_{\mathfrak{q}_3}(\xi^{j-t}).$$
		Thus (d4) holds.
	\end{proof}

	\section{Filling-in of fjords\label{S4}}
	
	\subsection{Uniform bounds for area rates of measurable subsets of a square}
	We denote by $\mathbb{B}(z,r)$ the Euclidean ball in $\mathbb{C}$ centering at $z$ with radius $r>0$. The following lemma follows from [Lemma $6.23$, \cite{AL22}] or [Lemma $3.1$, \cite{QQZ23}].
	\begin{lemma}
		\label{L1}Let $c>0$ and $0<\eta<1$ be two real numbers.
		For any large enough positive integer $N$, we have:
		for any square $S\subseteq\mathbb{C}$ with side length $l$ and any nonempty measurable subset $E$ of $S$,
		if there exist $N$ mappings $r_1,r_2,\cdots,r_N$ from $E$ to $\mathbb{R}$
		and $N$ mappings $y_n:E\to\mathbb{C}$, $1\leq n\leq N$
		such that for all $x\in E$,
		the following conditions hold:
		\begin{itemize}
			\item[{\rm(a5)}] $0<r_1(x)\leq l\cdot\eta$,
			\item[{\rm(b5)}] $0<r_{n+1}(x)\leq\eta\cdot r_n(x)$ for all $1\leq n<N$,
			\item[{\rm(c5)}] $y_n(x)\in\mathbb{B}(x,c\cdot r_n(x))$ for all $1\leq n\leq N$,
			\item[{\rm(d5)}] $\mathbb{B}(y_n(x),r_n(x))\cap E=\emptyset$ for all $1\leq n\leq N$,
		\end{itemize}
		then
		$${\rm area}(E)\leq\lambda^N\cdot{\rm area}(S),$$
		where $0<\lambda<1$ is a constant only depending on $c$ and $\eta$.
	\end{lemma}

	\begin{lemma}
		\label{key4}Let $S\subseteq\mathbb{C}$ be a square with side length $l$, $E$ be a nonempty measurable subset of $S$ and $F$ be a nonempty measurable subset of $E$.
		Assume $0<\lambda<1$ be a real number.
		If there exists a map $r$ from $F$ to $\mathbb{R}$ such that for all $x\in F$
		\begin{itemize}
			\item[{\rm(a6)}] $0<r(x)\leq l$,
			\item[{\rm(b6)}] $\mathbb{B}(x,r(x))\cap (S\setminus E)=\emptyset$,
			\item[{\rm(c6)}] ${\rm area}(\mathbb{B}(x,r(x))\cap F)\leq\lambda\cdot{\rm area}(\mathbb{B}(x,r(x)))$,
		\end{itemize}
		then we have
		$${\rm area}(F)\leq c\cdot\lambda\cdot{\rm area}(E),$$
		where $c>0$ is a universal constant.
	\end{lemma}
	\begin{proof}
		Applying the Besicovitch covering theorem to
		$\{\overline{\mathbb{B}(x,r(x))}:x\in F\}$ (Due to (a6)), we have that there exist $K$ (a universal constant) subsets $\{\overline{\mathbb{B}(x,r(x))}:x\in F_j\}$, $j=1,2,\cdots,K$ of $\{\overline{\mathbb{B}(x,r(x))}:x\in F\}$ such that 
		\begin{itemize}
			\item[(I)] $F\subseteq\bigcup_{x\in\cup_{j=1}^K F_j}\overline{\mathbb{B}(x,r(x))}$, 
			\item[(II)] for all $1\leq j\leq K$ and all $x,y\in F_j$ with $x\not=y$, $\overline{\mathbb{B}(x,r(x))}\cap\overline{\mathbb{B}(y,r(y))}=\emptyset$.
		\end{itemize}
		Then
		\begin{align*}
			{\rm area}(F)&={\rm area}(F\cap\bigcup_{x\in\cup_{j=1}^K F_j}\overline{\mathbb{B}(x,r(x))})\ {\rm(Due\ to\ (I))}\\
			&={\rm area}(\bigcup_{x\in\cup_{j=1}^K F_j}(F\cap\overline{\mathbb{B}(x,r(x))}))\\
			&\leq\sum_{1\leq j\leq K}\sum_{x\in F_j}{\rm area}(F\cap\overline{\mathbb{B}(x,r(x))})\\
			&\leq\sum_{1\leq j\leq K}\sum_{x\in F_j}\lambda\cdot{\rm area}(\overline{\mathbb{B}(x,r(x))})\ {\rm(Due\ to\ (c6))}\\
			&\leq\sum_{1\leq j\leq K}4\lambda\cdot{\rm area}(E)\ {\rm(Due\ to\ (a6),\ (b6)\ and\ (II))}\\
			&=4K\lambda\cdot{\rm area}(E).
		\end{align*}
		
	\end{proof}

	\begin{lemma}
		\label{key3}Let $S\subseteq\mathbb{C}$ be a square with center $a$, $E$ be a measurable subset of $S$ and
		$0<\lambda<1$ be a real number such that ${\rm area}(E)\leq\lambda\cdot{\rm area}(S)$. Assume that
		$f$ is a univalent map on ${\rm int}(S)$. Then there exists an Euclidean ball $B$ {\rm(}$\subseteq f({\rm int}(S))${\rm)} centering at $f(a)$ such that
		$${\rm area}(B\cap f(E))\leq c\cdot\lambda\cdot{\rm area}(B),$$
		where $c$ {\rm(}$>0${\rm)} is an absolute constant.
	\end{lemma}
	\begin{proof}
		Let $l$ be the side length of $S$ and $B_{l/2}$ be the ball centering at $a$ with radius $\frac{l}{2}$. Then $f$ a univalent map on $B_{l/2}$. Let $B_{l/4}$ be the ball centering at $a$ with radius $\frac{l}{4}$.
		Then
		$${\rm area}(E\cap B_{l/4})\leq \lambda\cdot\frac{{\rm area}(S)}{{\rm area}(B_{l/4})}\cdot{\rm area}(B_{l/4})=\frac{16}{\pi}\lambda\cdot{\rm area}(B_{l/4}).$$
		Applying the distortion theorem to $f|_{B_{l/2}}$, we have
		\begin{align*}
			{\rm area}(f(E)\cap f(B_{l/4}))&=\int_{E\cap B_{l/4}}|f'(z)|^2dxdy\\
			&\asymp|f'(a)|^2\cdot{\rm area}(E\cap B_{l/4})\\
			&\leq|f'(a)|^2\cdot\frac{16}{\pi}\lambda\cdot{\rm area}(B_{l/4})\\
			&=\frac{16}{\pi}\lambda\cdot\int_{B_{l/4}}|f'(a)|^2dxdy\\
			&\asymp\frac{16}{\pi}\lambda\cdot\int_{B_{l/4}}|f'(z)|^2dxdy\\
			&=\frac{16}{\pi}\lambda\cdot{\rm area}(f(B_{l/4})).
		\end{align*}
		By the distortion theorem, there exists an Euclidean ball $B$ {\rm(}$\subseteq f(B_{l/4})${\rm)} centering at $f(a)$ such that
		$${\rm area}(B)\asymp {\rm area}(f(B_{l/4})).$$
		Then
		\begin{align*}
			{\rm area}(f(E)\cap B)&\leq{\rm area}(f(E)\cap f(B_{l/4}))\\
			&\preceq\frac{16}{\pi}\lambda\cdot{\rm area}(f(B_{l/4}))\\
			&\asymp\frac{16}{\pi}\lambda\cdot{\rm area}(B).
		\end{align*}
		Thus
		$${\rm area}(B\cap f(E))\leq c\cdot\lambda\cdot{\rm area}(B),$$
		where $c$ {\rm(}$>0${\rm)} is an absolute constant.
		
	\end{proof}

	\subsection{Pre-filling-in of fjords\label{s4.2}}
	For any interval $I$ of level $n$ {\rm(}$\geq\mathfrak{q}_4${\rm)}, if $\hat{\Delta}_{\alpha}^{n+1}\not=\hat{\Delta}_{\alpha}^{n}$, then $\mathcal{J}(\beta_I)\not=\emptyset$ and $\mathcal{J}(\tilde{\beta}_I)\not=\emptyset$, where $\beta_I$ is the dam based on $I$. Furthermore, we have the following lemma.
	\begin{lemma}
		\label{l20257192}For any  $z\in\mathcal{J}(\beta_I)$, there exists a square $S_z$ centering at $z$ such that $S_z\subseteq\mathcal{J}(\tilde{\beta}_I)$ and ${\rm diam}_{\mathbb{C}\setminus\overline{\Delta_{\alpha}}}(S_z)\asymp1$.
	\end{lemma}
	\begin{proof}
		For any $z\in\mathcal{J}(\beta_I)$, we denote by $\mathbb{B}_z$ the largest Euclidean ball centering at $z$ contained in $\mathcal{J}(\tilde{\beta}_I)$. We define $S_z$
		as the square centering at $z$ having a side  vertical to the $x$-axis with the length $l_z=\frac{\sqrt{2}}{4}{\rm diam}(\mathbb{B}_z)$.
		Then $S_z\subseteq\mathbb{B}_z\subseteq\mathcal{J}(\tilde{\beta}_I)$.
		If $\partial\mathbb{B}_z\cap\partial\Delta_{\alpha}\not=\emptyset$, then by Lemma \ref{l25070716} and $\mathbb{B}_z\subseteq\Lambda_{\alpha}^n\subseteq\Lambda_{\alpha}^{\mathfrak{q}_4}$,
		we have ${\rm diam}_{\mathbb{C}\setminus\overline{\Delta_{\alpha}}}(S_z)\asymp1$.
		If $\partial\mathbb{B}_z\cap\partial\Delta_{\alpha}=\emptyset$, then $\partial\mathbb{B}_z\cap\tilde{\beta}_I\not=\emptyset$. Thus by Lemma \ref{l25070716} 
		\begin{equation}
			\label{e20260330a2}{\rm diam}_{\mathbb{C}\setminus\overline{\Delta_{\alpha}}}(\mathbb{B}_z)\geq{\rm dist}_{\mathbb{C}\setminus\overline{\Delta_{\alpha}}}(\beta_I,\tilde{\beta}_I)\asymp{\rm dist}_{\hat{\mathbb{C}}\setminus\overline{\Delta_{\alpha}}}(\beta_I,\tilde{\beta}_I).
		\end{equation}
		By (\ref{e20251214a1})
		we have
		$\mathcal{W}\left(\overline{\mathcal{J}(\tilde{\beta}_n)\setminus\mathcal{J}(\beta_n)}\right)\gg1.$
		By Corollary \ref{c20260331a}
		\begin{equation}
			\label{e20260721a}{\rm dist}_{\hat{\mathbb{C}}\setminus\overline{\Delta_{\alpha}}}(\beta_I,\tilde{\beta}_I)\gg1.
		\end{equation}
		By (\ref{e20260330a2}) and (\ref{e20260721a}), we have
		$${\rm diam}_{\mathbb{C}\setminus\overline{\Delta_{\alpha}}}(\mathbb{B}_z)\gg1.$$
		Observe $\mathbb{B}_z\subseteq\Lambda_{\alpha}^{\mathfrak{q}_4}$. Then Lemma \ref{l25070716} gives
		${\rm dist}(z,\Delta_{\alpha})\gg{\rm dist}(\mathbb{B}_z,\Delta_{\alpha})$. Thus
		${\rm dist}(S_z,\Delta_{\alpha})\asymp l_z$.
		At last, by 
		Lemma \ref{l25070716} we have
		$${\rm diam}_{\mathbb{C}\setminus\overline{\Delta_{\alpha}}}(S_z)\asymp1.$$
	\end{proof}
	
	We call $S_z$, described in the above lemma, a filled-in square of $\mathcal{J}(\beta_I)$ centering at $z$. We denote by $l_z$ the side length of $S_z$.
	
	Assume $\hat{\Delta}_{\alpha}^{n+s_0+1}\not=\hat{\Delta}_{\alpha}^{n+s_0}$ with $s_0\geq\chi+1$ and $n\geq2\mathfrak{q}_4$. 
	For all $I\in\mathfrak{D}_{n+s_0}$ and $z\in\mathcal{J}(\beta_I)$, we take $w\in S_z$ and assume that there exist
	positive integers $m_1<m_2<\cdots<m_k$ ($k\gg_{\mathfrak{q}_3}1$) such that for all $1\leq j\leq k$,
	$$P_{\alpha}^{\comp m_j}(w)\in\tilde{\Omega}_{c_0}^{n_j}\setminus\Lambda_\alpha^{n_j+\mathfrak{q}_3},\ n_j\geq n+1+\chi$$ and
	$$\left\{P_{\alpha}^{\comp t}(w):\ 0\leq t\leq m_k\right\}\subseteq\Lambda_{\alpha}^{n+1+\chi}.$$
	By Lemma \ref{l20260330a} there exist Euclidean balls ${\bf B}_j$ ($1\leq j\leq k$) such that
	\begin{itemize}
		\item ${\rm diam}({\bf B}_{j+t})\leq C\xi^t\cdot{\rm diam}({\bf B}_j)$ for all $1\leq j<j+t\leq k$, where $C$ is a constant depending on $\mathfrak{q}_3$ and independent of $\alpha$,
		\item ${\rm diam}({\bf B}_j)={\rm dist}(w,\partial\Delta_{\alpha})\cdot O_{\mathfrak{q}_3}(\xi^j)$ for all $1\leq j\leq k$,
		\item ${\rm dist}({\bf B}_j,w)=O_{\mathfrak{q}_3}({\rm diam}({\bf B}_j))$ for all $1\leq j\leq k$,
		\item for all $1\leq j\leq k$ and all $w'\in{\bf B}_j$, $$\left\{P_{\alpha}^{\comp t}(w'):\ 0\leq t\leq m_j\right\}\subseteq\Lambda_{\alpha}^{n+1}\ {\rm and}\ P_{\alpha}^{\comp m_j}(w')\in\hat{U}_0^{n_j+\mathfrak{l}}\cap\tilde{\Omega}_{c_0}^{n_j}\setminus\Lambda_\alpha^{n_j+\mathfrak{q}_3}.$$
		Moreover, for all $1\leq t\leq j$ and $m_{t-1}\leq s\leq m_t$ (where $m_0=0$), there exists a curve $\tilde{\gamma}_{j}^{(s)}$ connecting $P_{\alpha}^{\comp s}(w')$ and $P_{\alpha}^{\comp s}(w)$ such that
		$l_{\mathbb{C}\setminus\overline{\Delta_{\alpha}}}(\tilde{\gamma}_{j}^{(s)})=O_{\mathfrak{q}_3}(\xi^{j-t})$
		and 
		$P_{\alpha}^{\comp (m_j-s)}\comp\tilde{\gamma}_{j}^{(s)}=\tilde{\gamma}_{j}^{(m_j)}$.
	\end{itemize}
	
	We take ${\bf t}$ ($\asymp_{\mathfrak{q}_3} 1$) large enough so that
	$C\xi^{\bf t}<\eta<1$ for some real number $\eta$. It follows from the proof of Lemma \ref{l20257192} that $l_z\asymp{\rm dist}(w,\partial\Delta_{\alpha})$.
	Then there exists
	a sufficiently big positive integer $j_0$ ($\asymp_{\mathfrak{q}_3} 1$) such that ${\rm diam}({\bf B}_{j_0})\leq l_z\cdot{\bf\eta}$, since ${\rm diam}({\bf B}_j)={\rm dist}(w,\partial\Delta_{\alpha})\cdot O_{\mathfrak{q}_3}(\xi^j)$ for all $1\leq j\leq k$.
	We take $c>1$ (depending on $\mathfrak{q}_3$) so that ${\rm dist}({\bf B}_j,w)\leq\frac{c-1}{2}\cdot{\rm diam}({\bf B}_j)$ for all $1\leq j\leq k$.
	Set $j_s:=j_0+s{\bf t}$ for all $0\leq s\leq\lfloor\frac{k-j_0}{\bf t}\rfloor$.
	Then we have
	\begin{itemize}
		\item ${\rm diam}({\bf B}_{j_0})\leq l_z\cdot{\bf\eta}$,
		\item ${\rm diam}({\bf B}_{j_{s+1}})\leq {\bf\eta}\cdot{\rm diam}({\bf B}_{j_s})$ for all $0\leq s\leq\lfloor\frac{k-j_0}{\bf t}\rfloor-1$,
		\item ${\rm dist}({\bf B}_{j_s},w)\leq\frac{c-1}{2}\cdot{\rm diam}({\bf B}_{j_s})$ for all $0\leq s\leq\lfloor\frac{k-j_0}{\bf t}\rfloor$,
		\item $P_{\alpha}^{\comp t}({\bf B}_{j_s})\subseteq\Lambda_{\alpha}^{n+1}$\
		{ for} { all} $0\leq s\leq\lfloor\frac{k-j_0}{\bf t}\rfloor$ and $0\leq t\leq m_{j_s}$,
		\item
		$P_{\alpha}^{\comp m_{j_s}}({\bf B}_{j_s})\subseteq\hat{U}_0^{n_{j_s}+\mathfrak{l}}\cap\tilde{\Omega}_{c_0}^{n_{j_s}}\setminus\Lambda_{\alpha}^{n_{j_s}+\mathfrak{q}_3}$ for all $0\leq s\leq\lfloor\frac{k-j_0}{\bf t}\rfloor$,
		\item for all $0\leq s\leq\lfloor\frac{k-j_0}{\bf t}\rfloor$, $w'\in{\bf B}_{j_s}$, $1\leq t\leq j_s$ and $m_{t-1}\leq l\leq m_t$,
		there exists a curve $\tilde{\gamma}_{j_s}^{(l)}$ connecting $P_{\alpha}^{\comp l}(w')$ and $P_{\alpha}^{\comp l}(w)$ such that
		$l_{\mathbb{C}\setminus\overline{\Delta_{\alpha}}}(\tilde{\gamma}_{j_s}^{(l)})=O_{\mathfrak{q}_3}(\eta^{j_s-t})$
		and 
		$P_{\alpha}^{\comp (m_{j_s}-l)}\comp\tilde{\gamma}_{j_s}^{(l)}=\tilde{\gamma}_{j_s}^{(m_{j_s})}$.
	\end{itemize}

	\subsection{Pseudo-filling-in of fjords}
	Firstly, we take two positive integers $p$ and $\mathfrak{l}$ such that $p\gg\mathfrak{q}_4$ and $1\leq\mathfrak{l}\ll\mathfrak{q}_2$.
	We fix a positive integer $n$ ($\geq2\mathfrak{q}_4$) and assume that 
	\begin{equation}
		\label{e20260410a}\hat{\Delta}_{\alpha}^{n+p+\mathfrak{l}}\not=\hat{\Delta}_{\alpha}^{n+p}\ {\rm and}\
		\hat{\Delta}_{\alpha}^{n+j}=\hat{\Delta}_{\alpha}^{n+p}\ {\rm for\ all}\ \mathfrak{l}\leq j\leq p.
	\end{equation}
	We fix a positive integer $\mathfrak{q}_5$ such that $\mathfrak{q}_4\ll\mathfrak{q}_5\ll p$.
	We set
	$$\hat{U}_{0,p}^{n}:=\overline{(\hat{U}_0^{n+p}\cap\Omega_{c_0}^{n+2\mathfrak{q}_5})\setminus\Omega_{c_0}^{n+p-2\mathfrak{q}_5}}.$$
	
	\noindent $4.3.1.$ {\bf Thickening of $\hat{U}_{0,p}^{n}$.}
	
	\vspace{0.2cm}
	\noindent We fix a positive integer $N$ and will thicken $\hat{U}_{0,p}^{n}$ $N$ times.
	To do it, we firstly introduce some notations.
	We view	$$\tilde{\mathcal{S}}^{(1)}(\hat{U}_{0,p}^{n}):=\overline{\Omega_{c_0}^{n+\mathfrak{q}_5}\setminus\Omega_{c_0}^{n+2\mathfrak{q}_5}}$$
	as a rectangle with $\partial^{v,1}\tilde{\mathcal{S}}^{(1)}(\hat{U}_{0,p}^{n})=\overline{\partial\Omega_{c_0}^{n+\mathfrak{q}_5}\setminus\partial\Delta_{\alpha}}$ and $\partial^{v,0}\tilde{\mathcal{S}}^{(1)}(\hat{U}_{0,p}^{n})=\overline{\partial\Omega_{c_0}^{n+2\mathfrak{q}_5}\setminus\partial\Delta_{\alpha}}$.
	We view	$$\tilde{\mathcal{S}}^{(2)}(\hat{U}_{0,p}^{n}):=\overline{\Omega_{c_0}^{n+p-2\mathfrak{q}_5}\setminus\Omega_{c_0}^{n+p-\mathfrak{q}_5}}$$
	as a rectangle with $\partial^{v,0}\tilde{\mathcal{S}}^{(2)}(\hat{U}_{0,p}^{n})=\overline{\partial\Omega_{c_0}^{n+p-2\mathfrak{q}_5}\setminus\partial\Delta_{\alpha}}$ and $\partial^{v,1}\tilde{\mathcal{S}}^{(2)}(\hat{U}_{0,p}^{n})=\overline{\partial\Omega_{c_0}^{n+p-\mathfrak{q}_5}\setminus\partial\Delta_{\alpha}}$.
	Due to Lemma \ref{l7301}, we let $\mathfrak{q}_5$ sufficiently big (if necessary, we could increase $p$) so that 
	\begin{equation}
		\label{e20260406}\mathcal{W}(\tilde{\mathcal{S}}^{(1)}(\hat{U}_{0,p}^{n}))\geq\Delta\ {\rm and}\ \mathcal{W}(\tilde{\mathcal{S}}^{(2)}(\hat{U}_{0,p}^{n}))\geq\Delta.
	\end{equation}
	Since $\Delta\gg1$, by Corollary \ref{c20260331a} we have that
	$${\rm dist}_{\hat{\mathbb{C}}\setminus\overline{\Delta_{\alpha}}}(\partial\Omega_{c_0}^{n+\mathfrak{q}_5}, \partial\Omega_{c_0}^{n+2\mathfrak{q}_5})\gg1\ {\rm and}\ {\rm dist}_{\hat{\mathbb{C}}\setminus\overline{\Delta_{\alpha}}}(\partial\Omega_{c_0}^{n+p-2\mathfrak{q}_5}, \partial\Omega_{c_0}^{n+p-\mathfrak{q}_5})\gg1.$$ Then by Lemma \ref{l25070716} we have 
	\begin{equation}
		\label{e7241}{\rm dist}_{\mathbb{C}\setminus\overline{\Delta_{\alpha}}}(\partial\Omega_{c_0}^{n+\mathfrak{q}_5}, \partial\Omega_{c_0}^{n+2\mathfrak{q}_5})\gg1\ {\rm and}\ {\rm dist}_{\mathbb{C}\setminus\overline{\Delta_{\alpha}}}(\partial\Omega_{c_0}^{n+p-2\mathfrak{q}_5}, \partial\Omega_{c_0}^{n+p-\mathfrak{q}_5})\gg1.
	\end{equation}
	Let
	$$h_1: \tilde{\mathcal{S}}^{(1)}(\hat{U}_{0,p}^{n})\to E_{x_1}$$ be the conformal map for some positive real number $x_1$ with $$h_1^{-1}(\{(x_1,y):0\leq y\leq 1\})=\partial^{v,1}\tilde{\mathcal{S}}^{(1)}(\hat{U}_{0,p}^{n})$$
	and $$h_1^{-1}(\{(0,y):0\leq y\leq 1\})=\partial^{v,0}\tilde{\mathcal{S}}^{(1)}(\hat{U}_{0,p}^{n}).$$
	For all positive integer $1\leq j\leq N$, we denote by $\tilde{\mathcal{S}}_{j}^{(1)}(\hat{U}_{0,p}^{n})$ the rectangle $$h_1^{-1}\{(x,y):\frac{j-1}{N}x_1\leq x\leq\frac{j}{N}x_1, 0\leq y\leq 1\}$$
	with $\partial^{v,0}\tilde{\mathcal{S}}_{j}^{(1)}(\hat{U}_{0,p}^{n})=h_1^{-1}\{(\frac{j-1}{N}x_1,y):0\leq y\leq 1\}$ and $\partial^{v,1}\tilde{\mathcal{S}}_{j}^{(1)}(\hat{U}_{0,p}^{n})=h_1^{-1}\{(\frac{j}{N}x_1,y):0\leq y\leq 1\}$.
	Let
	$$h_2: \tilde{\mathcal{S}}^{(2)}(\hat{U}_{0,p}^{n})\to E_{x_2}$$ be the conformal map for some positive real number $x_2$ with $$h_2^{-1}\{(x_2,y):0\leq y\leq 1\}=\partial^{v,1}\tilde{\mathcal{S}}^{(2)}(\hat{U}_{0,p}^{n})$$
	and $$h_2^{-1}\{(0,y):0\leq y\leq 1\}=\partial^{v,0}\tilde{\mathcal{S}}^{(2)}(\hat{U}_{0,p}^{n}).$$
	For all positive integer $1\leq j\leq N$, we denote by $\tilde{\mathcal{S}}_{j}^{(2)}(\hat{U}_{0,p}^{n})$ the rectangle $$h_2^{-1}\{(x,y):\frac{j-1}{N}x_2\leq x\leq\frac{j}{N}x_2,0\leq y\leq 1\}$$
	with $\partial^{v,0}\tilde{\mathcal{S}}_{j}^{(2)}(\hat{U}_{0,p}^{n})=h_2^{-1}\{(\frac{j-1}{N}x_2,y):0\leq y\leq 1\}$ and $\partial^{v,1}\tilde{\mathcal{S}}_{j}^{(2)}(\hat{U}_{0,p}^{n})=h_2^{-1}\{(\frac{j}{N}x_2,y):0\leq y\leq 1\}$.
	
	For all $I\in\bigcup\limits_{j=0}^{\mathfrak{l}-1}\mathfrak{D}_{n+p+j}$ with $\beta_I\not=\emptyset$, we denote by $\tilde{\mathcal{S}}^{(3)}(\beta_I)$ the rectangle $\overline{\mathcal{J}(\tilde{\beta}_I)\setminus\mathcal{J}(\beta_I)}$ with 
	$\partial^{v,0}\tilde{\mathcal{S}}^{(3)}(\beta_I)=\beta_I$ and $\partial^{v,1}\tilde{\mathcal{S}}^{(3)}(\beta_I)=\tilde{\beta}_I$.
	Let
	$$h_{3,I}: \tilde{\mathcal{S}}^{(3)}(\beta_I)\to E_{x_{3,I}}$$ be the conformal map for some positive real number $x_{3,I}$ with $$h_{3,I}^{-1}\{(x_{3,I},y):0\leq y\leq 1\}=\partial^{v,1}\tilde{\mathcal{S}}^{(3)}(\beta_I)$$
	and $$h_{3,I}^{-1}\{(0,y):0\leq y\leq 1\}=\partial^{v,0}\tilde{\mathcal{S}}^{(3)}(\beta_I).$$
	For all positive integer $1\leq j\leq N$, we denote by $\tilde{\mathcal{S}}_j^{(3)}(\beta_I)$ the rectangle $$h_{3,I}^{-1}\{(x,y):\frac{j-1}{N}x_{3,I}\leq x\leq\frac{j}{N}x_{3,I},0\leq y\leq 1\}$$
	with $\partial^{v,0}\tilde{\mathcal{S}}_j^{(3)}(\beta_I)=h_{3,I}^{-1}\{(\frac{j-1}{N}x_{3,I},y):0\leq y\leq 1\}$ and $\partial^{v,1}\tilde{\mathcal{S}}_j^{(3)}(\beta_I)=h_{3,I}^{-1}\{(\frac{j}{N}x_{3,I},y):0\leq y\leq 1\}$.
	We set
	$$\tilde{\mathcal{S}}_{j}^{(3)}(\hat{U}_{0,p}^{n}):=\bigcup_{s=0}^{\mathfrak{l}-1}\bigcup_{I\in\mathfrak{D}_{n+p+s},\beta_I\not=\emptyset}\tilde{\mathcal{S}}_j^{(3)}(\beta_I)$$
	and let $U_0(\tilde{\mathcal{S}}_{j}^{(3)}(\hat{U}_{0,p}^{n}))$ be the union of components of
	$P_{\alpha}^{-1}(\tilde{\mathcal{S}}_{j}^{(3)}(\hat{U}_{0,p}^{n}))$ intersecting $\hat{U}_0^{n+p}$.
	
	Now we can define thickenings of $\hat{U}_{0,p}^{n}$. For all positive integer $1\leq j\leq N$, the $j$-th thickening of $\hat{U}_{0,p}^{n}$ is defined by
	$$\mathcal{S}_{j}(\hat{U}_{0,p}^{n}):=\mathcal{S}_j^{(1)}(\hat{U}_{0,p}^{n})\cup\mathcal{S}_j^{(2)}(\hat{U}_{0,p}^{n})\cup\mathcal{S}_j^{(3)}(\hat{U}_{0,p}^{n}),$$
	where
	$$\mathcal{S}_j^{(1)}(\hat{U}_{0,p}^{n}):=\left(\hat{U}_0^{n+p}\cup\bigcup\limits_{s=1}^jU_0(\tilde{\mathcal{S}}_{s}^{(3)}(\hat{U}_{0,p}^{n}))\right)\cap\tilde{\mathcal{S}}_j^{(1)}(\hat{U}_{0,p}^{n}),$$
	$$\mathcal{S}_j^{(2)}(\hat{U}_{0,p}^{n}):=\left(\hat{U}_0^{n+p}\cup\bigcup\limits_{s=1}^jU_0(\tilde{\mathcal{S}}_{s}^{(3)}(\hat{U}_{0,p}^{n}))\right)\cap\tilde{\mathcal{S}}_j^{(2)}(\hat{U}_{0,p}^{n})$$
	and
	$$\mathcal{S}_{j}^{(3)}(\hat{U}_{0,p}^{n}):=U_0(\tilde{\mathcal{S}}_{j}^{(3)}(\hat{U}_{0,p}^{n}))\cap\left(\overline{\Omega_{c_0}^{n+2\mathfrak{q}_5}\setminus\Omega_{c_0}^{n+p-2\mathfrak{q}_5}}\cup\bigcup_{s=1}^j\tilde{\mathcal{S}}_{s}^{(1)}(\hat{U}_{0,p}^{n})\cup\bigcup_{s=1}^j\tilde{\mathcal{S}}_{s}^{(2)}(\hat{U}_{0,p}^{n})\right).$$
	Figure \ref{f20260720f} illustrates thickenings of $\hat{U}_{0,p}^{n}$ for the case of $N=2$.
	\begin{figure}
		\centering
		\includegraphics[scale=0.4]{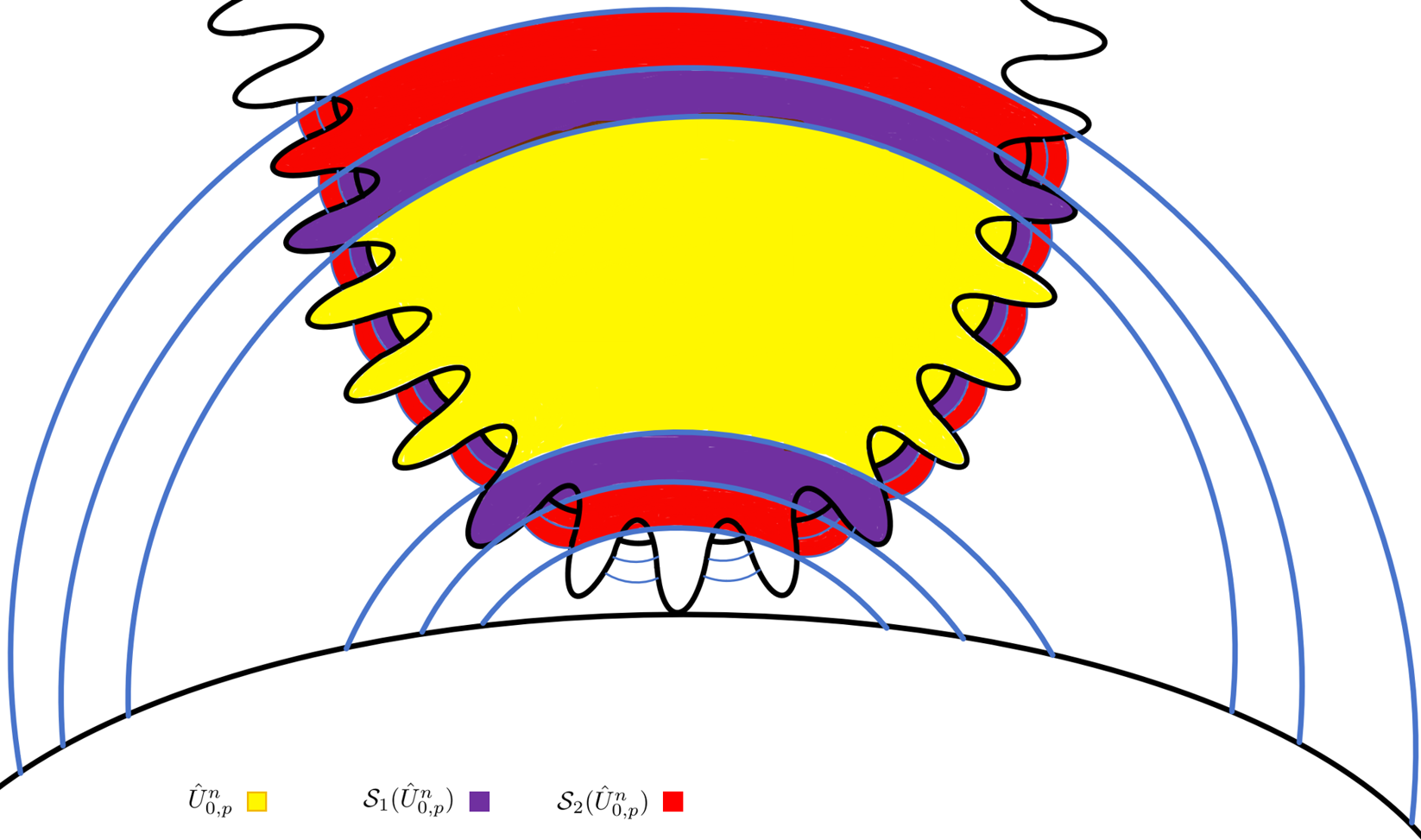}
		\caption{Thickenings of $\hat{U}_{0,p}^{n}$ for $N=2$.}
		\label{f20260720f}
	\end{figure}
	
	For all $1\leq j\leq N$, the following equality is  easy to be checked:
	\begin{align}
		\label{e20260406c}\hat{U}_{0,p}^{n}\cup\bigcup\limits_{s=1}^{j}\mathcal{S}_s(\hat{U}_{0,p}^{n})=&
		\left(\hat{U}_0^{n+p}\cup\bigcup\limits_{s=1}^jU_0(\tilde{\mathcal{S}}_{s}^{(3)}(\hat{U}_{0,p}^{n}))\right)\qquad\qquad\qquad\qquad\qquad\qquad
	\end{align}
	$$\qquad\qquad\qquad\qquad\qquad\cap\left(\overline{\Omega_{c_0}^{n+2\mathfrak{q}_5}\setminus\Omega_{c_0}^{n+p-2\mathfrak{q}_5}}\cup\bigcup_{s=1}^j\tilde{\mathcal{S}}_{s}^{(1)}(\hat{U}_{0,p}^{n})\cup\bigcup_{s=1}^j\tilde{\mathcal{S}}_{s}^{(2)}(\hat{U}_{0,p}^{n})\right).$$
	\begin{lemma}
		\label{l12615}There exists a strictly increasing real function $\kappa:\mathbb{R}_+\to\mathbb{R}_+$, independent of $\alpha$, such that
		for all $N\geq2$, $1\leq j\leq N-1$ and $z\in\hat{U}_{0,p}^{n}\cup\bigcup\limits_{s=1}^j\mathcal{S}_s(\hat{U}_{0,p}^{n})$ with $P_{\alpha}(z)\not\in\hat{\Delta}_{\alpha}^{n+p+\mathfrak{l}}$, we have that for all $z'\in\mathbb{C}$ with the curve $\gamma$ connecting $z$ and $z'$ satisfying $l_{\mathbb{C}\setminus\overline{\Delta_{\alpha}}}(P_{\alpha}(\gamma))<
		\kappa(\frac{\Delta}{N})$, $$z'\in\hat{U}_{0,p}^{n}\cup\left(\bigcup_{s=1}^{j}\mathcal{S}_s(\hat{U}_{0,p}^{n})\right)\cup{\rm int}(\mathcal{S}_{j+1}(\hat{U}_{0,p}^{n})).$$
	\end{lemma}
	\begin{proof}
		For all $x>0$, we denote by $\mathcal{H}_x$ the set of all rectangles $H$ satisfying the following properties:
		\begin{itemize}
			\item $H\setminus(\partial^{h,0}H\cup\partial^{h,1}H)\subseteq\mathbb{C}\setminus\overline{\Delta_{\alpha}}$,
			\item $\partial^{h,0}H\subseteq\partial\Delta_{\alpha}$ and $\partial^{h,1}H\subseteq\partial\Delta_{\alpha}$,
			\item $\mathcal{W}(H)=x$.
		\end{itemize}
		For all $x>0$, we let $\kappa(x)$ be the infimum of ${\rm dist}_{\mathbb{C}\setminus\overline{\Delta_{\alpha}}}(\partial^{v,0}H,\partial^{v,1}H)$ for all $H\in\mathcal{H}_x$. It is clear that $\kappa(x)$ is independent of $\alpha$ and $\lim\limits_{x\to0^+}\kappa(x)=0$.
		By Corollary \ref{c20260331a} and the Schwarz lemma,
		$$\kappa(x)\geq\inf\limits_{H\in\mathcal{H}_x}{\rm dist}_{\hat{\mathbb{C}}\setminus\overline{\Delta_{\alpha}}}(\partial^{v,0}H,\partial^{v,1}H)\geq\log(\frac{e^{4\pi/x}+1}{e^{4\pi/x}-1}).$$
		This implies that $\kappa(x)$ is a strictly increasing positive function.
		
		Since $z\in\hat{U}_{0,p}^{n}\cup\bigcup\limits_{s=1}^j\mathcal{S}_s(\hat{U}_{0,p}^{n})$ and $P_{\alpha}(z)\not\in\hat{\Delta}_{\alpha}^{n+p+\mathfrak{l}}$, we have $P_{\alpha}(z)\in\mathcal{J}(\beta_I)\cup
		\bigcup\limits_{s=1}^j\tilde{\mathcal{S}}_s^{(3)}(\beta_I)$
		for some $I\in\bigcup\limits_{j=0}^{\mathfrak{l}-1}\mathfrak{D}_{n+p+j}$.
		Let $z'\in\mathbb{C}$ with the curve $\gamma$ connecting $z$ and $z'$ satisfying $l_{\mathbb{C}\setminus\overline{\Delta_{\alpha}}}(P_{\alpha}(\gamma))<
		\kappa(\frac{\Delta}{N})$. Observe that $\mathcal{W}(\tilde{\mathcal{S}}_s^{(3)}(\beta_I))\geq\Delta/N$ for all $1\leq s\leq N$.
		Then 
		\begin{equation}
			\label{e20260405a}P_{\alpha}(\gamma)\subseteq{\rm int}(\mathcal{J}(\beta_I)\cup
			\bigcup\limits_{s=1}^{j+1}\tilde{\mathcal{S}}_s^{(3)}(\beta_I)).
		\end{equation}
		Let $\Gamma_{\alpha}$ be the same as that in Corollary \ref{c20251027} and
		$h_{\Gamma_{\alpha}}^{(1)},h_{\Gamma_{\alpha}}^{(2)}$ be the single-value analytic branches of $P_{\alpha}^{-1}$ on 
		$\hat{\mathbb{C}}\setminus\Gamma_{\alpha}$ such that $h_{\Gamma_{\alpha}}^{(1)}(0)=0$ and $h_{\Gamma_{\alpha}}^{(2)}(0)\not=0$.
		It is known that $P_{\alpha}(c_{0})$ is not contained in the closure of any big fjord of level $\geq-1$, see Section \ref{s2.3}. Then
		\begin{equation}
			\label{e20260406d}\mathcal{J}(\beta_I)\cup\bigcup\limits_{s=1}^{N}\tilde{\mathcal{S}}_s^{(3)}(\beta_I)\subseteq\overline{\mathcal{J}(\tilde{\beta}_I)}\subseteq\tilde{\Delta}_{\alpha}^{n+p}\setminus\{P_{\alpha}(c_{0})\}.
		\end{equation}
		Corollary \ref{c20251027} gives 
		\begin{equation}
			\label{e20260406e}\tilde{\Delta}_{\alpha}^{n+p}\setminus\{P_{\alpha}(c_0)\}\subseteq\hat{\mathbb{C}}\setminus\Gamma_{\alpha}, 
		\end{equation}
		and hence
		$$P_{\alpha}(\gamma)\subseteq\mathcal{J}(\beta_I)\cup
		\bigcup\limits_{s=1}^{N}\tilde{\mathcal{S}}_s^{(3)}(\beta_I)\subseteq\hat{\mathbb{C}}\setminus\Gamma_{\alpha}.$$
		Then $h_{\Gamma_{\alpha}}^{(1)}(P_{\alpha}(\gamma))$ and $h_{\Gamma_{\alpha}}^{(2)}(P_{\alpha}(\gamma))$ are exactly two lifts of $P_{\alpha}(\gamma)$, and 
		\begin{equation}
			\label{e20260405b}h_{\Gamma_{\alpha}}^{(1)}(P_{\alpha}(\gamma))\cap h_{\Gamma_{\alpha}}^{(2)}(\hat{\mathbb{C}}\setminus\Gamma_{\alpha})=\emptyset.
		\end{equation}
		
		Recall the definitions of 
		$\hat{U}_{0}^{n+p}$ and $U_0(\tilde{\mathcal{S}}_{s}^{(3)}(\hat{U}_{0,p}^{n})), 1\leq s\leq N$, that is $\hat{U}_{0}^{n+p}$ is
		the closure of the component of $P_{\alpha}^{-1}(\hat{\Delta}_{\alpha}^{n+p})\setminus\{c_0\}$ not containing $0$
		and
		$U_0(\tilde{\mathcal{S}}_{s}^{(3)}(\hat{U}_{0,p}^{n}))$ is the union of components of
		$P_{\alpha}^{-1}(\tilde{\mathcal{S}}_{s}^{(3)}(\hat{U}_{0,p}^{n}))$ intersecting $\hat{U}_0^{n+p}$. Then it is clear that
		$$\hat{U}_0^{n+p}\setminus\{c_0\}=h_{\Gamma_{\alpha}}^{(2)}\left(\hat{\Delta}_{\alpha}^{n+p}\setminus\{P_{\alpha}(c_0)\}\right)\ {\rm and}\ 
		U_0(\tilde{\mathcal{S}}_{s}^{(3)}(\hat{U}_{0,p}^{n}))=h_{\Gamma_{\alpha}}^{(2)}\left(\tilde{\mathcal{S}}_{s}^{(3)}(\hat{U}_{0,p}^{n})\right), 1\leq s\leq N.$$
		Then by (\ref{e20260406c}), (\ref{e20260406d}), (\ref{e20260406e}) and the above two equalities, we have 
		\begin{align*}
			z\in\hat{U}_{0,p}^{n}\cup\bigcup\limits_{s=1}^{j}\mathcal{S}_s(\hat{U}_{0,p}^{n})&\subseteq
			\left(\hat{U}_0^{n+p}\setminus\{c_0\}\right)\cup\bigcup\limits_{s=1}^jU_0(\tilde{\mathcal{S}}_{s}^{(3)}(\hat{U}_{0,p}^{n}))\\
			&=h_{\Gamma_{\alpha}}^{(2)}\left(\hat{\Delta}_{\alpha}^{n+p}\setminus\{P_{\alpha}(c_0)\}\right)\cup \bigcup\limits_{s=1}^jh_{\Gamma_{\alpha}}^{(2)}\left(\tilde{\mathcal{S}}_{s}^{(3)}(\hat{U}_{0,p}^{n})\right)\\
			&\subseteq h_{\Gamma_{\alpha}}^{(2)}\left(\tilde{\Delta}_{\alpha}^{n+p}\setminus\{P_{\alpha}(c_0)\}\right)\\
			&\subseteq h_{\Gamma_{\alpha}}^{(2)}(\hat{\mathbb{C}}\setminus\Gamma_{\alpha}).
		\end{align*}
		Then
		\begin{equation*}
			\gamma\cap h_{\Gamma_{\alpha}}^{(2)}(\hat{\mathbb{C}}\setminus\Gamma_{\alpha})\not=\emptyset.
		\end{equation*}
		By combining (\ref{e20260405b}) and the fact that $\gamma$ is a lift of $P_{\alpha}(\gamma)$, it
		follows that $$\gamma=h_{\Gamma_{\alpha}}^{(2)}(P_{\alpha}(\gamma)).$$
		
		Now on the one hand, by (\ref{e20260405a}) we have
		\begin{equation}
			\label{e07171}\gamma=h_{\Gamma_{\alpha}}^{(2)}(P_{\alpha}(\gamma))\subseteq h_{\Gamma_{\alpha}}^{(2)}\left({\rm int}(\mathcal{J}(\beta_I)\cup
			\bigcup\limits_{s=1}^{j+1}\tilde{\mathcal{S}}_s^{(3)}(\beta_I))\right)
		\end{equation}
		\begin{equation*}
			\subseteq{\rm int}(\hat{U}_{0}^{n+p}\cup\bigcup\limits_{s=1}^{j+1}U_0(\tilde{\mathcal{S}}_{s}^{(3)}(\hat{U}_{0,p}^{n}))).
		\end{equation*}
		On the other hand, together with
		$$z\in\hat{U}_{0,p}^{n}\cup\bigcup\limits_{s=1}^j\mathcal{S}_s(\hat{U}_{0,p}^{n})\subseteq\overline{\Omega_{c_0}^{n+2\mathfrak{q}_5}\setminus\Omega_{c_0}^{n+p-2\mathfrak{q}_5}}\cup\bigcup\limits_{s=1}^j\tilde{\mathcal{S}}_{s}^{(1)}(\hat{U}_{0,p}^{n})\cup\bigcup\limits_{s=1}^j\tilde{\mathcal{S}}_{s}^{(2)}(\hat{U}_{0,p}^{n})$$
		and $$l_{\mathbb{C}\setminus\overline{\Delta_{\alpha}}}(\gamma)<
		\kappa(\frac{\Delta}{N})\ {\rm(by\ the\ Schwarz\ lemma)},$$
		(\ref{e20260406}) gives
		\begin{equation}
			\label{e20260721}\gamma\subseteq{\rm int}\left(\overline{\Omega_{c_0}^{n+2\mathfrak{q}_5}\setminus\Omega_{c_0}^{n+p-2\mathfrak{q}_5}}\cup\bigcup\limits_{s=1}^{j+1}\tilde{\mathcal{S}}_{s}^{(1)}(\hat{U}_{0,p}^{n})\cup\bigcup\limits_{s=1}^{j+1}\tilde{\mathcal{S}}_{s}^{(2)}(\hat{U}_{0,p}^{n})\right).
		\end{equation}
		Thus it follows from (\ref{e20260406c}), (\ref{e07171}) and  (\ref{e20260721}) that $\gamma\subseteq{\rm int}(\hat{U}_{0,p}^{n}\cup\bigcup_{s=1}^{j+1}\mathcal{S}_s(\hat{U}_{0,p}^{n}))$.
		Observe that $\hat{U}_{0,p}^{n}\cup\left(\bigcup_{s=1}^{j}\mathcal{S}_s(\hat{U}_{0,p}^{n})\right)$ is closed. Thus
		$${\rm int}(\hat{U}_{0,p}^{n}\cup\bigcup_{s=1}^{j+1}\mathcal{S}_s(\hat{U}_{0,p}^{n}))
		\subseteq\hat{U}_{0,p}^{n}\cup\left(\bigcup_{s=1}^{j}\mathcal{S}_s(\hat{U}_{0,p}^{n})\right)\cup{\rm int}(\mathcal{S}_{j+1}(\hat{U}_{0,p}^{n})).$$
		It follows that
		\begin{align*}
			\gamma\subseteq\hat{U}_{0,p}^{n}\cup\left(\bigcup_{s=1}^{j}\mathcal{S}_s(\hat{U}_{0,p}^{n})\right)\cup{\rm int}(\mathcal{S}_{j+1}(\hat{U}_{0,p}^{n})).
		\end{align*}
		This implies
		$$z'\in\hat{U}_{0,p}^{n}\cup\left(\bigcup_{s=1}^{j}\mathcal{S}_s(\hat{U}_{0,p}^{n})\right)\cup{\rm int}(\mathcal{S}_{j+1}(\hat{U}_{0,p}^{n})).$$
		
	\end{proof}
	
	\noindent $4.3.2.$ {\bf Pseudo-filling-in of fjords.}
	
	\vspace{0.2cm}
	\noindent For all positive integers $t$ and $s$ with $t>s$, we set
	$$K^{t-s}_{\alpha}:=\left\{z:\ P_{\alpha}^{\comp j}(z)\in\Lambda_{\alpha}^{t-s}\ {\rm for\ all}\ j\geq0\right\}$$
	and denote by $E_s^{\mathfrak{l}}(\hat{\Delta}_{\alpha}^{t})$ the union:
	$$\left\{z:\ \exists k\geq0\ {\rm s.t.}\ P_{\alpha}^{\comp k}(z)\in \hat{\Delta}_{\alpha}^{t+\mathfrak{l}-1}\ {\rm and}\
	P_{\alpha}^{\comp j}(z)\in\Lambda_{\alpha}^{t-s}\ {\rm for}\ {\rm all}\ 0\leq j\leq k-1\right\}\bigcup K^{t-s}_{\alpha}.$$
	Let $X$ and $Y$ be two subsets of $\mathbb{C}$ such that the area of $X$ satisfies $0<{\rm area}(X)<+\infty$. By ${\rm dens}_{X}Y$ we mean the density of $Y$ in $X$, that is
	$${\rm dens}_{X}Y=\frac{{\rm area}(X\cap Y)}{{\rm area}(X)}.$$
	We fix $n\geq2\mathfrak{q}_4$, $N\gg_{\mathfrak{q}_3}1$, $p\geq N\gg{\mathfrak{q}_5}$, $p\gg_{N}\mathfrak{q}_3$ and $1\leq\mathfrak{l}\ll\mathfrak{q}_2$. 
	Assume that (\ref{e20260410a}) holds. The following lemma is the main result of this section.
	\begin{lemma}
		\label{key1}There exists an absolute constant $0<\bm{\lambda}<1$ such that for all $I\in\bigcup\limits_{j=0}^{\mathfrak{l}-1}\mathfrak{D}_{n+p+j}$ with $\mathcal{J}(\beta_I)\not=\emptyset$ and all $z\in\mathcal{J}(\beta_I)$,
		$${\rm dens}_{S_z}S_z\setminus E_p^{\mathfrak{l}}(\hat{\Delta}_{\alpha}^{n+p+1})\leq\bm{\lambda}^N.$$
	\end{lemma}
	\begin{proof}
		We denote by $\mathcal{R}^{(1)}$ the set
		of $w\in S_z\setminus E_p^{\mathfrak{l}}(\hat{\Delta}_{\alpha}^{n+p+1})$ such that the orbit of $w$ essentially visits $\tilde{\Omega}_{c_0}^{n+p}$ at least $N$ times before leaving $\Lambda_{\alpha}^{n+1+\chi}$ and
		by $\mathcal{R}^{(2)}$ the set of
		$w\in S_z\setminus E_p^{\mathfrak{l}}(\hat{\Delta}_{\alpha}^{n+p+1})$ such that the orbit of $w$ visits $\hat{U}_{0,p}^{n}$ $0$ times before leaving $\Lambda_{\alpha}^{n+1+\chi}$.
		We define
		$$\mathcal{N}:=\left\{{\bf n}=(m,t)\in\mathbb{N}_+\times\mathbb{N}_+:0<m<N,\ 0<t<N\right\}.$$
		For any ${\bf n}=(m,t)\in\mathcal{N}$, we denote by
		$\tilde{E}_{{\bf n}}$ the subset of $S_z$ consisting of $w\in S_z\setminus E_p^{\mathfrak{l}}(\hat{\Delta}_{\alpha}^{n+p+1})$ such that the orbit of $w$ visits $\hat{U}_{0,p}^{n}\cup\bigcup\limits_{j=1}^{t-1}\mathcal{S}_{j}(\hat{U}_{0,p}^{n})$ $m$ times and doesn't visit ${\rm int}(\mathcal{S}_t(\hat{U}_{0,p}^{n}))$ before leaving $\Lambda_{\alpha}^{n+1+\chi}$.
		
		Next, we will prove
		\begin{itemize}
			\item[(1)] $S_z\setminus E_p^{\mathfrak{l}}(\hat{\Delta}_{\alpha}^{n+p+1})\subseteq\mathcal{R}^{(1)}
			\cup\mathcal{R}^{(2)}\cup\bigcup_{{{\bf n}}\in\mathcal{N}}\tilde{E}_{{\bf n}}$,
			\item[(2)]
			${\rm area}(\mathcal{R}^{(1)})\leq\lambda_1^N\cdot{\rm area}(S_z),$ where $0<\lambda_1<1$ is a constant determined by $\mathfrak{q}_3$ and independent of $z$ and $\alpha$,
			\item[(3)]
			${\rm area}(\mathcal{R}^{(2)})\leq\lambda_2^p\cdot{\rm area}(S_z),$ where $0<\lambda_2<1$ is a constant determined by $\mathfrak{q}_3$ and independent of $z$ and $\alpha$,
			\item[(4)] for all ${\bf n}=(m,t)\in\mathcal{N}$,
			${\rm area}(\tilde{E}_{{\bf n}})\leq\lambda_2^{\frac{p}{3}}\cdot{\rm area}(S_z)$.
		\end{itemize}
		\emph{The proof of (1)}: 
		We take $w\in S_z\setminus E_p^{\mathfrak{l}}(\hat{\Delta}_{\alpha}^{n+p+1})$ and $w\not\in\mathcal{R}^{(2)}\cup\bigcup_{{\bf n}\in\mathcal{N}}\tilde{E}_{\bf n}$. Then the orbit of $w$ visits $\hat{U}_{0,p}^{n}\cup\bigcup_{1\leq j\leq N-1}\mathcal{S}_{j}(\hat{U}_{0,p}^{n})$ at least $N$ times before leaving $\Lambda_{\alpha}^{n+1+\chi}$. By Lemmas \ref{l20257191}, \ref{l20257193} and \ref{l20257192}, we have that
		\begin{itemize}
			\item[(I)] the orbit of $w$ will essentially visit $\tilde{\Omega}_{c_0}^{n+p}$
			at least $1$ times before firstly visiting $\hat{U}_{0,p}^{n}\cup\bigcup_{1\leq j\leq N-1}\mathcal{S}_{j}(\hat{U}_{0,p}^{n})$;
			\item[(II)] the orbit of $w$ will essentially visit $\tilde{\Omega}_{c_0}^{n+p}$
			at least $1$ times between two times visiting $\hat{U}_{0,p}^{n}\cup\bigcup_{1\leq j\leq N-1}\mathcal{S}_{j}(\hat{U}_{0,p}^{n})$;
			\item[(III)] the orbit of $w$ will essentially visit $\tilde{\Omega}_{c_0}^{n+p}$
			at least $1$ times after the $N$-th visiting $\hat{U}_{0,p}^{n}\cup\bigcup_{1\leq j\leq N-1}\mathcal{S}_{j}(\hat{U}_{0,p}^{n})$ before leaving $\Lambda_{\alpha}^{n+1+\chi}$.
		\end{itemize}
		Indeed, by Lemma \ref{l20257192},  $w\in\mathcal{J}(\tilde{\beta}_I)$. By Lemma \ref{l20257191}, there exists a positive integer $J_0$ such that 
		\begin{equation}
			\label{e20260512a}P_{\alpha}^{\comp J_0}(w)\in\Omega_{c_0}^{n+p}
		\end{equation}
		and
		\begin{equation}
			\label{e20260512b}
			P_{\alpha}^{\comp j}(w)\in\Lambda_{\alpha}^{n+p},\ 0\leq j\leq J_0.
		\end{equation}
		By Lemma \ref{l20257193}, for all $j\geq\mathfrak{q}_4$,
		$$\tilde{U}_0^{-1}\cap\left(\tilde{\Delta}_{\alpha}^{-1}\cup(\Lambda_{\alpha}^{j+\mathfrak{q}_2}\setminus\Omega_{c_0}^j)\right)=\{c_0\}.$$
		Then $\tilde{U}_0^{-1}\cap(\Lambda_{\alpha}^{n+p-\mathfrak{q}_5+\mathfrak{q}_2}\setminus\Omega_{c_0}^{n+p-\mathfrak{q}_5})=\emptyset$ and hence $\Lambda_{\alpha}^{n+p}\cap\tilde{U}_0^{-1}\setminus\Omega_{c_0}^{n+p-\mathfrak{q}_5}=\emptyset.$
		It follows from (\ref{e20260406c}) that 
		$$\hat{U}_{0,p}^{n}\cup\bigcup_{1\leq j\leq N-1}\mathcal{S}_{j}(\hat{U}_{0,p}^{n})\subseteq\tilde{U}_0^{-1}\setminus\Omega_{c_0}^{n+p-\mathfrak{q}_5}.$$
		Then
		\begin{equation}
			\label{e7201}\Lambda_{\alpha}^{n+p}\cap\left(\hat{U}_{0,p}^{n}\cup\bigcup_{1\leq j\leq N-1}\mathcal{S}_{j}(\hat{U}_{0,p}^{n})\right)=\emptyset.
		\end{equation}
		Thus
		\begin{equation}
			\label{e20260512c}\{P_{\alpha}^{\comp j}(w): 0\leq j\leq J_0\}\cap\left(\hat{U}_{0,p}^{n}\cup\bigcup_{1\leq j\leq N-1}\mathcal{S}_{j}(\hat{U}_{0,p}^{n})\right)=\emptyset.
		\end{equation}
		By (\ref{e20260512a}) we have $P_{\alpha}^{\comp J_0}(w)\in\tilde{\Omega}_{c_0}^{n+p}$ or $P_{\alpha}^{\comp J_0}(w)\in\Omega_{c_0}^{n+p}\setminus\tilde{\Omega}_{c_0}^{n+p}$.
		Firstly, we consider the case: $P_{\alpha}^{\comp J_0}(w)\in\tilde{\Omega}_{c_0}^{n+p}$. 
		If $P_{\alpha}^{\comp J_0}(w)\not\in\Lambda_{\alpha}^{n+p+\mathfrak{q}_3}$, then (I) holds. If $P_{\alpha}^{\comp J_0}(w)\in\Lambda_{\alpha}^{n+p+\mathfrak{q}_3}$, then
		by Lemma \ref{l819} and $P_{\alpha}^{\comp(J_0+\Delta J)}(w)\not\in\Lambda_{\alpha}^{n+p}\cup\overline{\Delta_{\alpha}}$ for some positive integer $\Delta J$ (Due to $w\in S_z\setminus E_p^{\mathfrak{l}}(\hat{\Delta}_{\alpha}^{n+p+1})$), there exists a positive integer $m_1$ such that 
		\begin{equation}
			\label{e20260512d}P_{\alpha}^{\comp (J_0+m_1)}(w)\in\tilde{\Omega}_{c_0}^{n+p}\setminus\Lambda_{\alpha}^{n+p+\mathfrak{q}_3} \ {\rm and}\ \left\{P_{\alpha}^{\comp(J_0+s)}(z):\ 0\leq s\leq m_1\right\}\subseteq\Lambda_{\alpha}^{n+p}.
		\end{equation}
		By combining (\ref{e7201}), it follows that
		\begin{equation}
			\label{e20260512e}\left\{P_{\alpha}^{\comp(J_0+s)}(z):\ 0\leq s\leq m_1\right\}\cap\left(\hat{U}_{0,p}^{n}\cup\bigcup_{1\leq j\leq N-1}\mathcal{S}_{j}(\hat{U}_{0,p}^{n})\right)=\emptyset.
		\end{equation}
		Thus it follows from $(\ref{e20260512b})$, $(\ref{e20260512c})$, $(\ref{e20260512d})$ and $(\ref{e20260512e})$ that (I) holds. Next, we consider the other case: $P_{\alpha}^{\comp J_0}(w)\in\Omega_{c_0}^{n+p}\setminus\tilde{\Omega}_{c_0}^{n+p}$. Then $P_{\alpha}^{\comp J_0}(w)\in\mathcal{J}(\tilde{\beta}_{I'})$ for some interval $I'$ of level $m'\geq n+p$.
		Thus we can repeat the above argument to $P_{\alpha}^{\comp J_0}(w)$. Observe that $P_{\alpha}^{\comp(J_0+\Delta J)}(w)\not\in\Lambda_{\alpha}^{n+p}\cup\overline{\Delta_{\alpha}}$ and
		$(\ref{e20260512b})$. So
		the second case can occur at most finitely many times, and hence the first case must occur. Then (I) holds. 
		
		In the remaining two cases (II) and (III), if $P_{\alpha}^{\comp j_2}(w)\in\hat{U}_{0,p}^{n}\cup\bigcup_{1\leq j\leq N-1}\mathcal{S}_{j}(\hat{U}_{0,p}^{n})$ for some positive integer $j_2$ and the orbit of $P_{\alpha}^{\comp (j_2+1)}(w)$ will escape from $\Lambda_{\alpha}^{n+p}\cup\overline{\Delta_{\alpha}}$, then $P_{\alpha}^{\comp (j_2+1)}(w)\in\mathcal{J}(\tilde{\beta}_{I''})$ for some interval
		$I''$ of level $m''\geq n+p$. In the same way as proving (I), we can obtain that $P_{\alpha}^{\comp (j_2+1)}(w)$ will essentially visit $\tilde{\Omega}_{c_0}^{n+p}$
		at least $1$ times before escaping from $\Lambda_{\alpha}^{n+p}\cup\overline{\Delta_{\alpha}}$. This implies that both (II) and (III) hold.
		
		By (I), (II) and (III), the orbit of $w$ will essentially visit $\tilde{\Omega}_{c_0}^{n+p}$
		at least $N$ times before leaving $\Lambda_{\alpha}^{n+1+\chi}$.
		This implies $w\in\mathcal{R}^{(1)}$,
		which completes the proof of (1).
		
		\vspace{0.2cm}
		\noindent\emph{The proof of (2)}: For all $w\in\mathcal{R}^{(1)}$, we have that $w\in S_z\setminus E_p^{\mathfrak{l}}(\hat{\Delta}_{\alpha}^{n+p+1})$ and the orbit of $w$ essentially visits $\tilde{\Omega}_{c_0}^{n+p}$ at least $N$ times before leaving $\Lambda_{\alpha}^{n+1+\chi}$.
		Then there exist
		positive integers $m_{1}<m_{2}<\cdots<m_{N-1}$ such that for all $1\leq j\leq N-1$,
		$$P_{\alpha}^{\comp m_j}(w)\in\tilde{\Omega}_{c_0}^{n+p}\setminus\Lambda_{\alpha}^{n+p+\mathfrak{q}_3}$$
		and
		$$\left\{P_{\alpha}^{\comp s}(w):\ 0\leq s\leq m_j\right\}\subseteq\Lambda_{\alpha}^{n+1+\chi}.$$
		
		Since $I\in\bigcup\limits_{j=0}^{\mathfrak{l}-1}\mathfrak{D}_{n+p+j}$ and $\mathcal{J}(\beta_I)\not=\emptyset$, there exists $s_0\in\{p,p+1,\cdots,p+\mathfrak{l}-1\}$ such that
		$\hat{\Delta}_{\alpha}^{n+s_0+1}\not=\hat{\Delta}_{\alpha}^{n+s_0}$ and $I$ is of level $n+s_0$. Then $s_0\geq p\gg\mathfrak{q}_4$, $n\geq2\mathfrak{q}_4$, $N-1\gg_{\mathfrak{q}_3}1$,
		$m_1<m_2<\cdots<m_{N-1}$ and for all $1\leq j\leq N-1$,
		$$P_{\alpha}^{\comp m_j}(w)\in\tilde{\Omega}_{c_0}^{n_j}\setminus\Lambda_\alpha^{n_j+\mathfrak{q}_3},\ n_j=n+p\geq n+1+\chi$$ and
		$$\left\{P_{\alpha}^{\comp t}(w):\ 0\leq t\leq m_{N-1}\right\}\subseteq\Lambda_{\alpha}^{n+1+\chi}.$$
		Thus by Section \ref{s4.2}, we have that there exist a sequence
		$\left\{j_s:=j_0+s{\bf t}\right\}_{s=0}^{\lfloor\frac{N-1-j_0}{\bf t}\rfloor}$ ($j_0\asymp_{\mathfrak{q}_3}1$ and ${\bf t}\asymp_{\mathfrak{q}_3}1$) and Euclidean balls ${\bf B}_{j_s}$ ($0\leq s\leq\lfloor\frac{N-1-j_0}{\bf t}\rfloor$) such that
		\begin{itemize}
			\item[(a7)] ${\rm diam}({\bf B}_{j_0})\leq l_z\cdot{\bf\eta}$,
			\item[(b7)] ${\rm diam}({\bf B}_{j_{s+1}})\leq {\bf\eta}\cdot{\rm diam}({\bf B}_{j_s})$ for all $0\leq s\leq\lfloor\frac{N-1-j_0}{\bf t}\rfloor-1$,
			\item[(c7)] ${\rm dist}({\bf B}_{j_s},w)\leq\frac{c-1}{2}\cdot{\rm diam}({\bf B}_{j_s})$ for all $0\leq s\leq\lfloor\frac{N-1-j_0}{\bf t}\rfloor$,
			\item[(d7)] $P_{\alpha}^{\comp t}({\bf B}_{j_s})\subseteq\Lambda_{\alpha}^{n+1}$
			for all $0\leq s\leq\lfloor\frac{N-1-j_0}{\bf t}\rfloor$ and $0\leq t\leq m_{j_s}$,
			\item[(e7)]
			$P_{\alpha}^{\comp m_{j_s}}({\bf B}_{j_s})\subseteq\hat{U}_0^{n+p+\mathfrak{l}}\cap\tilde{\Omega}_{c_0}^{n+p}\setminus\Lambda_{\alpha}^{n+p+\mathfrak{q}_3}$ for all $0\leq s\leq\lfloor\frac{N-1-j_0}{\bf t}\rfloor$.
		\end{itemize}
		By (d7) and (e7), for all $0\leq s\leq\lfloor\frac{N-1-j_0}{\bf t}\rfloor$, the orbit of
		every point of ${\bf B}_{j_s}$ stays in $\Lambda_{\alpha}^{n+1}$ all the time or visits 
		$\hat{U}_0^{n+p+\mathfrak{l}}$ before leaving $\Lambda_{\alpha}^{n+1}$, and hence
		${\bf B}_{j_s}\subseteq E_p^{\mathfrak{l}}(\hat{\Delta}_{\alpha}^{n+p+1})$.
		This implies ${\bf B}_{j_s}\cap \mathcal{R}^{(1)}=\emptyset$.
		Observe ${\lfloor\frac{N-1-j_0}{\bf t}\rfloor}\gg_{\mathfrak{q}_3}1$. Thus by Lemma \ref{L1} and (a7)(b7)(c7), we have
		$${\rm dens}_{S_z}\mathcal{R}^{(1)}\leq \lambda^{\lfloor\frac{N-1-j_0}{\bf t}\rfloor}\leq\lambda_1^N,$$
		where $0<\lambda<1$ is a constant only depending on $c$ and $\eta$ (determined by $\mathfrak{q}_3$), and $\lambda_1=\lambda^{\frac{1}{2\bf t}}$ (Due to $N-1\gg_{\mathfrak{q}_3}1$).

		\noindent\emph{The proof of (3)}: For all $w\in \mathcal{R}^{(2)}$,
		we have that
		for all $w\in S_z\setminus E_p^{\mathfrak{l}}(\hat{\Delta}_{\alpha}^{n+p+1})$ and the orbit of $w$ visits $\hat{U}_{0,p}^{n}$ $0$ times before leaving $\Lambda_{\alpha}^{n+1+\chi}$.
		
		Since $w\in S_z\setminus E_p^{\mathfrak{l}}(\hat{\Delta}_{\alpha}^{n+p+1})$, by Lemma \ref{l20257192} we have $w\in\mathcal{J}(\tilde{\beta}_I)\subseteq\Lambda_{\alpha}^{n+p}$ and 
		$P_{\alpha}^{\comp m}(w)\not\in\Lambda_{\alpha}^{n+1}\cup\overline{\Delta_{\alpha}}$ for some positive integer $m$.
		By Lemma \ref{l819}
		we have that there exist positive integers $\tilde{m}_1<\tilde{m}_2<\cdots<\tilde{m}_k$ ($k=\lfloor\frac{p-1-2\chi}{\mathfrak{q}_3}\rfloor$) such that for all $1\leq j\leq k$,
		$$P_{\alpha}^{\comp \tilde{m}_j}(w)\in\tilde{\Omega}_{c_0}^{n+p-j\mathfrak{q}_3}\setminus\Lambda_{\alpha}^{n+p-(j-1)\mathfrak{q}_3}$$
		and
		$$\left\{P_{\alpha}^{\comp s}(w):\ 0\leq s\leq \tilde{m}_k\right\}\subseteq\Lambda_{\alpha}^{n+1+2\chi}.$$
		By (\ref{e20260410a}) we have that for all $1\leq j\leq k$,
		$\hat{U}_0^{n+p-j\mathfrak{q}_3+\mathfrak{l}}=\hat{U}_0^{n+p}$.
		Then for all $$j\in\left\{2+\lfloor\frac{2\mathfrak{q}_5}{\mathfrak{q}_3}\rfloor,\ 3+\lfloor\frac{2\mathfrak{q}_5}{\mathfrak{q}_3}\rfloor,\ \cdots,\ \lfloor\frac{p-1-2\chi-2\mathfrak{q}_5}{\mathfrak{q}_3}\rfloor\right\},$$
		we have
		\begin{equation}
			\label{e20260410b1}\hat{U}_0^{n+p-j\mathfrak{q}_3+\mathfrak{l}}\cap\tilde{\Omega}_{c_0}^{n+p-j\mathfrak{q}_3}\setminus\Lambda_{\alpha}^{n+p-(j-1)\mathfrak{q}_3}\subseteq\overline{\hat{U}_0^{n+p}\cap\Omega_{c_0}^{n+1+2\chi+2\mathfrak{q}_5}\setminus\Omega_{c_0}^{n+p-2\mathfrak{q}_5}}\subseteq\hat{U}_{0,p}^{n}.
		\end{equation}
		We denote by $\tilde{k}$ the number of elements in the set $\left\{2+\lfloor\frac{2\mathfrak{q}_5}{\mathfrak{q}_3}\rfloor,\ 3+\lfloor\frac{2\mathfrak{q}_5}{\mathfrak{q}_3}\rfloor,\ \cdots,\ \lfloor\frac{p-1-2\chi-2\mathfrak{q}_5}{\mathfrak{q}_3}\rfloor\right\}$.
		For all $1\leq j\leq\tilde{k}$, we set
		$m_j:=\tilde{m}_{1+\lfloor\frac{2\mathfrak{q}_5}{\mathfrak{q}_3}\rfloor+j}$. Then
		$p>\tilde{k}\gg_{\mathfrak{q}_3}1$ (Due to $p\gg\mathfrak{q}_5$ and $p\gg_{\mathfrak{q}_3}\mathfrak{q}_3$),
		$m_1<m_2<\cdots<m_{\tilde{k}}$ and for all $1\leq j\leq\tilde{k}$,
		$$P_{\alpha}^{\comp m_j}(w)\in\tilde{\Omega}_{c_0}^{n_j}\setminus\Lambda_\alpha^{n_j+\mathfrak{q}_3},\ n_j=n+p-\left(1+\lfloor\frac{2\mathfrak{q}_5}{\mathfrak{q}_3}\rfloor+j\right)\mathfrak{q}_3\geq n+1+2\chi$$ and
		$$\left\{P_{\alpha}^{\comp t}(w):\ 0\leq t\leq m_{\tilde{k}}\right\}\subseteq\Lambda_{\alpha}^{n+1+2\chi}.$$
		Observe that
		$\hat{\Delta}_{\alpha}^{(n+\chi)+(s_0-\chi)+1}\not=\hat{\Delta}_{\alpha}^{(n+\chi)+(s_0-\chi)}$, $s_0-\chi\geq p-\chi\gg\mathfrak{q}_4$ and $n+\chi\geq2\mathfrak{q}_4$.
		Thus by Section \ref{s4.2}, we have that there exist a sequence
		$\left\{j_s:=j_0+s{\bf t}\right\}_{s=0}^{\lfloor\frac{\tilde{k}-j_0}{\bf t}\rfloor}$ ($j_0\asymp_{\mathfrak{q}_3}1$ and ${\bf t}\asymp_{\mathfrak{q}_3}1$)
		and Euclidean balls ${\bf B}_{j_s}$ ($0\leq s\leq\lfloor\frac{\tilde{k}-j_0}{\bf t}\rfloor$) such that
		\begin{itemize}
			\item[(a8)] ${\rm diam}({\bf B}_{j_0})\leq l_z\cdot{\bf\eta}$,
			\item[(b8)] ${\rm diam}({\bf B}_{j_{s+1}})\leq {\bf\eta}\cdot{\rm diam}({\bf B}_{j_s})$ for all $0\leq s\leq\lfloor\frac{\tilde{k}-j_0}{\bf t}\rfloor-1$,
			\item[(c8)] ${\rm dist}({\bf B}_{j_s},w)\leq\frac{c-1}{2}\cdot{\rm diam}({\bf B}_{j_s})$ for all $0\leq s\leq\lfloor\frac{\tilde{k}-j_0}{\bf t}\rfloor$,
			\item[(d8)] $P_{\alpha}^{\comp t}({\bf B}_{j_s})\subseteq\Lambda_{\alpha}^{n+1+\chi}$ for all $0\leq s\leq\lfloor\frac{\tilde{k}-j_0}{\bf t}\rfloor$ and $0\leq t\leq m_{j_s}$,
			\item[(e8)]
			$P_{\alpha}^{\comp m_{j_s}}({\bf B}_{j_s})\subseteq\hat{U}_0^{n+p-(1+\lfloor\frac{2\mathfrak{q}_5}{\mathfrak{q}_3}\rfloor+j_s){\mathfrak{q}_3}+\mathfrak{l}}\cap\tilde{\Omega}_{c_0}^{n+p-(1+\lfloor\frac{2\mathfrak{q}_5}{\mathfrak{q}_3}\rfloor+j_s)\mathfrak{q}_3}\setminus\Lambda_{\alpha}^{n+p-(\lfloor\frac{2\mathfrak{q}_5}{\mathfrak{q}_3}\rfloor+j_s)\mathfrak{q}_3}$ for all $0\leq s\leq\lfloor\frac{\tilde{k}-j_0}{\bf t}\rfloor$.
		\end{itemize}
		By (d8), (e8) and (\ref{e20260410b1}), we have that
		for all $1\leq s\leq\lfloor\frac{\tilde{k}-j_0}{\bf t}\rfloor$ and all $w'\in{\bf B}_{j_s}$, $w'\in E_p^{\mathfrak{l}}(\hat{\Delta}_{\alpha}^{n+p+1})$ or the orbit of
		$w'$ visits
		$\hat{U}_{0,p}^{n}$ at least $1$ times before leaving $\Lambda_{\alpha}^{n+1+\chi}$. Thus for all $1\leq s\leq\lfloor\frac{\tilde{k}-j_0}{\bf t}\rfloor$, sets
		$\mathcal{R}_2$ and
		${\bf B}_{j_s}$
		are disjoint. Observe $\lfloor\frac{\tilde{k}-j_0}{\bf t}\rfloor\gg_{\mathfrak{q}_3}1$.
		Then by Lemma \ref{L1} and (a8)(b8)(c8), we have
		$${\rm dens}_{S_z}\mathcal{R}_2\leq \lambda^{\lfloor\frac{\tilde{k}-j_0}{\bf t}\rfloor}.$$
		Since $p\gg\mathfrak{q}_5$ and $p\gg_{N}\mathfrak{q}_3$, we have that $\tilde{k}\geq\frac{p}{2{\mathfrak{q}_3}}$ and $\tilde{k}\gg_{\mathfrak{q}_3}1$, and hence $\lfloor\frac{\tilde{k}-j_0}{\bf t}\rfloor\geq\frac{p}{4{\bf t}\mathfrak{q}_3}$.
		We take $\lambda_2=\lambda^{\frac{1}{4{\bf t}\mathfrak{q}_3}}$. Then
		$${\rm dens}_{S_z}\mathcal{R}_2\leq \lambda^{\lfloor\frac{\tilde{k}-j_0}{\bf t}\rfloor}\leq\lambda_2^p.$$

		\noindent\emph{The proof of (4)}: For any ${\bf n}=(m,t)\in\mathcal{N}$ and
		$w\in\tilde{E}_{{\bf n}}$, we have that $w\in S_z\setminus E_p^{\mathfrak{l}}(\hat{\Delta}_{\alpha}^{n+p+1})$ and the orbit of $w$ visits $\hat{U}_{0,p}^{n}\cup\bigcup_{1\leq j<t}\mathcal{S}_{j}(\hat{U}_{0,p}^{n})$ $m$ times and doesn't visit ${\rm int}(\mathcal{S}_t(\hat{U}_{0,p}^{n}))$ before leaving $\Lambda_{\alpha}^{n+1+\chi}$.
		
		Let $P_{\alpha}^{\comp k_m}(w)$ correspond to the $m$-th visiting $\hat{U}_{0,p}^{n}\cup\bigcup_{1\leq j<t}\mathcal{S}_{j}(\hat{U}_{0,p}^{n})$ of the orbit of $w$.
		Then the orbit of $w$ doesn't visit $\hat{U}_{0,p}^{n}\cup\bigcup_{1\leq j<t}\mathcal{S}_{j}(\hat{U}_{0,p}^{n})$ after $P_{\alpha}^{\comp k_m}(w)$ before leaving $\Lambda_{\alpha}^{n+1+2\chi}$. Observe that $w\in S_z\setminus E_p^{\mathfrak{l}}(\hat{\Delta}_{\alpha}^{n+p+1})$ and $\{P_{\alpha}^{\comp j}(w): 0\leq j\leq k_m\}\subseteq\Lambda_{\alpha}^{n+1+\chi}$.
		Then $P_{\alpha}^{\comp(k_m+1)}(w)\in\mathcal{J}(\tilde{\beta}_I)$ $(\subseteq\Lambda_{\alpha}^{n+p})$ for some $I\in\bigcup\limits_{j=0}^{\mathfrak{l}-1}\mathfrak{D}_{n+p+j}(\alpha)$ and $P_{\alpha}^{\comp k_0}(P_{\alpha}^{\comp(k_m+1)}(w))\not\in\Lambda_{\alpha}^{n+1}\cup\overline{\Delta_{\alpha}}$ for some positive integer $k_0$.
		By Lemma \ref{l819} there exist
		positive integers $k_m+1<\tilde{m}_{1}<\tilde{m}_{2}<\cdots<\tilde{m}_{k}$ ($k=\lfloor\frac{p-1-2\chi}{\mathfrak{q}_3}\rfloor$) such that
		\begin{equation}
			\label{e20260421d1}P_{\alpha}^{\comp \tilde{m}_j}(w)\in\tilde{\Omega}_{c_0}^{n+p-j\mathfrak{q}_3}\setminus\Lambda_{\alpha}^{n+p-(j-1)\mathfrak{q}_3},\ \forall j\in\{1,2,\cdots,k\}
		\end{equation}
		and
		\begin{equation}
			\label{e20260421d2}\left\{P_{\alpha}^{\comp s}(w):\ k_m+1\leq s\leq \tilde{m}_{k}\right\}\subseteq\Lambda_{\alpha}^{n+1+2\chi}.
		\end{equation}
		Since $\{P_{\alpha}^{\comp j}(w): 0\leq j\leq k_m\}\subseteq\Lambda_{\alpha}^{n+1+\chi}$, we have
		\begin{equation}
			\label{e20260421d3}\left\{P_{\alpha}^{\comp s}(w):\ 0\leq s\leq \tilde{m}_{k}\right\}\subseteq\Lambda_{\alpha}^{n+1+\chi}.
		\end{equation}
		It follows from (\ref{e20260410b1}) in the proof of $(3)$ that for all $j\in\left\{2+\lfloor\frac{2\mathfrak{q}_5}{\mathfrak{q}_3}\rfloor,\ 3+\lfloor\frac{2\mathfrak{q}_5}{\mathfrak{q}_3}\rfloor,\ \cdots,\ \lfloor\frac{p-1-2\chi-2\mathfrak{q}_5}{\mathfrak{q}_3}\rfloor\right\}$, we have
		\begin{equation}
			\label{e202601051}\hat{U}_0^{n+p-j\mathfrak{q}_3+\mathfrak{l}}\cap\tilde{\Omega}_{c_0}^{n+p-j\mathfrak{q}_3}\setminus\Lambda_{\alpha}^{n+p-(j-1)\mathfrak{q}_3}\subseteq\hat{U}_{0,p}^{n}.
		\end{equation}
		As that in the proof of $(3)$, we denote by $\tilde{k}$ the number of elements in the set $$\left\{2+\lfloor\frac{2\mathfrak{q}_5}{\mathfrak{q}_3}\rfloor,\ 3+\lfloor\frac{2\mathfrak{q}_5}{\mathfrak{q}_3}\rfloor,\ \cdots,\ \lfloor\frac{p-1-2\chi-2\mathfrak{q}_5}{\mathfrak{q}_3}\rfloor\right\},$$
		we set
		$m_j:=\tilde{m}_{1+\lfloor\frac{2\mathfrak{q}_5}{\mathfrak{q}_3}\rfloor+j}$ for all $1\leq j\leq\tilde{k}$ and $m_0:=0$, and by Section \ref{s4.2}, (\ref{e20260421d1}) and (\ref{e20260421d3}), we can obtain that there exist a sequence
		$\left\{j_s:=j_0+s{\bf t}\right\}_{s=0}^{\lfloor\frac{\tilde{k}-j_0}{\bf t}\rfloor}$ ($j_0\asymp_{\mathfrak{q}_3}1$ and ${\bf t}\asymp_{\mathfrak{q}_3}1$)
		and Euclidean balls ${\bf B}_{j_s}$ ($0\leq s\leq\lfloor\frac{\tilde{k}-j_0}{\bf t}\rfloor$) such that
		\begin{itemize}
			\item[(a9)] ${\rm diam}({\bf B}_{j_0})\leq l_z\cdot{\bf\eta}$,
			\item[(b9)] ${\rm diam}({\bf B}_{j_{s+1}})\leq {\bf\eta}\cdot{\rm diam}({\bf B}_{j_s})$ for all $0\leq s\leq\lfloor\frac{\tilde{k}-j_0}{\bf t}\rfloor-1$,
			\item[(c9)] ${\rm dist}({\bf B}_{j_s},w)\leq\frac{c-1}{2}\cdot{\rm diam}({\bf B}_{j_s})$ for all $0\leq s\leq\lfloor\frac{\tilde{k}-j_0}{\bf t}\rfloor$,
			\item[(d9)] $P_{\alpha}^{\comp t}({\bf B}_{j_s})\subseteq\Lambda_{\alpha}^{n+1}$ for all $0\leq s\leq\lfloor\frac{\tilde{k}-j_0}{\bf t}\rfloor$ and $0\leq t\leq m_{j_s}$,
			\item[(e9)]
			$P_{\alpha}^{\comp m_{j_s}}({\bf B}_{j_s})\subseteq\hat{U}_0^{n+p-(1+\lfloor\frac{2\mathfrak{q}_5}{\mathfrak{q}_3}\rfloor+j_s)\mathfrak{q}_3+\mathfrak{l}}\cap\tilde{\Omega}_{c_0}^{n+p-(1+\lfloor\frac{2\mathfrak{q}_5}{\mathfrak{q}_3}\rfloor+j_s)\mathfrak{q}_3}\setminus\Lambda_{\alpha}^{n+p-(\lfloor\frac{2\mathfrak{q}_5}{\mathfrak{q}_3}\rfloor+j_s)\mathfrak{q}_3}$ for all $0\leq s\leq\lfloor\frac{\tilde{k}-j_0}{\bf t}\rfloor$,
			\item[(f9)]
			for all $0\leq s\leq\lfloor\frac{\tilde{k}-j_0}{\bf t}\rfloor$, $w'\in{\bf B}_{j_s}$, $1\leq t\leq j_s$ and $m_{t-1}\leq l\leq m_t$,
			there exists a curve $\tilde{\gamma}_{j_s}^{(l)}$ connecting $P_{\alpha}^{\comp l}(w')$ and $P_{\alpha}^{\comp l}(w)$ such that
			$l_{\mathbb{C}\setminus\overline{\Delta_{\alpha}}}(\tilde{\gamma}_{j_s}^{(l)})=O_{\mathfrak{q}_3}(\eta^{j_s-t})$
			and 
			$P_{\alpha}^{\comp (m_{j_s}-l)}\comp\tilde{\gamma}_{j_s}^{(l)}=\tilde{\gamma}_{j_s}^{(m_{j_s})}$.
		\end{itemize}
		By (f9), we have that for all $0\leq s\leq\lfloor\frac{\tilde{k}-j_0}{\bf t}\rfloor$, $0\leq l\leq m_{j_s}$ and $w'\in{\bf B}_{j_s}$,
		\begin{equation}
			\label{e20260515a}{\rm dist}_{\mathbb{C}\setminus\overline{\Delta_{\alpha}}}\left(P_{\alpha}^{\comp l}(w'),P_{\alpha}^{\comp l}(w)\right)=O_{\mathfrak{q}_3}(1).
		\end{equation}
		It follows from Claim $1$,
		(\ref{e20260421d2}) and (\ref{e20260515a}) that
		\begin{itemize}
			\item[(d'9)] $P_{\alpha}^{\comp t}({\bf B}_{j_s})\subseteq\Lambda_{\alpha}^{n+1+\chi}$ for all $0\leq s\leq\lfloor\frac{\tilde{k}-j_0}{\bf t}\rfloor$ and $k_m+1\leq t\leq m_{j_s}$.
		\end{itemize}
		
		Now we claim that
		for all $\lfloor\frac{\tilde{k}-j_0}{2{\bf t}}\rfloor\leq s\leq\lfloor\frac{\tilde{k}-j_0}{\bf t}\rfloor$ and $w'\in {\bf B}_{j_s}\cap S_z$, 
		we have that if $w'\in S_z\setminus E_p^{\mathfrak{l}}(\hat{\Delta}_{\alpha}^{n+p+1})$, then one of the following cases holds:
		\begin{itemize}
			\item the orbit of $w'$ visits $\mathcal{S}_t(\hat{U}_{0,p}^{n})$ before leaving $\Lambda_{\alpha}^{n+1+\chi}$;
			\item the orbit of $w'$ leaves $\Lambda_{\alpha}^{n+1+\chi}$ before visiting $\hat{U}_{0,p}^{n}\cup\bigcup_{1\leq j<t}\mathcal{S}_{j}(\hat{U}_{0,p}^{n})$ $m$ times;
			\item the orbit of $w'$ visits $\hat{U}_{0,p}^{n}\cup\bigcup_{1\leq j<t}\mathcal{S}_{j}(\hat{U}_{0,p}^{n})$ at least $m+1$ times before leaving $\Lambda_{\alpha}^{n+1+\chi}$.
		\end{itemize}
		Indeed, for all $\lfloor\frac{\tilde{k}-j_0}{2{\bf t}}\rfloor\leq s\leq\lfloor\frac{\tilde{k}-j_0}{\bf t}\rfloor$ and $w'\in{\bf B}_{j_s}\cap S_z$ with $w'\in S_z\setminus E_p^{\mathfrak{l}}(\hat{\Delta}_{\alpha}^{n+p+1})$, by (f9) we have that for all $1\leq t\leq j_s$ and $m_{t-1}\leq l\leq m_t$,
		$$l_{\mathbb{C}\setminus\overline{\Delta_{\alpha}}}(\tilde{\gamma}_{j_s}^{(l)})=O_{\mathfrak{q}_3}(\eta^{j_s-t})\ {\rm and}\ P_{\alpha}^{\comp (m_{j_s}-l)}\comp\tilde{\gamma}_{j_s}^{(l)}=\tilde{\gamma}_{j_s}^{(m_{j_s})}.$$
		Observe $0<\eta<1$ is determined by $\mathfrak{q}_3$. Then $\frac{1}{-\ln\eta}\asymp_{\mathfrak{q}_3}1$ and hence $\frac{1}{-\ln\eta}\asymp_{N}1$.
		Since $s\geq\lfloor\frac{\tilde{k}-j_0}{2{\bf t}}\rfloor\geq\frac{p}{4{\bf t}\mathfrak{q}_3}\gg_{N}1$, we have
		$-j_s\ln\eta\gg_N1$ and hence  $\frac{1}{\eta^{j_s}}\gg_N1$. By combining $\frac{1}{\kappa(\frac{\Delta}{N})}\preceq_N1$, it follows that
		\begin{equation}
			\label{e4.1}l_{\mathbb{C}\setminus\overline{\Delta_{\alpha}}}(\tilde{\gamma}_{j_s}^{(k_m+1)})=O_{\mathfrak{q}_3}(\eta^{j_s-1})=O_{N}(\eta^{j_s})=\eta^{j_s}O_{N}(1)<\kappa(\frac{\Delta}{N}).
		\end{equation}
		For all $1\leq j\leq m$, we let $P_{\alpha}^{\comp k_j}(w)$ correspond to the $j$-th visiting $\hat{U}_{0,p}^{n}\cup\bigcup_{1\leq j<t}\mathcal{S}_{j}(\hat{U}_{0,p}^{n})$ of the orbit of $w$.
		It follows from $w\in\tilde{E}_{(m,t)}$ that
		\begin{equation}
			\label{e20260414a1}P_{\alpha}^{\comp (k_j+1)}(w)\not\in\hat{\Delta}_{\alpha}^{n+p+\mathfrak{l}}.
		\end{equation}
		By the Schwarz lemma and (\ref{e4.1}) we have
		\begin{equation}
			\label{e20260417a1}
			l_{\mathbb{C}\setminus\overline{\Delta_{\alpha}}}(\tilde{\gamma}_{j_s}^{(l)})<\kappa(\frac{\Delta}{N}),\ 0\leq l\leq k_m,
		\end{equation}
		in particular, we have that for all $1\leq j\leq m$,
		\begin{equation}
			\label{e20260414a2}
			l_{\mathbb{C}\setminus\overline{\Delta_{\alpha}}}(\tilde{\gamma}_{j_s}^{(k_j+1)})<\kappa(\frac{\Delta}{N}).
		\end{equation}
		Together with (\ref{e20260414a1}), (\ref{e20260414a2}),
		$P_{\alpha}\comp\tilde{\gamma}_{j_s}^{(k_j)}=\tilde{\gamma}_{j_s}^{(k_j+1)}$, $P_{\alpha}^{\comp k_j}(w)\in\hat{U}_{0,p}^{n}\cup\bigcup_{1\leq j<t}\mathcal{S}_{j}(\hat{U}_{0,p}^{n})$ and the fact that $\tilde{\gamma}_{j_s}^{(k_j)}$ connects $P_{\alpha}^{\comp k_j}(w)$ and $P_{\alpha}^{\comp k_j}(w')$, Lemma \ref{l12615} gives that for all $1\leq j\leq m$,
		\begin{equation}
			\label{e20260420d}P_{\alpha}^{\comp k_j}(w')\in\hat{U}_{0,p}^{n}\cup\bigcup_{1\leq j< t}\mathcal{S}_{j}(\hat{U}_{0,p}^{n})\cup{\rm int}(\mathcal{S}_{t}(\hat{U}_{0,p}^{n})).
		\end{equation}
		By $w'\in S_z\setminus E_p^{\mathfrak{l}}(\hat{\Delta}_{\alpha}^{n+p+1})$, we have that the orbit of $w'$ doesn't stay in $\Lambda_{\alpha}^{n+1+\chi}$ all the time. Now we assume that the orbit of $w'$ doesn't visit ${\rm int}(\mathcal{S}_t(\hat{U}_{0,p}^{n}))$ before leaving $\Lambda_{\alpha}^{n+1+\chi}$ and doesn't leave $\Lambda_{\alpha}^{n+1+\chi}$ before visiting $\hat{U}_{0,p}^{n}\cup\bigcup_{1\leq j<t}\mathcal{S}_{j}(\hat{U}_{0,p}^{n})$ $m$ times. So to prove the claim, we need only to 
		prove that the orbit of $w'$ visits $\hat{U}_{0,p}^{n}\cup\bigcup_{1\leq j<t}\mathcal{S}_{j}(\hat{U}_{0,p}^{n})$ at least $m+1$ times before leaving $\Lambda_{\alpha}^{n+1+\chi}$.
		For all $1\leq j\leq m$,
		we let $P_{\alpha}^{\comp k'_j}(w)$ correspond to the $j$-th visiting $\hat{U}_{0,p}^{n}\cup\bigcup_{1\leq j< t}\mathcal{S}_{j}(\hat{U}_{0,p}^{n})\cup{\rm int}(\mathcal{S}_{t}(\hat{U}_{0,p}^{n}))$ of the orbit of $w'$.
		By the above assumption we have 
		\begin{equation}
			\label{e20260420b}P_{\alpha}^{\comp k_j'}(w')\in\hat{U}_{0,p}^{n}\cup\bigcup_{1\leq j<t}\mathcal{S}_{j}(\hat{U}_{0,p}^{n}),\ \forall j\in \{1,\cdots,m\}
		\end{equation}
		and further we will prove that 
		\begin{equation}
			\label{e20260420a}k'_j=k_j,\ \forall j\in \{1,\cdots,m\}.
		\end{equation}
		In fact, if not, then it follows from (\ref{e20260420d}) that $k'_{j_0}<k_{j_0}$ for some $j_0\in\{1,\cdots,m\}$, and we let $j_0$ be the smallest one satisfying the above inequality. The smallestness gives 
		$k'_{j_0}<k_1$ or $k_{j_0-1}<k'_{j_0}<k_{j_0}$,
		and hence 
		\begin{equation}
			\label{e20260417a3}k'_{j_0}\not\in\{k_{1},k_{2},\cdots,k_{m}\}.
		\end{equation}
		Since the orbit of $w'$  doesn't leave $\Lambda_{\alpha}^{n+1+\chi}$ before visiting $\hat{U}_{0,p}^{n}\cup\bigcup_{1\leq j<t}\mathcal{S}_{j}(\hat{U}_{0,p}^{n})$ $m$ times, we have 
		\begin{equation}
			\label{e20260420c}\{P_{\alpha}^{\comp t}(w'):\ 0\leq t\leq k_m'\}\subseteq\Lambda_{\alpha}^{n+1+\chi},
		\end{equation}
		and then by $w'\in S_z\setminus E_p^{\mathfrak{l}}(\hat{\Delta}_{\alpha}^{n+p+1})$,
		\begin{equation}
			\label{e20260417a2}P_{\alpha}^{\comp (k'_{j_0}+1)}(w')\not\in\hat{\Delta}_{\alpha}^{n+p+\mathfrak{l}}.
		\end{equation}
		Thus together with (\ref{e20260417a2}), 
		$l_{\mathbb{C}\setminus\overline{\Delta_{\alpha}}}(\tilde{\gamma}_{j_s}^{(k_{j_0}'+1)})<\kappa(\frac{\Delta}{N})$ (Due to (\ref{e20260417a1})), 
		$P_{\alpha}\comp\tilde{\gamma}_{j_s}^{(k_{j_0}')}=\tilde{\gamma}_{j_s}^{(k_{j_0}'+1)}$, $P_{\alpha}^{\comp k_{j_0}'}(w')\in\hat{U}_{0,p}^{n}\cup\bigcup_{1\leq j<t}\mathcal{S}_{j}(\hat{U}_{0,p}^{n})$ (Due to (\ref{e20260420b})) and the fact that $\tilde{\gamma}_{j_s}^{(k_{j_0}')}$ connects $P_{\alpha}^{\comp k_{j_0}'}(w)$ and $P_{\alpha}^{\comp k_{j_0}'}(w')$, Lemma \ref{l12615} gives that 
		$$P_{\alpha}^{\comp k_{j_0}'}(w)\in\hat{U}_{0,p}^{n}\cup\bigcup_{1\leq j< t}\mathcal{S}_{j}(\hat{U}_{0,p}^{n})\cup{\rm int}(\mathcal{S}_{t}(\hat{U}_{0,p}^{n})).$$
		Together with $k_{j_0}'<k_{j_0}\leq k_m$, the above assumption and definitions of $k_j$, $1\leq j\leq m$, it gives that
		$k_{j_0}'\in\{1,2,\cdots,k_m\}$, which contradicts (\ref{e20260417a3}).
		Thus $k'_j=k_j$ for all $1\leq j\leq m$.
		By (d'9), (e9) and (\ref{e202601051}) we have that the orbit of $w'$ visits $\hat{U}_{0,p}^{n}\cup\bigcup_{1\leq j<t}\mathcal{S}_{j}(\hat{U}_{0,p}^{n})$ at least $1$ times after $P_{\alpha}^{\comp k_m}(w')$ before leaving $\Lambda_{\alpha}^{n+1+\chi}$. 
		Then it, together with (\ref{e20260420b}), (\ref{e20260420a}) and (\ref{e20260420c}), gives that the orbit of $w'$ visits $\hat{U}_{0,p}^{n}\cup\bigcup_{1\leq j<t}\mathcal{S}_{j}(\hat{U}_{0,p}^{n})$ at least $m+1$ times before leaving $\Lambda_{\alpha}^{n+1+\chi}$. This completes the proof of the claim.
		
		The above claim implies that
		$\tilde{E}_{{\bf n}}\cap {\bf B}_{j_s}=\emptyset$ for all $\lfloor\frac{\tilde{k}-j_0}{2{\bf t}}\rfloor\leq s\leq\lfloor\frac{\tilde{k}-j_0}{\bf t}\rfloor$.
		By Lemma \ref{L1} and (a9)(b9)(c9), we have that
		for all ${\bf n}=(m,t)\in\mathcal{N}$,
		$${\rm area}(\tilde{E}_{{\bf n}})\leq\lambda^{\lfloor\frac{\tilde{k}-j_0}{2{\bf t}}\rfloor}\cdot{\rm area}(S_z)\leq\lambda_2^{\frac{p}{3}}\cdot{\rm area}(S_z).$$
		This completes the proof of ($4$).
		
		\vspace{0.2cm}
		Now, the lemma follows easily from (1), (2), (3) and (4). In fact, by (1), (2), (3) and (4), we have
		\begin{align*}
			{\rm area}(S_z\setminus E_p^{\mathfrak{l}}(\hat{\Delta}_{\alpha}^{n+p+1}))&\leq{\rm area}(\mathcal{R}^{(1)})+{\rm area}(\mathcal{R}^{(2)})+\sum_{{\bf n}\in\mathcal{N}}{\rm area}(\tilde{E}_{{\bf n}})\\
			&\leq\lambda_1^N\cdot{\rm area}(S_z)+\lambda_2^p\cdot{\rm area}(S_z)+(N-1)^2\lambda_2^{\frac{p}{3}}\cdot{\rm area}(S_z)\\
			&=(\lambda_1^N+\lambda_2^p+(N-1)^2\lambda_2^{\frac{p}{3}})\cdot{\rm area}(S_z)\\
			&\leq N^2\sqrt[3]{\max\{\lambda_1,\lambda_2\}}^N\cdot{\rm area}(S_z)\ {\rm(Due\ to\ \emph{p}\geq\emph{N})}\\
			&\leq\sqrt[4]{\max\{\lambda_1,\lambda_2\}}^N\cdot{\rm area}(S_z).\ {\rm(Due\ to\ \emph{N}\gg_{\mathfrak{q}_3}1)}
		\end{align*}
		We take
		$\bm{\lambda}=\sqrt[4]{\max\{\lambda_1,\lambda_2\}}$. Then
		we have
		$${\rm area}(S_z\setminus E_p^{\mathfrak{l}}(\hat{\Delta}_{\alpha}^{n+p+1}))\leq\bm{\lambda}^N\cdot{\rm area}(S_z).$$
		This completes the proof of the lemma.
		
	\end{proof}

	\subsection{Real-filling-in of fjords\label{s891}}
	For all $n\geq2\mathfrak{q}_4$, $1\leq\mathfrak{l}\ll\mathfrak{q}_2$ and $0<q<1$,
	we take $N$ large enough such that $N\gg_{\mathfrak{q}_3}1$, $N\gg\mathfrak{q}_5$ and $\bm{\lambda}^N<q$. We let $\mathfrak{p}$ sufficiently big so that $\mathfrak{p}\geq N$,  $\mathfrak{p}\gg_{N}\mathfrak{q}_3$ and
	\begin{equation}
		\label{e0418}\lambda_2^\mathfrak{p}<q^2,
	\end{equation}
	where $\lambda_2$ are the same as that in (3) in the proof of Lemma \ref{key1}.
	
	We let ${\bf k}$ be a positive integer and $p_s$ ($1\leq s\leq{\bf k}$) be ${\bf k}$ positive integers such that $p_s\geq\mathfrak{p}$ for all $1\leq s\leq{\bf k}$.
	We assume that  $\hat{\Delta}_{\alpha}^{n+j}=\hat{\Delta}_{\alpha}^{n+p_1}$ for all $1\leq j\leq p_1$,
	$\hat{\Delta}_{\alpha}^{n+\sum_{s=1}^{t-1}p_s+j}=\hat{\Delta}_{\alpha}^{n+\sum_{s=1}^{t}p_s}$ for $2\leq t\leq {\bf k}$ and $\mathfrak{l}\leq j\leq p_t$,
	$\hat{\Delta}_{\alpha}^{n+\sum_{s=1}^{\bf k}p_s+\mathfrak{l}}=\overline{\Delta_{\alpha}}$ and $\hat{\Delta}_{\alpha}^{n+\sum_{s=1}^{\bf k}p_s+\mathfrak{l}}\not=\hat{\Delta}_{\alpha}^{n+\sum_{s=1}^{\bf k}p_s}$. By Lemma \ref{key1} we have that
	for all $1\leq t\leq {\bf k}$ and $I\in\bigcup\limits_{j=0}^{\mathfrak{l}-1}\mathfrak{D}_{n+\sum_{s=1}^tp_s+j}$, if $\mathcal{J}(\beta_I)\not=\emptyset$, then for all $z\in\mathcal{J}(\beta_I)$,
	\begin{equation}
		\label{e1071}{\rm dens}_{S_z}S_z\setminus E_{p_t}^\mathfrak{l}(\hat{\Delta}_{\alpha}^{n+\sum_{s=1}^tp_s+1})\leq\bm{\lambda}^N<q.
	\end{equation}
	The section is devoted to proving the following theorem.
	\begin{theorem}
		\label{key5}For all $1\leq t\leq {\bf k}$ and $I\in\bigcup\limits_{j=0}^{\mathfrak{l}-1}\mathfrak{D}_{n+\sum_{s=1}^tp_s+j}$, if $\mathcal{J}(\beta_I)\not=\emptyset$, then for all $z\in\mathcal{J}(\beta_I)$, the filled-in Julia set $K(P_{\alpha})$ of $P_{\alpha}$ satisfies
		$${\rm dens}_{S_z}S_z\setminus K(P_{\alpha})\leq\frac{1}{2c^3},$$
		where $c>1$ is a universal constant.
	\end{theorem}
	To prove Theorem \ref{key5}, we need the following three lemmas.
	
	\begin{lemma}
		\label{l7271}For all $c>1$ and $0<q<\frac{1}{12c^3}$, we
		let $\{b_m\}_{m=1}^{\infty}$ be a sequence of positive real numbers such that
		$b_1=q$
		and for all $m\geq2$,
		$$b_m=q+c\left(cb_{m-1}(q+cb_{m-1})+qb_{m-1}\right).$$
		Then we have
		\begin{itemize}
			\item[{\rm(a11)}]
			$b_m\leq b_{m+1}$ for all $m\geq1${\rm;}
			\item[{\rm(b11)}]
			$b_m\leq\frac{1}{2c^3}$ for all $m\geq1$.
		\end{itemize}
	\end{lemma}
	\begin{proof}
		(a11):
		Firstly, we have $$b_2=q+c\left(cb_{1}(q+cb_{1})+qb_{1}\right)>q=b_1.$$
		We define $$f(x)=q+c\left(cx(q+cx)+qx\right),\ x>0.$$
		Evidently, $f(x)$ is strictly increasing and $b_m=f(b_{m-1})$.
		
		Applying $f(x)$ to two sides of $b_2>b_1$, we have
		$b_3>b_2$.
		In the same way, we have
		$b_m>b_{m-1}$ for all $m\geq2$.
		Thus (a11) holds.
		
		\vspace{0.2cm}
		\noindent (b11): To prove (b11), we need only to prove that
		for all $m\geq1$, $b_m<\frac{1-c(c+1)q}{2c^3}$. We prove it by induction. Firstly,
		it follows from $c>1$ and $q<\frac{1}{12c^3}$ that
		$b_1=q<\frac{1-c(c+1)q}{2c^3}$.
		We suppose  $b_{k-1}<\frac{1-c(c+1)q}{2c^3}$, where  $k\geq2$.
		Then 
		\begin{align*}
			b_k&=q+c(cb_{k-1}(q+cb_{k-1})+qb_{k-1})\\
			&=q+c(c+1)qb_{k-1}+c^3b_{k-1}^2\\
			&\leq q+c(c+1)qb_{k-1}+\frac{(1-c(c+1)q)^2}{4c^3}\\
			&\leq q+c(c+1)qb_{k-1}+\frac{(1-c(c+1)q)^2}{2c^3}-q
		\end{align*}
		and hence
		\begin{align*}
			b_k-\frac{1-c(c+1)q}{2c^3}\leq c(c+1)q\left(b_{k-1}-\frac{1-c(c+1)q}{2c^3}\right)<0,
		\end{align*}
		that is
		\begin{align*}
			b_k<\frac{1-c(c+1)q}{2c^3}.
		\end{align*}
		Thus for all $m\geq1$,
		$$b_m<\frac{1-c(c+1)q}{2c^3}<\frac{1}{2c^3}.$$
	\end{proof}
	
	\begin{lemma}
		\label{l7262}For all $0<b<1$ and all positive integers $2\leq t\leq{\bf k}$, we assume that for all $m\in\{{\bf k}, {\bf k}-1, \cdots, t\}$ and
		$I\in\bigcup\limits_{j=0}^{\mathfrak{l}-1}\mathfrak{D}_{n+\sum_{s=1}^mp_s+j}$, if $\mathcal{J}(\beta_I)\not=\emptyset$, then for all $z\in\mathcal{J}(\beta_I)$, 
		\begin{equation*}
			{\rm dens}_{S_z}S_z\setminus E_{\sum_{s=t-1}^{\bf k}p_s+\chi}^{\mathfrak{l}}(\hat{\Delta}_{\alpha}^{n+\sum_{s=1}^{\bf k}p_s+1})\leq b.
		\end{equation*}
		Then we have that if $\hat{\Delta}_{\alpha}^{n+\sum_{s=1}^{t-1}p_s+\mathfrak{l}}\not=\hat{\Delta}_{\alpha}^{n+\sum_{s=1}^{t-1}p_s}$, then for all $I\in\bigcup\limits_{j=0}^{\mathfrak{l}-1}\mathfrak{D}_{n+\sum_{s=1}^{t-1}p_s+j}$ with $\mathcal{J}(\beta_I)\not=\emptyset$ and $z\in\mathcal{J}(\beta_I)$,
		$${\rm dens}_{S_z}S_z\setminus E_{\sum_{s=t-1}^{\bf k}p_s+\chi}^{\mathfrak{l}}(\hat{\Delta}_{\alpha}^{n+\sum_{s=1}^{\bf k}p_s+1})\leq q+cb,$$
		where $c>1$ is a universal constant and $q$ is the same as that in {\rm(}\ref{e0418}{\rm)} and {\rm(}\ref{e1071}{\rm)}.
	\end{lemma}
	\begin{proof}
		We consider the case: $\hat{\Delta}_{\alpha}^{n+\sum_{s=1}^{t-1}p_s+\mathfrak{l}}\not=\hat{\Delta}_{\alpha}^{n+\sum_{s=1}^{t-1}p_s}$, $I\in\bigcup\limits_{j=0}^{\mathfrak{l}-1}\mathfrak{D}_{n+\sum_{s=1}^{t-1}p_s+j}$ with $\mathcal{J}(\beta_I)\not=\emptyset$ and $z\in\mathcal{J}(\beta_I)$.
		For all $w\in S_z\cap E_{p_{t-1}}^{\mathfrak{l}}(\hat{\Delta}_{\alpha}^{n+\sum_{s=1}^{t-1}p_s+1})\setminus E_{\sum_{s=t-1}^{\bf k}p_s}^{\mathfrak{l}}(\hat{\Delta}_{\alpha}^{n+\sum_{s=1}^{\bf k}p_s+1})$, there exists a nonnegative integer $k$ such that
		$P_{\alpha}^{\comp k}(w)\in\hat{\Delta}_{\alpha}^{n+\sum_{s=1}^{t-1}p_s+\mathfrak{l}}\setminus\hat{\Delta}_{\alpha}^{n+\sum_{s=1}^{\bf k}p_s+\mathfrak{l}}$ and $P_{\alpha}^{\comp j}(w)\in\Lambda_{\alpha}^{n+\sum_{s=1}^{t-2}p_s+1}$ for all $0\leq j\leq k$.
		Then there exists $I'\in\bigcup\limits_{j=t}^{\bf k}\bigcup\limits_{s=0}^{\mathfrak{l}-1}\mathfrak{D}_{n+\sum_{s=1}^{j}p_s+s}$ such that $P_{\alpha}^{\comp k}(w)\in\mathcal{J}(\beta_{I'})$.
		By the assumption we have
		\begin{equation*}
			{\rm dens}_{S_{P_{\alpha}^{\comp k}(w)}}S_{P_{\alpha}^{\comp k}(w)}\setminus E_{\sum_{s=t-1}^{\bf k}p_s+\chi}^{\mathfrak{l}}(\hat{\Delta}_{\alpha}^{n+\sum_{s=1}^{\bf k}p_s+1})\leq b.
		\end{equation*}
		Since $S_{P_{\alpha}^{\comp k}(w)}$ doesn't intersect the postcritical set $\mathcal{O}(P_{\alpha})$, there exists a one-value analytic branch $\psi$ of $P_{\alpha}^{-k}$ on ${\rm int} (S_{P_{\alpha}^{\comp k}(w)})$ such that $\psi(P_{\alpha}^{\comp k}(w))=w$.
		By Lemma \ref{key3} there exists an Euclidean ball $B$ {\rm(}$\subseteq\psi({\rm int}(S_{P_{\alpha}^{\comp k}(w)}))${\rm)} centering at $w$ such that
		\begin{equation}
			\label{e20260426a}{\rm area}\left(B\cap\psi({\rm int}(S_{P_{\alpha}^{\comp k}(w)})\setminus E_{\sum_{s=t-1}^{\bf k}p_s+\chi}^{\mathfrak{l}}(\hat{\Delta}_{\alpha}^{n+\sum_{s=1}^{\bf k}p_s+1}))\right)\leq c_1\cdot b\cdot{\rm area}(B),
		\end{equation}
		where $c_1$ {\rm(}$>1${\rm)} is an absolute constant. 
		It follows from Claim $1$, the Schwarz lemma, Lemma \ref{l20257192} and $P_{\alpha}^{\comp j}(w)\in\Lambda_{\alpha}^{n+\sum_{s=1}^{t-2}p_s+1}$ ($0\leq j\leq k$) that
		$$\psi({\rm int}(S_{P_{\alpha}^{\comp k}(w)})\cap E_{\sum_{s=t-1}^{\bf k}p_s+\chi}^{\mathfrak{l}}(\hat{\Delta}_{\alpha}^{n+\sum_{s=1}^{\bf k}p_s+1}))\subseteq
		E_{\sum_{s=t-1}^{\bf k}p_s+\chi}^{\mathfrak{l}}(\hat{\Delta}_{\alpha}^{n+\sum_{s=1}^{\bf k}p_s+1})$$
		and hence
		$$\psi({\rm int}(S_{P_{\alpha}^{\comp k}(w)})\setminus E_{\sum_{s=t-1}^{\bf k}p_s+\chi}^{\mathfrak{l}}(\hat{\Delta}_{\alpha}^{n+\sum_{s=1}^{\bf k}p_s+1}))\supseteq
		\psi({\rm int}(S_{P_{\alpha}^{\comp k}(w)}))\setminus
		E_{\sum_{s=t-1}^{\bf k}p_s+\chi}^{\mathfrak{l}}(\hat{\Delta}_{\alpha}^{n+\sum_{s=1}^{\bf k}p_s+1}).$$
		By combining (\ref{e20260426a}), it follows that
		$${\rm area}\left(B\cap E_{p_{t-1}}^{\mathfrak{l}}(\hat{\Delta}_{\alpha}^{n+\sum_{s=1}^{t-1}p_s+1})\setminus
		E_{\sum_{s=t-1}^{\bf k}p_s+\chi}^{\mathfrak{l}}(\hat{\Delta}_{\alpha}^{n+\sum_{s=1}^{\bf k}p_s+1})\right)\leq c_1\cdot b\cdot{\rm area}(B).$$
		By Lemmas \ref{l25070716} and \ref{l20257192} we have 
		${\rm diam}(S_z)\asymp{\rm dist}(S_z,\mathcal{O}(P_{\alpha}))$. It, together with $w\in S_z$, gives
		\begin{equation}
			\label{e20260424a}{\rm diam}(B)\leq2{\rm dist}(w,\mathcal{O}(P_{\alpha}))\asymp{\rm diam}(S_z)=\sqrt{2}l_z.
		\end{equation}
		We denote by $\tilde{B}$ the Euclidean ball centering at $w$ with radius $\min\{l_z, \frac{{\rm diam}(B)}{2}\}$.
		(\ref{e20260424a}) gives ${\rm diam}(\tilde{B})\asymp{\rm diam}(B)$, and hence
		$${\rm area}\left(\tilde{B}\cap E_{p_{t-1}}^{\mathfrak{l}}(\hat{\Delta}_{\alpha}^{n+\sum_{s=1}^{t-1}p_s+1})\setminus
		E_{\sum_{s=t-1}^{\bf k}p_s+\chi}^{\mathfrak{l}}(\hat{\Delta}_{\alpha}^{n+\sum_{s=1}^{\bf k}p_s+1})\right)\leq c_2c_1\cdot b\cdot{\rm area}(\tilde{B}),$$
		where $c_2>1$ is a universal constant.
		Together with the above inequality and $0<\frac{{\rm diam}(\tilde{B})}{2}=\min\{l_z, \frac{{\rm diam}(B)}{2}\}\leq l_z$, Lemma \ref{key4} gives
		$${\rm dens}_{S_z}S_z\cap E_{p_{t-1}}^{\mathfrak{l}}(\hat{\Delta}_{\alpha}^{n+\sum_{s=1}^{t-1}p_s+1})\setminus
		E_{\sum_{s=t-1}^{\bf k}p_s+\chi}^{\mathfrak{l}}(\hat{\Delta}_{\alpha}^{n+\sum_{s=1}^{\bf k}p_s+1})\leq c_3c_2c_1b,$$
		where $c_3$ {\rm(}$>1${\rm)} is an absolute constant.
		It follows from (\ref{e1071}) that
		$${\rm dens}_{S_z}S_z\setminus E_{p_{t-1}}^{\mathfrak{l}}(\hat{\Delta}_{\alpha}^{n+\sum_{s=1}^{t-1}p_s+1})\leq q.$$
		Since
		\begin{align*}
			S_z\setminus E_{\sum_{s=t-1}^{\bf k}p_s+\chi}^{\mathfrak{l}}(\hat{\Delta}_{\alpha}^{n+\sum_{s=1}^{\bf k}p_s+1})\subseteq&\left(S_z\setminus E_{p_{t-1}}^{\mathfrak{l}}(\hat{\Delta}_{\alpha}^{n+\sum_{s=1}^{t-1}p_s+1})\right)\\
			&\bigcup\left(S_z\cap E_{p_{t-1}}^{\mathfrak{l}}(\hat{\Delta}_{\alpha}^{n+\sum_{s=1}^{t-1}p_s+1})\setminus E_{\sum_{s=t-1}^{\bf k}p_s+\chi}^{\mathfrak{l}}(\hat{\Delta}_{\alpha}^{n+\sum_{s=1}^{\bf k}p_s+1})\right),
		\end{align*}
		we have
		$${\rm dens}_{S_z}S_z\setminus E_{\sum_{s=t-1}^{\bf k}p_s+\chi}^{\mathfrak{l}}(\hat{\Delta}_{\alpha}^{n+\sum_{s=1}^{\bf k}p_s+1})\leq q+cb.$$
		where $c=c_3c_2c_1$ ($>1$).
	\end{proof}

	\begin{lemma}
		\label{l7272}For all $0<b<1$ and all positive integers $2\leq t\leq{\bf k}$, we assume that for all $m\in\{{\bf k}, {\bf k}-1, \cdots, t\}$ and
		$I\in\bigcup\limits_{j=0}^{\mathfrak{l}-1}\mathfrak{D}_{n+\sum_{s=1}^{m}p_s+j}$, if $\mathcal{J}(\beta_I)\not=\emptyset$, then for all $z\in\mathcal{J}(\beta_I)$,
		\begin{equation*}
			{\rm dens}_{S_z}S_z\setminus E_{\sum_{s=t}^{\bf k}p_s+\chi}^{\mathfrak{l}}(\hat{\Delta}_{\alpha}^{n+\sum_{s=1}^{\bf k}p_s+1})\leq b.
		\end{equation*}
		Then for all $m\in\{{\bf k}, {\bf k}-1, \cdots, t\}$ and
		$I\in\bigcup\limits_{j=0}^{\mathfrak{l}-1}\mathfrak{D}_{n+\sum_{s=1}^{m}p_s+j}$, if $\mathcal{J}(\beta_I)\not=\emptyset$, then for all $z\in\mathcal{J}(\beta_I)$, we have
		$${\rm dens}_{S_z}S_z\setminus E_{\sum_{s=t-1}^{\bf k}p_s+\chi}^{\mathfrak{l}}(\hat{\Delta}_{\alpha}^{n+\sum_{s=1}^{\bf k}p_s+1})
		\leq cb(q+cb)+q^2,
		$$
		where both $c>1$ and $0<q<1$ are the same as that in Lemma \ref{l7262}.
	\end{lemma}
	\begin{proof}
		For all $m\in\{{\bf k}, {\bf k}-1, \cdots, t\}$ and all
		$I\in\bigcup\limits_{j=0}^{\mathfrak{l}-1}\mathfrak{D}_{n+\sum_{s=1}^{m}p_s+j}$, if $\mathcal{J}(\beta_I)\not=\emptyset$, then for all $z\in\mathcal{J}(\beta_I)$,
		we denote by $I_1^{(t)}$ the set consisting of $w\in S_z\setminus E_{\sum_{s=t-1}^{\bf k}p_s}^{\mathfrak{l}}(\hat{\Delta}_{\alpha}^{n+\sum_{s=1}^{\bf k}p_s+1})$ such that the orbit of $w$ doesn't visit $\hat{U}_{0,p_{t-1}}^{n+\sum_{s=1}^{t-2}p_s}$ before leaving $\Lambda_{\alpha}^{n+\sum_{s=1}^{t-2}p_s+1}$; we denote by $I_2^{(t)}$ the set consisting of $w\in S_z\setminus E_{\sum_{s=t-1}^{\bf k}p_s}^{\mathfrak{l}}(\hat{\Delta}_{\alpha}^{n+\sum_{s=1}^{\bf k}p_s+1})$ such that the orbit of $w$ visits $\hat{U}_{0,p_{t-1}}^{n+\sum_{s=1}^{t-2}p_s}$ before leaving $\Lambda_{\alpha}^{n+\sum_{s=1}^{t-2}p_s+1}$.
		
		Next, we prove
		\begin{equation}
			\label{e117a}{\rm dens}_{S_z}I_1^{(t)}\leq q^2
		\end{equation}
		and
		\begin{equation}
			\label{e117b}{\rm dens}_{S_z}I_2^{(t)}\setminus E_{\sum_{s=t-1}^{\bf k}p_s+\chi}^{\mathfrak{l}}(\hat{\Delta}_{\alpha}^{n+\sum_{s=1}^{\bf k}p_s+1})\leq cb(q+cb),
		\end{equation}
		respectively.
		Firstly, we prove
		$${\rm dens}_{S_z}I_1^{(t)}\leq q^2.$$
		For all $w\in I_1^{(t)}$,
		the orbit of $w$ doesn't visit $\hat{U}_{0,p_{t-1}}^{n+\sum_{s=1}^{t-2}p_s}$ before leaving $\Lambda_{\alpha}^{n+\sum_{s=1}^{t-2}p_s+1+\chi}$.
		Since $w\in S_z\setminus E_{\sum_{s=t-1}^{\bf k}p_s}^{\mathfrak{l}}(\hat{\Delta}_{\alpha}^{n+\sum_{s=1}^{\bf k}p_s+1})$, by Lemma \ref{l20257192} we have $w\in S_z\subseteq\mathcal{J}(\tilde{\beta}_I)\subseteq\Lambda_{\alpha}^{n+\sum_{s=1}^{m}p_s}\subseteq\Lambda_{\alpha}^{n+\sum_{s=1}^{t-1}p_s}$ and 
		$P_{\alpha}^{\comp j}(w)\not\in\Lambda^{n+\sum_{s=1}^{t-2}p_s+1}\cup\overline{\Delta_{\alpha}}$ for some positive integer $j$. Then in the same way as the proof of (3) in the proof of Lemma \ref{key1}, we can obtain
		$${\rm dens}_{S_z}I_1^{(t)}\leq \lambda_2^{p_{t-1}}\leq\lambda_2^{\mathfrak{p}}.$$
		By combining (\ref{e0418}), it follows that
		$${\rm dens}_{S_z}I_1^{(t)}\leq q^2.$$
		
		Now we prove ${\rm dens}_{S_z}I_2^{(t)}\setminus E_{\sum_{s=t-1}^{\bf k}p_s+\chi}^{\mathfrak{l}}(\hat{\Delta}_{\alpha}^{n+\sum_{s=1}^{\bf k}p_s+1})\leq cb(q+cb).$
		For all $w\in I_2^{(t)}$, there exists a nonnegative integer $k$ such that $P_{\alpha}^{\comp k}(w)\in\hat{U}_{0,p_{t-1}}^{n+\sum_{s=1}^{t-2}p_s}$ and $P_{\alpha}^{\comp j}(w)\in\Lambda_{\alpha}^{n+\sum_{s=1}^{t-2}p_s+1}$ for all $0\leq j\leq k$.
		By combining $w\in S_z\setminus E_{\sum_{s=t-1}^{\bf k}p_s}^{\mathfrak{l}}(\hat{\Delta}_{\alpha}^{n+\sum_{s=1}^{\bf k}p_s+1})$, it follows that
		$$P_{\alpha}^{\comp (k+1)}(w)\in\bigcup\limits_{j=t-1}^{\bf k}\bigcup\limits_{s=0}^{\mathfrak{l}-1}\hat{\Delta}_{\alpha}^{n+\sum_{s=1}^{j}p_s+s}\setminus\hat{\Delta}_{\alpha}^{n+\sum_{s=1}^{j}p_s+s+1}.$$
		If $P_{\alpha}^{\comp (k+1)}(w)\in
		\bigcup\limits_{s=0}^{\mathfrak{l}-1}\hat{\Delta}_{\alpha}^{n+\sum_{s=1}^{t-1}p_s+s}\setminus\hat{\Delta}_{\alpha}^{n+\sum_{s=1}^{t-1}p_s+s+1}$,
		then there exists $I\in\bigcup\limits_{s=0}^{\mathfrak{l}-1}\mathfrak{D}_{n+\sum_{s=1}^{t-1}p_s+s}$ such that $P_{\alpha}^{\comp (k+1)}(w)\in\mathcal{J}(\beta_I)$.
		By Lemma \ref{l7262} and the assumption we have
		\begin{equation}
			\label{e7245a}{\rm dens}_{S_{P_{\alpha}^{\comp (k+1)}(w)}}S_{P_{\alpha}^{\comp (k+1)}(w)}\setminus E_{\sum_{s=t-1}^{\bf k}p_s+\chi}^{\mathfrak{l}}(\hat{\Delta}_{\alpha}^{n+\sum_{s=1}^{\bf k}p_s+1})\leq q+cb.
		\end{equation}
		If $P_{\alpha}^{\comp (k+1)}(w)\in\bigcup\limits_{j=t}^{\bf k}\bigcup\limits_{s=0}^{\mathfrak{l}-1}\hat{\Delta}_{\alpha}^{n+\sum_{s=1}^{j}p_s+s}\setminus\hat{\Delta}_{\alpha}^{n+\sum_{s=1}^{j}p_s+s+1}$,
		then there exists $I\in\bigcup\limits_{j=t}^{\bf k}\bigcup\limits_{s=0}^{\mathfrak{l}-1}\mathfrak{D}_{n+\sum_{s=1}^{j}p_s+s}$ such that $P_{\alpha}^{\comp (k+1)}(w)\in\mathcal{J}(\beta_I)$.
		By the assumption we have
		$${\rm dens}_{S_{P_{\alpha}^{\comp (k+1)}(w)}}S_{P_{\alpha}^{\comp (k+1)}(w)}\setminus E_{\sum_{s=t}^{\bf k}p_s+\chi}^{\mathfrak{l}}(\hat{\Delta}_{\alpha}^{n+\sum_{s=1}^{\bf k}p_s+1})\leq b$$
		and hence
		\begin{equation}
			\label{e7245}{\rm dens}_{S_{P_{\alpha}^{\comp (k+1)}(w)}}S_{P_{\alpha}^{\comp (k+1)}(w)}\setminus E_{\sum_{s=t-1}^{\bf k}p_s+\chi}^{\mathfrak{l}}(\hat{\Delta}_{\alpha}^{n+\sum_{s=1}^{\bf k}p_s+1})\leq q+cb.
		\end{equation}
		Since $S_{P_{\alpha}^{\comp(k+1)}(w)}$ doesn't intersect the postcritical set $\mathcal{O}(P_{\alpha})$, there exists a one-value analytic branch $\psi$ of $P_{\alpha}^{-k-1}$ on ${\rm int} (S_{P_{\alpha}^{\comp(k+1)}(w)})$ such that $\psi(P_{\alpha}^{\comp(k+1)}(w))=w$. By Lemma \ref{l20257192} and the Schwarz lemma we have that for all $0\leq j\leq k+1$,
		\begin{equation}
			\label{e724a}{\rm diam}_{\mathbb{C}\setminus\overline{\Delta_{\alpha}}}P_{\alpha}^{\comp j}(\psi({\rm int}(S_{P_{\alpha}^{\comp(k+1)}(w)})))\preceq1.
		\end{equation}
		Observe that $P_{\alpha}^{\comp k}(w)\in\hat{U}_{0,p_{t-1}}^{n+\sum_{s=1}^{t-2}p_s}\subseteq\overline{\Omega_{c_0}^{n+\sum_{s=1}^{t-2}p_s+2\mathfrak{q}_5}\setminus\Omega_{c_0}^{n+\sum_{s=1}^{t-1}p_s-2\mathfrak{q}_5}}$ and $P_{\alpha}^{\comp k}(w)\in P_{\alpha}^{\comp k}(\psi({\rm int}(S_{P_{\alpha}^{\comp(k+1)}(w)})))$. Thus (\ref{e7241})
		gives
		\begin{equation}
			\label{e7242}P_{\alpha}^{\comp k}(\psi({\rm int}(S_{P_{\alpha}^{\comp(k+1)}(w)})))\subseteq\Omega_{c_0}^{n+\sum_{s=1}^{t-2}p_s+\mathfrak{q}_5}\setminus\Omega_{c_0}^{n+\sum_{s=1}^{t-1}p_s-\mathfrak{q}_5}.
		\end{equation}
		Moreover, since $S_{P_{\alpha}^{\comp(k+1)}(w)}\subseteq\mathcal{J}(\tilde{\beta}_I)\subseteq\tilde{\Delta}_{\alpha}^{-1}\setminus\{P_{\alpha}(c_0)\}$ and $P_{\alpha}^{\comp k}(w)\in\hat{U}_{0,p_{t-1}}^{n+\sum_{s=1}^{t-2}p_s}\subseteq\tilde{U}_0^{-1}$, we have
		\begin{equation}
			\label{e7243}P_{\alpha}^{\comp k}(\psi({\rm int}(S_{P_{\alpha}^{\comp(k+1)}(w)})))\subseteq\tilde{U}_0^{-1}\setminus\{c_0\}.
		\end{equation}
		By Lemma \ref{l20257193},
		\begin{equation}
			\label{e7244}\tilde{U}_0^{-1}\cap\left(\overline{\Delta_{\alpha}}\cup\Lambda_{\alpha}^{n+\sum_{s=1}^{t-1}p_s+1-\chi}\setminus\Omega_{c_0}^{n+\sum_{s=1}^{t-1}p_s-\mathfrak{q}_5}\right)=\{c_0\}.\ {\rm(Note\ \mathfrak{q}_2-\mathfrak{q}_5<1-\chi)}
		\end{equation}
		Thus by (\ref{e7242}), (\ref{e7243}) and (\ref{e7244}), we have
		\begin{equation}
			\label{e117}P_{\alpha}^{\comp k}(\psi({\rm int}(S_{P_{\alpha}^{\comp(k+1)}(w)})))\cap(\Lambda_{\alpha}^{n+\sum_{s=1}^{t-1}p_s+1-\chi}\cup\overline{\Delta_{\alpha}})={\emptyset}.
		\end{equation}
		By (\ref{e7245a}), (\ref{e7245}) and  Lemma \ref{key3}, there exists an Euclidean ball $B$ {\rm(}$\subseteq\psi({\rm int}(S_{P_{\alpha}^{\comp(k+1)}(w)}))${\rm)} centering at $w$ such that
		\begin{equation}
			\label{e724b}{\rm area}\left(B\cap\psi(S_{P_{\alpha}^{\comp (k+1)}(w)}\setminus E_{\sum_{s=t-1}^{\bf k}p_s+\chi}^{\mathfrak{l}}(\hat{\Delta}_{\alpha}^{n+\sum_{s=1}^{\bf k}p_s+1}))\right)\leq c_1\cdot(q+cb)\cdot{\rm area}(B),
		\end{equation}
		where $c_1$ is the same as that in the proof of Lemma \ref{l7262}.
		It follows from (\ref{e724a}),
		$P_{\alpha}^{\comp j}(w)\in\Lambda_{\alpha}^{n+\sum_{s=1}^{t-2}p_s+1}$ ($0\leq j\leq k+1$) and Claim $1$ that
		$$\psi(S_{P_{\alpha}^{\comp(k+1)}(w)}\cap E_{\sum_{s=t-1}^{\bf k}p_s+\chi}^{\mathfrak{l}}(\hat{\Delta}_{\alpha}^{n+\sum_{s=1}^{\bf k}p_s+1}))\subseteq E_{\sum_{s=t-1}^{\bf k}p_s+\chi}^{\mathfrak{l}}(\hat{\Delta}_{\alpha}^{n+\sum_{s=1}^{\bf k}p_s+1}),$$
		and hence
		$$\psi(S_{P_{\alpha}^{\comp(k+1)}(w)}\setminus E_{\sum_{s=t-1}^{\bf k}p_s+\chi}^{\mathfrak{l}}(\hat{\Delta}_{\alpha}^{n+\sum_{s=1}^{\bf k}p_s+1}))\supseteq
		\psi(S_{P_{\alpha}^{\comp(k+1)}(w)})\setminus
		E_{\sum_{s=t-1}^{\bf k}p_s+\chi}^{\mathfrak{l}}(\hat{\Delta}_{\alpha}^{n+\sum_{s=1}^{\bf k}p_s+1}).$$
		By combining (\ref{e724b}), it follows that
		$${\rm area}\left(B\cap I_2^{(t)}\setminus
		E_{\sum_{s=t-1}^{\bf k}p_s+\chi}^{\mathfrak{l}}(\hat{\Delta}_{\alpha}^{n+\sum_{s=1}^{\bf k}p_s+1})\right)\leq c_1\cdot(q+cb)\cdot{\rm area}(B).$$
		By Lemmas \ref{l25070716} and \ref{l20257192} we have 
		${\rm diam}(S_z)\asymp{\rm dist}(S_z,\mathcal{O}(P_{\alpha}))$. By combining $w\in S_z$, it follows that
		\begin{equation}
			\label{e20260425a}{\rm diam}(B)\leq2{\rm dist}(w,\mathcal{O}(P_{\alpha}))\asymp{\rm diam}(S_z)=\sqrt{2}l_z.
		\end{equation}
		We denote by $\tilde{B}$ the Euclidean ball centering at $w$ with radius $\min\{l_z, \frac{{\rm diam}(B)}{2}\}$.
		(\ref{e20260425a}) gives ${\rm diam}(\tilde{B})\asymp{\rm diam}(B)$, and hence
		\begin{equation}
			\label{e7271}{\rm area}\left(\tilde{B}\cap I_2^{(t)}\setminus
			E_{\sum_{s=t-1}^{\bf k}p_s+\chi}^{\mathfrak{l}}(\hat{\Delta}_{\alpha}^{n+\sum_{s=1}^{\bf k}p_s+1})\right)\leq c_2c_1\cdot(q+cb)\cdot{\rm area}(\tilde{B}),
		\end{equation}
		where $c_2$ is the same as that in the proof of Lemma \ref{l7262}.
		Since $\tilde{B}\subseteq B\subseteq\psi({\rm int}(S_{P_{\alpha}^{\comp(k+1)}(w)}))$, (\ref{e117}) and $\hat{\Delta}_{\alpha}^{n+\sum_{s=1}^{\bf k}p_s+\mathfrak{l}}=\overline{\Delta_{\alpha}}$,
		we have
		\begin{equation}
			\label{e7272}\tilde{B}\cap E_{\sum_{s=t}^{\bf k}p_s+\chi}^{\mathfrak{l}}(\hat{\Delta}_{\alpha}^{n+\sum_{s=1}^{\bf k}p_s+1})=\emptyset.
		\end{equation}
		Together with (\ref{e7271}), (\ref{e7272}), $0<\frac{{\rm diam}(\tilde{B})}{2}=\min\{l_z, \frac{{\rm diam}(B)}{2}\}\leq l_z$ and the assumption, Lemma \ref{key4} gives
		$${\rm dens}_{S_z}I_2^{(t)}\setminus E_{\sum_{s=t-1}^{\bf k}p_s+\chi}^{\mathfrak{l}}(\hat{\Delta}_{\alpha}^{n+\sum_{s=1}^{\bf k}p_s+1})\leq c_3c_2c_1b(q+cb)=cb(q+cb),$$
		where both $c_3$ and $c$ are the same as that in the proof of Lemma \ref{l7262}.
		
		At last, by  (\ref{e117a}) and (\ref{e117b}) we have
		\begin{align*}
			{\rm dens}_{S_z}S_z\setminus E_{\sum_{s=t-1}^{\bf k}p_s+\chi}^{\mathfrak{l}}(\hat{\Delta}_{\alpha}^{n+\sum_{s=1}^{\bf k}p_s+1})&={\rm dens}_{S_z}(I_1^{(t)}\cup I_2^{(t)})\setminus E_{\sum_{s=t-1}^{\bf k}p_s+\chi}^{\mathfrak{l}}(\hat{\Delta}_{\alpha}^{n+\sum_{s=1}^{\bf k}p_s+1})\\
			&\leq cb(q+cb)+q^2.
		\end{align*}
		
	\end{proof}

	\begin{proof}[The proof of Theorem \ref{key5}:]
		Let $c$ be the same as that in the proof of Lemma \ref{l7262}, and let $q$, $\{b_j\}$ and the above $c$ satisfy the relations in the lemma \ref{l7271}.
		
		Firstly, we prove that for all $1\leq m\leq{\bf k}$ and $I\in\bigcup\limits_{j=0}^{\mathfrak{l}-1}\mathfrak{D}_{n+\sum_{s=1}^mp_s+j}$, if $\mathcal{J}(\beta_I)\not=\emptyset$, then for all $z\in\mathcal{J}(\beta_I)$, we have
		\begin{equation}
			\label{e1116a}{\rm dens}_{S_z}S_z\setminus E_{\sum_{s=m}^{\bf k}p_s+\chi}^{\mathfrak{l}}(\hat{\Delta}_{\alpha}^{n+\sum_{s=1}^{\bf k}p_s+1})\leq b_{{\bf k}-m+1}.
		\end{equation}
		In fact, since $\hat{\Delta}_{\alpha}^{n+\sum_{s=1}^{\bf k}p_s+\mathfrak{l}}\not=\hat{\Delta}_{\alpha}^{n+\sum_{s=1}^{\bf k}p_s}$, by (\ref{e1071}) we have that for $I\in\bigcup\limits_{j=0}^{\mathfrak{l}-1}\mathfrak{D}_{n+\sum_{s=1}^{\bf k}p_s+j}$ with $\mathcal{J}(\beta_I)\not=\emptyset$ and all $z\in\mathcal{J}(\beta_I)$, we have
		$${\rm dens}_{S_z}S_z\setminus E_{p_{\bf k}}^{\mathfrak{l}}(\hat{\Delta}_{\alpha}^{n+\sum_{s=1}^{\bf k}p_s+1})\leq q=b_1.$$
		We suppose that (\ref{e1116a}) holds for $m={\bf k}, {\bf k}-1, \cdots, t>1$.
		Next, we prove that (\ref{e1116a}) holds for $m=t-1$.
		By the above supposition and (a11) in Lemma \ref{l7271}, for all $m\in\{{\bf k}, {\bf k}-1, \cdots, t\}$, we have that for all
		$I\in\bigcup\limits_{j=0}^{\mathfrak{l}-1}\mathfrak{D}_{n+\sum_{s=1}^{m}p_s+j}$, if $\mathcal{J}(\beta_I)\not=\emptyset$, then for all $z\in\mathcal{J}(\beta_I)$, we have
		\begin{align*}
			{\rm dens}_{S_z}S_z\setminus E_{\sum_{s=t}^{\bf k}p_s+\chi}^{\mathfrak{l}}(\hat{\Delta}_{\alpha}^{n+\sum_{s=1}^{\bf k}p_s+1})&\leq
			{\rm dens}_{S_z}S_z\setminus E_{\sum_{s=m}^{\bf k}p_s+\chi}^{\mathfrak{l}}(\hat{\Delta}_{\alpha}^{n+\sum_{s=1}^{\bf k}p_s+1})\\
			&\leq b_{{\bf k}-m+1}\\
			&\leq b_{{\bf k}-t+1}.
		\end{align*}
		Then by Lemma \ref{l7272} and (a11) in Lemma \ref{l7271} we have
		\begin{align*}
			{\rm dens}_{S_z}S_z\setminus E_{\sum_{s=t-1}^{\bf k}p_s+\chi}^{\mathfrak{l}}(\hat{\Delta}_{\alpha}^{n+\sum_{s=1}^{\bf k}p_s+1})
			&\leq
			cb_{{\bf k}-t+1}(q+cb_{{\bf k}-t+1})+q^2\\
			&\leq cb_{{\bf k}-t+1}(q+cb_{{\bf k}-t+1})+qb_{{\bf k}-t+1}.
		\end{align*}
		Now for all
		$I\in\bigcup\limits_{j=0}^{\mathfrak{l}-1}\mathfrak{D}_{n+\sum_{s=1}^{t-1}p_s+j}$, if $\mathcal{J}(\beta_I)\not=\emptyset$, then for all $z\in\mathcal{J}(\beta_I)$,
		by Lemma \ref{l7262}, we have
		\begin{align*}
			{\rm dens}_{S_z}S_z\setminus E_{\sum_{s=t-1}^{\bf k}p_s+\chi}^{\mathfrak{l}}(\hat{\Delta}_{\alpha}^{n+\sum_{s=1}^{\bf k}p_s+1})
			&\leq q+c(cb_{{\bf k}-t+1}(q+cb_{{\bf k}-t+1})+qb_{{\bf k}-t+1})\\
			&= b_{{\bf k}-(t-1)+1}.
		\end{align*}
		Thus by induction (\ref{e1116a}) holds.
		
		At last, by (b11) in Lemma \ref{l7271} and (\ref{e1116a}), we have that for all $1\leq m\leq{\bf k}$ and $I\in\bigcup\limits_{j=0}^{\mathfrak{l}-1}\mathfrak{D}_{n+\sum_{s=1}^mp_s+j}$, if $\mathcal{J}(\beta_I)\not=\emptyset$, then for all $z\in\mathcal{J}(\beta_I)$, we have
		$${\rm dens}_{S_z}S_z\setminus K(P_{\alpha})\leq{\rm dens}_{S_z}S_z\setminus E_{\sum_{s=m}^{\bf k}p_s+\chi}^{\mathfrak{l}}(\hat{\Delta}_{\alpha}^{n+\sum_{s=1}^{\bf k}p_s+1})\leq b_{{\bf k}-m+1}\leq\frac{1}{2c^3}.$$
		Thus the proof of the lemma is completed.
		
	\end{proof}

	\section{The proof of two main theorems}
	
	Before proving two main theorems, to satisfy the assumption condition at the begining of Section \ref{s891}, we first adjust the original pseudo-Siegel disks.
	Let $\{\hat{\Delta}_{\alpha}^n\}_{n\geq-1}$ be a geodesic filling-in of $\Delta_{\alpha}$ satisfying (P1) and (P2) in Section \ref{s2.3}. Let $\{F_{n}\}$ be the corresponding geodesic pre-filling-ins of $\Delta_{\alpha}$.
	Recall $\alpha=[a_1,a_2,\cdots]$ is an eventually golden mean number. We let ${\bf h}=2$ if $a_1>1$ and let ${\bf h}=3$ if $a_1=1$. For every positive integer $\mathfrak{M}$, we define 
	a new geodesic filling-in of $\Delta_{\alpha}$: $\{\hat{\Delta}_{\alpha, \mathfrak{M}}^n\}_{n\geq-1}$ as follows: 
	\begin{itemize}
		\item $\hat{\Delta}_{\alpha, \mathfrak{M}}^n=\overline{\Delta_{\alpha}}$ for all sufficiently large $n$,
		\item for all $n\geq0$,
		$$\hat{\Delta}_{\alpha,\mathfrak{M}}^{n-1}=
		\left\{\begin{matrix}
			\hat{\Delta}_{\alpha,\mathfrak{M}}^n,&a_{n-1+{\bf h}}\leq\mathfrak{M} \ {\rm or}\ \hat{\Delta}_{\alpha}^n=\hat{\Delta}_{\alpha}^{n-1}\\
			\hat{\Delta}_{\alpha,\mathfrak{M}}^n\cup F_{n-1},&{\rm else}
		\end{matrix}\right..
		$$
	\end{itemize}
	It follows from Dudko and Lyubich's construction (refer [Section 11.4, \cite{DL}]) that the following fact holds:
	
	\vspace{0.2cm}
	\noindent\emph{{\rm\bf Fact \uppercase\expandafter{\romannumeral1}:} For every positive integer $\mathfrak{M}$, there exists a geodesic filling-in of $\Delta_{\alpha}$: $\{\hat{\Delta}_{\alpha}^n\}_{n\geq-1}$ satisfying (P1) and (P2) in Section \ref{s2.3} so that $\hat{\Delta}_{\alpha}^n=\hat{\Delta}_{\alpha, \mathfrak{M}}^n$ for all $n\geq-1$. Note that parameters $C,M$ in (P1) and the parameter $\Delta$ in (P2) in Section \ref{s2.3} may depend on $\mathfrak{M}$, but independent of $\alpha$.}
	
	\vspace{0.2cm}
	\noindent In fact, in Dudko and Lyubich's construction, 
	on the one hand, in [Theorem 10.3, \cite{DL}] as long as $\Delta_{\alpha}$ does not have an external level-($n-1$) parabolic rectangle of width $\sqrt{\bf K}$, $\hat{\Delta}_{\alpha}^{n-1}=\hat{\Delta}_{\alpha}^n$. On the other hand, by [Thereom 4.1, \cite{DL}] if there is an external level-($n-1$) parabolic rectangle of sufficient width, then $a_{n-1+{\bf h}}$ will be sufficiently big. Thus 
	for every fixed $\mathfrak{M}$, 
	to ensure $\{\hat{\Delta}_{\alpha}^n\}_{n\geq-1}=\{\hat{\Delta}_{\alpha, \mathfrak{M}}^n\}_{n\geq-1}$,
	we only need to increase ${\bf K}$ in [Theorems 10.3 and 10.4, \cite{DL}] appropriately so that for all $n\geq0$, if there is an external level-($n-1$) parabolic rectangle of width $\sqrt{\bf K}$, then $a_{n-1+{\bf h}}>\mathfrak{M}$ holds. In this case, as Dudko and Lyubich's construction, parameters $C,M,\Delta$ are independent of $\alpha$.

	\begin{proof}[The proof of Theorem \ref{T2}]
		For all $\mathfrak{M}\geq1$ and $\mathfrak{l}\geq1$, by the above fact \uppercase\expandafter{\romannumeral1}, we have pseudo-Siegel disks: $\{\hat{\Delta}_{\alpha}^m\}_{m\geq-1}$ so that $\hat{\Delta}_{\alpha}^m=\hat{\Delta}_{\alpha, \mathfrak{M}}^m$ for all eventually golden mean $\alpha$.
		We let $\mathfrak{p}$ sufficiently big so that the all conditions at the first paragraph of Section \ref{s891} hold. For all
		$\alpha=[a_1,a_2,\cdots]\in\mathcal{C}_{\mathfrak{l},\mathfrak{p}}^{(\mathfrak{M})}$,
		we let $\alpha_k:=[a_1,a_2,\cdots,a_k,1,1,\cdots]$ be a perturbation of $\alpha$. By the upper semi-continuity of areas of filled-in Julia sets (see [Proposition $2$, \cite{BC}]), we have $$\limsup\limits_{k\to+\infty}{\rm area}(K(P_{\alpha_k}))\leq{\rm area}(K(P_{\alpha}))={\rm area}(J(P_{\alpha})).$$
		To prove Theorem \ref{T2}, we only need to
		prove
		$$\limsup\limits_{k\to+\infty}{\rm area}(K(P_{\alpha_k}))>0.$$
		
		Since $\alpha\in\mathcal{C}_{\mathfrak{l},\mathfrak{p}}^{(\mathfrak{M})}$, we have that
		there exists a positive integer $j_0$ such that for all $j_1>j_2\geq j_0$ with $a_{j_1}>\mathfrak{M}$ and $a_{j_2}>\mathfrak{M}$, we have $j_1-j_2\geq\mathfrak{p}$ or $j_1-j_2\leq\mathfrak{l}-1$. We fix a positive integer $n$ such that $n\geq\max\{2\mathfrak{q}_4+{\bf h}, j_0\}$ and for all $s\in\{n+1,n+2,\cdots, n+\mathfrak{p}-1\}$, $a_s\leq\mathfrak{M}$.
		Then there exists a sequence $\{p_{k}\}_{k=1}^\infty$ of positive integers such that 
		\begin{itemize}
			\item for all $k\geq1$,
			$p_k\geq\mathfrak{p}$,
			\item for all $n+1\leq m\leq n+p_1-1$,
			$a_m\leq\mathfrak{M}$,
			\item for all $k\geq1$ and $n+\sum_{s=1}^kp_s+\mathfrak{l}\leq m\leq n+\sum_{s=1}^{k+1}p_s-1$, $a_m\leq\mathfrak{M}$.
		\end{itemize}
		Thus $a_m>\mathfrak{M}$ may occur only for $n+\sum_{s=1}^kp_s\leq m\leq n+\sum_{s=1}^kp_s+\mathfrak{l}-1$ with $k\in\mathbb{N}_+$.
		Since $\alpha$ is of non-Brjuno type, $a_m>\mathfrak{M}$ must occur infinite times.
		
		For every $k=n+\sum_{s=1}^{\tilde{\bf k}}p_s+\mathfrak{l}$ with a positive integer $\tilde{\bf k}$, we have that $\hat{\Delta}_{\alpha_k}^m=\hat{\Delta}_{\alpha_k, \mathfrak{M}}^m$ for all $m\geq-1$.
		Then we have
		\begin{itemize}
			\item for $1\leq j\leq p_1$, $\hat{\Delta}_{\alpha_k}^{n-{\bf h}+j}=\hat{\Delta}_{\alpha_k}^{n-{\bf h}+p_1}$,
			\item
			for $2\leq t\leq \tilde{\bf k}$ and $\mathfrak{l}\leq j\leq p_t$,  
			$\hat{\Delta}_{\alpha_k}^{n-{\bf h}+\sum_{s=1}^{t-1}p_s+j}=\hat{\Delta}_{\alpha_k}^{n-{\bf h}+\sum_{s=1}^{t}p_s}$,
			\item  $\hat{\Delta}_{\alpha_k}^{n-{\bf h}+\sum_{s=1}^{\tilde{\bf k}}p_{s}+\mathfrak{l}}=\overline{\Delta_{\alpha_k}}$.
		\end{itemize}
		If $\hat{\Delta}_{\alpha_k}^{n-{\bf h}+1}\setminus\overline{\Delta_{\alpha_k}}\not=\emptyset$, then for all $z\in\hat{\Delta}_{\alpha_k}^{n-{\bf h}+1}\setminus\overline{\Delta_{\alpha_k}}$, there exist ${\bf k}\in\{1,2,\cdots,\tilde{\bf k}\}$ and $j\in\{0,1,2,\cdots,\mathfrak{l}-1\}$ such that
		$\hat{\Delta}_{\alpha_k}^{n-{\bf h}+\sum_{s=1}^{\bf k}p_{s}+j}\not=\hat{\Delta}_{\alpha_k}^{n-{\bf h}+\sum_{s=1}^{\bf k}p_{s}+j+1}$ and $z\in\mathcal{J}(\beta_I)$ for some $I\in\mathfrak{D}_{n-{\bf h}+\sum_{s=1}^{\bf k}p_s+j}$. In this case, by Theorem \ref{key5} we have
		\begin{equation}
			\label{e8111}{\rm dens}_{S_z} K(P_{\alpha_k})>\frac{1}{2}.
		\end{equation}
		
		Now we consider $\hat{\Delta}_{\alpha_k}^{n-{\bf h}+1}$. It follows from Lemma \ref{l84a} that $\hat{\Delta}_{\alpha_k}^{n-{\bf h}+1}$ contains a $\mathfrak{K}$-quasidisk $B$ whose diameter has a positive low bound depending on $n$, but not on $k$. For all $z\in B\setminus\overline{\Delta_{\alpha_k}}$, we denote by $B_z$ the biggest closed Euclidean ball centering at $z$ contained in $S_z$. Then ${\rm diam}(B_z)\leq2{\rm diam}(\hat{\Delta}_{\alpha_k}^{n-{\bf h}+1})$ (Due to $z\in\hat{\Delta}_{\alpha_k}^{n-{\bf h}+1}$, $\overline{\Delta_{\alpha_k}}\subseteq\hat{\Delta}_{\alpha_k}^{n-{\bf h}+1}$ and $B_z\subseteq S_z\subseteq\mathbb{C}\setminus\overline{\Delta_{\alpha_k}}$) and ${\rm dens}_{B_z} K(P_{\alpha_k})\geq1-2/\pi$ (Due to (\ref{e8111})). If ${\rm area}(B\setminus\overline{\Delta_{\alpha_k}})>0$, then by applying the Besicovitch covering theorem to $\{B_z:z\in B\setminus\overline{\Delta_{\alpha_k}}\}$, we have that there exists a family of pairwise disjoint $B_z$ ($z\in B\setminus\overline{\Delta_{\alpha_k}}$) such that the area of these $B_z$ union is at least compatible with the area of $B\setminus\overline{\Delta_{\alpha_k}}$, and hence the area of $K(P_{\alpha_k})$ is at least compatible with the area of $B\setminus\overline{\Delta_{\alpha_k}}$. Observe that $B\cap\overline{\Delta_{\alpha_k}}\subseteq K(P_{\alpha_k})$.
		This implies that the area of $K(P_{\alpha_k})$ 
		is at least compatible with that of $B$ and hence
		\begin{equation*}
			\label{e20260507c}{\rm area}(K(P_{\alpha_k}))\succeq{\rm area}(B).
		\end{equation*}
		It follows that ${\rm area}(K(P_{\alpha_k}))$
		has a positive lower bound not depending on $k$. Thus
		$$\limsup\limits_{k\to+\infty}{\rm area}(K(P_{\alpha_k}))>0.$$
		This completes the proof.
	\end{proof}
	
	\begin{lemma}
		\label{l8141}For all eventually golden mean number $\alpha$ and $n\geq4$, we have
		\begin{itemize}
			\item[{\rm(a10)}] $\gamma_{c_0}^{n}\cap\hat{\Delta}_{\alpha}^{-1,-1}=\{x_{q_{n}},x_{q_{n+1}-q_n}\}$,
			\item[{\rm(b10)}] ${\rm diam}(\gamma_{c_0}^{n})\to0$ and ${\rm dist}(c_0,\gamma_{c_0}^{n})\to0$  as $n\to+\infty$, both are uniform on $\alpha$.
		\end{itemize}
	\end{lemma}
	\begin{proof}
		(a10) For all $I\in\bigcup\limits_{s=-1}^\infty\mathfrak{D}_s$ with  $\mathcal{J}(\beta_I)\not=\emptyset$, by (P2) in Section \ref{s2.3}
		there exists a non-winding rectangle $\mathcal{R}_{I}$ based on $I$ (with respect to $\overline{\Delta_{\alpha}}$) separating (non-winding) $\beta_{I}$ and (non-winding) $\hat{\beta}_{I}$ on $\mathbb{C}\setminus\Delta_{\alpha}$ such that the width $\mathcal{W}(\mathcal{R}_{I})$ is greater than a universal $\Delta+4\gg1$. Moreover, the combinatorial distance between two endpoints of $\hat{\beta}_{I}$ and two endpoints of $I$ is at least $11\iota_{m+1}$, where $m$ is the level of $I$. Let $\beta_{I}^{-1}$ (resp. $\mathcal{R}_{I}^{-1}$) be the component of $P_{\alpha}^{-1}(\beta_{I})$ (resp. $P_{\alpha}^{-1}(\mathcal{R}_{I})$) intersecting $\overline{\Delta_{\alpha}}$, respectively.
		Since the combinatorial distance between two endpoints of $\hat{\beta}_{I}$ and two endpoints of $I$ is at least $11\iota_{m+1}$,
		we have that there exists a unique interval $I^{-1}$ of
		$\mathfrak{D}_m$ such that both $\beta_{I}^{-1}$ and $\mathcal{R}_{I}^{-1}$ are based on $I^{-1}$.
		Then $O_{I^{-1}}(\beta_{I}^{-1})\subseteq O_{I^{-1}}(\mathcal{R}_{I}^{-1})\setminus\mathcal{R}_{I}^{-1}$ and the combinatorial distance between $\partial^{h,0}\mathcal{R}_{I}^{-1}\cup\partial^{h,1}\mathcal{R}_{I}^{-1}$ ($=I^{-1}\cap\mathcal{R}_{I}^{-1}$) and two endpoints of $I^{-1}$ is at least $10\iota_{m+1}$.
		Observe that $\mathcal{W}(\mathcal{R}_{I}^{-1})=\mathcal{W}(\mathcal{R}_{I})\gg1$. 
		By Lemma \ref{l825a} there exists a geodesic $\gamma\subseteq\mathcal{R}_{I}^{-1}$ connecting $\partial^{h,0}\mathcal{R}_{I}^{-1}$ and $\partial^{h,1}\mathcal{R}_{I}^{-1}$ with respect to the hyperbolic metric on $\hat{\mathbb{C}}\setminus\overline{\Delta_{\alpha}}$. Thus 
		\begin{equation}
			\label{e20260428a}O_{I^{-1}}(\beta_{I}^{-1})\subseteq O_{I^{-1}}(\gamma)
		\end{equation}
		and $\overline{O_{I^{-1}}(\gamma)}\cap I^{-1}$ has the combinatorial distance at least $10\iota_{m+1}$ with two endpoints of $I^{-1}$.
		Observe that both $x_{q_n}$ and $x_{q_{n+1}-q_n}$ have the combinatorial distance at most $2\iota_{m+1}$ with endpoints of all intervals of $\mathfrak{D}_m$.
		Thus $\{x_{q_n},x_{q_{n+1}-q_n}\}\cap\overline{O_{I^{-1}}(\gamma)}=\emptyset$. By combining (\ref{e20260428a}), it follows that $\gamma_{c_0}^n\cap O_{I^{-1}}(\beta_I^{-1})\subseteq\gamma_{c_0}^n\cap O_{I^{-1}}(\gamma)=\emptyset$, and hence
		$\gamma_{c_0}^{n}\cap\hat{\Delta}_{\alpha}^{-1,-1}=\gamma_{c_0}^{n}\cap\overline{\Delta_{\alpha}}=\{x_{q_{n}},x_{q_{n+1}-q_n}\}$.
		
		\vspace{0.2cm}
		\noindent(b10) Let $n$ ($\geq\mathfrak{q}_2$) and $k$ ($\gg1$) be two positive integers. For all $1\leq s\leq2$, we let $I_{s,l}$ and  $I_{s,r}$ be the two components of $$\overline{I_{[x_{q_{n+1+(s-1)k}-1},x_{q_{n+2+(s-1)k}-1}]}\setminus  I_{[x_{q_{n+1+sk}-1},x_{q_{n+2+sk}-1}]}}$$
		with
		$I_{s,l}<I_{[x_{q_{n+1+sk}-1},x_{q_{n+2+sk}-1}]}<I_{s,r}$.
		For all $1\leq s\leq2$,
		we denote by 
		$\hat{I}_{s,l}$ and  $\hat{I}_{s,r}$ projections of
		$I_{s,l}$ and  $I_{s,r}$ onto $\partial\hat{\Delta}_{\alpha}^{n+1}$, respectively.
		We define $\mathcal{G}_s$ as the interior of the union $\mathcal{G}^-_{\hat{\Delta}_{\alpha}^{n+1}}(\hat{I}_{s,l},\hat{I}_{s,r})\cup\hat{\mathcal{G}}^+_{\overline{\Delta_{\alpha}}}(\hat{I}_{s,l},\hat{I}_{s,r})$.
		By Corollary \ref{c831}, we have that for all $1\leq s\leq2$, ${\rm mod}(\mathcal{G}_s)\gg1$.
		By Lemma \ref{l841} and $n\geq\mathfrak{q}_2$, 
		for all $1\leq s\leq2$, $\mathcal{G}_s\subseteq\mathbb{C}$ and $P_{\alpha}(c_0)$ is contained in the bounded component of $\mathbb{C}\setminus\mathcal{G}_s$. 
		For all $1\leq s\leq2$, we define $\mathcal{G}_s^{-1}:=P_{\alpha}^{-1}(\mathcal{G}_s)$.
		Since $P_{\alpha}(c_0)$ is contained in the bounded component of $\mathbb{C}\setminus\mathcal{G}_s$, we have that $P_{\alpha}|_{\mathcal{G}_s^{-1}}$ is a covering of degree $2$ from $\mathcal{G}_s^{-1}$ to $\mathcal{G}_s$. Then ${\rm mod}(\mathcal{G}_s^{-1})=\frac{{\rm mod}(\mathcal{G}_s)}{2}\gg1$.
		
		For all $1\leq s\leq2$, we let $I_{s,l}^{-1}$ and $I_{s,r}^{-1}$ be pre-images of
		$I_{s,l}$ and $I_{s,r}$
		contained in $\partial\Delta_{\alpha}$, respectively.
		We denote by $\mathcal{R}_{s,+}^{(-1)}$ the rectangle bounded by $I_{s,l}^{-1}\cup I_{s,r}^{-1}\cup\partial^{v,0}\overline{\mathcal{G}_s^{-1}\setminus\hat{\Delta}_{\alpha}^{n+1,-1}}\cup\partial^{v,1}\overline{\mathcal{G}_s^{-1}\setminus\hat{\Delta}_{\alpha}^{n+1,-1}}$ with
		$\partial^{h,0}\mathcal{R}_{s,+}^{(-1)}=I_{s,l}^{-1}$ and $\partial^{h,1}\mathcal{R}_{s,+}^{(-1)}=I_{s,r}^{-1}$. 
		By (\ref{e85a})
		$$\mathcal{W}(\mathcal{R}_{s,+}^{(-1)})\succeq{\rm mod}(\mathcal{G}_s)\gg1.$$
		
		Observe that $\partial^{h,0}\mathcal{R}_{2,+}^{(-1)}<I_{[x_{q_{n+3k}},x_0]}\cup I_{[x_{q_{n+3k+1}-q_{n+3k}},x_0]}<\partial^{h,1}\mathcal{R}_{2,+}^{(-1)}$.
		Then it follows from Lemma \ref{l825a} and $\mathcal{W}(\mathcal{R}_{2,+}^{(-1)})\gg1$ that
		$$\gamma_{c_0}^{n+3k}\setminus\{x_{q_{n+3k}},x_{q_{n+3k+1}-q_{n+3k}}\}\subseteq{\rm int}( O_{I_{[x_{q_{n+1+k}},x_0]}\cup I_{[x_{q_{n+2+k}},x_0]}}(\mathcal{R}_{2,+}^{(-1)})).$$
		This implies $\gamma_{c_0}^{n+3k}\cap\mathcal{R}_{1,+}^{(-1)}=\emptyset$, and hence
		$\gamma_{c_0}^{n+3k}\cap\overline{\mathcal{G}_1^{-1}\setminus\hat{\Delta}_{\alpha}^{n+1,-1}}=\emptyset$. Together with $\gamma_{c_0}^{n+3k}\cap{\rm int}(\hat{\Delta}_{\alpha}^{n+1,-1})=\emptyset$ (Due to (a10)), it gives 
		$\gamma_{c_0}^{n+3k}\cap\mathcal{G}_1^{-1}=\emptyset$.
		Observe that $\mathcal{G}_1^{-1}$ separates $\{\infty, P_{\alpha}(c_0)\}$ and $\{x_{q_{n+3k}}, c_0\}$.
		Thus $\mathcal{G}_1^{-1}$ separates $\{\infty, P_{\alpha}(c_0)\}$ and $\{\gamma_{c_0}^{n+3k}, c_0\}$.
		By Corollary \ref{c831} ${\rm mod}(\mathcal{G}_1^{-1})=\frac{{\rm mod}(\mathcal{G}_1)}{2}$ converges to $+\infty$ as $k\to+\infty$, uniform on $\alpha$.
		Then the diameter of the bounded component of $\mathbb{C}\setminus\mathcal{G}_1^{-1}$ converges to $0$ as $k\to+\infty$, uniform on $\alpha$.
		Since the bounded component of $\mathbb{C}\setminus\mathcal{G}_1^{-1}$ contains $\gamma_{c_0}^{n+3k}$ and $c_0$, we have that
		${\rm diam}(\gamma_{c_0}^{n+3k})\to0$ and ${\rm dist}(c_0,\gamma_{c_0}^{n+3k})\to0$ as $k\to+\infty$, uniform on $\alpha$.
		
	\end{proof}
	
	\begin{proof}[The proof of Theorem \ref{T3}]
		For all $\mathfrak{M}\geq1$ and $\mathfrak{l}\geq1$, by the above fact \uppercase\expandafter{\romannumeral1}, we have pseudo-Siegel disks: $\{\hat{\Delta}_{\alpha}^m\}_{m\geq-1}$ so that $\hat{\Delta}_{\alpha}^m=\hat{\Delta}_{\alpha, \mathfrak{M}}^m$ for all eventually golden mean $\alpha$.
		We let $\mathfrak{p}$ sufficiently big so that the all conditions at the first paragraph of Section \ref{s891} hold. For all
		$\alpha=[a_1,a_2,\cdots]\in\mathcal{E}_{\mathfrak{l},\mathfrak{p}}^{(\mathfrak{M})}$,
		we let $\alpha_k:=[a_1,a_2,\cdots,a_k,1,1,\cdots]$ be a perturbation of $\alpha$. We denote by $\Delta_{\alpha}$ (resp. $\Delta_{\alpha_k}$) the Siegel disk of $P_{\alpha}$ (resp. $P_{\alpha_k}$) centering at $0$.
		Let $\phi_{\alpha}$ (resp. $\phi_{\alpha_k}$) be a conformal map from $\mathbb{D}$ to $\Delta_{\alpha}$ (resp. $\Delta_{\alpha_k}$) such that $P_{\alpha}(\phi_{\alpha}(z))=\phi_{\alpha}(e^{2\pi i\alpha}z)$ (resp. $P_{\alpha_k}(\phi_{\alpha_k}(z))=\phi_{\alpha_k}(e^{2\pi i\alpha_k}z)$) for all $z\in\mathbb{D}$.
		For all $0<r<1$, we define $$\Delta_{\alpha}(r):=\phi_{\alpha}(\mathbb{B}(0,r))$$
		and
		$$\Delta_{\alpha_k}(r):=\phi_{\alpha_k}(\mathbb{B}(0,r)).$$
		For all $0<r<1$ and any Brjuno number $\theta$, we 
		define $$K_r(P_{\theta}):=K(P_{\theta})\cap\left\{z\in\mathbb{C}:\{P_{\theta}^{\comp s}(z)\}_{s=0}^{+\infty}\subseteq\mathbb{C}\setminus\Delta_{\alpha}(r)\right\}.$$
		Similar to the upper semi-continuity of the area of $K(P_{\theta})$ on $\theta$, it is clear that ${\rm area}(K_r(P_{\theta}))$ also has the upper semi-continuity on $\theta$. Thus we have that for all $0<r<1$, 
		\begin{equation}
			\label{e8171}\limsup\limits_{k\to+\infty}{\rm area}(K_r(P_{\alpha_k}))\leq{\rm area}(K_r(P_{\alpha})).
		\end{equation}
		Observe that $K_r(P_{\alpha})$ is decreasing on $r$ and
		$$\bigcap\limits_{0<r<1}K_r(P_{\alpha})=J(P_{\alpha}).$$
		Then
		$$\lim\limits_{r\to1^-}{\rm area}(K_r(P_{\alpha}))={\rm area}(J(P_{\alpha})).$$
		Applying $r\to1^-$ to two sides of (\ref{e8171}), we have
		$$\limsup\limits_{r\to1^-}\left(\limsup\limits_{k\to+\infty}{\rm area}(K_r(P_{\alpha_k}))\right)\leq\lim\limits_{r\to1^-}{\rm area}(K_r(P_{\alpha}))={\rm area}(J(P_{\alpha})).$$
		To prove Theorem \ref{T3}, we only need to
		prove that there exists $\epsilon>0$ such that for all $0<r'<1$,
		$$\limsup\limits_{k\to+\infty}{\rm area}(K_{r'}(P_{\alpha_k}))>\epsilon.$$
		
		Since $\alpha\in\mathcal{E}_{\mathfrak{l},\mathfrak{p}}^{(\mathfrak{M})}$, we have that $c_{0,\alpha}\not\in\overline{\Delta_{\alpha}}$. By [Lemma 1.3, \cite{ST00}] $c_{0,\alpha}$ is a recurrent critical point. Thus $P_{\alpha}(c_{0,\alpha})\not\in\overline{\Delta_{\alpha}}$.
		We write
		$\delta:={\rm dist}(\Delta_{\alpha},P_{\alpha}(c_{0,\alpha}))>0$.
		By (b10) in Lemma \ref{l8141} for all sufficiently big $n$, $P_{\alpha_k}(\gamma_{c_{0,\alpha_k}}^n)\in\mathbb{B}(P_{\alpha_k}(c_{0,\alpha_k}),\delta)$ for all $k\geq1$.
		We fix $n$ sufficiently big so that $n\geq2\mathfrak{q}_4$ and 
		\begin{equation}
			\label{e20260127a}P_{\alpha_k}(\gamma_{c_{0,\alpha_k}}^n)\in\mathbb{B}(P_{\alpha_k}(c_{0,\alpha_k}),\delta)\ {\rm for\ all}\ k\geq1.
		\end{equation}
		It follows from [Proposition $10$, \cite{R99}] (also see [Corollary $4$, \cite{ABC04}]) that
		$\phi_{\alpha_k}'(0)\to\phi_{\alpha}'(0)$ as $k\to+\infty$. Then
		for all $0<r<1$, $\phi_{\alpha_k}$ converges uniformly to $\phi_{\alpha}$ as $k\to+\infty$ on $\mathbb{B}(0,r)$. Observe that $P_{\alpha_k}(c_{0,\alpha_k})$ converges to $P_{\alpha}(c_{0,\alpha})$ as $k\to+\infty$. Then it follows that for sufficiently big $k$, 
		\begin{equation}
			\label{e20260127b} 
			{\rm dist}(\Delta_{\alpha_k}(r),P_{\alpha_k}(c_{0,\alpha_k}))>\delta.
		\end{equation}
		Thus by (\ref{e20260127a}) and (\ref{e20260127b}) we have that for all $0<r<1$, there exists a positive integer $K$ such that for all $k\geq K$, we have 
		$$P_{\alpha_k}\left(\partial\Omega_{c_{0,\alpha_k}}^{n}\right)\cap\Delta_{\alpha_k}(r)=\emptyset$$
		and hence
		\begin{equation}
			\label{e20260429a}P_{\alpha_k}\left(\overline{\Omega_{c_{0,\alpha_k}}^{n}}\right)\cap\Delta_{\alpha_k}(r)=\emptyset.
		\end{equation}
		since for all $0<r<1$, $\phi_{\alpha_k}$ converges uniformly to $\phi_{\alpha}$ as $k\to+\infty$ on $\mathbb{B}(0,r)$, we have that for all $r'$ with $0<r'<r<1$, there exists a positive integer $K'$ such that for all $k\geq K'$, we have 
		\begin{equation}
			\label{e8182}\Delta_{\alpha}(r')\subseteq\Delta_{\alpha_k}(r).
		\end{equation}
		
		Since $\alpha\in\mathcal{E}_{\mathfrak{l},\mathfrak{p}}^{(\mathfrak{M})}$, we have that
		there exists a positive integer $j_0$ such that for all $j_1>j_2\geq j_0$ with $a_{j_1}>\mathfrak{M}$ and $a_{j_2}>\mathfrak{M}$, we have $j_1-j_2\geq\mathfrak{p}$ or $j_1-j_2\leq\mathfrak{l}-1$. We fix a positive integer $n_0$ such that $n_0\geq\max\{n+\mathfrak{q}_2+\chi+{\bf h}, j_0\}$ and for all $s\in\{n_0+1,n_0+2,\cdots, n_0+\mathfrak{p}-1\}$, $a_s\leq\mathfrak{M}$.
		Then there exists a sequence $\{p_{k}\}_{k=1}^\infty$ of positive integers such that 
		\begin{itemize}
			\item for all $k\geq1$,
			$p_k\geq\mathfrak{p}$,
			\item for all $n_0+1\leq m\leq n_0+p_1-1$,
			$a_m\leq\mathfrak{M}$,
			\item for all $k\geq1$ and $n_0+\sum_{s=1}^kp_s+\mathfrak{l}\leq m\leq n_0+\sum_{s=1}^{k+1}p_s-1$, $a_m\leq\mathfrak{M}$.
		\end{itemize}
		Thus $a_m>\mathfrak{M}$ may occur only for $n_0+\sum_{s=1}^kp_s\leq m\leq n_0+\sum_{s=1}^kp_s+\mathfrak{l}-1$ with $k\in\mathbb{N}_+$.
		Since $\alpha\in\mathcal{E}_{\mathfrak{l},\mathfrak{p}}^{(\mathfrak{M})}$, $a_m>\mathfrak{M}$ must occur infinite times.
		
		For every $k=n_0+\sum_{s=1}^{\tilde{\bf k}}p_s+\mathfrak{l}\geq\max\{K,K'\}$ with a positive integer $\tilde{\bf k}$, we have that
		pseudo-Siegel disks $\hat{\Delta}_{\alpha_k}^m=\hat{\Delta}_{\alpha_k, \mathfrak{M}}^m$ for all $m\geq-1$.
		Then we have
		\begin{itemize}
			\item for $1\leq j\leq p_1$, $\hat{\Delta}_{\alpha_k}^{n_0-{\bf h}+j}=\hat{\Delta}_{\alpha_k}^{n_0-{\bf h}+p_1}$,
			\item
			for $2\leq t\leq\tilde{\bf k}$ and $\mathfrak{l}\leq j\leq p_t$,  
			$\hat{\Delta}_{\alpha_k}^{n_0-{\bf h}+\sum_{s=1}^{t-1}p_s+j}=\hat{\Delta}_{\alpha_k}^{n_0-{\bf h}+\sum_{s=1}^{t}p_s}$,
			\item  $\hat{\Delta}_{\alpha_k}^{n_0-{\bf h}+\sum_{s=1}^{\tilde{\bf k}}p_{s}+\mathfrak{l}}=\overline{\Delta_{\alpha_k}}$.
		\end{itemize}
		If $\hat{\Delta}_{\alpha_k}^{n_0-{\bf h}+1}\setminus\overline{\Delta_{\alpha_k}}\not=\emptyset$, then for all $z\in\hat{\Delta}_{\alpha_k}^{n_0-{\bf h}+1}\setminus\overline{\Delta_{\alpha_k}}$, there exist ${\bf k}\in\{1,2,\cdots,\tilde{\bf k}\}$ and $j\in\{0,1,2,\cdots,\mathfrak{l}-1\}$ such that
		$\hat{\Delta}_{\alpha_k}^{n_0-{\bf h}+\sum_{s=1}^{\bf k}p_{s}+j}\not=\hat{\Delta}_{\alpha_k}^{n_0-{\bf h}+\sum_{s=1}^{\bf k}p_{s}+j+1}$ and $z\in\mathcal{J}(\beta_I)$ for some $I\in\mathfrak{D}_{n_0-{\bf h}+\sum_{s=1}^{\bf k}p_s+j}$. In this case, 
		by the proof of Theorem \ref{key5} we have
		\begin{equation*}
			{\rm dens}_{S_z}S_z\setminus E_{\sum_{s=t}^{\bf k}p_s+\chi}^{\mathfrak{l}}(\hat{\Delta}_{\alpha_k}^{n_0-{\bf h}+\sum_{s=1}^{\bf k}p_s+1})\leq\frac{1}{2c^3}<\frac{1}{2}.
		\end{equation*}
		For all $z'\in\mathbb{C}$, if there exists $m\geq1$ such that $P_{\alpha_k}^{\comp m}(z')\in\hat{\Delta}_{\alpha_k}^{n_0-{\bf h}+\sum_{s=1}^{\bf k}p_s+\mathfrak{l}}=\overline{\Delta_{\alpha_k}}$ and $\{P_{\alpha_k}^{\comp s}(z')\}_{s=0}^{m-1}\subseteq\Lambda_{\alpha_k}^{n+\mathfrak{q}_2}$, then by Lemma \ref{l20257193}  $P_{\alpha_k}^{\comp (m-1)}(z')\in\Omega_{c_{0,\alpha_k}}^{n}$ and hence by (\ref{e20260429a}) 
		$P_{\alpha_k}^{\comp m}(z')\in\overline{\Delta_{\alpha_k}}\setminus\Delta_{\alpha_k}(r)$.
		By combining (\ref{e8182}), it follows that for all $j\geq0$, $P_{\alpha_k}^{\comp j}(z')\not\in\Delta_{\alpha}(r')$.
		Together with $n_0\geq n+\mathfrak{q}_2+\chi+{\bf h}$, this implies
		$$E_{\sum_{s=t}^{\bf k}p_s+\chi}^{\mathfrak{l}}(\hat{\Delta}_{\alpha_k}^{n_0-{\bf h}+\sum_{s=1}^{\bf k}p_s+1})\setminus\overline{\Delta_{\alpha_k}}\subseteq K_{r'}(P_{\alpha_k}).$$
		Thus
		\begin{equation}
			\label{e8161}{\rm dens}_{S_z}K_{r'}(P_{\alpha_k})\geq{\rm dens}_{S_z} E_{\sum_{s=t}^{\bf k}p_s+\chi}^{\mathfrak{l}}(\hat{\Delta}_{\alpha_k}^{n_0-{\bf h}+\sum_{s=1}^{\bf k}p_s+1})>\frac{1}{2}.
		\end{equation}
		
		Now we consider $\hat{\Delta}_{\alpha_k}^{n_0-{\bf h}+1}$.
		For all positive integer $t\geq2$, we let $I_{l,t}$ and  $I_{r,t}$ be the two components of 
		$$\overline{I_{[x_{q_{n_0-{\bf h}+1}-1},x_{q_{n_0-{\bf h}+2}-1}]}\setminus I_{[x_{q_{n_0-{\bf h}+1+t}-1},x_{q_{n_0-{\bf h}+2+t}-1}]}}$$
		with
		$I_{l,t}<I_{[x_{q_{n_0-{\bf h}+1+t}-1},x_{q_{n_0-{\bf h}+2+t}-1}]}<I_{r,t}$.
		We denote by 
		$\hat{I}_{l,t}$,   $\hat{I}_{r,t}$ and $\hat{I}_{[x_{q_{n_0-{\bf h}+1+t}-1},x_{q_{n_0-{\bf h}+2+t}-1}]}$ projections of
		$I_{l,t}$, $I_{r,t}$ and $I_{[x_{q_{n_0-{\bf h}+1+t}-1},x_{q_{n_0-{\bf h}+2+t}-1}]}$ onto $\partial\hat{\Delta}_{\alpha_k}^{n_0-{\bf h}+1}$, respectively.
		We define $\mathcal{G}_t$ as the interior of the union $\mathcal{G}^-_{\hat{\Delta}_{\alpha_k}^{n_0-{\bf h}+1}}(\hat{I}_{l,t},\hat{I}_{r,t})\cup\hat{\mathcal{G}}^+_{\overline{\Delta_{\alpha_k}}}(\hat{I}_{l,t},\hat{I}_{r,t})$.
		Let $\gamma_t^-$ be the geodesic connecting $x_{q_{n_0-{\bf h}+1+t}-1}$ and $x_{q_{n_0-{\bf h}+2+t}-1}$ with respect to the hyperbolic metric on ${\rm int}(\hat{\Delta}_{\alpha_k}^{n_0-{\bf h}+1})$. Let $D_t^-$ ($\subseteq\mathbb{C}$) be the closed region bounded by $\gamma_t^-$ and $\hat{I}_{[x_{q_{n_0-{\bf h}+1+t}-1},x_{q_{n_0-{\bf h}+2+t}-1}]}$. By Lemma \ref{l84a} we have that $D_t^-$ is a closed $\mathfrak{K}'$-quasidisk, where $\mathfrak{K}'$ is a universal constant.
		Observe that $\mathcal{G}_t$ separates $\{\infty,c_{0,\alpha_k}\}$ and $\{P_{\alpha_k}(c_{0,\alpha_k}),D_t^-\}$.
		By Corollary \ref{c831}, we have ${\rm mod}(\mathcal{G}_t)\to+\infty$ as $t\to+\infty$, uniform on $\alpha_k$. Thus there exists a positive integer $t_0$ (independent of $\alpha_k$)
		such that for all $t\geq t_0$, we have
		$D_t^-\subseteq\mathbb{B}(P_{\alpha_k}(c_{0,\alpha_k}),\frac{\delta}{2})$ and hence by (\ref{e20260127b}) $D_t^-\cap\Delta_{\alpha_k}(r)=\emptyset$.
		Since $D_{t_0}^-$ is a closed $\mathfrak{K}'$-quasidisk containing $x_{q_{n_0-{\bf h}+1+t_0}-1}$ and $x_{q_{n_0-{\bf h}+2+t_0}-1}$,
		${\rm area}(D_{t_0}^-)$ has a positive lower bound, independent of $k$.
		Similar to the proof of Theorem \ref{T2}, applying (\ref{e8182}), (\ref{e8161}) and $D_{t_0}^-\cap\Delta_{\alpha_k}(r)=\emptyset$ to $D_{t_0}^-$, we can obtain
		$$\limsup\limits_{k\to+\infty}{\rm area}(K_{r'}(P_{\alpha_k}))\succeq{\rm area}(D_{t_0}^-).$$
		The conclusion follows from the arbitrariness of $r$ and $r'$.
		
	\end{proof}

\end{document}